\documentclass[11pt,leqno]{amsart}
\usepackage{graphicx}
\usepackage{stmaryrd, amssymb,color,bm}

\newcommand{\abs}[1]{\vert#1\vert}

\newcommand{\jump}[1]{\llbracket#1\rrbracket}
\newcommand{\mybf}[1]{\boldsymbol{#1}}
\newcommand{\dsp}{\displaystyle}

\newcommand{\bu}{{\mathbf u}}

\newcommand{\ovV}{\overline{V}}
\newcommand{\ovv}{\overline{v}}
\newcommand{\uU}{\underline{U}}
\newcommand{\unu}{\underline{\nu}}

\newcommand{\R}{{\mathbb R}}
\newcommand{\N}{{\mathbb N}}
\newcommand{\dt}{\partial_t}
\newcommand{\dz}{\partial_z}
\newcommand{\dx}{\partial_x}

\newcommand{\cS}{{\mathcal S}}

\newcommand{\eps}{\varepsilon}

\newcommand{\ux}{\underline{x}}
\newcommand{\uu}{\underline{u}}
\newcommand{\ubu}{\underline{{\boldsymbol{u}}}}
\newcommand{\cA}{\mathcal{A}}

\newcommand{\cU}{{\mathcal U}}
\newcommand{\bcU}{\boldsymbol{\mathcal U}}
\newcommand{\WW}{{\mathbb W}}
\newcommand{\vu}{u_{\vert_{x=0}}}
\newcommand{\vuin}{{u^{\rm in}}_{\vert_{x=0}}}
\newcommand{\vf}{f_{\vert_{x=0}}}

\newcommand{\bpi}{{\boldsymbol \pi}}
\newcommand{\bvarphi}{{\boldsymbol \varphi}}

\makeatletter
\newcommand{\opnorm}{\@ifstar\@opnorms\@opnorm}
\newcommand{\@opnorms}[1]{%
  \left|\mkern-1.5mu\left|\mkern-1.5mu\left|
   #1
  \right|\mkern-1.5mu\right|\mkern-1.5mu\right|
}
\newcommand{\@opnorm}[2][]{%
  \mathopen{#1|\mkern-1.5mu#1|\mkern-1.5mu#1|}
  #2
  \mathclose{#1|\mkern-1.5mu#1|\mkern-1.5mu#1|}
}

\newtheorem{theorem}{Theorem}
\newtheorem{proposition}{Proposition}
\newtheorem{lemma}{Lemma}
\newtheorem{remark}{Remark}
\newtheorem{example}{Example}
\newtheorem{definition}{Definition}
\newtheorem{notation}{Notation}
\newtheorem{assumption}{Assumption}

\begin{document}

\title[Hyperbolic free boundary problems and applications]{
Hyperbolic free boundary problems and applications \\ to wave-structure interactions}
\author{Tatsuo Iguchi and David Lannes}
\address{Tatsuo Iguchi, Department of Mathematics, Faculty of Science and Technology, Keio University, 
3-14-1 Hiyoshi, Kohoku-ku, Yokohama 223-8522, Japan}
\address{David Lannes, Institut de Math\'ematiques de Bordeaux, Universit\'e de Bordeaux et CNRS UMR 5251, 
351 Cours de la Lib\'eration, 33405 Talence Cedex, France}
\maketitle

\begin{abstract}
Motivated by a new kind of initial boundary value problem (IBVP) with a free boundary arising 
in wave-structure interaction, we propose here a general approach to one-dimensional IBVP 
as well as transmission problems. 
For general strictly hyperbolic $2\times 2$ quasilinear hyperbolic systems, 
we derive new sharp linear estimates with refined dependence on the source term and control 
on the traces of the solution at the boundary. 
These new estimates are used to obtain sharp results for quasilinear IBVP and transmission problems, 
and for fixed, moving, and free boundaries. In the latter case, two kinds of evolution equations are considered. 
The first one is of ``kinematic type'' in the sense that the velocity of the interface has the same regularity
 as the trace of the solution. 
Several applications that fall into this category are considered: 
the interaction of waves with a lateral piston, and a new version of the well-known stability of shocks 
(classical and undercompressive) that improves the results of the general theory 
by taking advantage of the specificities of the one-dimensional case. 
We also consider ``fully nonlinear'' evolution equations characterized by the fact that the velocity 
of the interface is one derivative more singular than the trace of the solution. 
This configuration is the most challenging; it is motivated by a free boundary problem arising in wave-structure 
interaction, namely, the evolution of the contact line between a floating object and the water. 
This problem is solved as an application of the general theory developed here.
\end{abstract}

\section{Introduction}
\subsection{General setting}

This article is devoted to a general analysis of free boundary and free transmission hyperbolic problems 
in the one dimensional case. 
It is mainly motivated by a new kind of free boundary problem arising in the study of wave-structure 
interactions and for which the evolution of the free boundary is governed by a singular equation.

\medbreak
In order to explain the singular  structure of this problem, let us recall some results on hyperbolic 
initial boundary value problems (a good reference on this subject is the book \cite{benzoniserre2007}). 
Let us for instance consider a general quasilinear equation of the form 
\[
\dt U+ A(U)\dx U=0
\]
for $t>0$ and $x\in \R$. It is well known that if the system is Friedrichs symmetrizable, i.e., 
if there exists a positive definite matrix $S(u)$ such that $S(u)A(u)$ is symmetric, 
then the associated initial value problem is well-posed in $C([0,T];H^s(\R))$ if $s>d+1/2$ 
(with $d=1$ is the space dimension). 
The proof is based on the study of the linearized system and an iterative scheme. 
If we consider the same equation on $\R_+$, and impose a boundary condition on $U$ at $x=0$, 
then the corresponding initial boundary value problem might not be well-posed, 
even if the system is Friedrichs symmetrizable. 
Well-posedness is however ensured if there exists a Kreiss symmetrizer which, as the Friedrichs symmetrizer, 
transforms the system into a symmetric system, but with the additional property that the boundary condition 
for this symmetric system is striclty dissipative (roughly speaking, this means that the trace of the solution 
at the boundary is controled by the natural energy estimate). 
The construction of such a Kreiss symmetrizer in extremely delicate and is usually done under the so-called 
uniform Lopatinski\u{\i} condition which can formally be derived as a stability condition 
for the normal mode solutions of the linearized equations with frozen coefficients. 
Under such a condition (and additional compatibility conditions between the boundary and initial data), 
a unique solution can again be constructed (though with many more technical issues) via estimates 
on the linearized system and an iterative scheme. 
The typical result for quasilinear initial boundary value problems satisfying the aforementioned condition, 
as announced in \cite{RauchMassey} and proved in \cite{mokrane1987}, is that the equations are well-posed 
but with higher regularity requirements, and more importantly, with a loss of half a derivative 
with respect to the initial and boundary data.

\medbreak
In some situation, the boundary of the domain on which the equations are cast depends on time. 
In dimension $d=1$ for instance, this means that instead of working on $\R_+$, one works on $(\ux(t),+\infty)$, 
where the function $\ux$ is either a known function (boundary in forced motion) or an unknown function 
determined by an equation involving the solution $U$ of the hyperbolic system, typically, 
\[
\dot \ux(t)=\chi(U_{\vert_{x=\ux(t)}})
\]
for some smooth function $\chi$ (we shall say that this kind of boundary evolution of ``kinematic type'' 
because, as for kinematic boundary conditions, the regularity of $\dot\ux$ is the same as the regularity 
of the solution at the boundary). 
Such problems are called free boundary hyperbolic problems.

\medbreak
It is noteworthy that, up to a doubling of the dimension of the system of equations under consideration, 
the considerations above can be extended to transmission problems, 
where two possibly different hyperbolic systems are considered on the two different sides of an interface, 
and where the boundary condition is replaced by a condition involving the traces of the solution on both sides. 
One of the most famous transmission problems with a free boundary is the stability of shocks. 
The problem consists in finding  solutions to a quasilinear hyperbolic system that are smooth on both sides 
of a moving interface and whose traces on the interface satisfy the Rankine--Hugoniot condition. 
In dimension $d=1$, this latter condition provides an evolution equation for the interface of the same form as above.

Showing the well-posedness of free boundary hyperbolic problems requires new ingredients and in particular, 
\begin{itemize}
\item
A diffeomorphism must be used to transform the problem into a boundary value problem with a fixed boundary. 
\item
A change of unknown must be introduced to study the linearized equation. 
Indeed, with the standard linearization procedure, a derivative loss occurs due to the dependence 
of the transformed problem on the diffeomorphism. This loss is removed by working with 
so-called Alinhac's good unknown. 
\end{itemize}
The proof of the stability of multidimensional shocks is a celebrated achievement of Majda 
\cite{majda1,majda2,majda3}, with improvements in \cite{metivier2001}. 
Since the proof relies on the theory of initial boundary value problems, the same loss of half a derivative 
with respect to the initial and boundary data is observed.

\medbreak
The free boundary problem that motivates this work is the evolution of the contact line between 
a floating object and the water, in the situation where the motion of the waves is assumed to be 
governed by the (hyperbolic) nonlinear shallow water equations, and in horizontal dimension $d=1$. 
In a simplified version, this problem can be reduced to a free boundary hyperbolic problem, 
but with a more singular evolution equation for the free boundary, which is of the form 
\[
U(t,\ux(t))=U_{\rm i}(t,\ux(t)),
\]
where $U_{\rm i}$ is a known function (for the contact line problem, this condition expresses the fact 
that the surface elevation and the horizontal flux of the water are continuous across the contact point). 
Time differentiating this condition yields an evolution equation for $\ux$ of the form 
\[
\dot \ux(t) = \chi\bigl( (\dt U)_{\vert_{x=\ux(t)}}, (\dx U)_{\vert_{x=\ux(t)}}, 
 (\dt U_{\rm i})_{\vert_{x=\ux(t)}}, (\dx U_{\rm i})_{\vert_{x=\ux(t)}} \bigr). 
\]
The standard procedure for free boundary hyperbolic problems descrived above does not work with such a 
boundary equation, because there is obviously a loss of one derivative in the estimates: 
the boundary condition is fully nonlinear. 
In order to handle this new difficulty without using a Nash--Moser type scheme, 
we propose to work with a second order linearization and introduce a second order Alinhac's good unknown 
in order to cancel out the terms responsible for the derivative losses.

\medbreak
Proving the well-posedness of this fully nonlinear free boundary hyperbolic problem also requires sharp and 
new estimates for one-dimensional hyperbolic initial boundary values problems that are of independent interest. 
One-dimensional hyperbolic boundary value problems are generally dealt with using the method of 
characteristics \cite{li1985boundary}. 
In the Sobolev setting, there is no specific work dealing with the one-dimensional setting, and the general 
multidimensional results are used, with their drawbacks: 
high regularity requirements and derivative loss with respect to the boundary and initial data. 
These drawbacks however can  easily be bypassed by taking advantage of the specificities of the one-dimensional case, 
and in particular of the explicit construction of the Kreiss symmetrizers. 
For this reason, we propose in this article a general study of initial boundary value problems 
(as well as transmission problems) for fixed, moving, and free boundaries. 
This study is based on the new sharp estimates developed to solve the fully nonlinear free boundary problem 
mentioned above and fully exploits the specificities of the one-dimensional case. 
In particular, the high regularity requirements and the derivative loss of the general theory are removed. 
This is for instance of interest to solve the problem of transparent conditions for hyperbolic systems. 
We use this general approach to solve several problems coming from wave-structure interactions, 
as well as other problems such as conservation laws with a discontinuous flux and the stability 
of one-dimensional standards and nonstandards shocks. 
Another advantage of our approach is that it is much more elementary than the general results, 
and does not require refined paradifferential calculus for instance.

\subsection{Organization of the paper}
Section \ref{sect2} is devoted to the study of several kinds of free boundary problems for 
$2\times2$ quasilinear (strictly) hyperbolic systems. 
The case of non homogeneous linear initial boundary value problems with variable coefficients and 
a fix boundary is considered first in \S \ref{sect2VC}. 
The main focus is the derivation of a sharp estimate, given in Theorem \ref{theoIBVP1}, which requires only 
a weak control in time of the source term (weaker than $L^1(0,T)$, which is itself weaker than the standard 
$L^2(0,T)$ that can be found in the literature \cite{benzoniserre2007}) and which provides a better control 
of the trace of the solution at the boundary. 
We first assume the existence of a Kreiss symmertrizer and derive a priori weighted $L^2$-estimates in 
\S \ref{sectapL2}, and higher order estimates in \S \ref{sect2HO}. 
In order to complete the proof of Theorem \ref{theoIBVP1}, the main step, 
performed in \S \ref{secttheoIBVP1} is the explicit construction of a Kreiss symmetrizer under an explicit 
Lopatinski\u{\i} condition. 
In \S \ref{sectapplQL}, these linear estimates are used to prove the well-posedness of quasilinear systems; 
Theorem \ref{theoIBVP2} provides a sharp result for such systems, which takes advantage of the specifities 
of the one-dimensional case and improves the results provided by the general (multi-dimensional) theorems. 
It can for instance be used to improve the existing results concerning transparent boundary conditions 
for the nonlinear shallow water equations. 
In \S \ref{sectVCm} we go back to the analysis of linear initial boundary value problems, 
but this time on a moving domain, i.e., in the case where the domain on which the equations are cast is 
$(\ux(t),\infty)$, with $\ux$ assumed here to be a known function. Using a diffeomorphism that maps 
$\R_+$ to $(\ux(t),\infty)$ for all times, this problem is transformed into an initial boundary value problem 
with fix boundary, but whose coefficients depend on the diffeomorphism. 
One could apply Theorem \ref{theoIBVP1} to this problem, but would lose an unecessary derivative 
in the dependence on the diffeomorphism. 
This loss is avoided in Theorem \ref{theoIBVP3} by applying Theorem \ref{theoIBVP1} to the system satisfied 
by Alinhac's good unknown; 
in order to get a sharp result in terms of regularity requirements on the initial data, the sharp dependence 
on the source terms proved in Theorem \ref{theoIBVP1} is necessary at this point. 
These linear estimates are then used in \S \ref{sectFB1} to study quasilinear initial boundary value problems 
with free boundary, i.e., where the function $\ux(t)$ is no longer assumed to be known, 
but satisfies an evolution equation. 
The case of an evolution equation of ``kinematic'' type is considered first, so that a diffeomorphism of 
``Lagrangian'' type can be used and a solution constructed by an iterative scheme based on the linear estimates 
of Theorem \ref{theoIBVP3}. 
The more complicated case of fully nonlinear boundary conditions of the type mentioned above is addressed 
in \S \ref{sectVCm2}. 
To handle this problem, another kind of diffeomorphism must be used and a generalization of Alinhac's good 
unknown to the second order must be introduced to remove the loss of derivative induced by the fully nonlinear 
boundary condition. 
A more general type of fully nonlinear condition is also considered in \S \ref{sectext}, 
where a coupling with a system of ODEs is allowed.

\medbreak
As an illustration of the fact that the theory developed above for $2\times 2$ initial boundary value problems 
can be generalized to systems involving a higher number of equations, we propose in Section \ref{secttransmission} 
a rather detailed study of transmission problems. 
More precisely, we consider two $2\times 2$ hyperbolic systems cast on both sides of an interface, 
and coupled through transmission conditions at the interface. 
Such transmission problems can be transformed into $4\times 4$ initial boundary value problems to which 
the above theory can be adapted. 
Linear transmission problems are first considered in \S \ref{sectVCtransm}, the main step being the construction 
of a Kreiss symmetrizer whose nature depends on the number of characteristics pointing towards the interface; 
the nonlinear case is then considered in \S \ref{sectappltransmQL}. 
Moving interfaces are then treated in \S \ref{secttransmmov} for linear systems and an application to 
free boundary transmission problems with ``kinematic'' boundary condition is given in \S \ref{secttransmkin}.

\medbreak
A first application of the general theory described above to wave-structure interactions is given in 
Section \ref{sectlatpis}. 
The problem consists in studying the interaction of waves in shallow water with a lateral piston. 
The nonlinear shallow water equations are a quasilinear hyperbolic problem that falls into the class studied above. 
The domain is a half-line delimited by a piston which can move under the pressure force exerted by the wave. 
Its motion (and therefore the position of the boundary) is given by the resolution of a second order ODE in time 
(Newton's equation) coupled with the nonlinear shallow water equations. 
The key step is to show that this evolution equation is essentially of ``kinematic'' type so that 
the results of \S \ref{sectFB1} can be applied.

\medbreak
In Section \ref{sectfloat} we present the problem that motivated this work, namely, the description of the 
evolution of the contact line between a floating body and the surface of the water in the shallow water regime. 
We recall in \S \ref{sectpresfloat} the derivation of the equations proposed in \cite{Lannes2017} to describe 
this problem and investigate first, in \S \ref{sectfixfloat}, the case of a fixed floating body. 
We show that the problem can be reduced to an initial boundary value problem with free boundary governed 
by a fully nonlinear equation, which allows us to use the results of \S \ref{sectVCm2}. 
The extension to the case of a floating object with a prescribed motion is then presented in 
\S \ref{sectprescfloat} and the more complicated case of a freely floating object is studied in 
\S \ref{sectfreefloat}. 
For this latter case, the evolution of the contact point is more complicated because it is coupled with 
the three dimensional Newton equation for the solid (on the vertical and horizontal coordinates of 
the center of mass and on the rotation angle). 
Technical computations are postponed to Appendix \ref{appreform}.

\medbreak
We finally present in Section \ref{sectsev} several applications of our results on transmission problems. 
The first one, considered in \S \ref{sectdiscf} is a general $2\times 2$ system of conservation laws 
with a discontinuous flux (a typical example is provided by the nonlinear shallow water equations 
over a discontinuous topography). 
We then investigate in \S \ref{sectshocks} the stability of one-dimensional shocks 
(both classical and undercompressive); using our sharp one-dimensional results, 
we are able to improve the results one would obtain by considering the one-dimensional case 
in the general multi-dimensional theory of \cite{majda1,majda2,majda3,metivier2001} for classical shocks 
and \cite{coulombel2003} for undercompressive shocks.

\subsection{General notations}\label{sectnot}
- We write $\Omega_T = (0,T)\times \R_+$.\\
- The notation $\partial$ stands for either $\dx$ or $\dt$, so that $\partial f\in L^\infty(\Omega_T)$ for instance, 
means $\dx f\in L^\infty(\Omega_T)$ and $\dt f\in L^\infty(\Omega_T)$.\\
- We denote by $\cdot$ the $\R^2$ scalar product and by $(\cdot,\cdot)_{L^2}$ the $L^2(\R_+)$ scalar product. \\
- If $A$ is a vector or matrix, and $X$ a functional space, we simply write $A\in X$ to express the fact that 
all the elements of $A$ belong to $X$. \\
- In order to define smooth solutions of hyperbolic systems in $\Omega_T=(0,T)\times \R_+$, 
it is convenient to introduce the space $\WW^m(T)$ as 
\[
\WW^m(T)=\bigcap_{j=0}^l C^j([0,T];H^{m-j}(\R_+)),
\]
with associated norm 
\[
\| u \|_{\WW^m(T)} = \sup_{t\in[0,T]} \opnorm{ u(t) }_{m}
\quad \mbox{ with }\quad 
\opnorm{ u(t) }_m = \sum_{j=0}^m \| \dt^j u(t) \|_{H^m(\R_+)}.
\]
We have in particular $H^{m+1}(\Omega_T)\subset \WW^m(T) \subset H^m(\Omega_T)$. \\
- In order to control the boundary regularity of the solution, it is convenient to use the norm 
\[
|\vu|_{m,t} = \biggl( \sum_{j=0}^m|(\dx^ju)_{\vert_{x=0}}|_{H^{m-j}(0,t)}^2 \biggr)^\frac12
 = \biggl( \sum_{|\alpha| \leq m}|(\partial^\alpha u)_{\vert_{x=0}}|_{L^2(0,t)}^2 \biggr)^\frac12.
\]
- We also use weighted norms with an exponential function $e^{-\gamma t}$ for $\gamma>0$ defined by 
\begin{align*}
& |g|_{L_\gamma^2(0,t)} = \biggl( \int_0^t e^{-2\gamma t'}|g(t')|^2{\rm d}t' \biggr)^\frac12, \qquad
  |g|_{H_\gamma^m(0,t)} = \biggl( \sum_{j=0}^m |\dt^jg|_{L_\gamma^2(0,t)}^2 \biggr)^\frac12, \\
& \opnorm{ u(t) }_{m,\gamma} = e^{-\gamma t}\opnorm{ u(t) }_m, \qquad
 \| u \|_{\WW_\gamma^m(T)} = \sup_{t\in[0,T]} \opnorm{ u(t) }_{m,\gamma}, \\
& |\vu|_{m,\gamma,t} = \biggl( \sum_{j=0}^m|(\dx^ju)_{\vert_{x=0}}|_{H_\gamma^{m-j}(0,t)}^2 \biggr)^\frac12. \\
\end{align*}

\medskip
\noindent
{\bf Acknowledgement} \ 
This work was carried out when T. I. was visiting Universit\'e de Bordeaux 
on his sabbatical leave during the 2017 academic year. 
He is very grateful to the members of Institut de Math\'ematiques de Bordeaux 
for their kind hospitality and for fruitful discussions. 
T. I. was partially supported by JSPS KAKENHI Grant Number JP17K18742 and JP17H02856. 
D. L. is partially supported by the Del Duca Fondation, the Conseil R\'egional d'Aquitaine 
and the ANR-17-CE40-0025 NABUCO.

\section{Hyperbolic initial boundary value problems with a free boundary}\label{sect2}
This section is devoted to the analysis of a general class of initial boundary value problems, 
with a boundary that can be either fixed, in prescribed motion, or freely moving. 
We refer to \S \ref{sectnot} for the notations used, and in particular for the definition of the functional spaces.

\subsection{Variable coefficients linear $2\times 2$ initial boundary value problems}\label{sect2VC}
The aim of this section is to provide an existence theorem with sharp estimates for 
a general linear initial boundary value problem with variable coefficients of the following form, 
\begin{equation}\label{systVC}
\begin{cases}
 \dt u + A(t,x)\dx u + B(t,x) u = f(t,x) &\mbox{in}\quad \Omega_T, \\
 u_{\vert_{t=0}} = u^{\rm in}(x) & \mbox{on}\quad \R_+, \\
 \nu(t) \cdot u_{\vert_{x=0}} = g(t)& \mbox{on}\quad (0,T), 
\end{cases}
\end{equation}
where $u$, $u^{\rm in}$, $f$, and $\nu$ are $\R^2$-valued functions and $g$ is real-valued function, 
while $A$ and $B$ take their values in the space of $2\times2$ real-valued matrices. 
We also make the following assumption on the hyperbolicity of the system and on the boundary condition. 

\begin{assumption}\label{asshyp}
There exists $c_0>0$ such that the following assertions hold. 
\begin{enumerate}
\item[{\bf i.}]
$A\in W^{1,\infty}(\Omega_T), \; B\in L^\infty(\Omega_T), \; \nu \in C([0,T])$.
\item[{\bf ii.}]
For any $(t,x) \in \Omega_T$, the matrix $A(t,x)$ has eigenvalues $\lambda_+(t,x)$ and $-\lambda_-(t,x)$ satisfying 
\[
\lambda_\pm(t,x)\geq c_0. 
\]
\item[{\bf iii.}]
(The uniform Kreiss--Lopatinski\u{\i} condition.)
Denoting by ${\bf e}_+(t,x)$ a unit eigenvector associated to the eigenvalue $\lambda_+(t,x)$ of $A(t,x)$, 
for any $t\in[0,T]$ we have 
\[
|\nu(t,0) \cdot \mathbf{e}_{+}(t,0)| \geq c_0.
\]
\end{enumerate}
\end{assumption}

\begin{example}\label{ex1}
A typical example of application is to consider the linearized shallow water equations 
with a boundary condition on the horizontal water flux $q$. 
This system has the form 
\[
\begin{cases}
\dt \zeta + \dx q = 0, \\
\dt q +2 \frac{\underline{q}}{\underline{h}}\dx q
 + \Bigl( \mathtt{g}\underline{h}- \frac{\underline{q}^2}{\underline{h}^2} \Bigr)\dx \zeta = 0
\end{cases}
\]
with initial and boundary conditions 
\[
(\zeta,q)_{\vert_{t=0}} = (\zeta^{\rm in},q^{\rm in})
 \quad\mbox{ and }\quad q_{\vert_{x=0}} = g,
\]
where $\mathtt{g}$ is the gravitational constant. 
This problem is of the form \eqref{systVC} with $u = (\zeta,q)^{\rm T}$, $B=0$, $f=0$, $\nu=(0,1)^{\rm T}$, and 
\begin{equation}\label{eqASW}
A(t,x)=A(\underline{u})=
\begin{pmatrix}
 0 & 1 \\ 
 \mathtt{g}\underline{h}-\frac{\underline{q}^2}{\underline{h}^2} & 2\frac{\underline{q}}{\underline{h}}
\end{pmatrix}.
\end{equation}
The eigenvalues $\pm\lambda_{\pm}$ and the corresponding unit eigenvectors $\mathbf{e}_{\pm}$ of $A$ 
are given by $\lambda_{\pm} = \sqrt{\mathtt{g}\underline{h}}\pm \frac{\underline{q}}{\underline{h}}$ 
and ${\bf e}_{\pm} = \frac{1}{\sqrt{1+\lambda_{\pm}^2}}(1,\pm\lambda_{\pm})^{\rm T}$, so that 
Assumption \ref{asshyp} is satisfied provided that $\underline{h},\underline{q} \in W^{1,\infty}(\Omega_T)$, and 
\[
\underline{h}(t,x) \geq c_0, \qquad 
\sqrt{\mathtt{g}\underline{h}(t,x)} \pm \frac{\underline{q}(t,x)}{\underline{h}(t,x)}\geq c_0
\]
with some positive constant $c_0$ independent of $(t,x)\in\Omega_T$. 
\end{example}

\begin{notation}\label{dualnorm}
In order to define an appropriate norm to the source term $f(t,x)$ in \eqref{systVC}, 
it is convenient to use the following norm to functions of $t$ 
\[
S_{\gamma,T}^*(f) = \sup_{\varphi} \biggl\{ \biggl| \int_0^T e^{-2\gamma t}f(t)\varphi(t){\rm d}t \biggr|
 \,;\, \sup_{t\in[0,T]}e^{-\gamma t}|\varphi(t)|
  + \biggl( \gamma \int_0^T e^{-2\gamma t}|\varphi(t)|^2{\rm d}t\biggr)^\frac12 \leq 1 \biggr\},
\]
which is the norm of the dual space to $L_\gamma^{\infty}(0,T) \cap L_\gamma^2(0,T)$ equipped with the norm 
\[
\sup_{t\in[0,T]}e^{-\gamma t}|\varphi(t)|
  + \biggl( \gamma \int_0^T e^{-2\gamma t}|\varphi(t)|^2{\rm d}t\biggr)^\frac12
\]
associated to the inner product of $L_\gamma^2(0,T)$. 
\end{notation}

It is easy to check that $S_{\gamma,t}^*(f)$ is a nondecreasing function of $t\geq0$ for each fixed $f$ 
and that $S_{\gamma,t}^*(f)$ is monotone with respect to $f$ in the sense that if $0\leq f_1(t) \leq f_2(t)$ 
for $t\in[0,T]$, then we have $S_{\gamma,t}^*(f_1) \leq S_{\gamma,t}^*(f_2)$ for $t\in[0,T]$. 
Moreover, we have 
\[
S_{\gamma,T}^*(f) \leq \int_0^T e^{-\gamma t}|f(t)| {\rm d}t \quad\mbox{and}\quad
S_{\gamma,T}^*(f) \leq \biggl( \frac{1}{\gamma} \int_0^T e^{-2\gamma t}|f(t)|^2{\rm d}t\biggr)^\frac12.
\]

\begin{remark}\label{rememb}
The first of these two inequalities implies an $L^2$-type control through the Cauchy--Schwarz inequality, 
\[
\int_0^T e^{-\gamma t}|f(t)| {\rm d}t \leq \sqrt{T}\biggl(  \int_0^T e^{-2\gamma t}|f(t)|^2{\rm d}t\biggr)^\frac12,
\]
but with a right-hand side involving a factor ${\sqrt{T}}$. This is not the case for the $L^2$-type control 
(with respect to time) deduced from $S_{\gamma,T}^*(f) $ and this improvement allows to derive energy estimates 
with an exponential growth in Theorems \ref{theoIBVP1}, \ref{theoIBVP3}, and \ref{theoIBVP1transm} for instance. 
\end{remark}

The main result of this section is the following theorem (see \S \ref{sectnot} for the definition of .$ \WW^{m-1}(T)$ and of the various weighted norms used in the statement).

\begin{theorem}\label{theoIBVP1}
Let $m\geq1$ be an integer, $T>0$, and assume that Assumption \ref{asshyp} is satisfied for some $c_0>0$. 
Assume moreover that there are constants $0<K_0\leq K$ such that 
\[
\begin{cases}
\frac{1}{c_0}, \Vert A \Vert_{L^\infty(\Omega_T)}, \abs{\nu}_{L^\infty(0,T)} \leq K_0, \\
\Vert A\Vert_{W^{1,\infty}(\Omega_T)}, \Vert B \Vert_{L^\infty(\Omega_T)}, 
 \Vert(\partial A,\partial B)\Vert_{ \WW^{m-1}(T)}, \abs{\nu}_{W^{m,\infty}(0,T)} \leq K. 
\end{cases}
\]
Then, for any data $u^{\rm in} \in H^m(\R_+)$, $g\in H^m(0,T)$, and $f\in H^m(\Omega_T)$ satisfying 
the compatibility conditions up to order $m-1$ in the sense of Definition \ref{defcompVC} below, 
there exists a unique solution $u \in \WW^m(T)$ to the initial boundary value problem \eqref{systVC}. 
Moreover, the following estimate holds for any $t\in[0,T]$ and any $\gamma \geq C(K)$:
\begin{align*}
& \opnorm{ u(t) }_{m,\gamma} + \biggl( \gamma\int_0^t\opnorm{ u(t') }_{m,\gamma}^2{\rm d}t' \biggr)^\frac12
 + \vert \vu \vert_{m,\gamma,t} \\
&\leq C(K_0)\bigl( \opnorm{ u(0) }_{m} + \vert g \vert_{H_\gamma^m(0,t)} 
 + |\vf|_{m-1,\gamma,t} + S_{\gamma,t}^*(\opnorm{ \dt f(\cdot) }_{m-1}) \bigr). 
\end{align*}
Particularly, we have 
\begin{align*}
& \opnorm{ u(t) }_{m} + \vert \vu \vert_{m,t} \\
&\leq C(K_0)e^{C(K)t} \biggl( \opnorm{ u(0) }_{m} + \vert g \vert_{H^m(0,t)} 
 + |\vf|_{m-1,t} + \int_0^t \opnorm{ \dt f(t') }_{m-1}{\rm d}t' \biggr). 
\end{align*}
\end{theorem}

\begin{remark}
The estimates provided by the theorem are a refinement of classical estimates that can be found 
in the extensive literature on initial boundary value problems 
(see for instance \cite{schochet1986, metivier2001, benzoniserre2007, metivier2012}). 

{\bf i.} \ 
With the exception of \cite{metivier2001}, these references provide a control of the source term 
in $L^2$-norm with respect to time; 
it turns out that such a control is not enough to handle ``fully nonlinear'' boundary conditions 
as in \S \ref{sectVCm2} below. 
In \cite{metivier2001}, a more precise upper bound involving only the $L^1$-norm in time of $f$ is provided, 
but only for constant coefficient symmetric systems. 
The above theorem extends this result to variable coefficients systems and also refines it 
since it provides a control in terms of $S^*_{\gamma,t}$ instead of $L^1$. 
This latter refinement is important for instance to get low regularity results 
-- $\WW^2(T)$ instead of $\WW^3(T)$ -- in Theorems \ref{theoIBVP2}, \ref{theoIBVP4}, \ref{theoIBVP5}, 
\ref{theoIBVP6}, and \ref{theoIBVP2transm}. 

{\bf ii.} \ 
The estimates of the theorem provide a control of $\abs{\vu}_{m,t}$ and not only of $\abs{\vu}_{H^m(0,t)}$. 

{\bf iii.} \ 
In addition to the classical $L^\infty(0,T)$ upper bound on $t\mapsto \opnorm{ u(t) }_m$, 
our estimates provide a control of its $L^1(0,T)$-norm which is uniform with respect to $t$ 
(see the comments in Remark \ref{rememb} above) which is typical of weghted estimates 
\cite{metivier2012,benzoniserre2007}. 
This term is essential in the derivation of the higher-order estimates 
(see the proof of Proposition \ref{propVC2}). 
\end{remark}

\begin{remark}
The assumption $\vert \nu \vert_{W^{m,\infty}(0,T)}\leq K$ can be weakened into 
$\vert \nu \vert_{W^{1,\infty} \cap W^{m-1,\infty}(0,T)} \leq K$ and $\vert \dt^m \nu \vert_{L^2(0,T)}\leq K$ 
(this is a particular case of Theorem \ref{theoIBVP3} below with $\ux\equiv 0$).
\end{remark}

\subsubsection{Compatibility conditions}
From the interior equations, denoting $u_k=\dt^k u$, we have 
\[
u_1=-A\dx u -Bu +f.
\]
More generally, differentiating the equation $k$-times with respect to $t$, we have 
a recursion relation 
\[
u_{k+1} = -\sum_{j=0}^k \begin{pmatrix} k \\ j \end{pmatrix}\{
 (\dt^{k-j}A)\dx u_j + (\dt^{k-j}B)u_j \} + \dt^kf.
\]
For a smooth solution $u$, $u_k^{\rm in} = {u_k}_{\vert_{t=0}}$ is therefore given inductively by 
$u_0^{\rm in} = u^{\rm in}$ and 
\begin{equation}\label{defu0k}
u_{k+1}^{\rm in} =-\sum_{j=0}^k \begin{pmatrix} k \\ j \end{pmatrix}\{
 (\dt^{k-j}A)_{\vert_{t=0}}\dx u_j^{\rm in} + (\dt^{k-j}B)_{\vert_{t=0}}u_j^{\rm in} \} + (\dt^k f)_{\vert_{t=0}}. 
\end{equation}
The boundary condition $\nu(t)\cdot u_{\vert_{x=0}}=g$ also implies that 
\[
\dt^k \big(\nu(t)\cdot {u}_{\vert_{x=0}}\big)=\dt^k g.
\]
On the edge $\{t=0,x=0\}$, smooth enough solutions must therefore satisfy 
\begin{equation}\label{compkVC}
\sum_{j=0}^k \binom{k}{j} (\dt^j \nu)_{\vert_{t=0}}\cdot {u^{\rm in}_{k-j}}_{\vert_{x=0}} = (\dt^k g)_{\vert_{t=0}}.
\end{equation}

\begin{definition}\label{defcompVC}
Let $m\geq1$ be an integer. 
We say that the data $u^{\rm in}\in H^m(\R_+)$, $f\in H^m(\Omega_T)$, and $g \in H^m(0,T)$ 
for the initial boundary value problem \eqref{systVC} satisfy the compatibility condition at order $k$ 
if the $\{u_j^{\rm in}\}_{j=0}^{m}$ defined in \eqref{defu0k} satisfy \eqref{compkVC}. 
We also say that the data satisfy the compatibility conditions up to order $m-1$ if they satisfy the 
compatibility conditions at order $k$ for $k=0,1,\ldots,m-1$. 
\end{definition}

\subsubsection{A priori $L^2$-estimate}\label{sectapL2}
We prove here an $L^2$ a priori estimate using the following assumption, 
which will be verified later as a consequence of Assumption \ref{asshyp}.

\begin{assumption}\label{assVC}
There exists a symmetric matrix $S(t,x) \in {\mathcal M}_2(\R)$ such that 
for any $(t,x)\in\Omega_T$ $S(t,x)A(t,x)$ is symmetric and the following conditions hold. 
\begin{enumerate}
\item[{\bf i.}]
There exist constants $\alpha_0,\beta_0>0$ such that for any 
$(v,t,x)\in \R^2\times \Omega_T$ we have 
\[
\alpha_0 |v|^2 \leq v^{\rm T} S(t,x) v \leq \beta_0 |v|^2.
\]
\item[{\bf ii.}]
There exist constants $\alpha_1,\beta_1>0$ such that for any 
$(v,t)\in \R^2\times (0,T)$ we have 
\[
v^{\rm T} S(t,0)A(t,0) v \leq -\alpha_1 |v|^2 + \beta_1 |\nu(t) \cdot v|^2.
\]
\item[{\bf iii.}]
There exists a constant $\beta_2$ such that 
\[
\| \dt S + \dx (SA) - 2SB \|_{L^2(\Omega_T)\to L^2(\Omega_T)} \leq \beta_2. 
\]
\end{enumerate}
\end{assumption}

\begin{notation}\label{notin}
We denote by $\beta_0^{\rm in}\leq \beta_0$ any constant such that 
the inequality in {\bf i} of the assumption is satisfied at $t=0$. 
\end{notation}

In the $L^2$ a priori estimate provided by the proposition, the control of the source term by 
$S_{\gamma,t}^*( \|f(\cdot)\|_{L^2} )$ is crucial to get the refined higher order estimates of 
Theorem \ref{theoIBVP1}.

\begin{proposition}\label{propNRJ1}
Under Assumption \ref{assVC}, there are constants  
\[
\mathfrak{c}_0 = C\Bigl( \frac{\beta_0^{\rm in}}{\alpha_0},\frac{\beta_0^{\rm in}}{\alpha_1} \Bigr)
 \quad\mbox{ and }\quad
\mathfrak{c}_1 = C\Big( \frac{\beta_0}{\alpha_0},\frac{\beta_1}{\alpha_0},\frac{\alpha_0}{\alpha_1} \Big)
\]
such that for any $u \in H^1(\Omega_T)$ solving \eqref{systVC}, any $t\in [0,T]$, and any 
$\gamma\geq\frac{\beta_2}{\alpha_0}$, 
the following inequality holds. 
\begin{align*}
&\opnorm{u(t)}_{0,\gamma} + \biggl(\gamma \int_0^t \opnorm{u(t')}_{0,\gamma}^2 {\rm d}t' \biggr)^\frac12
 + |u_{\vert_{x=0}}|_{L_\gamma^2(0,t)} \\
&\leq \mathfrak{c}_0\|u^{\rm in}\|_{L^2}
 + \mathfrak{c}_1\bigl( |g|_{L_\gamma^2(0,t)} + S_{\gamma,t}^*( \|f(\cdot)\|_{L^2} ) \bigr),
\end{align*}
where we recall that $S_{\gamma,t}^*( \|f(\cdot)\|_{L^2} )$ is defined in Notation \ref{dualnorm}. 
\end{proposition}

\begin{proof}
Multiplying the first equation of \eqref{systVC} by $S$ and taking the $L^2(\Omega_t)$ scalar product with 
$e^{-2\gamma t}u$, 
we get after integration by parts, 
\begin{align*}
& e^{-2\gamma t}(Su(t),u(t))_{L^2} + 2\gamma\int_0^t e^{-2\gamma t'}(Su,u)_{L^2}{\rm d}t'
 - \int_0^t e^{-2\gamma t'}(SA u\cdot u )_{\vert_{x=0}}{\rm d}t' \\
&= (S_{\vert_{t=0}}u^{\rm in},u^{\rm in})_{L^2}
 + \int_0^t e^{-2\gamma t'}((\dt S+\dx (SA)-2 SB)u + 2Sf,u)_{L^2}{\rm d}t'.
\end{align*}
Using Assumption \ref{assVC} with Notation \ref{notin}, this yields 
\begin{align*}
&\alpha_0 \opnorm{u(t)}_{0,\gamma}^2 + (2\alpha_0\gamma - \beta_2)\int_0^t \opnorm{u(t')}_{0,\gamma}^2{\rm d}t'
 + \alpha_1 |u_{\vert_{x=0}}|^2_{L_\gamma^2(0,t)} \\
&\leq \beta_0^{\rm in} \|u^{\rm in}\|_{L^2}^2 + \beta_1 |g|^2_{L_\gamma^2(0,t)} 
 + 2\beta_0 \int_0^t e^{-2\gamma t'}\|f(t')\|_{L^2} \|u(t')\|_{L^2} {\rm d}t'.
\end{align*}
We evaluate the last term as 
\begin{align*}
& \int_0^t e^{-2\gamma t'}\|f(t')\|_{L^2} \|u(t')\|_{L^2} {\rm d}t' \\
&\leq S_{\gamma,t}^*(\|f(\cdot)\|_{L^2}) \biggl\{ \|u\|_{\WW_\gamma^0(t)}
 + \biggl(\gamma\int_0^t \opnorm{u(t')}_{0,\gamma}^2{\rm d}t' \biggr)^\frac12 \biggr\} \\
&\leq S_{\gamma,t}^*(\|f(\cdot)\|_{L^2}) \|u\|_{\WW_\gamma^0(t)}
 + \frac{\beta_0}{\alpha_0}S_{\gamma,t}^*(\|f(\cdot)\|_{L^2})^2 
 + \frac14\frac{\alpha_0}{\beta_0}\gamma \int_0^t\opnorm{u(t')}_{0,\gamma}^2{\rm d}t'
\end{align*}
and we deduce that 
\begin{align}\label{estuuu}
&\opnorm{u(t)}_{0,\gamma}^2
 + \frac{\gamma}{2}\int_0^t \opnorm{u(t')}_{0,\gamma}^2{\rm d}t'
 + \frac{\alpha_1}{\alpha_0} |u_{\vert_{x=0}}|^2_{L_\gamma^2(0,t)} \\
&\leq \frac{\beta_0^{\rm in}}{\alpha_0} \|u^{\rm in}\|_{L^2}^2 + \frac{\beta_1}{\alpha_0} |g|^2_{L_\gamma^2(0,t)} 
 + 2\frac{\beta_0}{\alpha_0}S_{\gamma,t}^*(\|f(\cdot)\|_{L^2}) \|u\|_{\WW_\gamma^0(t)}
 + 2\biggl( \frac{\beta_0}{\alpha_0}S_{\gamma,t}^*(\|f(\cdot)\|_{L^2}) \biggr)^2 \nonumber \\
&\leq \frac{\beta_0^{\rm in}}{\alpha_0} \|u^{\rm in}\|_{L^2}^2 + \frac{\beta_1}{\alpha_0} |g|^2_{L_\gamma^2(0,t)} 
 + \frac12\|u\|_{\WW_\gamma^0(t)}^2
 + 4\biggl( \frac{\beta_0}{\alpha_0}S_{\gamma,t}^*(\|f(\cdot)\|_{L^2}) \biggr)^2 \nonumber
\end{align}
for $\gamma\geq\frac{\beta_2}{\alpha_0}$. 
Particularly, we have 
\[
\frac12\|u\|_{\WW_\gamma^0(t)}^2
\leq \frac{\beta_0^{\rm in}}{\alpha_0} \|u^{\rm in}\|_{L^2}^2 + \frac{\beta_1}{\alpha_0} |g|^2_{L_\gamma^2(0,t)} 
 + 4\biggl( \frac{\beta_0}{\alpha_0}S_{\gamma,t}^*(\|f(\cdot)\|_{L^2}) \biggr)^2.
\]
Plugging this into \eqref{estuuu}, we obtain the desired estimate. 
\end{proof}

\subsubsection{Product and commutator estimates}
To obtain higher order a priori estimates, we need to use calculus inequalities. 
By the standard Sobolev imbedding theorem $H^1(\R_+) \subseteq L^\infty(\R_+)$, 
we can easily obtain the following lemma.

\begin{lemma}\label{ineq1}
Let $m\geq1$ be an integer. 
There exists a constant $C$ such that the following inequalities hold:
\begin{enumerate}
\setlength{\itemsep}{3pt}
\item[{\bf i.}]
$\opnorm{ u(t)v(t) }_m \leq C(\|u(t)\|_{L^\infty(\R_+)} + \opnorm{ \partial u(t) }_{m-1}) \opnorm{ v(t)}_m$,
\item[{\bf ii.}]
$\|[\partial^\alpha,u(t)]v(t)\|_{L^2(\R_+)}
 \leq C(\|\partial u(t)\|_{L^\infty(\R_+)} + \opnorm{ \partial u(t) }_{m-1}) \opnorm{ v(t) }_{m-1}$ 
 if \ $|\alpha| \leq m$,
\item[{\bf iii.}]
$\|\partial[\partial^\alpha,u(t)]v(t)\|_{L^2(\R_+)}
 \leq C(\|\partial u(t)\|_{L^\infty(\R_+)} + \opnorm{ \partial u(t) }_{m-1}) \opnorm{ v(t) }_{m-1}$ 
 if \ $|\alpha| \leq m-1$,
\item[{\bf iv.}]
$\|\partial[\partial^\alpha;u(t),v(t)]\|_{L^2(\R_+)}
 \leq C\opnorm{ \partial u(t) }_{m-2} \opnorm{ \partial v(t) }_{m-2}$ if \ $2\leq |\alpha| \leq m-1$,
\end{enumerate}
where $[\partial^\alpha;u,v] = \partial^\alpha(uv)-(\partial^\alpha u)v-u(\partial^\alpha v)$ is a 
symmetric commutator. 
\end{lemma}

The following Moser-type inequality is a direct consequence of the above lemma.

\begin{lemma}\label{ineq2}
Let $\mathcal{U}$ be an open set in $\R^N$, $F \in C^\infty(\mathcal{U})$, and $F(0)=0$. 
If $m \in \N$ and $u \in \WW^m(T)$ takes its value in a compact set 
$\mathcal{K} \subset \mathcal{U}$, then for any $t\in[0,T]$ we have 
\[
\opnorm{ (F(u))(t) }_m \leq C(\|u\|_{W^{[m/2],\infty}(\Omega_t)}) \opnorm{ u(t) }_m,
\]
where $[m/2]$ is the integer part of $m/2$. 
\end{lemma}

We also need Moser-type inequalities for the trace at the boundary of the nonlinear terms, 
as in the following lemma.

\begin{lemma}\label{ineq3}
Let $\mathcal{U}$ be an open set in $\R^N$, $F \in C^\infty(\mathcal{U})$, and $F(0)=0$. 
If $m \in \N$ and $u=u(t,x)$ takes its value in a compact set $\mathcal{K} \subset \mathcal{U}$, 
then we have 
\begin{enumerate}
\setlength{\itemsep}{3pt}
\item[{\bf i.}]
$|F(u)_{\vert_{x=0}}|_{m,t} 
 \leq C( \sum_{|\alpha| \leq [m/2]}|(\partial^\alpha u)_{\vert_{x=0}}|_{L^\infty(0,t)} )|\vu|_{m,t}$,
\item[{\bf ii.}]
$|F(u)_{\vert_{x=0}}|_{m,t} \leq C( \|u\|_{\WW^{[m/2]+1}(t)} )|\vu|_{m,t}$,
\item[{\bf iii.}]
$|\dt(F(u))_{\vert_{x=0}}|_{m,t}
 \leq C( \|u\|_{\WW^{m}(t)}, \|u\|_{L^\infty(\Omega_T)})
 (|(\dt u)_{\vert_{x=0}}|_{m,t} + \|\dt u\|_{\WW^{m}(t)}|\vu|_{m,t})$, 
\end{enumerate}
where $[m/2]$ is the integer part of $m/2$. 
\end{lemma}

\begin{proof}
The proof of {\bf i} is straightforward and {\bf i} together with the Sobolev imbedding theorem 
$H^1(\R_+) \subseteq L^\infty(\R_+)$ yields {\bf ii}. 
We will prove {\bf iii}. 
The case $m=0$ is obvious so that we assume $m\geq1$. 
In view of $\partial^\alpha\dt(F(u)) = F'(u)\partial^\alpha\dt u + [\partial^\alpha,F'(u)]\dt u$, we have 
\begin{align*}
|\dt(F(u))_{\vert_{x=0}}|_{m,t}
&\leq C|(\dt u)_{\vert_{x=0}}|_{m,t}
 + C\|\dt u\|_{W^{m-1,\infty}(\Omega_t)}\sum_{1 \leq |\alpha| \leq m}|\partial^\alpha F'(u)|_{L^2(0,t)} \\
&\leq C|(\dt u)_{\vert_{x=0}}|_{m,t}
 + C(\|u\|_{\WW^{[m/2]+1}(t)})\|\dt u\|_{\WW^m(t)}|\vu|_{m,t}.
\end{align*}
Since $[m/2]+1 \leq m$, we obtain the desired inequality. 
\end{proof}

\begin{lemma}\label{ineq4}
There exists an absolute constant $C$ such that for any $\gamma>0$ and any integer $m\geq1$ we have 
\begin{align}
\label{ineq5}
& e^{-\gamma t}|u(t)| + \biggl( \gamma\int_0^t e^{-2\gamma t'}|u(t')|^2{\rm d}t' \biggr)^\frac12
 \leq C\bigl( |u(0)| + S_{\gamma,t}^*(|\dt u|) \bigr), \\
\label{ineq6}
& |u_{\vert_{x=0}}|_{m-1,\gamma,t}
 \leq C( \gamma^{-\frac12}\opnorm{u(0)}_m + \gamma^{-1}|u_{\vert_{x=0}}|_{m,\gamma,t} ), \\
\label{ineq7}
& \opnorm{u(t)}_{m-1,\gamma} + \biggl( \gamma\int_0^t \opnorm{u(t')}_{m-1,\gamma}^2{\rm d}t' \biggr)^\frac12
 \leq C\bigl( \opnorm{u(0)}_{m-1} + S_{\gamma,t}^*(\opnorm{\dt u(\cdot)}_{m-1}) \bigr).
\end{align}
\end{lemma}

\begin{proof}
Integrating the identity 
\[
\frac{\rm d}{{\rm d}t}(e^{-2\gamma t}|u(t)|^2) + 2\gamma e^{-2\gamma t}|u(t)|^2
= 2e^{-2\gamma t}u(t)\cdot\dt u(t),
\]
we have 
\[
e^{-2\gamma t}|u(t)|^2 + 2\gamma \int_0^t e^{-2\gamma t'}|u(t')|^2{\rm d}t' 
= |u(0)|^2 + 2\int_0^t e^{-2\gamma t'}u(t')\cdot\dt u(t'){\rm d}t'.
\]
The last term is evaluated as 
\begin{align*}
2\int_0^t e^{-2\gamma t'}u(t')\cdot\dt u(t'){\rm d}t'
&\leq 2\int_0^t e^{-2\gamma t'}|u(t')||\dt u(t')|{\rm d}t' \\
&\leq 2S_{\gamma,t}^*(|\dt u|)\biggl\{ \sup_{t'\in[0,t]}e^{-\gamma t'}|u(t')|
 + \biggl(\gamma \int_0^t e^{-2\gamma t'}|u(t')|^2{\rm d}t' \biggr)^\frac12 \biggr\} \\
&\leq \frac12\sup_{t'\in[0,t]}e^{-2\gamma t'}|u(t')|^2
 + \gamma \int_0^t e^{-2\gamma t'}|u(t')|^2{\rm d}t'
 +3S_{\gamma,t}^*(|\dt u|)^2,
\end{align*}
so that we obtain \eqref{ineq5}. 
Similarly, we can show \eqref{ineq7}. 
As a corollary of \eqref{ineq5}, we have 
\[
|u|_{L_\gamma^2(0,t)} \leq C( \gamma^{-\frac12}|u(0)| + \gamma^{-1}|\dt u|_{L_\gamma^2(0,t)} ).
\]
Applying this inequality to $(\partial^\alpha u)_{\vert_{x=0}}$, summing the resulting inequality over 
$|\alpha| \leq m-1$, and using the Sobolev imbedding theorem $H^1(\R_+) \subseteq L^\infty(\R_+)$, 
we obtain \eqref{ineq6}. 
\end{proof}

\subsubsection{Higher order a priori estimate}\label{sect2HO}
We can now state the generalization of Proposition \ref{propNRJ1} to higher order Sobolev spaces.

\begin{proposition}\label{propVC2}
Let $m\geq1$ be an integer, $T>0$, and assume that Assumption \ref{assVC} is satisfied. 
Assume moreover that there are two constants $0<K_0\leq K$ such that 
\[
\begin{cases}
\mathfrak{c}_0, \mathfrak{c}_1, \|A\|_{L^\infty(\Omega_T)}, 
 \|A^{-1}\|_{L^\infty(\Omega_T)}, |\nu|_{L^\infty(0,T)} \leq K_0, \\
\frac{\beta_2}{\alpha_0}, \|A\|_{W^{1,\infty}(\Omega_T)}, \|B\|_{L^\infty(\Omega_T)}, 
 \|(\partial A,\partial B)\|_{\WW^{m-1}(T)}, |\nu|_{W^{m,\infty}(0,T)} \leq K,
\end{cases}
\]
where $\mathfrak{c}_0$ and $\mathfrak{c}_1$ are as in Proposition \ref{propNRJ1}. 
Then, every solution $u\in H^{m+1}(\Omega_{T})$ to the initial boundary value problem \eqref{systVC} satisfies, 
for any $t\in[0,T]$ and any $\gamma \geq C(K)$, 
\begin{align*}
&\opnorm{ u(t) }_{m,\gamma} + \biggl( \gamma\int_0^t \opnorm{ u(t') }_{m,\gamma}^2{\rm d}t' \biggr)^\frac12
 + |\vu|_{m,\gamma,t} \\
&\leq C(K_0)\bigl( \opnorm{ u(0) }_{m} + \abs{g}_{H_\gamma^m(0,t)} 
 + \abs{\vf}_{m-1,\gamma,t} + S_{\gamma,t}^*(\opnorm{ \dt f(t') }_{m-1}) \bigr).
\end{align*}
\end{proposition}

\begin{proof}
Let $u_m = \dt^m u$. 
Then, $u_m$ solves 
\[
\begin{cases}
\dt u_m + A(t,x)\dx u_m + B(t,x) u_m = f_m & \mbox{in}\quad \Omega_T, \\
{u_m}_{\vert_{t=0}} = (\dt^m u)_{\vert_{t=0}} & \mbox{on}\quad \R_+, \\
\nu(t) \cdot {u_m}_{\vert_{x=0}}= g_m(t) & \mbox{on}\quad (0,T),
\end{cases}
\]
where 
\[
\begin{cases}
f_m = \dt^m(f-Bu) - [\dt^m,A]\dx u, \\
g_m = \dt^m g- [\dt^m,\nu] \cdot u_{\vert_{x=0}}.
\end{cases}
\]
Applying Proposition \ref{propNRJ1} we obtain 
\begin{align*}
&\opnorm{u_m(t)}_{0,\gamma} + \biggl(\gamma \int_0^t \opnorm{u_m(t')}_{0,\gamma}^2 {\rm d}t' \biggr)^\frac12
 + |{u_m}_{\vert_{x=0}}|_{L_\gamma^2(0,t)} \\
&\leq \mathfrak{c}_0\opnorm{u(0)}_m 
 + \mathfrak{c}_1\bigl( |g_m|_{L_\gamma^2(0,t)} + S_{\gamma,t}^*( \|f_m(\cdot)\|_{L^2} ) \bigr).
\end{align*}
On the other hand, it follows from Lemma \ref{ineq1} that 
\[
\begin{cases}
\|f_m(t)\|_{L^2} \leq \opnorm{ \dt f(t) }_{m-1} + C(K)\opnorm{ u(t) }_m, \\
|g_m|_{L_\gamma^2(0,t)} \leq |g|_{H_\gamma^m(0,t)} + C(K)|\vu|_{m-1,\gamma,t}.
\end{cases}
\]
Therefore, we obtain 
\begin{align}\label{hestpre1}
&\opnorm{u_m(t)}_{0,\gamma} + \biggl(\gamma \int_0^t \opnorm{u_m(t')}_{0,\gamma}^2 {\rm d}t' \biggr)^\frac12
 + |{u_m}_{\vert_{x=0}}|_{L_\gamma^2(0,t)} \\
&\leq  C(K_0)\bigl( \opnorm{ u(0) }_m + |g|_{H_\gamma^m(0,t)}
 + S_{\gamma,t}^*( \opnorm{\dt f(\cdot)}_{m-1} ) \bigr) \nonumber \\
&\quad
 + C(K)\bigl( |\vu|_{m-1,\gamma,t} + S_{\gamma,t}^*(\opnorm{ u(t') }_m) \bigr). \nonumber
\end{align}
We proceed to control the other derivatives. 
Let $k$ and $l$ be nonnegative integers satisfying $k+l \leq m-1$. 
Applying $\dt^k\dx^l$ to the equation, we get 
\[
\dt^{k+1}\dx^l u + A \dt^k \dx^{l+1} u = \dt^k\dx^l(f-Bu)-[\dt^k\dx^l,A]\dx u=:f_{k,l}.
\]
By using these two expressions of $f_{k,l}$ together with Lemma \ref{ineq1} we see that 
\[
\begin{cases}
\|f_{k,l}(0)\|_{L^2} \leq C(K_0)\opnorm{u(0)}_m, \\
\|\dt f_{k,l}(t)\|_{L^2} \leq \opnorm{\dt f(t)}_{m-1} + C(K)\opnorm{u(t)}_m, \\
|f_{k,l \vert_{x=0}}|_{L_\gamma^2(0,t)} \leq |\vf|_{m-1,\gamma,t} + C(K)|\vu|_{m-1,\gamma,t}.
\end{cases}
\]
We have now the relation $\dt^k \dx^{l+1} u = A^{-1}(f_{k,l}-\dt^{k+1}\dx^l u)$ so that 
\[
\begin{cases}
\|\dt^k \dx^{l+1} u(t)\|_{L^2} \leq C(K_0)( \|\dt^{k+1} \dx^l u(t)\|_{L^2} + \|f_{k,l}(t)\|_{L^2}), \\
|(\dt^k \dx^{l+1} u)_{\vert_{x=0}}|_{L_\gamma^2(0,t)} \leq C(K_0)( 
 |(\dt^{k+1} \dx^l u)_{\vert_{x=0}}|_{L_\gamma^2(0,t)} + |f_{k,l \vert_{x=0}}|_{L_\gamma^2(0,t)} ).
\end{cases}
\]
Therefore, 
\begin{align*}
& \opnorm{ \dt^k \dx^{l+1} u(t) }_{0,\gamma}
 + \biggl(\gamma\int_0^t\opnorm{ \dt^k \dx^{l+1} u(t') }_{0,\gamma}^2{\rm d}t'\biggr)^\frac12
 + |(\dt^k \dx^{l+1} u)_{\vert_{x=0}}|_{L_\gamma^2(0,t)} \\
&\leq C(K_0)\biggl\{
 \opnorm{ \dt^{k+1} \dx^l u(t) }_{0,\gamma}
 + \biggl(\gamma\int_0^t\opnorm{ \dt^{k+1} \dx^l u(t') }_{0,\gamma}^2{\rm d}t'\biggr)^\frac12
 + |(\dt^{k+1} \dx^l u)_{\vert_{x=0}}|_{L_\gamma^2(0,t)} \\
&\phantom{ \leq C(K_0)\biggl\{ }
 + \opnorm{ f_{k,l}(t) }_{0,\gamma}
 + \biggl(\gamma\int_0^t\opnorm{ f_{k,l}(t') }_{0,\gamma}^2{\rm d}t'\biggr)^\frac12
 + |f_{k,l \vert_{x=0}}|_{L_\gamma^2(0,t)} \biggr\}.
\end{align*}
Here, by Lemma \ref{ineq4} we have 
\begin{align*}
& \opnorm{ f_{k,l}(t) }_{0,\gamma}
 + \biggl(\gamma\int_0^t\opnorm{ f_{k,l}(t') }_{0,\gamma}^2{\rm d}t'\biggr)^\frac12 \\
&\leq C\bigl( \|f_{k,l}(0)\|_{L^2} + S_{\gamma,t}^*(\|\dt f_{k,l}(\cdot)\|_{L^2}) \bigr) \\
&\leq C(K_0)\bigl( \opnorm{u(0)}_m + S_{\gamma,t}^*(\opnorm{\dt f(\cdot)}_{m-1}) \bigr)
 + C(K)S_{\gamma,t}^*(\opnorm{u(\cdot)}_m).
\end{align*}
By using the above inequality inductively, we obtain 
\begin{align*}
& \opnorm{ u(t) }_{m,\gamma}
 + \biggl(\gamma\int_0^t\opnorm{ u(t') }_{m,\gamma}^2{\rm d}t'\biggr)^\frac12
 + |\vu|_{m,\gamma,t} \\
&\leq C(K_0)\biggl\{
 \opnorm{u(0)}_m + S_{\gamma,t}^*(\opnorm{\dt f(\cdot)}_{m-1}) + |\vf|_{m-1,\gamma,t} \\
&\phantom{ \leq C(K_0)\biggl\{ }
 + \opnorm{ u_m(t) }_{0,\gamma}
 + \biggl(\gamma\int_0^t\opnorm{ u_m(t') }_{0,\gamma}^2{\rm d}t'\biggr)^\frac12
 + |u_{m \vert_{x=0}}|_{L_\gamma^2(0,t)} \\
&\phantom{ \leq C(K_0)\biggl\{ }
 + \opnorm{ u(t) }_{m-1,\gamma}
 + \biggl(\gamma\int_0^t\opnorm{ u(t') }_{m-1,\gamma}^2{\rm d}t'\biggr)^\frac12 \biggr\} \\
&\quad
 + C(K)\bigl( |\vu|_{m-1,\gamma,t} + S_{\gamma,t}^*(\opnorm{u(\cdot)}_m) \bigr).
\end{align*}
This together with \eqref{hestpre1} and Lemma \ref{ineq4} implies 
\begin{align*}
& \opnorm{ u(t) }_{m,\gamma}
 + \biggl(\gamma\int_0^t\opnorm{ u(t') }_{m,\gamma}^2{\rm d}t'\biggr)^\frac12
 + |\vu|_{m,\gamma,t} \\
&\leq C(K_0)\bigl( \opnorm{ u(0) }_m + |g|_{H_\gamma^m(0,t)} + |\vf|_{m-1,\gamma,t}
 + S_{\gamma,t}^*( \opnorm{\dt f(\cdot)}_{m-1} ) \bigr) \nonumber \\
&\quad
 + C(K)\bigl( |\vu|_{m-1,\gamma,t} + S_{\gamma,t}^*(\opnorm{ u(t') }_m) \bigr) \\
&\leq C(K_0)\bigl( \opnorm{ u(0) }_m + |g|_{H_\gamma^m(0,t)} + |\vf|_{m-1,\gamma,t}
 + S_{\gamma,t}^*( \opnorm{\dt f(\cdot)}_{m-1} ) \bigr) \nonumber \\
&\quad
 + C(K)\biggl\{ \gamma^{-\frac12}\opnorm{ u(0) }_m 
  + \gamma^{-1}\biggl( \gamma\int_0^t \opnorm{u(t')}_{m,\gamma}^2{\rm d}t'\biggr)^\frac12 
  + \gamma^{-1}|\vu|_{m,\gamma,t} \biggr\}.
\end{align*}
Therefore, by taking $\gamma$ sufficiently large compared to $C(K)$, we obtain the desired estimate 
(note that this would not be possible without the second term of the left-hand side). 
\end{proof}

\subsubsection{Proof of Theorem \ref{theoIBVP1}}\label{secttheoIBVP1}

Under Assumption \ref{assVC}, the existence and uniqueness of a solution $u\in \WW^m(T)$ to \eqref{systVC} 
can be deduced from Proposition \ref{propVC2} and the compatibility condition along classical lines 
(see for instance \cite{metivier2001,metivier2012,benzoniserre2007}). 
We still have to prove that the assumptions made in the statement of Theorem \ref{theoIBVP1} imply that 
Assumption \ref{assVC} is satisfied. 
This is given by the following lemma.

\begin{lemma}\label{lemsymmetrizer}
Let $c_0>0$ be such that Assumption \ref{asshyp} is satisfied. 
There exist a symmetrizer $S\in W^{1,\infty}(\Omega_T)$ and constants 
$\alpha_0,\alpha_1$ and $\beta_0,\beta_1,\beta_2$ such that Assumption \ref{assVC} is satisfied.
Moreover, we have 
\[
\mathfrak{c}_0 \leq C\Bigl( \frac{1}{c_0}, \| A_{\vert_{t=0}} \|_{L^\infty(\R_+)} \Bigr)
\quad\mbox{and}\quad 
\mathfrak{c}_1 \leq C\Bigl( \frac{1}{c_0},\|A\|_{L^{\infty}(\Omega_T)} \Bigr),
\]
where $\mathfrak{c}_0$ and $\mathfrak{c}_1$ are as defined in Proposition \ref{propNRJ1}, and we also have 
\[
\frac{\beta_2}{\beta_0} \leq C\Bigl( \frac{1}{c_0},\Vert A\Vert_{W^{1,\infty}(\Omega_T)},
 \Vert B\Vert_{L^\infty(\Omega_T)}\Bigr).
\]
\end{lemma}

This lemma is a simple consequence of the following proposition and its proof, 
which characterizes the uniform Kreiss--Lopatinski\u{\i} condition {\bf iii} in Assumption \ref{asshyp}.

\begin{proposition}\label{propBC}
Suppose that the condition {\bf ii} in Assumption \ref{asshyp}, $\vert \nu(t) \vert \geq c_0$, 
and $\vert A(t,x) \vert \leq 1/c_0$ hold for some positive constant $c_0$. 
Then, the following four statements are all equivalent. 
\begin{enumerate}
\setlength{\itemsep}{3pt}
\item[{\bf i.}]
There exist a symmetrizer $S\in W^{1,\infty}(\Omega_T)$ and positive constants $\alpha_0$ and $\beta_0$ 
such that $\alpha_0{\rm Id} \leq S(t,x) \leq \beta_0{\rm Id}$ and that for any $v \in \R^2$ satisfying 
$\nu(t)\cdot v = 0$ we have 
\[
v^{\rm T}S(t,0)A(t,0)v \leq 0.
\]

\item[{\bf ii.}]
There exist a symmetrizer $S\in W^{1,\infty}(\Omega_T)$ and positive constants $\alpha_0$, $\beta_0$, 
$\alpha_1$, and $\beta_1$ such that $\alpha_0{\rm Id} \leq S(t,x) \leq \beta_0{\rm Id}$ and that for any 
$v \in \R^2$ we have 
\[
v^{\rm T}S(t,0)A(t,0)v \leq -\alpha_1|v|^2 + \beta_1|\nu(t)\cdot v|^2.
\]

\item[{\bf iii.}]
There exists a positive constant $\alpha_0$ such that 
\[
|\pi_{-}(t,0)\nu(t)^\perp| \geq \alpha_0,
\]
where $\pi_{\pm}(t,x)$ is the eigenprojector associated to the eigenvalue $\pm\lambda_{\pm}(t,x)$ of $A(t,x)$. 

\item[{\bf iv.}]
There exists a positive constant $\alpha_0$ such that 
\[
|\nu(t)\cdot\mathbf{e}_{+}(t,0)| \geq \alpha_0,
\]
where $\mathbf{e}_{\pm}(t,x)$ is the unit eigenvector associated to the eigenvalue $\pm\lambda_{\pm}(t,x)$ of $A(t,x)$. 
\end{enumerate}
\end{proposition}

\begin{proof}
We note that the eigenprojector $\pi_{\pm}(t,x)$ is given explicitly by 
\[
\pi_{+}(t,x) = \frac{A(t,x)+\lambda_{-}(t,x){\rm Id}}{\lambda_{+}(t,x)+\lambda_{-}(t,x)}, \qquad
\pi_{-}(t,x) = -\frac{A(t,x)-\lambda_{+}(t,x){\rm Id}}{\lambda_{+}(t,x)+\lambda_{-}(t,x)}
\]
and that under the assumption $\lambda_{\pm}(t,x)$ and $|\pi_{\pm}(t,x)|$ are bounded from above by a constant 
depending on $c_0$. 
We see that 
\begin{align*}
|\nu(t) \cdot \mathbf{e}_{+}(t,0)|
&= |\nu(t)^\perp \cdot \mathbf{e}_{+}(t,0)^\perp|
 = |(\pi_-(t,0)\nu(t)^\perp) \cdot \mathbf{e}_{+}(t,0)^\perp|
 \leq |\pi_-(t,0)\nu(t)^\perp|
\end{align*}
and that 
\begin{align*}
|\pi_-(t,0)\nu(t)^\perp|
&= |(\nu(t)^\perp \cdot \mathbf{e}_{+}(t,0)^\perp)\pi_-(t,0)\mathbf{e}_{+}(t,0)^\perp|
 \leq |\pi_-(t,0)||\nu(t) \cdot \mathbf{e}_{+}(t,0)|.
\end{align*}
These imply the equivalence of {\bf iii} and {\bf iv}. 
Obviously, {\bf ii} implies {\bf i}. 

We proceed to show that {\bf i} implies {\bf iii}. 
By the assumption we have 
\[
(\nu(t)^\perp)^{\rm T} S(t,0)A(t,0)\nu(t)^\perp \leq 0,
\]
which together with the spectral decomposition 
\[
A(t,x) = \lambda_{+}(t,x)\pi_{+}(t,x) - \lambda_{-}(t,x)\pi_{-}(t,x)
\]
implies 
\begin{align*}
c_0\alpha_0 |\pi_{+}(t,0)\nu(t)^\perp|^2
\leq & \lambda_{+}(t,0)(\pi_{+}(t,0)\nu(t)^\perp)^{\rm T}S(t,0)\pi_{+}(t,0)\nu(t)^\perp \\
\leq & (\lambda_{-}(t,0)-\lambda_{+}(t,0))(\pi_{+}(t,0)\nu(t)^\perp)^{\rm T}S(t,0)\pi_{-}(t,0)\nu(t)^\perp \\
 & +\lambda_{-}(t,0)(\pi_{-}(t,0)\nu(t)^\perp)^{\rm T}S(t,0)\pi_{-}(t,0)\nu(t)^\perp \\
\leq & \beta_0|\lambda_{-}(t,0)-\lambda_{+}(t,0)| |\pi_{+}(t,0)\nu(t)^\perp| |\pi_{-}(t,0)\nu(t)^\perp| \\
 & + \beta_0\lambda_{-}(t,0)|\pi_{-}(t,0)\nu(t)^\perp|^2.
\end{align*}
Particularly, we have 
\[
c_0\alpha_0 |\pi_{+}(t,0)\nu(t)^\perp|^2
\leq \biggl( \frac{\beta_0^2|\lambda_{-}(t,0)-\lambda_{+}(t,0)|^2}{c_0\alpha_0}
 + 2\beta_0\lambda_{-}(t,0) \biggr)|\pi_{-}(t,0)\nu(t)^\perp|^2.
\]
Therefore, in view of $c_0 \leq |\nu(t)| \leq |\pi_{-}(t,0)\nu(t)^\perp| + |\pi_{+}(t,0)\nu(t)^\perp|$ 
we obtain the desired inequality in the statement {\bf iii}. 

Finally, we will show that {\bf iii} implies {\bf ii}. 
This is the most important part of this proposition. 
We want to show that for a suitably large $M>1$, a symmetrizer $S(t,x)$ satisfying the conditions 
in the statement {\bf ii} is provided by the formula 
\[
S(t,x) = \pi_+(t,x)^{\rm T}\pi_+(t,x) + M\pi_-(t,x)^{\rm T}\pi_-(t,x),
\]
so that the first point of {\bf ii} is satisfied with $\alpha_0=1$ and $\beta_0=M$. 
By the definition of $\pi_\pm$, we compute indeed that 
\[
SA = \lambda_+ \pi_+^{\rm T}\pi_+ - M \lambda_- \pi_-^{\rm T}\pi_-,
\]
which is obviously symmetric. 
For the second point of {\bf ii}, just remark that 
\[
v^{\rm T}SAv =\lambda_+ |\pi_+ v|^2 - M\lambda_- |\pi_- v|^2.
\]
We need to show that this quantity is negative on the kernel $\R \nu^\perp$ of the boundary condition. 
Under the hypothesis we can assume that $|\nu(t)|=1$ without loss of generality. 
Then, we see that 
\begin{align*}
-|\pi_{-}v|^2
&= -|(\nu^\perp \cdot v)\pi_{-}\nu^\perp + (\nu \cdot v)\pi_{-}\nu|^2 \\
&\leq -\frac12|\nu^\perp \cdot v|^2|\pi_{-}\nu^\perp|^2 + |\nu \cdot v|^2|\pi_{-}\nu|^2 \\
&\leq -\frac12|\pi_{-}\nu^\perp|^2|v|^2 + (|\pi_{-}\nu|^2+|\pi_{-}\nu^\perp|^2)|\nu \cdot v|^2
\end{align*}
and that 
\begin{align*}
|\pi_{+}v|^2
&= |(\nu^\perp \cdot v)\pi_{+}\nu^\perp + (\nu \cdot v)\pi_{+}\nu|^2 \\
&\leq 2|\pi_{+}\nu^\perp|^2|\nu^\perp \cdot v|^2 + 2|\pi_{+}\nu|^2|\nu \cdot v|^2 \\
&\leq 4|\pi_{+}\nu^\perp|^2|v|^2 + 4(|\pi_{+}\nu^\perp|^2+|\pi_{+}\nu|^2)|\nu \cdot v|^2.
\end{align*}
Therefore, we obtain 
\begin{align*}
v^{\rm T}SAv
\leq & -\lambda_{-}|\pi_{-}\nu^\perp|^2\biggl(
 \frac{M}{2} - 4\frac{\lambda_{+}}{\lambda_{-}}\frac{|\pi_{+}\nu^\perp|^2}{|\pi_{-}\nu^\perp|^2} \biggr)|v|^2 \\
 & + \bigl\{ \lambda_{-}M(|\pi_{-}\nu|^2+|\pi_{-}\nu^\perp|^2) + 4\lambda_{+}(|\pi_{+}\nu^\perp|^2+|\pi_{+}\nu|^2)
  \bigr\} |\nu \cdot v|^2
\end{align*}
Taking for instance 
$M = 2+8\sup_{\Omega_T}\frac{\lambda_+}{\lambda_-}\frac{|\pi_+ \nu^\perp|^2}{|\pi_- \nu^\perp|^2}$, 
we easily obtain the desired inequality in the statement {\bf ii}. 
\end{proof}

\subsection{Application to quasilinear $2\times 2$ initial boundary value problems}\label{sectapplQL}
The aim of this section is to use the results of the previous section to handle general quasilinear 
boundary value problems of the form 
\begin{equation}\label{systQL}
\begin{cases}
\dt u + A(u)\dx u + B(t,x)u = f(t,x) & \mbox{in}\quad \Omega_T, \\
u_{\vert_{t=0}} = u^{\rm in}(x) & \mbox{on}\quad \R_+,\\
\Phi(t,\vu)= g(t) & \mbox{on}\quad (0,T),
\end{cases}
\end{equation}
where $u$, $u^{\rm in}$, and $f$ are $\R^2$-valued functions, $g$ and $\Phi$ are  real-valued functions, 
while $A$ and $B$ take their values in the space of $2\times2$ real-valued matrices. 
We also make the following assumption on the hyperbolicity of the system and on the boundary condition.

\begin{assumption}\label{asshypQL}
Let $\mathcal{U}$ be an open set in $\R^2$, which represents a phase space of $u$. 
The following conditions hold. 
\begin{enumerate}
\setlength{\itemsep}{3pt}
\item[{\bf i.}]
$A \in C^\infty(\cU)$.
\item[{\bf ii.}]
For any $u \in \cU$, the matrix $A(u)$ has eigenvalues $\lambda_+(u)$ and $-\lambda_-(u)$ satisfying 
\[
\lambda_{\pm}(u) > 0. 
\]
\item[{\bf iii.}]
There exist a diffeomorphism $\Theta: \cU\to \Theta(\cU)\subset \R^2$ 
and $\nu\in C([0,T])$ such that for any $t\in [0,T]$ and any $u\in \cU$ we have 
\[
\Phi(t,u) = \nu(t) \cdot \Theta(u)
\quad\mbox{and}\quad
\abs{ \nabla_u \Phi(t,u) \cdot {\bf e}_+(u) } > 0,
\]
where  ${\bf e}_+ (u)$ is a unit eigenvector associated to the eigenvalue $\lambda_+(u)$ of $A(u)$.
\end{enumerate}
\end{assumption}

\begin{remark}\label{remphi}
In the case of a linear boundary condition as the we considered for Theorem \ref{theoIBVP1}, 
we have $\Phi(t,u) = \nu(t)\cdot u$ so that by taking $\Theta(u)=u$, the third point of the assumption reduces to 
\[
|\nu(t) \cdot \mathbf{e}_{+}(u)| > 0.
\]
\end{remark}

\begin{remark}
If $\Phi(t,u)=\Phi(u)$ is independent of $t$ and if for some $u^0$ we have 
$\abs{\nabla_{u} \Phi(t,u^0)\cdot {\bf e}_+(u^0)}>0$, 
then by the inverse function theorem and up to shrinking $\cU$ to a sufficiently small neighborhood of $u^0$, 
the existence of a diffeomorphism $\Theta$ satisfying the properties of point {\bf iii} is automatic.
\end{remark}

\begin{example}\label{ex2}
For the nonlinear shallow water equations 
\[
\dt u +A(u)\dx u=0
\]
with $u=(\zeta,q)^{\rm T}$ and $A(u)$ as given by \eqref{eqASW}, 
whose linear version has been considered in Example \ref{ex1}, 
the first two points of the assumption are equivalent to 
\[
h>0, \qquad 
\sqrt{\mathtt{g}{h}} \pm \frac{{q}}{{h}}>0 \qquad (\mbox{with } h=h_0+\zeta).
\]
The condition {\bf iii} of the assumption depends of course on the boundary condition under consideration. 
Let us consider here two important examples:
\begin{itemize}
\item
Boundary condition on the horizontal water flux, that is, $q_{\vert_{x=0}}=g$. 
As seen in Example \ref{ex1} and Remark \ref{remphi}, this corresponds to $\Phi(t,u) = \nu \cdot u$ with 
$\nu=(0,1)^{\rm T}$, and the condition {\bf iii} of the assumption is satisfied. 
\item
Boundary condition on the outgoing Riemann invariant, that is, 
$2(\sqrt{{\mathtt g}h}-\sqrt{{\mathtt g}h_0})+q/h=g$. 
We then have $\Phi(t,u) = \Phi(u) = 2(\sqrt{{\mathtt g}h}-\sqrt{{\mathtt g}h_0})+q/h$ 
and we can take the diffeomorphism defined on $\cU=\{(h,q)\in \R^2\,;\, h>0\}$ by 
\[
\Theta(h,q) = \big( 2(\sqrt{{\mathtt g}h}-\sqrt{{\mathtt g}h_0})+q/h,
 2(\sqrt{{\mathtt g}h}-\sqrt{{\mathtt g}h_0})-q/h\big)^{\rm T},
\]
where $2(\sqrt{{\mathtt g}h}-\sqrt{{\mathtt g}h_0})-q/h$ is the incoming Riemann invariant. 
Then, $\Phi(u) = \nu\cdot \Theta(u)$ with $\nu = (1,0)^{\rm T}$; 
moreover, we compute $\nabla_u \Phi = (1/h)(\lambda^-, 1)^{\rm T}$ so that all the conditions of the third point 
of the assumption are satisfied.
\end{itemize}
\end{example}

The main result is the following.

\begin{theorem}\label{theoIBVP2}
Let $m\geq 2$ be an integer, $B\in L^\infty(\Omega_T)$, $\partial B\in \WW^{m-1}(T)$, and 
assume that Assumption \ref{asshypQL} is satisfied with $\Theta \in C^\infty(\cU)$ and 
$\nu\in W^{m,\infty}(0,T)$. 
If $u^{\rm in }\in H^m(\R_+)$ takes its values in a compact and convex set $\mathcal{K}_0 \subset \cU$ 
and if the data $u^{\rm in}$, $f\in H^m(\Omega_T)$, and $g\in H^m(0,T)$ satisfy the compatibility conditions 
up to order $m-1$ in the sense of Definition \ref{defcompQL} below, 
then there exist $T_1 \in (0,T]$ and a unique solution $u\in \WW^m(T_1)$ 
to the initial boundary value problem \eqref{systQL}. 
Moreover, the trace of $u$ at the boundary $x=0$ belongs to $H^m(0,T_1)$ and $\abs{\vu}_{m,T_1}$ is finite.
\end{theorem}

\begin{remark}
There is a wide literature devoted to the analysis of quasilinear hyperbolic initial boundary value problems. 
For the general multi-dimensional case, assuming that the uniform Kreiss--Lopatinski\u{\i} condition holds, 
the existence is obtained for $m>(d+1)/2+1$, with a loss of $1/2$ derivative with respect to the boundary 
and initial data \cite{RauchMassey, mokrane1987} (see also \cite{benzoniserre2007}). 
Existence for $m>d/2+1$ without loss of derivative is obtained under the additional assumption that 
the system is Friedrichs symmetrizable \cite{schochet1986, metivier2012} but one cannot expect in general 
an $H^m(0,T_1)$ estimate for the trace of the solution at the boundary. 
In the particular one-dimensional case, a $C^1$ solution is constructed in \cite{li1985boundary} 
using the method of characteristics; 
more recently, in the Sobolev setting, it is shown in \cite{petcu2013one} that the general procedure of 
\cite{RauchMassey, mokrane1987} can be implemented in the particular case of the shallow water equations 
with transparent boundary conditions, that is, a boundary data on the outgoing Riemann invariant (see Example \ref{ex2} above): 
for data in $H^{7/2}$, a solution is constructed in $\WW^3(T)$. 
As said in Example \ref{ex2}, our result covers this situation and, by taking advantage of the specificities 
of the one-dimensional case proves existence in $\WW^m(T)$, with $m\geq 2$ and without loss of derivative, 
and provides an $H^m(0,T_1)$ trace estimate.
\end{remark}

\subsubsection{Compatibility conditions}
From the interior equations, denoting $u_k=\dt^k u$, we have 
\[
u_1 = -A(u)\dx u - Bu + f.
\]
More generally, by induction, we have 
\[
u_{k}=c_k(u,B,f),
\]
where $c_k(u,B,f)$ is a smooth function of $u$ and of its space derivatives of order at most $k$, 
and of the time and space derivatives of order lower than $k-1$ of $B$ and $f$. 
For a smooth solution $u$ to \eqref{systQL}, $u_k^{\rm in} = {u_k}_{\vert_{t=0}}$ is therefore given by 
\begin{equation}\label{defu0kter}
u_{k}^{\rm in}=c^{\rm in}_k(u,B,f),
\end{equation}
where $c^{\rm in}_k(u,B,f)=c_k(u,B,f)_{\vert_{t=0}}$. 
The boundary condition $\Phi( t,u_{\vert_{x=0}})=g$ also implies that
\[
\dt^k\Phi( t,u_{\vert_{x=0}}) = \dt^k g.
\]
On the edge $\{t=0,x=0\}$, smooth enough solutions must therefore satisfy 
\[
\begin{cases}
\Phi(0,\vuin)=g_{\vert_{t=0}} & k=0, \\
{u_1^{\rm in}}_{\vert_{x=0}}\cdot \nabla_u \Phi(0,\vuin)+\dt \Phi (0,\vuin)
 = (\dt g)_{\vert_{t=0}} & k=1,
\end{cases}
\]
and more generally, for any $k\geq 1$, 
\begin{equation}\label{compkQL}
{u_k^{\rm in}}_{\vert_{x=0}} \cdot \nabla_u \Phi(0,\vuin) 
 + F_k({u^{\rm in}_{0\leq j\leq k-1}}_{\vert_{x=0}})=(\dt^k g)_{\vert_{t=0}},
\end{equation}
where $F_k({u^{\rm in}_{1\leq j\leq k}}_{\vert_{x=0}})$ is a smooth function of its arguments 
that can be computed explicitly by induction. 

\begin{definition}\label{defcompQL}
Let $m\geq1$ be an integer. 
We say that the data $u^{\rm in}\in H^m(\R_+)$, $f\in H^m(\Omega_T)$, and $g \in H^m(0,T)$ 
for the initial boundary value problem \eqref{systQL} satisfy the compatibility condition at order $k$ 
if the $\{u_j^{\rm in}\}_{j=0}^m$ defined in \eqref{defu0kter} satisfy \eqref{compkQL}. 
We also say that the data satisfy the compatibility conditions up to order $m-1$ if they satisfy 
the compatibility conditions at order $k$ for $k=0,1,\ldots,m-1$. 
\end{definition}

\subsubsection{Proof of Theorem \ref{theoIBVP2}}
Without loss of generality, we can assume that $\Theta(0)=0$.
The first step is to linearize the boundary condition. 
Under Assumption \ref{asshypQL}, this is possible by introducing 
\[
v=\Theta(u),\qquad J(v)=d_{v} (\Theta^{-1}(v)), 
 \quad\mbox{ and }\quad {A}^\sharp(v)=J(v)^{-1} A(\Theta^{-1}(v)) J(v).
\]
Then, $u$ is a classical solution to \eqref{systQL} if and only if $v$ is a classical solution of 
\begin{equation}\label{systQLred}
\begin{cases}
\dt {v} + {A}^\sharp(v)\dx v + J(v)^{-1}B(t,x) \Theta^{-1}(v) = J(v)^{-1}f(t,x) & \mbox{in}\quad \Omega_T, \\
v_{\vert_{t=0}} = \Theta(u^{\rm in}(x)) & \mbox{on}\quad \R_+,\\
\nu(t)\cdot v_{\vert_{x=0}}= g(t) & \mbox{on}\quad (0,T)
\end{cases}
\end{equation}
with $\nu(t)$ as in Assumption \ref{asshypQL}. 
Let $\mathcal{K}_1$ be a compact and convex set in $\R^2$ satisfying 
$\mathcal{K}_0 \Subset \mathcal{K}_1 \Subset \mathcal{U}$. 
Then, there exists a constant $c_0 > 0$ such that for any $u \in \mathcal{K}_1$ and any $t\in[0,T]$ 
we have 
\begin{align*}
\lambda_{\pm}(u) \geq c_0 , \qquad |\nabla_u\Phi(t,u)\cdot \mathbf{e}_{+}(u)| \geq c_0. 
\end{align*}
%
Note that there exists a constant $\delta_0>0$ such that 
$\| v - \Theta(u^{\rm in})\|_{L^\infty} \leq \delta_0$ implies that 
$u = \Theta^{-1}(v)$ takes its values in $ \mathcal{K}_1$. 
We therefore construct a solution $v$ to \eqref{systQLred} satisfying 
$\| v(t)-\Theta(u^{\rm in})\|_{L^\infty} \leq \delta_0$ for $0\leq t\leq T_1$. 
The solution is classically constructed using the iterative scheme 
\begin{equation}\label{systiter}
\begin{cases}
\dt v^{n+1} + {A}^\sharp(v^n)\dx v^{n+1}  = f^n& \mbox{in}\quad \Omega_T, \\
{v^{n+1}}_{\vert_{t=0}} = \Theta(u^{\rm in}(x)) & \mbox{on}\quad \R_+, \\
 \nu(t)\cdot {v^{n+1}}_{\vert_{x=0}} = g(t) & \mbox{on}\quad (0,T), 
\end{cases}
\end{equation}
for all $n\in \N$ and with
$$
f^n(t,x)=J(v^n)^{-1}f(t,x) - J(v^n)^{-1}B(t,x)\Theta^{-1}(v^{n}).
$$
For the first iterate $u^0$, we choose a function $u^0\in H^{m+1/2}(\R\times \R_+)$ such that 
\[
(\dt^k u^0)_{\vert_{t=0}} = u_k^{\rm in} \quad\mbox{for}\quad k=0,1,\ldots,m 
\]
with $u_k^{\rm in}$ as defined in \eqref{defu0kter}. 
Such a choice ensures along a classical procedure \cite{metivier2001,metivier2012} that 
the data $(\Theta(u^{\rm in}), f^n, g)$ are compatible for the linear initial boundary value problem 
\eqref{systiter} in the sense of Definition \ref{defcompVC}. 
Moreover, $\opnorm{ v^n (0) }_m$ is independent of $n$, and there exists therefore $K_0$ such that
\[
\frac{1}{c_0}, \opnorm{ v^n(0) }_m, \| {A}^\sharp(v^n)\|_{L^\infty(\Omega_{T_1})}, 
 \| {A}^\sharp(v^n)^{-1}\|_{L^\infty(\Omega_{T_1})} \leq K_0,
\]
as long as $v^n$ satisfies $\|v^n(t) - \Theta(u^{\rm in})\|_{L^\infty} \leq \delta_0$ for $0\leq t\leq T_1$. 
We prove now that for $M$ large enough and $T_1$ small enough, for any $n \in \N$ we have 
\begin{equation}\label{assert}
\begin{cases}
\| v^n\|_{\WW^m(T_1)} + |{ v^n}_{\vert_{x=0}}|_{m,T_1} \leq M, \\
\| v^n(t)-\Theta(u^{\rm in})\|_{L^\infty} \leq \delta_0 \quad\mbox{for}\quad 0 \leq t\leq T_1.
\end{cases}
\end{equation}
The main tool to prove this assertion is to apply Theorem \ref{theoIBVP1} to \eqref{systiter}. 
In order to do so, we first need to check that Assumption \ref{asshyp} is satisfied. 
The only non trivial point to check is the third condition of this assumption. 
The fact that this is a consequence of Assumption \ref{asshypQL} for the original system \eqref{systQL} 
is proved in the following lemma.

\begin{lemma}
For any $ v \in \Theta(\cU)$, the matrix ${A}^\sharp( v)$ has two eigenvalues $\pm {\lambda}^\sharp_\pm(v)$ 
and associated eigenvectors ${{\bf e}}^\sharp_\pm(v)$ given by 
\[
{\lambda}^\sharp_\pm(v) = \lambda_\pm(\Theta^{-1}(v))\quad \mbox{ and }\quad
{{\bf e}}^\sharp_\pm(v) = J(v)^{-1}{\bf e}_\pm(\Theta^{-1}(v)).
\]
Moreover, denoting $u=\Theta^{-1}(v)$ we have 
\[
\nu(t)\cdot {\bf e}^\sharp_+(v) = \nabla_u\Phi(t,u)\cdot {\bf e}_+(u).
\]
\end{lemma}

\begin{proof}[Proof of the lemma]
The first part of the lemma is straightforward. 
For the second point, just notice that by definition of $\Theta$, one has 
$\nabla_u\Phi(t,u) = (\Theta'(u))^{\rm T}\nu(t)$. 
Since moreover $\Theta'(u) = (d_v (\Theta^{-1}(v)) )^{-1} = J(v)^{-1}$, we have 
\[
\nabla_u\Phi(t,u) \cdot {\bf e}_+(u) = \nu(t) \cdot J(v)^{-1} {\bf e}_+(\Theta^{-1}(v))
\]
and the result follows from the first point.
\end{proof}

We can therefore use Theorem \ref{theoIBVP1} to prove \eqref{assert} by induction. 
Since it is satisfied for $n=0$ for a suitable $M$ and $T_1$, 
we just need to prove that it holds at rank $n+1$ if it holds at rank $n$. 
There is $K=K(M)$ such that 
\[
\|{A}^\sharp(v^n)\|_{W^{1,\infty}(\Omega_{T_1})}, \|\partial ({A}^\sharp(v^n)) \|_{\WW^{m-1}(T_1)} \leq K.
\]
Taking a greater $K$ if necessary, we can assume also that 
$\|B\|_{L^\infty(\Omega_T)}$ and $\|\partial B\|_{\WW^{m-1}(T)} \leq K$ and therefore that 
\[
\opnorm{f^n(t)}_m\leq C(K)(1+\opnorm{f(t)}_m).
\]
It follows therefore from Theorem \ref{theoIBVP1} that 
\begin{align*}
& \|v^{n+1} \|_{\WW^m(T_1)} + |{v^{n+1}}_{\vert_{x=0}}|_{m,T_1} \\
&\leq C(K_0)e^{C(K)T_1} \Big( 1 + |g|_{H^m(0,T_1)}
 + |\vf|_{m-1,T_1} + C(K)\int_0^{T_1} (1+\opnorm{ f(t) }_m) {\rm d}t \Big). 
\end{align*}
We also have 
\[
\|v^{n+1}(t)-\Theta(u^{\rm in})\|_{L^\infty} \leq \|\dt v^{n+1}\|_{L^\infty(\Omega_{T_1})}T_1
 \leq C\|v^{n+1}\|_{\WW^2(T_1)}T_1.
\]
Therefore, by choosing $M$ large enough and $T_1$ small enough the claim is proved. 
The convergence is classically obtained by proving that $\{v^n\}_n$ is a Cauchy sequence 
and, therefore, convergent in $L^2$, and that the limit is actually in $\WW^m(T)$. 
We omit the details.

\subsection{Variable coefficients $2\times 2$ boundary value problems on moving domains}\label{sectVCm}
We now turn to consider initial boundary value problems that are still cast on a half-line, 
but instead of $\R_+$, we now consider $(\ux(t),+\infty)$, 
where the left boundary $\ux(t)$ is a time dependent function. 
We consider first linear problems with variable coefficients. 
For the sake of simplicity and to prepare the ground for applications to quasilinear systems, 
we consider a slightly less general system of equations than in \eqref{systVC}: 
the variable coefficient matrix $A(t,x)$ is of the form $A(\uU(t,x))$. 
More precisely, 
\begin{equation}\label{IBVPm}
\begin{cases}
\dt U + A(\uU)\dx U + {\mathtt B} U = F & \mbox{in}\quad (\ux(t),\infty) \quad\mbox{for}\quad t\in(0,T), \\
U_{\vert_{t=0}} = u^{\rm in}(x) & \mbox{on}\quad (0,\infty), \\
\nu(t)\cdot U_{\vert_{x=\ux(t)}} = g(t) & \mbox{on}\quad (0,T),
\end{cases}
\end{equation}
where without loss of generality we assumed $\ux(0)=0$. 
The first thing to do is of course to transform this initial boundary value problem on a moving domain 
into another one cast on a fix domain, say, $\R_+$. 
This is done through a diffeomorphism $\varphi(t,\cdot)$ that maps at all times $\R_+$ onto $(\ux(t),\infty)$  
and such that for any $t$, we have $\varphi(t,0)=\ux(t)$. 
Several choices are possible for $\varphi$ and shall be discussed later. 
At this point, we just assume that $\varphi\in C^1(\Omega_T)$ and that $\varphi(0,x)=x$. 
Composing the interior equation in \eqref{IBVPm} with the diffeomorphism $\varphi$ 
to work on the fix domain $(0,\infty)$, introducing the notations 
\[
u = U\circ \varphi, \qquad \uu = \uU\circ \varphi, \qquad \dt^\varphi u = (\dt U)\circ\varphi, 
 \qquad \dx^\varphi u = (\dx U)\circ\varphi,
\]
so that, in particular, 
\begin{equation}\label{dtphi}
\dx^\varphi=\frac{1}{\dx \varphi}\dx, \qquad \dt^\varphi= \dt - \frac{\dt \varphi}{\dx \varphi}\dx,
\end{equation}
and writing $B = {\mathtt B}\circ \varphi$ and $f=F\circ\varphi$, we obtain the following equation for $u$ 
\begin{equation}\label{equ}
\dt^\varphi u + A(\uu)\dx^\varphi u + B(t,x) u = f(t,x).
\end{equation}
The initial boundary value problem on a moving domain \eqref{IBVPm} can therefore be recast as 
an initial boundary value problem on a fix domain 
\begin{equation}\label{IBVPmT}
\begin{cases}
\dt u + \cA(\uu,\partial\varphi)\dx u + B(t,x)u  = f(t,x) & \mbox{in}\quad \Omega_T, \\
u_{\vert_{t=0}} = u^{\rm in}(x) & \mbox{on}\quad \R_+, \\
\nu(t)\cdot u_{\vert_{x=0}} = g(t) & \mbox{on}\quad (0,T),
\end{cases}
\end{equation}
with 
$$
\cA(\uu,\partial\varphi)=\frac{1}{\dx \varphi}\big( A(\uu) - (\dt\varphi) \mbox{Id}\big).
$$

If we want to apply Theorem \ref{theoIBVP1} to construct solutions to \eqref{IBVPmT}, 
it is necessary to get some information on the regularity of $\varphi$, 
which is of course related to the properties of the boundary coordinate $\ux(t)$. 
A direct application of Theorem \ref{theoIBVP1} requires that $\partial\varphi$ be in $\WW^m(T)$ 
in order to get solutions $u$ in $\WW^m(T)$. 
Using Alinhac's good unknown \cite{alinhac1989}, it is however possible to obtain refined regularity estimates, 
as shown in the following theorem which requires only the following assumption.

\begin{assumption}\label{asshypm}
We have $\uu \in W^{1,\infty}(\Omega_T)$, $\ux\in C^1([0,T])$, $\ux(0)=0$, 
and the diffeomorphism $\varphi$ is in $C^1(\Omega_T)$. 
Moreover, there exists a constant $c_0>0$ such that the following three conditions hold.
\begin{enumerate}
\setlength{\itemsep}{3pt}
\item[{\bf i.}]
There exists an open set $\cU\subset \R^2$ such that $A \in C^\infty(\cU)$ and that for any $u\in\cU$, 
the matrix $A(u)$ has eigenvalues $\lambda_+(u)$ and $-\lambda_-(u)$. 
Moreover, $\uu$ takes its values in a compact set $\mathcal{K}_0 \subset \cU$ and for any $(t,x)\in\Omega_T$ we have 
\[
\lambda_\pm(\uu(t,x))\mp \dt \varphi (t,x)\geq c_0
 \quad\mbox{ and }\quad \lambda_\pm(\uu(t,x))\geq c_0.
\]

\item[{\bf ii.}]
Denoting by ${\bf e}_+ (u)$ a unit eigenvector associated to the eigenvalue $\lambda_+(u)$ of $A(u)$, 
for any $t \in [0,T]$ we have 
\[
|\nu(t) \cdot {\bf e}_+(\uu(t,0))| \geq c_0.
\]

\item[{\bf iii.}]
The Jacobian of the diffeomorphism is uniformly bounded from below and from above, that is, 
for any $(t,x)\in\Omega_T$ we have 
\[
c_0 \leq \dx \varphi (t,x) \leq \frac{1}{c_0}.
\]
\end{enumerate}
\end{assumption}

\begin{example}\label{ex3}
Considering as in Example \ref{ex1} the linearized shallow water equations, 
but this time on a moving domain, Assumption \ref{asshypm} reduces to the conditions 
$\underline{h},\underline{q} \in W^{1,\infty}(\Omega_T)$ and 
\[
\underline{h}(t,x) \geq c_0, \quad
\sqrt{\mathtt{g}\underline{h}(t,x)} \pm \Bigl( 
 \frac{\underline{q}(t,x)}{\underline{h}(t,x)} - \dt \varphi(t,x) \Bigr) \geq c_0, \quad
\sqrt{\mathtt{g}\underline{h}(t,x)} \pm 
 \frac{\underline{q}(t,x)}{\underline{h}(t,x)} \geq c_0 
\]
with some positive constant $c_0$ independent of $(t,x) \in \Omega_T$. 
\end{example}

\begin{theorem}\label{theoIBVP3}
Let $m\geq1$ be an integer, $T>0$, and assume that Assumption \ref{asshypm} is satisfied for some $c_0>0$. 
Assume moreover that there are two constants $0<K_0\leq K$ such that 
\[
\begin{cases}
\frac{1}{c_0}, \opnorm{ \partial\widetilde{\varphi}(0) }_{m-1}, |\nu|_{L^\infty(0,T)}, 
 \|\partial\varphi\|_{L^\infty(\Omega_T)}, \|A\|_{L^\infty(\mathcal{K}_0)} \leq K_0, \\
\|\partial \widetilde\varphi \|_{\WW^{m-1}(T)}, \|\dt \varphi \|_{H^m(\Omega_T)}, 
 |(\partial^m\varphi)_{\vert_{x=0}}|_{L^\infty(0,T)} \leq K, \\
\|\uu \|_{W^{1,\infty}(\Omega_T) \cap \WW^{m}(T)}, \|B\|_{W^{1,\infty}(\Omega_T)}, 
 \|\partial B\|_{\WW^{m-1}(T)}, |\nu|_{W^{1,\infty} \cap W^{m-1,\infty}(0,T)},
 |\dt^m\nu|_{L^2(0,T)} \leq K,
\end{cases}
\]
where $\widetilde \varphi(t,x) = \varphi(t,x)-x$. 
Then, for any data $u^{\rm in }\in H^m(\R_+)$, $f\in H^m(\Omega_T)$, and $g\in H^m(0,T)$ satisfying the 
compatibility conditions up to order $m-1$ in the sense of Definition \ref{defcompVC}, 
there exists a unique solution $u\in \WW^m(T)$ to \eqref{IBVPmT}. 
Moreover, the following estimate holds for any $t\in[0,T]$ and any $\gamma \geq C(K)$: 
\begin{align*}
& \opnorm{ u(t) }_{m,\gamma} + \biggl(\gamma\int_0^t\opnorm{ u(t') }_{m,\gamma}^2{\rm d}t'\biggr)^\frac12
 + |\vu|_{m,\gamma,t} \\
&\leq  C(K_0)\bigl( 
 (1+|\dt^m\nu|_{L^2(0,t)})\opnorm{ u(0) }_m + |g|_{H_\gamma^m(0,t)} + |\vf|_{m-1,\gamma,t}
  + S_{\gamma,t}^*(\opnorm{ f(\cdot) }_m) \bigr).
\end{align*}
Particularly, we have 
\begin{align*}
& \opnorm{ u(t) }_m + |\vu|_{m,t} \\
&\leq  C(K_0)e^{C(K)t}\biggl( 
 (1+|\dt^m\nu|_{L^2(0,t)})\opnorm{ u(0) }_m + |g|_{H^m(0,t)} + |\vf|_{m-1,t}
  + \int_0^t \opnorm{ f(t') }_m{\rm d}t' \biggr).
\end{align*}
\end{theorem}


\subsubsection{Proof of Theorem \ref{theoIBVP3}}
A direct estimate in $\WW^m(T)$ for the solution of \eqref{IBVPmT} through Theorem \ref{theoIBVP1} 
is not possible because it would require that $\partial^2 \varphi\in \WW^{m-1}(T)$ while, 
under the assumptions made in the statement of the theorem, we only have $\partial^2 \varphi\in \WW^{m-2}(T)$. 
The key step is to derive a $\WW^{m-1}(T)$ estimate on $u$ as well as on 
$\dt^\varphi u = \dt u - (\dt\varphi) \dx^\varphi u$.

\begin{proposition}\label{propAl}
Under the assumptions of Theorem \ref{theoIBVP3}, there is a unique solution $u\in \WW^{m-1}(T)$ to 
\eqref{IBVPmT} satisfying 
\begin{align}\label{eqmm1}
\opnorm{ u(t) }_{0} + |\vu|_{0,t}
&\leq C(K_0)e^{C(K)t}\biggl( \opnorm{ u(0) }_{0} + |g|_{H^0(0,t)} 
 + \int_0^t \opnorm{ f(t') }_{0} {\rm d}t' \biggr)
\end{align}
in the case $m=1$ and 
\begin{align}\label{eqmm1.5}
&\opnorm{ u(t) }_{m-1} + |\vu|_{m-1,t} \\
&\leq C(K_0)e^{C(K)t}\biggl( \opnorm{ u(0) }_{m-1} + |g|_{H^{m-1}(0,t)} + |\vf|_{m-2,t}
 + \int_0^t \opnorm{ \dt f(t') }_{m-2} {\rm d}t' \biggr) \nonumber
\end{align}
in the case $m \geq 2$. 
Moreover, $\dt^\varphi u\in \WW^{m-1}(T)$ and we have 
\begin{align}\label{eqmm2}
& \opnorm{ \dt^\varphi u(t) }_{m-1,\gamma} 
 + \biggl( \gamma\int_0^t\opnorm{ \dt^\varphi u(t') }_{m-1,\gamma}^2{\rm d}t' \biggr)^\frac12
+ |(\dt^\varphi u)_{\vert_{x=0}}|_{m-1,\gamma,t} \\
&\leq C(K_0)\bigl( 
 (1+|\dt^m\nu|_{L^2(0,t)})\opnorm{ u(0) }_m + |g|_{H_\gamma^m(0,t)} + |\vf|_{m-1,\gamma,t}
  + S_{\gamma,t}^*(\opnorm{ f(\cdot) }_m) \bigr) \nonumber \\
&\quad + C(K)\bigl( S_{\gamma,t}^*(\opnorm{ u(\cdot) }_m) + |\vu|_{m-1,\gamma,t} \bigr). \nonumber
\end{align}
\end{proposition}

\begin{proof}[Proof of the proposition]
{\bf Step 1.}
We first show that there exists a solution $u\in \WW^{m-1}(T)$ to \eqref{IBVPmT} satisfying 
\eqref{eqmm1}--\eqref{eqmm1.5}. 
A direct application of Theorem \ref{theoIBVP1} almost yields the result, but with a constant $C(K')$ 
bigger than $C(K)$ in the sense that it depends on $\|\partial\varphi\|_{W^{1,\infty}(\Omega_T)}$ 
instead of $\|\partial\varphi\|_{L^{\infty}(\Omega_T)}$. 
The improved estimate claimed in \eqref{eqmm1}--\eqref{eqmm1.5} is made possible by the particular structure of 
the matrix $\cA(\uu,\partial\varphi)$, as shown in the following lemma which improves Lemma \ref{lemsymmetrizer}.

\begin{lemma}\label{lemsymmetrizerbis}
Suppose that Assumption \ref{asshypm} is satisfied. 
Then, there exist a symmetrizer $\cS\in W^{1,\infty}(\Omega_T)$ and constants 
$\alpha_0,\alpha_1$ and $\beta_0,\beta_1,\beta_2$ such that Assumption \ref{assVC} 
is satisfied for the initial boundary value problem \eqref{IBVPmT}. 
Moreover, we have 
\begin{align*}
& \mathfrak{c}_0 \leq C\Bigl(\frac{1}{c_0}, \|A(\uu^{\rm in})\|_{L^\infty(\R_+)}, 
 \|(\dt \varphi)_{\vert_{t=0}}\|_{L^\infty(\R_+)} \Bigr), \\
& \mathfrak{c}_1 \leq C\Bigl(\frac{1}{c_0}, \|A(\uu)\|_{L^\infty(\Omega_T)},
 \|\dt\varphi\|_{L^\infty(\Omega_T)}\Bigr),
\end{align*}
where $\uu^{\rm in}=\uu_{\vert_{t=0}}$ and ${\mathfrak c}_0$ and ${\mathfrak c}_1$ are as defined in 
Proposition \ref{propNRJ1}, and 
\[
\frac{\beta_2}{\beta_0} \leq C\Bigl( \frac{1}{c_0}, \|A(\uu)\|_{W^{1,\infty}(\Omega_T)},
 \|\dt\varphi\|_{L^{\infty}(\Omega_T)}, \|B\|_{L^\infty(\Omega_T)} \Bigr).
\]
\end{lemma}

\begin{proof}[Proof of the lemma]
The proof is an adaptation of the proof of Lemma \ref{lemsymmetrizer}. 
We still denote by $\pi_\pm$ the eigenprojector associated to the eigenvalues $\pm\lambda_{\pm}$ of $A(\uu)$. 
As a symmetrizer for $\cA(\uu,\varphi)$, we choose 
\[
\cS = (\dx \varphi) \bigl( \pi_+^{\rm T}\pi_+ + M \pi_-^{\rm T} \pi_- \bigr)
\]
with sufficiently large $M$. 
Since we have 
\begin{align*}
\beta_2
&= \|\dt \cS + \dx (\cS\cA) - 2\cS B\|_{L^\infty(\Omega_T)} \\
&= \|(\dx \varphi)\dt S + \dx (SA) - (\dt \varphi) \dx S - 2(\dx \varphi)SB\|_{L^\infty(\Omega_T)},
\end{align*}
where we denoted $S= \pi_+^{\rm T}\pi_+ + M \pi_-^{\rm T}\pi_-$, and since $\pi_\pm$ depends only on $A(\uu)$, 
we deduce the desired results. 
\end{proof}

Using  Lemma \ref{lemsymmetrizerbis} instead of Lemma \ref{lemsymmetrizer} in the proof of 
Theorem \ref{theoIBVP1} in the particular case of the initial boundary value problem \eqref{IBVPmT}, 
we get \eqref{eqmm1}--\eqref{eqmm1.5}.

\medskip
\noindent
{\bf Step 2.}
We prove here an extra regularity on $\dt^\varphi u $ that implies the inequality stated in the theorem. 
The main tool to get this extra regularity is Alinhac's good unknown \cite{alinhac1989}, 
which removes the loss of derivative due to the dependence on $\varphi$ in the coefficients 
of the initial boundary value problem \eqref{IBVPmT}. 
Differentiating with respect to time the interior equation in \eqref{IBVPmT}, 
and writing $\dot u=\dt u$, $\dot f=\dt f$, etc., we get 
\begin{equation}\label{eqbad}
\dt\dot{u} + \cA(\uu,\partial \varphi)\dx\dot{u} + A'(\uu)[\dot{\uu}]\dx^\varphi u
 + \mathcal{M}(\uu,\partial\varphi,\dx u)\partial\dot{\varphi} + B\dot{u} = \dot{f} - \dot{B}u
\end{equation}
with 
\[
\mathcal{M}(u,\partial\varphi,\dx u)\partial \dot \varphi
= - \bigl( (\dx\dot{\varphi}) \mathcal{A}(\uu,\partial \varphi)
 + (\dt\dot{\varphi}) \mbox{Id} \bigr) \dx^\varphi u.
\]
Obviously, the term $\mathcal{M}(\uu,\partial\varphi,\dx u)\partial\dot{\varphi}$ is responsible 
for the loss of one derivative, in the sense that a control of $\varphi $ in $ \WW^{m+1}(T)$ 
is required to control the $\WW^m(T)$ norm of $u$. 
This singular dependence is removed by working with Alinhac's good unknown 
$\dot{u}^\varphi = \dot{u} - \dot{\varphi} \dx^\varphi u$ instead of $\dot{u}$. 
The notations $\dot{f}^\varphi$ and $\dot{B}^\varphi$ are defined similarly. The following lemma is due to Alinhac \cite{alinhac1989} and can be checked by simple computations.
\begin{lemma}\label{lemeq}
With $\dot{u}^\varphi = \dot{u} - \dot{\varphi}\dx^\varphi u$, 
the equation \eqref{eqbad} can be rewritten under the form 
\[
\dt\dot{u}^\varphi + \cA(\uu,\partial \varphi)\dx\dot{u}^\varphi + A'(\uu)[\dot{\uu}^\varphi] \dx^\varphi u
 + B\dot{u}^\varphi = \dot{f}^\varphi - \dot{B}^\varphi u.
\]
\end{lemma}

\begin{remark}
We use the notations $\dot{u} = \dt u$ and $\dot{u}^\varphi = \dt^\varphi u$ to underline the fact that 
this is a general procedure that works for any linearization operator, not only time differentiation. 
\end{remark}

We can use \eqref{equ} to write 
\[
\dx^\varphi u = A(\uu)^{-1}(f-Bu-\dot u^\varphi),
\]
so that the lemma yields
\[
\dt\dot{u}^\varphi + \cA(\uu,\partial \varphi)\dx\dot{u}^\varphi + B_{(1)}\dot{u}^\varphi = f_{(1)},
\]
where
\begin{equation}\label{equx}
\begin{cases}
B_{(1)} = B - A'(\uu)[\dot{\uu}^\varphi]A(\uu)^{-1}, \\
f_{(1)} = \dot{f}^\varphi - A'(\uu)[\dot{\uu}^\varphi] A(\uu)^{-1}f - 
 (\dot{B}^\varphi - A'(\uu)[\dot{\uu}^\varphi] A(\uu)^{-1}B)u.
\end{cases}
\end{equation}
Therefore, $\dot{u}^\varphi = \dt^\varphi u$ solves an interior equation similar to those considered 
in Theorem \ref{theoIBVP1}. 
Let us now consider the initial and boundary conditions for $\dot{u}^\varphi$. 
For the initial condition, we have 
\[
(\dot u^\varphi)_{\vert_{t=0}} = u^{\rm in}_{(1)}
 \quad\mbox{with}\quad 
{u^{\rm in}_{(1)}} = (\dt u)_{\vert_{t=0}} - (\dt\varphi)_{\vert_{t=0}}\dx u^{\rm in}.
\]
For the boundary condition, let us differentiate with respect to time the boundary condition in \eqref{IBVPmT} 
to obtain $\nu(t) \cdot \dt u_{\vert_{x=0}} = \dt g - \nu'(t)\cdot u_{\vert_{x=0}}$ or equivalently 
\[
\nu(t) \cdot (\dot u^\varphi +\dot \ux \dx^\varphi u)_{\vert_{x=0}} = \dt g - \nu'(t) \cdot u_{\vert_{x=0}}.
\]
Using \eqref{equ}, this yields 
\[
\nu(t) \cdot \bigl((\mbox{Id}-\dot{x} A(\uu)^{-1})\dot u^\varphi\bigr)_{\vert_{x=0}}
 = \dt g - \nu'(t)\cdot u_{\vert_{x=0}}-\dot{x} \nu(t) \cdot A(\uu)^{-1} (f-Bu)_{\vert_{x=0}}.
\]
It follows that $\dot u^\varphi$ satisfies an initial boundary value problem of the form \eqref{systVC}, namely, 
\begin{equation}\label{systVCm}
\begin{cases}
\dt\dot{u}^\varphi + \cA(\uu,\partial\varphi)\dx\dot{u}^\varphi + B_{(1)} \dot{u}^\varphi = f_{(1)}
 & \mbox{in}\quad \Omega_T, \\
\dot{u}^\varphi_{\vert_{t=0}} = u^{\rm in}_{(1)} & \mbox{on}\quad \R_+, \\
\nu_{(1)}(t) \cdot \dot u^\varphi_{\vert_{x=0}} = g_{(1)} & \mbox{on}\quad (0,T),
\end{cases}
\end{equation}
where $f_{(1)}$ and $B_{(1)}$ are as in \eqref{equx} and
\begin{equation}\label{nu1}
\begin{cases}
g_{(1)} = \dt g - (\dt \nu) \cdot u_{\vert_{x=0}} - \dot \ux \nu \cdot A(\uu)^{-1} (f-Bu)_{\vert_{x=0}}, \\
\nu_{(1)} = (\mbox{Id}-\dot \ux A(\uu_{\vert_{x=0}})^{-1})^{\rm T}\nu. 
\end{cases}
\end{equation}
Concerning the boundary condition, we have the following lemma which shows that 
the initial boundary value problem \eqref{systVCm} satisfies condition {\bf iii} in Assumption \ref{asshyp}.

\begin{lemma}\label{mdBC}
Under Assumption \ref{asshypm}, for any $t\in[0,T]$ we have 
\[
|\nu_{(1)}(t) \cdot \mathbf{e}_+(\uu(t,0))| \geq \frac{c_0^2}{\lambda_{+}(\uu(t,0))}.
\]
\end{lemma}

\begin{proof}
We see that 
\begin{align*}
\nu_{(1)}(t) \cdot \mathbf{e}_+(\uu(t,0))
&= \nu(t) \cdot (\mbox{Id}-\dot{x}(t) A(\uu(t,0))^{-1})\mathbf{e}_+(\uu(t,0)) \\
&= \Bigl( 1-\frac{\dot{x}(t)}{\lambda_{+}(\uu(t,0))} \Bigr) \nu(t) \cdot \mathbf{e}_+(\uu(t,0)).
\end{align*}
Since $\dot{x}(t)=(\dt\varphi)(t,0)$, 
this gives the desired inequality. 
\end{proof}

Here, we see that 
\[
|\nu_{(1)}|_{L^\infty(0,T)} \leq C(K_0), \qquad
\|B_{(1)}\|_{L^\infty(\Omega_T)} \leq C(K)
\]
and that in the case $m\geq2$ 
\[
\|\partial B_{(1)}\|_{\WW^{m-2}(T)}, |\nu_{(1)}|_{W^{m-1,\infty}(0,T)} \leq C(K).
\]
Therefore, we can apply the result in Step 1 to obtain 
\begin{align}\label{eqmm3}
& \opnorm{ \dot{u}^\varphi(t) }_{m-1,\gamma}
 + \biggl( \gamma\int_0^t \opnorm{ \dot{u}^\varphi(t') }_{m-1,\gamma}^2{\rm d}t' \biggr)^\frac12
 + |{\dot{u}^\varphi}_{\;\,\vert_{x=0}}|_{m-1,t} \\
&\leq  C(K_0)\bigl( \opnorm{ \dot{u}^\varphi(0) }_{m-1} + |g_{(1)}|_{H_\gamma^{m-1}(0,t)}
 + |f_{(1) \vert_{x=0}}|_{m-2,\gamma,t} + S_{\gamma,t}^*(\opnorm{ f_{(1)}(\cdot) }_{m-1}) \bigr), \nonumber
\end{align}
where the term $|f_{(1) \vert_{x=0}}|_{m-2,\gamma,t}$ is dropped in the case $m=1$. 
Here, we have 
\[
\begin{cases}
\opnorm{ \dot{u}^\varphi(0) }_{m-1} \leq C(K_0) \opnorm{ u(0) }_m, \\
\opnorm{ f_{(1)}(t) }_{m-1} \leq C(K)( \opnorm{ f(t) }_m + \opnorm{ u(t) }_{m-1} ), \\
|f_{(1) \vert_{x=0}}|_{m-2,\gamma,t} \leq C(K)( |\vf|_{m-1,\gamma,t} + |\vu|_{m-1,\gamma,t} ).
\end{cases}
\]
Concerning the term $|g_{(1)}|_{H^{m-1}(0,t)}$, especially, the term $(\dt\nu) \cdot u_{\vert_{x=0}}$ 
we need to estimate it carefully, because we do not assume $\nu \in W^{m,\infty}(0,T)$. 
In the case $m=1$, we estimate it directly as 
\[
|(\dt\nu) \cdot u_{\vert_{x=0}}|_{L_\gamma^2(0,t)} \leq C(K)|\vu|_{L_\gamma^2(0,t)}.
\]
In the case $m\geq2$, we see that 
\begin{align*}
|(\dt\nu) \cdot u_{\vert_{x=0}}|_{H_\gamma^{m-1}(0,t)}
&\leq |\nu|_{W^{m-1,\infty}(0,t)}|\vu|_{m-1,\gamma,t}
 + |\dt^m\nu|_{L^2(0,t)} \sup_{t'\in[0,t]}e^{-\gamma t'}|u(t',0)| \\
&\leq C(K)|\vu|_{m-1,\gamma,t} + C|\dt^m\nu|_{L^2(0,t)} \opnorm{ u(0) }_{m-1},
\end{align*}
where we used $\sup_{t'\in[0,t]}e^{-\gamma t'}|u(t',0)| \leq C( \|u(0)\|_{H^1} + \gamma^{-\frac12}|\vu|_{1,\gamma,t})$, 
which is a simple consequence of \eqref{ineq5} in Lemma \ref{estuuu}. 
In any case, we have 
\begin{align*}
|g_{(1)}|_{H_\gamma^{m-1}(0,t)} \leq 
& |g|_{H_\gamma^m(0,t)}  + C|\dt^m\nu|_{L^2(0,t)} \opnorm{ u(0) }_{m-1}  + C(K)(|\vu|_{m-1,t}+|\vf|_{m-1,t}).
\end{align*}
Therefore, by \eqref{eqmm3} we obtain 
\begin{align*}
&  \opnorm{ \dot{u}^\varphi(t) }_{m-1,\gamma}
 + \biggl( \gamma\int_0^t \opnorm{ \dot{u}^\varphi(t') }_{m-1,\gamma}^2{\rm d}t' \biggr)^\frac12
 + |{\dot{u}^\varphi}_{\;\,\vert_{x=0}}|_{m-1,t} \\
& \leq C(K_0)\bigl( (1+|\dt^m\nu|_{L^2(0,t)})\opnorm{ u(0) }_m + |g|_{H^m(0,t)} \bigr) \\
&\quad
  + C(K)\bigl( |\vf|_{m-1,t} + |\vu|_{m-1,t}
   + S_{\gamma,t}^*( \opnorm{ f(\cdot) }_m) + S_{\gamma,t}^*(\opnorm{ u(\cdot) }_{m-1})  \bigr), 
\end{align*}
which shows $\dt^\varphi u \in \WW^{m-1}(T)$.

\medskip
\noindent
{\bf Step 3.}
Finally, we improve the above inequality to show \eqref{eqmm2}. 
It follows directly from Lemma \ref{lemeq} that we have also the equation for $\dot{u}^\varphi$ of the form 
\[
\dt\dot{u}^\varphi + \cA(\uu,\partial\varphi)\dx\dot{u}^\varphi = \widetilde{f}_{(1)}
\]
with 
\[
\widetilde{f}_{(1)} = \dt^\varphi f - A'(\uu)[\dt^\varphi\uu]\dx^\varphi u - \dt^\varphi(Bu).
\]
Moreover, we have \eqref{eqmm3} with $f_{(1)}$ replaced by $\widetilde{f}_{(1)}$. 
In order to give modified estimates for $\widetilde{f}_{(1)}$ and $g_{(1)}$, in the case of $m\geq2$ 
we use the following expressions 
\begin{align*}
\partial^\alpha \widetilde{f}_{(1)}
&= \dt^\varphi\partial^\alpha f
 + [\partial^\alpha,\dt^\varphi](\dt^\varphi u + A(\uu)\dx^\varphi u + Bu) \\
&\quad
 - \partial^\alpha(A'(\uu)[\dt^\varphi\uu]\dx^\varphi u + \dt^\varphi(Bu)), \\
\dt^k g_{(1)}
&= \dt^k(\dt g - (\dt \nu) \cdot u_{\vert_{x=0}})
 - \dot \ux \nu \cdot A(\uu)^{-1} \dt^k (f-Bu)_{\vert_{x=0}} \\
&\quad
 - [\dt^k, \dot \ux \nu \cdot A(\uu)^{-1}](\dt^\varphi u + A(\uu)\dx^\varphi u)_{\vert_{x=0}},
\end{align*}
where we used \eqref{equ}. 
These expressions together with Lemma \ref{ineq1} give 
\begin{align*}
& \opnorm{ \widetilde{f}_{(1)}(t) }_{m-1} \leq C(K_0)\opnorm{ f(t) }_m + C(K)\opnorm{ u(t) }_m, \\
& |g_{(1)}|_{H_\gamma^{m-1}(0,t)} + |\widetilde{f}_{(1) \vert_{x=0}}|_{m-2,\gamma,t} \\
& \leq C(K_0)( |\dt^m\nu|_{L^2(0,t)}\opnorm{ u(0) }_{m-1} + |g|_{H^m(0,t)} + |\vf|_{m-1,t}) 
 + C(K)|\vu|_{m-1,t},
\end{align*}
which yields \eqref{eqmm2}. 
The proof of Proposition \ref{propAl} is complete. 
\end{proof}

In order to  conclude the proof of Theorem \ref{theoIBVP3}, we need to show that Proposition \ref{propAl} 
provides a control of $u$ in $\WW^{m}(T)$.

\begin{lemma}\label{lemequivn}
Under the assumptions of Theorem \ref{theoIBVP3}, if $u$ solves \eqref{IBVPmT}, then we have 
\begin{align*}
&\opnorm{ \partial u(t) }_{m-1,\gamma}
 + \biggl( \gamma\int_0^t \opnorm{ \partial u(t') }_{m-1,\gamma}^2{\rm d}t' \biggr)^\frac12
 + |(\partial u)_{\vert_{x=0}}|_{m-1,t} \\
&\leq C(K_0)\biggl\{ \opnorm{ u(0) }_m + |\vf|_{m-1,\gamma,t} + S_{\gamma,t}^*(\opnorm{ \dt f(\cdot) }_{m-1}) \\
&\phantom{ \leq C(K_0)\biggl\{ }
 + \opnorm{ \dt^\varphi u(t) }_{m-1,\gamma}
 + \biggl( \gamma\int_0^t \opnorm{ \dt^\varphi u(t') }_{m-1,\gamma}^2{\rm d}t' \biggr)^\frac12
 + |(\dt^\varphi u)_{\vert_{x=0}}|_{m-1,t} \biggr\} \\
&\quad + C(K)\biggl\{ \biggl( \int_0^t \opnorm{ u(t') }_{m,\gamma}^2{\rm d}t' \biggr)^\frac12
 + |\vu|_{m-1,t} \biggr\}.
\end{align*}
\end{lemma}

\begin{proof}
We will use the same notation $\dot{u}^\varphi=\dt^\varphi u$ in the proof of Proposition \ref{propAl}. 
Then, \eqref{equ} can be written as 
\begin{equation}\label{eqn2}
\dot{u}^\varphi + A(\uu)\dx^\varphi u = f-Bu =: f_0.
\end{equation}
We first consider the case $m=1$. 
Here, it holds that 
\[
\begin{cases}
\|f_0(0)\|_{L^2} \leq C(K_0)\opnorm{ u(0) }_1, \\
\|\dt f_0(t)\|_{L^2} \leq \|\dt f(t)\|_{L^2} + C(K)\opnorm{ u(t) }_1, \\
|{f_0}_{\vert_{x=0}}|_{L_\gamma^2(0,t)} \leq |\vf|_{L_\gamma^2(0,t)} + C(K)|\vu|_{L_\gamma^2(0,t)}.
\end{cases}
\]
It follows from \eqref{eqn2} that 
\[
\dx u = (\dx \varphi)A(\uu)^{-1}(f_0 - \dot{u}^\varphi).
\]
We also have 
\[
\dt u = \dot{u}^\varphi - \frac{\dt\varphi}{\dx\varphi}\dx u.
\]
Therefore, we obtain 
\[
|\partial u(t,x)| \leq C(K_0)( |\dot{u}^\varphi(t,x)| + |f_0(t,x)|).
\]
By Lemma \ref{estuuu} we have 
\begin{align*}
& \opnorm{ f_0(t) }_{0,\gamma}
 + \biggl( \gamma\int_0^t \opnorm{ f_0(t') }_{0,\gamma}^2{\rm d}t' \biggr)^\frac12 \\
&\leq C\bigl( \|f_0(0)\|_{L^2} + S_{\gamma,t}^*(\|\dt f_0(\cdot)\|_{L^2}) \bigr) \\
&\leq C(K_0)\bigl( \opnorm{ u(0) }_1 + S_{\gamma,t}^*(\|\dt f(\cdot)\|_{L^2}) \bigr)
 + C(K)S_{\gamma,t}^*(\opnorm{ u(\cdot) }_1).
\end{align*}
Using the above inequalities, we get the desired estimate in the case $m=1$.

We proceed to consider the case $m\geq2$. 
Applying $\partial^\alpha$ with a multi-index $\alpha$ satisfying $|\alpha| \leq m-1$ to \eqref{eqn2} 
and using the identity 
\begin{equation}\label{id1}
\dx^\varphi\partial^\alpha u
 = \partial^\alpha\dx^\varphi u + (\dx^\varphi\partial^\alpha\varphi)\dx^\varphi u
  + (\dx\varphi)^{-1}[\partial^\alpha; \dx\varphi, \dx^\varphi u]
\end{equation}
with a symmetric commutator 
$[\partial^\alpha; v, w] = \partial^\alpha(vw)-(\partial^\alpha v)w-v(\partial^\alpha w)$, 
we obtain 
\begin{align*}
A(\uu)\dx^\varphi\partial^\alpha u + \partial^\alpha\dot{u}^\varphi
&= \partial^\alpha(f-Bu)-[\partial^\alpha,A(\uu)]\dx^\varphi u \\
&\quad
 +A(\uu)( (\dx^\varphi\partial^\alpha\varphi)\dx^\varphi u
  + (\dx\varphi)^{-1}[\partial^\alpha; \dx\varphi, \dx^\varphi u] ) \\
&=: f_{1,\alpha}.
\end{align*}
Here, by Lemma \ref{ineq1} it holds that 
\[
\begin{cases}
\|f_{1,\alpha}(0)\|_{L^2} \leq C(K_0)\opnorm{ u(0) }_m, \\
\|\dt f_{1,\alpha}(t)\|_{L^2} \leq C(K_0)\opnorm { \dt f(t) }_{m-1}
 + C(K)(1+\opnorm{ \dt\varphi(t) }_m)\opnorm{ u(t) }_m, \\
|{f_{1,\alpha}}_{\vert_{x=0}}|_{L_\gamma^2(0,t)}
 \leq |\vf|_{m-1,\gamma,t} + C(K)|\vu|_{m-1,\gamma,t}.
\end{cases}
\]
We also have
\[
\partial^\alpha\dx u = (\dx\varphi)A(\uu)^{-1}(f_{1,\alpha}-\partial^\alpha\dot{u}^\varphi),
\]
which will be used to evaluate $\dx u$. 
Applying $\partial^\alpha$ to the identity 
$\dt u = \dot{u}^\varphi + (\dt\varphi)\dx^\varphi u$ 
and using \eqref{id1} we obtain 
\begin{align*}
& \partial^\alpha\dt u - \partial^\alpha\dot{u}^\varphi - (\dt\varphi)(\dx\varphi)^{-1}\partial^\alpha\dx u \\
&= (\partial^\alpha\dt\varphi)\dx^\varphi u + [\partial^\alpha; \dt\varphi, \dx^\varphi u]
 -(\dt\varphi)(\dx\varphi)^{-1}( (\partial^\alpha\dx\varphi)\dx^\varphi u
  + [\partial^\alpha; \dx\varphi, \dx^\varphi u] ) \\
&=: f_{2,\alpha}.
\end{align*}
Here, by Lemma \ref{ineq1} it holds that 
\[
\begin{cases}
\|f_{2,\alpha}(0)\|_{L^2} \leq C(K_0)\opnorm{ u(0) }_m, \\
\|\dt f_{2,\alpha}(t)\|_{L^2} \leq C(K)(1+\opnorm{ \dt\varphi(t) }_m)\opnorm{ u(t) }_m, \\
|{f_{2,\alpha}}_{\vert_{x=0}}|_{L_\gamma^2(0,t)} \leq C(K)|\vu|_{m-1,\gamma,t}.
\end{cases}
\]
We also have 
\[
\partial^\alpha\dt u
= \partial^\alpha\dot{u}^\varphi + (\dt\varphi)(\dx\varphi)^{-1}\partial^\alpha\dx u + f_{2,\alpha},
\]
which will be used to evaluate $\dt u$. 
Therefore, we obtain 
\[
|\partial^\alpha\partial u(t,x)|
 \leq C(K_0)( |\partial^\alpha\dot{u}^\varphi(t,x)| +  |f_{1,\alpha}(t,x)| + |f_{2,\alpha}(t,x)| ),
\]
so that 
\begin{align*}
&\opnorm{ \partial u(t) }_{m-1,\gamma}
 + \biggl( \gamma\int_0^t \opnorm{ \partial u(t') }_{m-1,\gamma}^2{\rm d}t' \biggr)^\frac12
 + |(\partial u)_{\vert_{x=0}}|_{m-1,t} \\
&\leq C(K_0)\biggl\{ \opnorm{ \dot{u}^\varphi(t) }_{m-1,\gamma}
 + \biggl( \gamma\int_0^t \opnorm{ \dot{u}^\varphi(t') }_{m-1,\gamma}^2{\rm d}t' \biggr)^\frac12
 + |{\dot{u}^\varphi}_{\;\,\vert_{x=0}}|_{m-1,t} \\
&\quad
 + \sum_{|\alpha| \leq m-1,j=1,2}\biggl( \opnorm{ f_{j,\alpha}(t) }_{0,\gamma}
 + \biggl( \gamma\int_0^t \opnorm{ f_{j,\alpha}(t') }_{0,\gamma}^2{\rm d}t' \biggr)^\frac12
 + |f_{j,\alpha \vert_{x=0}}|_{L_\gamma^2(0,t)} \biggr) \biggr\}.
\end{align*}
Here, by Lemma \ref{estuuu} we see that 
\begin{align*}
& \opnorm{ f_{j,\alpha}(t) }_{0,\gamma}
 + \biggl( \gamma\int_0^t \opnorm{ f_{j,\alpha}(t') }_{0,\gamma}^2{\rm d}t' \biggr)^\frac12 \\
&\leq C\bigl( \|f_{j,\alpha}(0)\|_{L^2} + S_{\gamma,t}^*(\|\dt f_{j,\alpha}(\cdot)\|_{L^2}) \bigr) \\
&\leq C(K_0)\bigl( \opnorm{ u(0) }_m + S_{\gamma,t}^*(\opnorm{ \dt f(\cdot) }_{m-1}) \bigr) 
 + C(K)S_{\gamma,t}^*( (1+\opnorm{ \dt\varphi(\cdot) }_m)\opnorm{ u(\cdot) }_m )
\end{align*}
and that 
\begin{align*}
& S_{\gamma,t}^*( (1+\opnorm{ \dt\varphi(\cdot) }_m)\opnorm{ u(\cdot) }_m ) \\
&\leq \biggl(\frac{1}{\gamma}\int_0^t\opnorm{ u(t') }_{m,\gamma}^2{\rm d}t'\biggr)^\frac12
 + \int_0^te^{-\gamma t'}\opnorm{ \dt\varphi(t') }_m \opnorm{ u(t') }_m{\rm d}t' \\
&\leq \biggl(\frac{1}{\gamma}\int_0^t\opnorm{ u(t') }_{m,\gamma}^2{\rm d}t'\biggr)^\frac12
 + \|\dt\varphi\|_{H^m(\Omega_t)} \biggl(\int_0^t\opnorm{ u(t') }_{m,\gamma}^2{\rm d}t'\biggr)^\frac12.
\end{align*}
Summarizing the above inequalities, we obtain the desired estimate. 
\end{proof}

Now, it follows from the estimates in Proposition \ref{propAl} and Lemma \ref{lemequivn} together with 
Lemma \ref{ineq4} that 
\begin{align*}
&\opnorm{ u(t) }_{m,\gamma}
 + \biggl( \gamma\int_0^t \opnorm{ u(t') }_{m,\gamma}^2{\rm d}t' \biggr)^\frac12 + |\vu|_{m,t} \\
&\leq \opnorm{ \partial u(t) }_{m-1,\gamma}
 + \biggl( \gamma\int_0^t \opnorm{ \partial u(t') }_{m-1,\gamma}^2{\rm d}t' \biggr)^\frac12
 + |(\partial u)_{\vert_{x=0}}|_{m-1,t} \\
&\quad
 + \opnorm{ u(t) }_{m-1,\gamma}
 + \biggl( \gamma\int_0^t \opnorm{ u(t') }_{m-1,\gamma}^2{\rm d}t' \biggr)^\frac12 + |\vu|_{m-1,t} \\
&\leq C(K_0)\bigl( (1+|\dt^m\nu|_{L^2(0,t)})\opnorm{ u(0) }_m + |g|_{H_\gamma^m(0,t)}
 + |\vf|_{m-1,\gamma,t} + S_{\gamma,t}^*(\opnorm{ \dt f(\cdot) }_{m-1}) \bigr) \\
&\quad
 + C(K)\biggl\{ \gamma^{-\frac12}\biggl( \gamma\int_0^t \opnorm{ u(t') }_{m,\gamma}^2{\rm d}t' \biggr)^\frac12
  + \gamma^{-\frac12}\opnorm{ u(0) }_m + \gamma^{-1}|\vu|_{m,\gamma,t} \biggr\}. 
\end{align*}
Therefore, by taking $\gamma$ sufficiently large compared to $C(K)$, 
we obtain the desired estimate in Theorem \ref{theoIBVP3}. 
The proof of Theorem \ref{theoIBVP3} is complete.

\subsection{Application to free boundary problems with a boundary equation of ``kinematic'' type}\label{sectFB1}
We investigate here a general class of free boundary problems. 
We consider a quasilinear hyperbolic system cast on a moving domain $(\ux(t),\infty)$, 
\begin{equation}\label{IBVPfb}
\begin{cases}
\dt U + A(U)\dx U = 0 & \mbox{in}\quad (\ux(t),\infty) \quad\mbox{for}\quad t\in(0,T), \\
U_{\vert_{t=0}} = u^{\rm in}(x) & \mbox{on}\quad (\ux(0),\infty), \\
\unu\cdot U_{\vert_{x=\ux(t)}} = g(t) & \mbox{on}\quad (0,T),
\end{cases}
\end{equation}
and assume that the evolution of the boundary is governed by a nonlinear equation of the form 
\begin{equation}\label{eqFB}
\dot{\ux} = \mathcal{X}(U_{\vert_{x=\ux(t)}})
\end{equation}
for some smooth function ${\mathcal X}$.
The set of equations \eqref{IBVPfb}--\eqref{eqFB} is a free boundary problem. 
In the following, without loss of generality we assume $\ux(0)=0$. 
Using as in \S \ref{sectVCm} a diffeomorphism $\varphi(t,\cdot) : \R_+ \to (\ux(t),\infty)$, 
and recalling the notations 
\[
u = U\circ \varphi, \qquad \dx^\varphi = \frac{1}{\dx \varphi}\dx, \qquad 
 \dt^\varphi = \dt - \frac{\dt \varphi}{\dx \varphi}\dx, 
\]
the free boundary problem \eqref{IBVPfb}--\eqref{eqFB} can therefore be recast as 
an initial boundary value problem on a fixed domain, 
\begin{equation}\label{IBVPfbbis}
\begin{cases}
\dt u + \cA(u,\partial\varphi)\dx u = 0 & \mbox{in}\quad \Omega_T, \\
u_{\vert_{t=0}} = u^{\rm in}(x) & \mbox{on}\quad \R_+, \\
\unu\cdot u_{\vert_{x=0}} = g(t) & \mbox{on}\quad (0,T),
\end{cases}
\end{equation}
where $\unu\in \R^2$ is a constant vector and 
\[
\cA(u,\partial\varphi) = \frac{1}{\dx \varphi}\bigl( A(u) - (\dt\varphi) \mbox{Id}\bigr),
\]
complemented by the evolution equation 
\begin{equation}\label{eqFBbis}
\dot{\ux} = {\mathcal X}(u_{\vert_{x=0}}),\qquad \ux(0) = 0.
\end{equation}
As shown in \S \ref{sectVCm}, the regularity of $\varphi$ plays an important role in the analysis 
of the initial boundary value problem \eqref{IBVPfbbis}. 
It is therefore important to make an appropriate choice for the diffeomorphism. 
For a boundary equation of the form \eqref{eqFBbis} which is of ``kinematic'' type, 
a ``Lagrangian'' diffeomorphism is appropriate. 
In particular, in the second point of the lemma, the structure of $\varphi$ allows the control of 
$\dt \varphi$ in $\WW^m(T)$ (which involves $m+1$ derivatives of $\varphi$) by $u$ in $\WW^m(T)$ 
(which involves only $m$ derivative of $u$).

\begin{lemma}\label{lemdiffeo}
Let $\mathcal{U}$ be an open set in $\R^2$ and $\mathcal{X} \in C^\infty(\mathcal{U})$. 
Suppose that $u\in W^{1,\infty}(\Omega_T)$ takes its values in a compact and convex set 
$\mathcal{K}_1 \subset \mathcal{U}$ and that 
\[
\|u\|_{W^{1,\infty}(\Omega_T)}, \|\mathcal{X}\|_{W^{1,\infty}(\mathcal{K}_1)} \leq K.
\]
Then, $\ux \in C^1([0,T])$ can be defined by the ODE 
\[
\begin{cases}
\dot{\ux}(t) = {\mathcal X}(u_{\vert_{x=0}}(t)) \quad\mbox{for}\quad t\in(0,T), \\
\ux(0) = 0.
\end{cases}
\]
Moreover, there exists $T_1 \in (0,T]$ depending on $K$ such that the mapping 
$\varphi:\overline{\Omega_T}\to \R$ defined by 
\begin{equation}\label{diffeo}
\varphi(t,x)=x+\int_0^t {\mathcal X}(u(t',x)){\rm d}t'
\end{equation}
satisfies the following properties:
\begin{enumerate}
\setlength{\itemsep}{3pt}
\item[{\bf i.}]
We have $\varphi(t,0) = \ux(t)$ and that for any $t\in[0,T_1]$, 
$\varphi(t,\cdot)$ is a diffeomorphism mapping $\R_+$ onto $(\ux(t),\infty)$ and satisfying 
$\frac12 \leq \dx\varphi(t,x) \leq 2$. 

\item[{\bf ii.}]
If moreover $m\geq2$, $u\in \WW^m(T_1)$, and $\mathcal{X}(0)=0$, then we have, 
with $\widetilde{\varphi}(t,x)=\varphi(t,x)-x$,
\begin{align*}
& \opnorm{ \partial\widetilde{\varphi}(0) }_{m-1}, \|\partial\varphi\|_{L^\infty(\Omega_{T_1})}
 \leq C( \opnorm{ u(0) }_m ), \\
& \|\widetilde{\varphi}\|_{\WW^m(T_1)}, \|\dt\varphi\|_{\WW^m(T_1)}, 
 |(\partial^m\varphi)_{\vert_{x=0}}|_{L^\infty(0,T_1)} 
\leq C\bigl( \|u\|_{\WW^m(T_1)}, |\vu|_{m,T_1} \bigr).
\end{align*}
\end{enumerate}
\end{lemma}

We can now state the main result of this section, which holds under the following assumption.

\begin{assumption}\label{asshypQLFB}
Let $\mathcal{U}$ be an open set in $\R^2$, which represents a phase space of $u$. 
The following conditions hold.
\begin{enumerate}
\setlength{\itemsep}{3pt}
\item[{\bf i.}]
$A,\mathcal{X} \in C^\infty(\mathcal{U})$, $\mathcal{X}(0)=0$. 

\item[{\bf ii.}]
For any $u\in\cU$, the matrix $A(u)$ has eigenvalues $\lambda_+(u)$ and $-\lambda_-(u)$ satisfying 
\[
\lambda_\pm(u) > 0 \quad\mbox{and}\quad \lambda_\pm(u)\mp {\mathcal X}(u) > 0.
\]

\item[{\bf iii.}]
Denoting by ${\bf e}_+ (u)$ a unit eigenvector associated to the eigenvalue $\lambda_+(u)$ of $A(u)$, 
for any $u\in\mathcal{U}$ we have 
\[
|\unu\cdot {\bf e}_+(u)| > 0.
\]
\end{enumerate}
\end{assumption}

\begin{theorem}\label{theoIBVP4}
Let $m\geq 2$ be an integer. 
Suppose that Assumption \ref{asshypQLFB} is satisfied. 
If $u^{\rm in}\in H^m(\R_+)$ takes its values in a compact and convex set ${\mathcal K}_0\subset \cU$ and 
if the data $u^{\rm in}$ and $g \in H^m(0,T)$ satisfy the compatibility conditions up to order $m-1$ 
in the sense of Definition \ref{defcompfbp} below, then there exist $T_1 \in (0,T]$ and a unique solution 
$(u,\ux)$ to \eqref{IBVPfbbis}--\eqref{eqFBbis} with $u\in \WW^m(T_1)$, $\ux\in H^{m+1}(0,T_1)$, and 
$\varphi$ given by Lemma \ref{lemdiffeo}. 
\end{theorem}

\subsubsection{Compatibility conditions}
For the free boundary problem, $\ux(t)$ and $\varphi(t,x)$ are unknowns so that the interior equation 
$\dt u + \cA(u,\partial\varphi)\dx u = 0$ does not determine $(\dt^ku)_{\vert_{x=0}}$ directly in terms 
of the initial data $u^{\rm in}$ and its derivatives. 
In order to determine them, we need to use \eqref{diffeo}, or equivalently, 
the evolution equation $\dt\varphi=\mathcal{X}(u)$ at the same time. 

Suppose that $u$ is a smooth solution to \eqref{IBVPfbbis}--\eqref{eqFBbis}. 
We note that the interior equation in \eqref{IBVPfbbis} can be written as 
\[
\dt^\varphi u + A(u)\dx^\varphi u =0
\]
and that $\dt^\varphi$ and $\dx^\varphi$ commute. 
Therefore, denoting $u_{(k)} = (\dt^\varphi)^k u$ and using the above equation inductively, 
we have 
\[
u_{(k)} = c_{1,k}(u,\dx^\varphi u,\ldots,(\dx^\varphi)^ku),
\]
where $c_{1,k}$ is a smooth function of its arguments. 
In view of this, we define $u_{(k)}^{\rm in}$ by 
\begin{equation}\label{ukin}
u_{(k)}^{\rm in} = c_{1,k}(u^{\rm in},\dx u^{\rm in},\ldots,\dx^k u^{\rm in})
\end{equation}
for $k=1,2,\ldots$. 
Using the relation $\dt = \dt^\varphi+(\dt\varphi)\dx^\varphi$ inductively, 
we see that 
\[
\dt^k = (\dt^\varphi)^k + (\dt^k\varphi)\dx^\varphi
 + \sum_{l=2}^k\sum_{\substack{j_0+j_1+\cdots+j_l=k \\ 1\leq j_1,\ldots,j_l}}
  c_{l,j_0,\ldots,j_l}(\dt^{j_1}\varphi)\cdots(\dt^{j_l}\varphi)(\dt^\varphi)^{j_0}(\dx^\varphi)^l,
\]
so that denoting $u_{k}=\dt^k u$ and $\varphi_k=\dt^k\varphi$ we have 
\[
u_{k} = u_{(k)} + \varphi_k\dx^\varphi u
 + \sum_{l=2}^k\sum_{\substack{j_0+j_1+\cdots+j_l=k \\ 1\leq j_1,\ldots,j_l}}
  c_{l,j_0,\ldots,j_l}\varphi_{j_1}\cdots\varphi_{j_l}(\dx^\varphi)^lu_{(j_0)}.
\]
Particularly, denoting $u_{k}^{\rm in}=(\dt^k u)_{\vert_{t=0}}$ and 
$\varphi_k^{\rm in}=(\dt^k \varphi)_{\vert_{t=0}}$ we obtain 
\begin{equation}\label{dtku}
u_{k}^{\rm in} = u_{(k)}^{\rm in} + \varphi_k^{\rm in}(\dx u^{\rm in})
 + \sum_{l=2}^k\sum_{\substack{j_0+j_1+\cdots+j_l=k \\ 1\leq j_1,\ldots,j_l}}
  c_{l,j_0,\ldots,j_l}\varphi_{j_1}^{\rm in}\cdots\varphi_{j_l}^{\rm in}\dx^l u_{(j_0)}^{\rm in}.
\end{equation}
This implies that $u_{k}^{\rm in}$ is written in terms of $\varphi_j^{\rm in}$ and 
$\dx^j u^{\rm in}$ for $0\leq j\leq k$. 
On the other hand, differentiating the evolution equation $\dt\varphi = \mathcal{X}(u)$ 
$k$-times with respect to $t$, we have 
\[
\varphi_{k+1} = c_{2,k}(u,\dt u,\ldots,\dt^k u),
\]
where $c_{2,k}$ is a smooth function of its arguments. 
Therefore, we get 
\begin{equation}\label{dtkx}
\varphi_{k+1}^{\rm in} = c_{2,k}({u^{\rm in}},u_{1}^{\rm in},\ldots,u_{k}^{\rm in}).
\end{equation}
Using \eqref{dtku} and \eqref{dtkx} alternatively we can determine $u_{k}^{\rm in}$ and $\varphi_k^{\rm in}$. 
Now, the boundary condition $\unu\cdot u_{\vert_{x=0}}=g$ implies that 
\[
\unu \cdot \dt^k u_{\vert_{x=0}} = \dt^k g.
\]
On the edge $\{t=0,x=0\}$, smooth enough solutions must therefore satisfy 
\begin{equation}\label{compfbp}
\unu\cdot {u_{k}^{\rm in}}_{\vert_{x=0}} = (\dt^k g)_{\vert_{t=0}}.
\end{equation}

\begin{definition}\label{defcompfbp}
Let $m\geq1$ be an integer. We say that the data $u^{\rm in}\in H^m(\R_+)$ and $g \in H^m(0,T)$ for the 
initial boundary value problem \eqref{IBVPfbbis}--\eqref{eqFBbis} satisfy the compatibility condition at order $k$ 
if the $\{u_{j}^{\rm in}\}_{j=0}^m$ defined by \eqref{ukin}--\eqref{dtkx} satisfy \eqref{compfbp}. 
We also say that the data satisfy the compatibility conditions up to order $m-1$ if they satisfy the 
compatibility conditions at order $k$ for $k=0,1,\ldots,m-1$. 
\end{definition}

\begin{remark}
These compatibility conditions do not depend on the particular choice of the diffeomorphism $\varphi$ 
such as \eqref{diffeo}. 
The other choice of the diffeomorphism $\varphi : \R_+ \to (\ux(t),\infty)$ will give the same conditions. 
\end{remark}

\subsubsection{Proof of Theorem \ref{theoIBVP4}}
Let $\mathcal{K}_1$ be a compact and convex set in $\R^2$ satisfying 
$\mathcal{K}_0 \Subset \mathcal{K}_1 \Subset \mathcal{U}$. 
Then, there exists a constant $c_0 > 0$ such that for any $u \in \mathcal{K}_1$ we have 
\begin{align*}
\lambda_{\pm}(u) \geq c_0, \qquad \lambda_\pm(u)\mp {\mathcal X}(u) \geq c_0,
 \qquad |\unu \cdot \mathbf{e}_{+}(u)| \geq c_0. 
\end{align*}
We will construct the solution $u$ with values in $\mathcal{K}_1$. 
Note that there exists a constant $\delta_0>0$ such that 
$\|u-u^{\rm in}\|_{L^\infty} \leq \delta_0$ implies $u(x) \in \mathcal{K}_1$ for all $x\in\R_+$. 
Therefore, it is sufficient to construct the solution $u$ satisfying 
$\|u(t) - u^{\rm in}\|_{L^\infty} \leq \delta_0$ for $0\leq t\leq T_1$. 
The solution is classically constructed using the iterative scheme 
\begin{equation}\label{eqFB_n}
\varphi^n(t,x) = x + \int_0^t {\mathcal X}(u^n(t',x)){\rm d}t'
\end{equation}
and 
\begin{equation}\label{IBVPfbbis_n}
\begin{cases}
\dt u^{n+1} + \cA(u^n,\partial\varphi^n)\dx u^{n+1} = 0 & \mbox{in}\quad \Omega_T, \\
{u^{n+1}}_{\vert_{t=0}} = u^{\rm in}(x) & \mbox{on}\quad \R_+, \\
\unu\cdot {u^{n+1}}_{\vert_{x=0}} = g(t) & \mbox{on}\quad (0,T)
\end{cases}
\end{equation}
for all $n\in\N$. 
For the first iterate $u^0$, we choose a function $u^0 \in H^{m+1/2}(\R\times\R_+)$ such that 
$(\dt^k u^0)_{\vert_{t=0}}=u_{k}^{\rm in}$ for $0\leq k\leq m$ with $u_{k}^{\rm in}$ defined by 
\eqref{ukin}--\eqref{dtkx}. 
Then, for the initial boundary value problem \eqref{IBVPfbbis_n} to the unknowns $u^{n+1}$ 
the data $(u^{\rm in},g)$ satisfy the compatibility conditions up to order $m-1$ in the sense of 
Definition \ref{defcompVC}. 
Moreover, $\opnorm{ u^n(0) }_m$ is independent of $n$, and there exists therefore $K_0$ such that 
\[
\frac{1}{c_0}, \opnorm{ u^n(0) }_m, \opnorm{ \partial\widetilde{\varphi}(0) }_{m-1}, 
 \|\partial\varphi^n\|_{L^\infty(\Omega_{T_1})},|\unu|,\|A\|_{L^\infty(\mathcal{K}_1)} \leq K_0,
\]
as long as $\|u^n\|_{W^{1,\infty}(\Omega_T)} \leq K$ and $T_1\in (0,T]$ sufficiently small depending on $K$. 
We prove now that for $M$ large enough and $T_1$ small enough, for any $n \in \N$ we have 
\[
\begin{cases}
\|u^n\|_{\WW^m(T_1)} + |{u^n}_{\vert_{x=0}}|_{m,T_1} \leq M, \\
\|u^n(t)-u^{\rm in}\|_{L^\infty} \leq \delta_0 \quad\mbox{for}\quad 0 \leq t\leq T_1.
\end{cases}
\]
We prove this assertion by induction. Since it is satisfied for $n=0$ for a suitable $M$ and $T_1$, 
we just need to prove that if holds at rank $n+1$ if it holds at rank $n$. 
By the Sobolev imbedding theorem and Lemma \ref{lemdiffeo}, we have 
\[
\|u^n\|_{W^{1,\infty}(\Omega_{T_1})}, \|\widetilde{\varphi}^n\|_{\WW^m(T_1)}, \|\dt\varphi^n\|_{\WW^m(T_1)}, 
 |(\partial^m\varphi^n)_{\vert_{x=0}}|_{L^\infty(0,T_1)} \leq K(M).
\]
It follows therefore from Theorem \ref{theoIBVP3} that 
\[
\| u^{n+1}(t) \|_{\WW^m(T_1)} + |{u^{n+1}}_{\vert_{x=0}}|_{m,T_1} \leq 
 C(K_0)e^{C(M)t}( 1 + |g|_{H^m(0,T_1)} ).
\]
Choosing $M = 2C(K_0)( 1 + |g|_{H^m(0,T)} )$, it is possible to choose $T_1$ small enough to get that the 
right-hand side is smaller than $M$. 
We also have $\|u^{n+1}(t)-u^{\rm in}\|_{L^\infty} \leq C\|u^{n+1}\|_{\WW^2(T_1)}T_1 \leq \delta_0$ for 
$0\leq t\leq T_1$. 
Therefore, the claim is proved. 

We proceed to show that the sequence of approximate solutions $\{(u^n,\varphi^n)\}_n$ converges to the solution 
$(u,\varphi)$ to \eqref{IBVPfbbis}--\eqref{eqFBbis} satisfying $u \in \WW^m(T_1)$ and 
$\ux=\varphi_{\vert_{x=0}} \in H^{m+1}(0,T_1)$. 
We have 
\[
\begin{cases}
\dt (u^{n+2}-u^{n+1}) + \cA(u^n,\partial\varphi^n) \dx(u^{n+2}-u^{n+1}) = f^n & \mbox{in}\quad \Omega_T, \\
(u^{n+2}-u^{n+1})_{\vert_{t=0}} = 0 & \mbox{on}\quad \R_+, \\
\unu \cdot (u^{n+2}-u^{n+1})_{\vert_{x=0}} = 0 & \mbox{on}\quad (0,T)
\end{cases}
\]
with 
\[
f^n = -(\cA(u^{n+1},\partial\varphi^{n+1}) - \cA(u^n,\partial\varphi^n)) \dx u^{n+1}.
\]
It follows therefore from \eqref{eqmm1.5} in Proposition \ref{propAl} that 
\begin{align*}
& \opnorm{ (u^{n+2}-u^{n+1})(t) }_{m-1} + |(u^{n+2}-u^{n+1})_{\vert_{x=0}}|_{m-1,t} \\
&\leq C(M)\Bigl( |{f^n}_{\vert_{x=0}}|_{m-2,t} + \int_0^t \opnorm{ \dt f^n(t') }_{m-2}{\rm d}t' \Bigr) \\
&\leq C(M)\int_0^t ( \opnorm{ \dt f^n(t') }_{m-2} + |(\dt f^n)_{\vert_{x=0}}|_{m-2,t'} ){\rm d}t'
\end{align*}
for $0\leq t\leq T_1$, where we used Lemma \ref{ineq4} and the fact that 
$(\dt^k u^n)_{\vert_{t=0}} = u_{k}^{\rm in}$ does not depend on $n$. 
Here, we see that 
\begin{align*}
\|\dt f^n\|_{\WW^{m-2}(T_1)}
&\leq C(M) \|(u^{n+1}-u^{n}, \varphi^{n+1}-\varphi^{n}, \dt(\varphi^{n+1}-\varphi^{n}))\|_{\WW^{m-1}(T_1)} \\
&\leq  C(M) \|u^{n+1}-u^{n}\|_{\WW^{m-1}(T_1)}
\end{align*}
and that 
\begin{align*}
|(\dt f^n)_{\vert_{x=0}}|_{m-2,T_1}
&\leq C(M)\bigl( \|(u^{n+1}-u^{n}, \varphi^{n+1}-\varphi^{n}, \dt(\varphi^{n+1}-\varphi^{n}))\|_{\WW^{m-1}(T_1)} \\
&\quad
 + |(u^{n+1}-u^{n}, \varphi^{n+1}-\varphi^{n}, \dt(\varphi^{n+1}-\varphi^{n}))_{\vert_{x=0}}|_{m-1,T_1} \bigr) \\
&\leq C(M)\bigl( \|u^{n+1}-u^{n}\|_{\WW^{m-1}(T_1)} + |(u^{n+1}-u^{n})_{\vert_{x=0}}|_{m-1,T_1} \bigr),
\end{align*}
where we used Lemma \ref{ineq3}. 
Note that in the above inequalities, the quantity $ \dt(\varphi^{n+1}-\varphi^{n})$ has been controled in 
$\WW^{m-1}(T_1)$; a similar control of $ \dx(\varphi^{n+1}-\varphi^{n})$ is not possible and 
this is the reason why it is important to have $\opnorm{\dt f(t)}_{m-2}$ rather than $\opnorm{f(t)}_{m-1}$ 
in the right-hand side of \eqref{eqmm1.5} in Proposition \ref{propAl}. 
Therefore, by taking $T_1$ sufficiently small if necessary, we obtain 
\begin{align*}
& \|u^{n+2}-u^{n+1}\|_{\WW^{m-1}(T_1)} + |(u^{n+2}-u^{n+1})_{\vert_{x=0}}|_{m-1,T_1} \\
&\leq \frac12 \bigl( \|u^{n+1}-u^{n}\|_{\WW^{m-1}(T_1)} + |(u^{n+1}-u^{n})_{\vert_{x=0}}|_{m-1,T_1} \bigr).
\end{align*}
This together with an interpolation inequality 
$\|u\|_{W^{1,\infty}(\Omega_{T_1})}^2 \leq C\|u\|_{\WW^{m-1}(T_1)}\|u\|_{\WW^{m}(T_1)}$ 
shows that $\{(u^n,\widetilde{\varphi}^n)\}_n$ converges to $(u,\widetilde{\varphi})$ in 
$\WW^{m-1}(T_1) \cap W^{1,\infty}(\Omega_{T_1})$, so that $(u,\widetilde{\varphi})$ is a solution to 
\eqref{IBVPfbbis}--\eqref{eqFBbis}. 
Moreover, by standard compactness arguments we see that 
\[
\|u\|_{\WW^m(T_1)} + |\vu|_{m,T_1} \leq M.
\]
The regularity and the uniqueness of the solution stated in the theorem is obtained by  standard 
arguments so we omit them. 
The proof of Theorem \ref{theoIBVP4} is complete.

\subsection{Application to free boundary problems with a  fully nonlinear boundary equation}
\label{sectVCm2}
We now consider a $2\times2$ quasilinear hyperbolic system on a moving domain $(\ux(t),\infty)$: 
\begin{equation}\label{IBVPfb2}
\dt U+A(U)\dx U =0 \quad\mbox{in}\quad (\ux(t),\infty)
\end{equation}
with a fully nonlinear boundary condition 
\begin{equation}\label{fbBC}
U = U_{\rm i} \quad\mbox{on}\quad x=\ux(t),
\end{equation}
where $U_{\rm i} = U_{\rm i}(t,x)$ is a given $\R^2$-valued function, whereas $\ux(t)$ is unknown function. 
Compared to the free boundary problem \eqref{IBVPfb}--\eqref{eqFB}, 
the evolution equation of the boundary is implicitly contained in the above boundary condition. 
In fact, differentiating the boundary condition $U(t,\ux(t)) = U_{\rm i}(t,\ux(t))$ with respect to $t$ 
and taking the Euclidean inner product of the resulting equation with $\partial_x U - \partial_x U_{\rm i}$, 
we obtain 
\begin{equation}\label{eqFB2}
\dot{\ux} = \chi( (\partial U)_{\vert_{x=\ux}},(\partial U_{\rm i})_{\vert_{x=\ux}} ),
\end{equation}
where 
\[
\chi( \partial U, \partial U_{\rm i} )
= - \frac{( \dx U - \dx U_{\rm i} )\cdot( \dt U - \dt U_{\rm i} ) }{ |\dx U - \dx U_{\rm i}|^2 }.
\]
In view of this, a discontinuity of the spatial derivative $\dx U$ on the free boundary is crucial 
to the free boundary problem \eqref{IBVPfb2}--\eqref{fbBC} whereas $U$ itself is continuous. 
Compared to the boundary equation \eqref{eqFB} of kinematic type, 
\eqref{eqFB2} does not depend on $U$ itself but on its derivative $\partial U$. 
Therefore, \eqref{IBVPfb2}--\eqref{eqFB2} is more difficult than \eqref{IBVPfb}--\eqref{eqFB} in 
the previous subsection. 
We will use again a diffeomorphism $\varphi(t,\cdot): \R_+ \to (\ux(t),\infty)$ and put 
$u = U\circ\varphi$ and $u_{\rm i} = U_{\rm i}\circ\varphi$. 
Then, the free boundary problem \eqref{IBVPfb2}--\eqref{fbBC} is recast as a problem on the 
fixed domain: 
\begin{equation}\label{IBVPfbbis2}
\begin{cases}
\dt^\varphi u + A(u)\dx^\varphi u =0 & \mbox{in}\quad \Omega_T, \\
u_{\vert_{x=0}} = {u_{\rm i}}_{\vert_{x=0}} & \mbox{on}\quad (0,T). 
\end{cases}
\end{equation}
We impose the initial conditions of the form 
\begin{equation}\label{IC}
u_{\vert_{t=0}} = u^{\rm in}(x) \quad\mbox{on}\quad \R_+, \qquad \ux(0)=0.
\end{equation}
We also note that the equation \eqref{eqFB2} for the free boundary is then reduced to 
\begin{equation}\label{eqFB3}
\dot{\ux} = \chi( (\partial^\varphi u)_{\vert_{x=0}},(\partial^\varphi u_{\rm i})_{\vert_{x=0}} ).
\end{equation}

\begin{assumption}\label{asshypNLFB}
Let $\cU$ be an open set in $\R^2$, which represents a phase space of $u$.
\begin{enumerate}
\item[{\bf i.}]
$A \in C^\infty(\cU)$.
\item[{\bf ii.}]
There exists $c_0>0$ such that for any $u \in \cU$, the matrix $A(u)$ has eigenvalues 
$\lambda_+(u)$ and $-\lambda_-(u)$ satisfying $\lambda_{\pm}(u) \geq c_0$. 
\end{enumerate}
\end{assumption}

As before, this condition ensures that the system is strictly hyperbolic. 
We denote by ${\bf e}_{\pm}(u)$ normalized eigenvectors associated to the eigenvalues 
$\pm\lambda_{\pm}(u)$ of $A(u)$. 
They are uniquely determined up to a sign. 
Since both eigenvalues are simple, we have $\lambda_{\pm}, {\bf e}_{\pm} \in C^\infty(\cU)$ 
under an appropriate choice of the sign of ${\bf e}_{\pm}$. 
As mentioned above, a discontinuity of $\dx U$ at the free boundary is crucial so that 
we will work in a class of solutions satisfying 
\begin{equation}\label{disconti}
|(\dx^\varphi u - \dx^\varphi u_i)_{\vert_{x=0}}| \geq c_0
\end{equation}
for some positive constant $c_0$. 
The interior equation in \eqref{IBVPfbbis2} can be written as 
\[
\dt u + \cA(u,\partial\varphi) \dx u = 0,
\]
where $\cA(u,\partial\varphi) = (\dx\varphi)^{-1}(A(u) - (\dt\varphi){\rm Id})$. 
The eigenvalues of this matrix are $(\dx\varphi)^{-1}(\pm\lambda_{\pm}(u) - \dt\varphi)$, whereas the 
corresponding eigenvectors are ${\bf e}_{\pm}(u)$ which does not depend on $\partial\varphi$. 
In view of {\bf i} in Assumption 1, we also restrict a class of solution by 
\begin{equation}\label{egenv}
\lambda_{\pm}(u) \mp \dt\varphi \geq c_0 \quad\mbox{in}\quad (0,T)\times\R_+.
\end{equation}

We note that the boundary equation \eqref{eqFB3} is not of the kinematic type considered in \S \ref{sectFB1} 
so that we need to use a different diffeomorphism from the one given by Lemma \ref{lemdiffeo}. 
Let $\psi \in C_0^\infty(\R)$ be a cut-off function such that 
$\psi(x)=1$ for $|x| \leq 1$ and $=0$ for $|x| \geq 2$. 
We define the diffeomorphism by 
\begin{equation}\label{diffeo2}
\varphi(t,x) = x + \psi\Bigl(\frac{x}{\varepsilon}\Bigr)\ux(t),
\end{equation}
where $\varepsilon>0$ is a small parameter which will be determined later. 
As we will see below, under this choice of the diffeomorphism, \eqref{egenv} would be satisfied if 
the solution satisfies 
\begin{equation}\label{egenv2}
\lambda_{\pm}(u_{\vert_{x=0}}) \mp \dot{\ux} \geq 2c_0 \quad\mbox{on}\quad (0,T).
\end{equation}
The following lemma shows that this choice of diffeomorphism behaves differently than the Lagrangian 
diffeomorphism studied in Lemma \ref{lemdiffeo}; in particular, the latter has a better time regularity, 
while the former has a better space regularity.

\begin{lemma}\label{lemdiffeo2}
Suppose that $\ux \in C^1([0,T])$ satisfies $\ux(0)=0$ and $|\dot{\ux}|_{L^2(0,T)} \leq K$. 
Then, there exists $T_1 \in (0,T]$ depending on $\varepsilon$ and $K$ such that the mapping 
$\varphi:\overline{\Omega_T}\to \R$ defined by \eqref{diffeo2} satisfies the following properties: 
\begin{enumerate}
\item[{\bf i.}]
We have $\varphi(t,0) = \ux(t)$ and $\varphi(0,x)=x$ and for all $0\leq t \leq T_1$, 
$\varphi(t,\cdot)$ is a diffeomorphism mapping $\R_+$ onto $(\ux(t),\infty)$ and satisfying 
$\frac12 \leq \dx\varphi(t,x) \leq 2$. 

\item[{\bf ii.}]
For any nonnegative integers $k$ and $l$, we have 
\[
\|\dt^l\dx^k\widetilde{\varphi}(t)\|_{L^1\cap L^\infty(\R_+)}
\leq C(\varepsilon,k)|\dt^l\ux(t)|,
\]
where $\widetilde{\varphi}(t,x)=\varphi(t,x)-x$. 
Particularly, if moreover $m\geq2$ and $\ux \in H^m(0,T_1)$, then we have 
\begin{align*}
& \opnorm{ \partial\widetilde{\varphi}(0) }_{m-2}, \|\partial\varphi\|_{L^\infty(\Omega_{T_1})}
 \leq C(\varepsilon)\biggl( \sum_{j=0}^{m-1}|(\dt^j\ux)_{\vert_{t=0}}|
  + \sqrt{T_1}|\dot{\ux}|_{H^2(0,T_1)} \biggr), \\
& \|\widetilde{\varphi}\|_{\WW^{m-1}(T_1)}, \|\dt\varphi\|_{\WW^{m-1}(T_1)}, 
 |(\partial^{m-1}\varphi)_{\vert_{x=0}}|_{L^\infty(0,T_1)} 
\leq C(\varepsilon)|\ux|_{W^{m-1,\infty} \cap H^m(0,T_1)}.
\end{align*}
\end{enumerate}
\end{lemma}

\begin{theorem}\label{theoIBVP5}
Let $m\geq 2$ be an integer. 
Suppose that Assumption \ref{asshypNLFB} is satisfied. 
If $u^{\rm in}\in H^m(\R_+)$ takes its values in a compact and convex set ${\mathcal K}_0\subset \cU$ and 
if the data $u^{\rm in}$ and $U_{\rm i} \in W^{m,\infty}((0,T)\times(-\delta,\delta))$ satisfy 
\begin{enumerate}
\item[{\bf i.}]
$\lambda_{\pm}({u^{\rm in}}_{\vert_{x=0}}) \mp \ux_1^{\rm in} >0$,
\item[{\bf ii.}]
$(\dx u^{\rm in})_{\vert_{x=0}} - (\dx U_{\rm i})_{\vert_{t=x=0}} \ne0$,
\item[{\bf iii.}]
$((\dx u^{\rm in})_{\vert_{x=0}} - (\dx U_{\rm i})_{\vert_{t=x=0}})^\perp \cdot 
 \mathbf{e}_{+}({u^{\rm in}}_{\vert_{x=0}}) \ne0$,
\end{enumerate}
where $\ux_1^{\rm in}=(\dt\ux)_{\vert_{t=0}}$ will be determined by \eqref{xkin} below, 
and the compatibility conditions up to order $m-1$ in the sense of Definition \ref{defCC} below, 
then there exist $T_1 \in (0,T]$ and a unique solution 
$(u,\ux)$ to \eqref{IBVPfbbis2}--\eqref{IC} with 
$u, \dx u \in \WW^{m-1}(T_1)$, $\ux\in H^m(0,T_1)$, and $\varphi$ given by Lemma \ref{lemdiffeo2}. 
\end{theorem}

\begin{remark}\label{remarkIC}
Thanks to Proposition \ref{conalg} below, the condition {\bf iii} in the theorem can be replaced by 
\begin{enumerate}
\item[{\bf iii$'$.}]
$\mu_0 \cdot \mathbf{e}_{+}({u^{\rm in}}_{\vert_{x=0}}) \ne0$,
\end{enumerate}
where $\mu_0$ is the unit vector satisfying 
$\mu_0 \cdot (\dt U_{\rm i} + A(U_{\rm i})\dx U_{\rm i})_{\vert_{t=x=0}} =0$. 
This unit vector $\mu_0$ is uniquely determined up to the sign under the other assumptions 
of the theorem. 
\end{remark}

\subsubsection{Compatibility conditions}\label{sscomp}
Suppose that $u$ is a smooth solution to \eqref{IBVPfbbis2}--\eqref{IC}. 
We note that $\dt^\varphi$ and $\dx^\varphi$ commute. 
Denoting $u_{(k)} = (\dt^\varphi)^ku$ and using the interior equation in \eqref{IBVPfbbis2} inductively, 
we have 
\[
u_{(k)} = c_{1,k}(u,\dx^\varphi u,\ldots,(\dx^\varphi)^ku),
\]
where $c_{1,k}$ is a smooth function of its arguments. 
In view of this, we define $u_{(k)}^{\rm in}$ by 
\begin{equation}\label{ukin2}
u_{(k)}^{\rm in} = c_{1,k}(u^{\rm in},\dx u^{\rm in},\ldots,\dx^k u^{\rm in})
\end{equation}
for $k=1,2,\ldots$. 
We proceed to express $(\dt^k\ux)_{\vert_{t=0}}$ in terms of the initial data. 
Differentiating the boundary condition in \eqref{IBVPfbbis2} with respect to $t$, we have 
$\dt^ku=\dt^ku_{\rm i}$ on $x=0$. 
Using the relation $\dt = \dt^\varphi+(\dt\varphi)\dx^\varphi$ inductively, 
we see that 
\[
\dt^k = (\dt^\varphi)^k + (\dt^k\varphi)\dx^\varphi
 + \sum_{l=2}^k\sum_{\substack{j_0+j_1+\cdots+j_l=k \\ 1\leq j_1,\ldots,j_l}}
  c_{l,j_0,\ldots,j_l}(\dt^{j_1}\varphi)\cdots(\dt^{j_l}\varphi)(\dt^\varphi)^{j_0}(\dx^\varphi)^l,
\]
so that denoting $\ux_{k} = \dt^k\ux$ we have 
\begin{align*}
& u_{(k)} - (\dt^\varphi)^ku_{\rm i} + \ux_{k}(\dx^\varphi u - \dx^\varphi u_{\rm i}) \\
& + \sum_{l=2}^k\sum_{\substack{j_0+j_1+\cdots+j_l=k \\ 1\leq j_1,\ldots,j_l}}
  c_{l,j_0,\ldots,j_l}\ux_{(j_1)}\cdots\ux_{j_l} (\dx^\varphi)^l(u_{(j_0)} - (\dt^\varphi)^{j_0}u_{\rm i})
 = 0 \quad\mbox{on}\quad x=0.
\end{align*}
Decomposing this relation into the direction $\dx^\varphi u - \dx^\varphi u_{\rm i}$ and 
its perpendicular direction, we obtain 
\begin{align*}
\ux_{k} 
&= -\frac{\dx^\varphi u - \dx^\varphi u_{\rm i}}{|\dx^\varphi u - \dx^\varphi u_{\rm i}|^2}
 \cdot \biggl\{ u_{(k)} - (\dt^\varphi)^ku_{\rm i} \\
&\qquad
 + \sum_{l=2}^k\sum_{\substack{j_0+j_1+\cdots+j_l=k \\ 1\leq j_1,\ldots,j_l}}
  c_{l,j_0,\ldots,j_l}\ux_{j_1}\cdots\ux_{j_l} (\dx^\varphi)^l(u_{(j_0)} - (\dt^\varphi)^{j_0}u_{\rm i})
  \biggr\}_{\vert_{x=0}}
\end{align*}
and 
\begin{align*}
& (\dx^\varphi u - \dx^\varphi u_{\rm i})^\perp \cdot \biggl\{
 u_{(k)} - (\dt^\varphi)^ku_{\rm i} \\
& + \sum_{l=2}^k\sum_{\substack{j_0+j_1+\cdots+j_l=k \\ 1\leq j_1,\ldots,j_l}}
  c_{l,j_0,\ldots,j_l}\ux_{j_1}\cdots\ux_{j_l} (\dx^\varphi)^l(u_{(j_0)} - (\dt^\varphi)^{j_0}u_{\rm i})
  \biggr\}_{\vert_{x=0}}
 = 0,
\end{align*}
respectively. 
In view of this, we define $\ux_{k}^{\rm in}$ inductively by $\ux_{0}^{\rm in} = 0$ and 
\begin{align}\label{xkin}
\ux_{k}^{\rm in} 
&= -\frac{\dx u^{\rm in} - (\dx U_{\rm i})_{\vert_{t=0}}}{|\dx u^{\rm in} - (\dx U_{\rm i})_{\vert_{t=0}}|^2}
 \cdot \biggl\{ u_{(k)} ^{\rm in} - (\dt^k U_{\rm i})_{\vert_{t=0}} \\
&\qquad
 + \sum_{l=2}^k\sum_{\substack{j_0+j_1+\cdots+j_l=k \\ 1\leq j_1,\ldots,j_l}}
  c_{l,j_0,\ldots,j_l}\ux_{j_1}^{\rm in}\cdots\ux_{j_l}^{\rm in}
   \dx^l(u_{(j_0)}^{\rm in} - (\dt^{j_0}U_{\rm i})_{\vert_{t=0}}) \biggr\}_{\vert_{x=0}} \nonumber
\end{align}
for $k=1,2,\ldots$.

\begin{definition}\label{defCC}
Let $m\geq1$ be an integer. 
We say that the data $u^{\rm in} \in H^m(\R_+)$ and $U_{\rm i} \in W^{m,\infty}((0,T)\times(-\delta,\delta))$ 
for the initial boundary value problem 
\eqref{IBVPfbbis2}--\eqref{IC} satisfy the compatibility condition at order $k$ if 
$\{u_{(j)}^{\rm in}\}_{j=0}^m$ and $\{\ux_{(j)}^{\rm in}\}_{j=0}^{m-1}$ defined by \eqref{ukin2}--\eqref{xkin} 
satisfy ${u^{\rm in}}_{\vert_{x=0}} = U_{{\rm i} \, \vert_{t=x=0}}$ in the case $k=0$ and 
\begin{align*}
& (\dx u^{\rm in} - (\dx U_{\rm i})_{\vert_{t=0}})^\perp \cdot \biggl\{
 u_{(k)}^{\rm in} - (\dt^k U_{\rm i})_{\vert_{t=0}}\\
& + \sum_{l=2}^k\sum_{\substack{j_0+j_1+\cdots+j_l=k \\ 1\leq j_1,\ldots,j_l}}
  c_{l,j_0,\ldots,j_l}\ux_{(j_1)}^{\rm in}\cdots\ux_{(j_l)}^{\rm in}
   \dx^l(u_{(j_0)}^{\rm in} - (\dt^{j_0}U_{\rm i})_{\vert_{t=0}} )
  \biggr\}_{\vert_{x=0}}
 = 0
\end{align*}
in the case $k\geq 1$. 
We say also that the data $u^{\rm in}$ and $U_{\rm i}$ for \eqref{IBVPfbbis2}--\eqref{IC} 
satisfy the compatibility conditions up to order $m-1$ if they satisfy the compatibility 
conditions at order $k$ for $k=0,1,\ldots,m-1$. 
\end{definition}

Roughly speaking, the definition of $\ux_{k}^{\rm in}$ ensures the equality 
$\dt^k u =\dt^k u_{\rm i}$ at $x=t=0$ in the direction $\dx^\varphi u - \dx^\varphi u_i$, 
whereas the compatibility conditions ensure it in the perpendicular direction 
$(\dx^\varphi u - \dx^\varphi u_i)^\perp$.

We shall need to approximate $u^{\rm in}$ and $U_{\rm i}$ by more regular data which satisfy 
higher order compatibility conditions. 
Such an approximation is given by the following proposition.

\begin{proposition}\label{appdata}
Let $m$ and $s$ be integers satisfying $s>m\geq2$ and let $A \in C^\infty(\cU)$. 
If $u^{\rm in} \in H^m(\R_{+})$ takes its value in $\cU$ and if the data $u^{\rm in}$ and 
$U_{\rm i} \in W^{m,\infty}((0,T)\times(-\delta,\delta))$ satisfy 
\[
(\dx u^{\rm in})_{\vert_{x=0}} - (\dx U_{\rm i})_{\vert_{t=x=0}} \ne0
\]
and the compatibility conditions up to order $m-1$, then there exists 
$\{(u^{{\rm in},(n)},U_{\rm i}^{(n)})\}_n$ a sequence of data such that 
$(u^{{\rm in},(n)},U_{\rm i}^{(n)}) \in H^s(\R_{+}) \times W^{s,\infty}((0,T)\times(-\delta,\delta))$ 
converges to $(u^{\rm in},U_{\rm i})$ in $H^m(\R_{+})\times B^{m-1}([0,T]\times[-\delta,\delta])$ and 
satisfies the compatibility conditions up to order $s-1$. 
\end{proposition}

\begin{proof}
Once we fix $U_{\rm i}$, the compatibility condition at order $k$ is a nonlinear relation among 
$(\dx^j u^{\rm in})_{\vert_{x=0}}$ for $j=0,1,\ldots,k$. 
We need to know the explicit dependence of the highest order term $(\dx^k u^{\rm in})_{\vert_{x=0}}$ 
of the compatibility condition to show this proposition. 

The compatibility conditions at order $0$ and $1$ are given by 
$(u^{\rm in})_{\vert_{x=0}} = {U_{\rm i}}_{\vert_{t=x=0}}$ and 
\[
((\dx u^{\rm in})_{\vert_{x=0}}-(\dx U_{\rm i})_{\vert_{t=x=0}})^\perp \cdot
 (A({u^{\rm in}}_{\vert_{x=0}})(\dx u^{\rm in})_{\vert_{x=0}} + (\dt U_{\rm i})_{\vert_{t=x=0}}) = 0,
\]
respectively. 
We proceed to consider the compatibility condition at order $k$ in the case $k\geq2$. 
We will denote simply by LOT the terms containing $\dx^j u^{\rm in}$ for $j=0,1,\ldots,k-1$, 
$U_{\rm i}$, and its derivatives only, and not containing $\dx^k u^{\rm in}$. 
Then, we have 
\[
u_{(k)}^{\rm in} = (-A(u^{\rm in}))^k\dx^k u^{\rm in} + \mbox{LOT}
\]
and $\ux_j^{\rm in}=\mbox{LOT}$ for $0 \leq j\leq k-1$. 
Denoting $u_k^{\rm in}=(\dt^k u)_{\vert_{t=0}}$ and using the relation 
$\dt=\dt^\varphi+(\dt\varphi)\dx^\varphi$ inductively, we obtain 
\begin{align*}
u_{k}^{\rm in}
&= \sum_{j=0}^k \binom{k}{j} ((\dt\varphi)_{t=0})^j\dx^ju_{(k-j)}^{\rm in}
 + (\dt^k\varphi)_{\vert_{t=0}}\dx u^{\rm in} + \mbox{LOT} \\
&= ((\dt\varphi)_{t=0}\mbox{Id}-A(u^{\rm in}))^k \dx^k u^{\rm in}
 + (\dt^k\varphi)_{\vert_{t=0}}\dx u^{\rm in} + \mbox{LOT},
\end{align*}
so that 
\[
u_{k \vert_{x=0}}^{\rm in} = (\ux_1^{\rm in}\mbox{Id}-A({u^{\rm in}}_{\vert_{x=0}}))^k (\dx^k u^{\rm in})_{\vert_{x=0}}
 + \ux_k^{\rm in}(\dx u^{\rm in})_{\vert_{x=0}} + \mbox{LOT}.
\]
We also have 
\[
(\dt^k u_{\rm i})_{\vert_{t=x=0}} = \ux_k^{\rm in}(\dx U_{\rm i})_{\vert_{t=x=0}}+ \mbox{LOT}.
\]
Therefore, the compatibility condition at order $k$ is given by 
\[
((\dx u^{\rm in})_{\vert_{x=0}}-(\dx U_{\rm i})_{\vert_{t=x=0}})^\perp \cdot
 \{ (\ux_1^{\rm in}\mbox{Id}-A({u^{\rm in}}_{\vert_{x=0}}))^k (\dx^k u^{\rm in})_{\vert_{x=0}} 
  + \mbox{LOT} \} = 0.
\]
Once we obtain these expressions to the compatibility conditions, 
the approximation stated in the proposition is obtained along classical lines. 
See for instance \cite{RauchMassey}. 
\end{proof}

\subsubsection{Reduction to a system with quasilinear boundary conditions}
At  first glance the boundary condition in \eqref{IBVPfbbis2} is nothing but a nonhomogeneous 
Dirichlet boundary condition. 
However, $u_{\rm i}(t,0) = U_{\rm i}(t,\ux(t))$ depends on the unknown free boundary $\ux$, which would be determined 
from the unknown $\partial^\varphi u$ through the evolution equation \eqref{eqFB3}. 
Therefore, the boundary condition represents implicitly a nonlinear relation between $u$ and 
its derivatives, so that we will reduce \eqref{IBVPfbbis2} 
to a system with standard quasilinear boundary conditions to solve the initial value problem 
\eqref{IBVPfbbis2}--\eqref{IC}. 
Now, suppose that $u$ is a solution to \eqref{IBVPfbbis2}. 
Putting 
\begin{equation}\label{u2}
u_{(2)} = \dt^\varphi \dt^\varphi u,
\end{equation}
we will derive a system for $u$ and $u_{(2)}$ with quasilinear boundary conditions 
together with a quasilinear evolution equation for $\ux$. 
We note that $\dt^\varphi$ and $\dx^\varphi$ commute. 
Applying differential operators $\dt^\varphi$ and $\dx^\varphi$ to the first equation in \eqref{IBVPfbbis2}, 
we can express $\dt^\varphi \dx^\varphi u$ and $\dx^\varphi \dx^\varphi u$ in terms of $u_{(2)}$, 
$u$, and $\partial^\varphi u$ as 
\begin{equation}\label{d2u}
\begin{cases}
\dt^\varphi \dx^\varphi u
= (-A(u)^{-1}) ( u_{(2)} + A'(u)[\dt^\varphi u]\dx^\varphi u ), \\
\dx^\varphi \dx^\varphi u
= (-A(u)^{-1})^2 ( u_{(2)} + A'(u)[\dt^\varphi u]\dx^\varphi u )
 + (-A(u)^{-1})A'(u)[\dx^\varphi u]\dx^\varphi u.
\end{cases}
\end{equation}
Applying $\dt^\varphi \dt^\varphi$ to the first equation in \eqref{IBVPfbbis2} and using the above relations, 
we obtain 
\[
\dt^\varphi u_{(2)} + A(u)\dx^\varphi u_{(2)} + B(u,\partial^\varphi u) u_{(2)}
 = f_{(2)}(u,\partial^\varphi u),
\]
where 
\begin{align*}
& B(u,\partial^\varphi u) u_{(2)}
 = A'(u)[u_{(2)}]\dx^\varphi u - 2A'(u)[\dt^\varphi u]A(u)^{-1}u_{(2)}, \\
& f_{(2)}(u,\partial^\varphi u)
 = 2A'(u)[\dt^\varphi u]A(u)^{-1}A'(u)[\dt^\varphi u]\dx^\varphi u
  - 2 A''(u)[\dt^\varphi u,\dt^\varphi u]\dx^\varphi u.
\end{align*}
This is an equation for $u_{(2)}$. 
We proceed to derive a boundary condition for $u_{(2)}$ and an evolution equation for $\ux$. 
Differentiating the boundary condition $u=u_{\rm i}$ on $x=0$ with respect to $t$ twice and 
using the relation $\dt = \dt^\varphi + (\dt\varphi)\dx^\varphi$, we have 
\[
\dt^\varphi\dt^\varphi u + 2\dot{\ux} \dt^\varphi\dx^\varphi u + \dot{\ux}^2 \dx^\varphi\dx^\varphi u
 + \ddot{\ux}\dx^\varphi u
 = \dt^\varphi\dt^\varphi u_{\rm i} + 2\dot{\ux} \dt^\varphi\dx^\varphi u_{\rm i}
  + \dot{\ux}^2 \dx^\varphi\dx^\varphi u_{\rm i} + \ddot{\ux}\dx^\varphi u_{\rm i}
\]
on $x=0$, where we used $\dt\varphi(t,0)=\dot{\ux}(t)$. 
This together with \eqref{d2u} implies 
\[
({\rm Id} - \dot{\ux}A(u)^{-1})^2 u_{(2)} + \ddot{\ux}( \dx^\varphi u - \dx^\varphi u_{\rm i} )
 = g_1(\dot{\ux},u,\partial^\varphi u,\partial^\varphi\partial^\varphi u_{\rm i}),
\]
where 
\begin{align*}
&g_1(\dot{\ux},u,\partial^\varphi u,\partial^\varphi\partial^\varphi u_{\rm i}) \\
&= \bigl( 2\dot{\ux}A(u)^{-1} - \dot{\ux}^2(A(u)^{-1})^2 \bigr)A'(u)[\dt^\varphi u]\dx^\varphi u
 + \dot{\ux}^2A(u)^{-1}A'(u)[\dx^\varphi u]\dx^\varphi u \\
&\quad\;
 + \dt^\varphi\dt^\varphi u_{\rm i} + 2\dot{\ux}\dt^\varphi\dx^\varphi u_{\rm i}
 + \dot{\ux}^2 \dx^\varphi\dx^\varphi u_{\rm i}.
\end{align*}
Decomposing this relation into the direction $\dx^\varphi u - \dx^\varphi u_{\rm i}$ and 
its perpendicular direction, we obtain an evolution equation for $\ux$ as 
\[
\ddot{\ux} = \chi(\dot{\ux},u,u_{(2)},\partial^\varphi u,\partial^\varphi u_{\rm i},
 \partial^\varphi\partial^\varphi u_{\rm i}),
\]
where 
\begin{align*}
& \chi(\dot{\ux},u,u_{(2)},\partial^\varphi u,\partial^\varphi u_{\rm i},\partial^\varphi\partial^\varphi u_{\rm i}) \\
&= \frac{(\dx^\varphi u - \dx^\varphi u_{\rm i}) \cdot 
 \bigl( g_1(\dot{\ux},u,\partial^\varphi u,\partial^\varphi\partial^\varphi u_{\rm i}) 
  - ({\rm Id} - \dot{\ux}A(u)^{-1})^2 u_{(2)} \bigr)}{|\dx^\varphi u - \dx^\varphi u_{\rm i}|^2}
\end{align*}
and a boundary condition for $u_{(2)}$ as 
\[
\nu_{(2)} \cdot u_{(2)}
 = g_{(2)},
\]
where $\nu_{(2)}=\nu_{(2)}(\dot{\ux},u,\dx^\varphi u,\dx^\varphi u_{\rm i})$ and 
$g_{(2)} = g_{(2)}(\dot{\ux},u,\partial^\varphi u,\partial^\varphi u_{\rm i},
\partial^\varphi\partial^\varphi u_{\rm i})$ are defined by 
\begin{equation}\label{nu2g2}
\begin{cases}
\nu_{(2)} = (({\rm Id} - \dot{\ux}A(u)^{-1})^2)^{\rm T}(
 (\dx^\varphi u - \dx^\varphi u_{\rm i})^\perp), \\
g_{(2)} = (\dx^\varphi u - \dx^\varphi u_{\rm i})^\perp \cdot 
 g_1(\dot{\ux},u,\partial^\varphi u,\partial^\varphi\partial^\varphi u_{\rm i}).
\end{cases}
\end{equation}
Concerning a boundary condition for $u$, we would like to write it in the form $\nu \cdot u = g$. 
However, we have a high degree of freedom for choosing the vector $\nu$. 
From the point of view of the maximal dissipativity in the sense of {\bf ii} in Assumption 1, 
the most convenient choice is $\nu = \unu$, where 
\[
\unu = {\bf e}_+(u^{\rm in}(0)). 
\]
As before, we introduce the matrix $\cA(u,\partial\varphi) = (\dx\varphi)^{-1}(A(u) - (\dt\varphi){\rm Id})$. 
The eigenvalues of this matrix are $(\dx\varphi)^{-1}(\pm\lambda_{\pm}(u) - \dt\varphi)$, whereas the 
corresponding eigenvectors are ${\bf e}_{\pm}(u)$ which does not depend on $\partial\varphi$. 
Summarizing the above arguments, the initial value problem \eqref{IBVPfbbis2}--\eqref{IC} yields the following:
\begin{equation}\label{qleq1}
\begin{cases}
 \dt u + \cA(u,\partial\varphi) \dx u = 0 & \mbox{in}\quad \Omega_T, \\
 u_{\vert_{t=0}} = u^{\rm in}(x) & \mbox{on}\quad \R_+, \\
 \unu \cdot u_{\vert_{x=0}} = \unu \cdot {u_{\rm i}}_{\vert_{x=0}} & \mbox{on}\quad (0,T),
\end{cases}
\end{equation}
together with
\begin{equation}\label{qleq2}
\begin{cases}
 \dt u_{(2)} + \cA(u,\partial\varphi) \dx u_{(2)}
  + B(u,\partial^\varphi u)u_{(2)} = f_{(2)}(u,\partial^\varphi u) & \mbox{in}\quad \Omega_T, \\
 u_{(2) \vert_{t=0}} = u_{(2)}^{\rm in}(x) & \mbox{on}\quad \R_+, \\
 \nu_{(2)} \cdot u_{(2) \vert_{x=0}} = g_{(2) \vert_{x=0}} & \mbox{on}\quad (0,T),
\end{cases}
\end{equation}
and an equation for the evolution of the free boundary given by 
\begin{equation}\label{qleq3}
\begin{cases}
 \ddot{\ux} = \chi(\dot{\ux},u,u_{(2)},\partial^\varphi u,\partial^\varphi u_{\rm i},
  \partial^\varphi\partial^\varphi u_{\rm i})_{\vert_{x=0}}
  & \mbox{for}\quad t\in(0,T), \\
 \ux(0)=0, \quad \dot{\ux}(0)=x_{(1)}^{\rm in},
\end{cases}
\end{equation}
where the initial data $u_{(2)}^{\rm in}$ and $x_{(1)}^{\rm in}$ should be chosen appropriately 
for the equivalence of \eqref{qleq1}--\eqref{qleq3} with \eqref{IBVPfbbis2}--\eqref{IC} 
and will be given in the next subsection.

\begin{remark} \ 
{\bf i.} 
In place of $\dt^\varphi \dt^\varphi u$ we can also use $\dt^2 u - (\dt^2 \varphi)\dx^\varphi u$ as $u_{(2)}$. 
An advantage of the choice \eqref{u2} is that the reduction and calculations become a little bit simpler. 

{\bf ii.} 
It is essential to differentiate \eqref{IBVPfbbis2} twice in time to derive a system 
with quasilinear boundary conditions. 
For example, the first derivative $u_{(1)}=\dt^\varphi u$ satisfies a boundary condition 
\[
(A(u)^{-1}u_{(1)} + \dx^\varphi u_{\rm i})^\perp \cdot (u_{(1)} - \dt^\varphi u_{\rm i})_{\vert_{x=0}}=0
 \quad\mbox{on}\quad (0,T),
\]
which is still nonlinear in $u_{(1)}$. 
\end{remark}

Then, we will analyze maximal dissipativity for \eqref{qleq2} in the sense of {\bf ii} in Assumption 1, 
that is, the positivity of $|\nu_{(2)} \cdot {\bf e}_+|$. 
The following proposition characterizes this condition algebraically under the restrictions 
\eqref{disconti} and \eqref{egenv}.

\begin{proposition}\label{conalg}
Suppose that $u$ together with $\ux$ is a smooth solution to \eqref{IBVPfbbis2} satisfying 
\eqref{disconti} and \eqref{egenv} and that $\nu_{(2)}$ is defined by \eqref{nu2g2}. 
Then, there exists a unique unit vector $\mu=\mu(t)$ up to the sign such that 
\[
\mu \cdot (\dt^\varphi u_{\rm i} + A(u_{\rm i})\dx^\varphi u_{\rm i})_{\vert_{x=0}} = 0.
\]
Moreover, we have the following identity on $x=0$:
\[
|\nu_{(2)} \cdot {\bf e}_+|
 = \frac{(\lambda_{+} - \dot{\ux})^3}{\lambda_+^2}
  \frac{|\dx^\varphi u - \dx^\varphi u_{\rm i}|}{|(\dot{\ux}{\rm Id}-A(u))^{\rm T}\mu|}
  |\mu \cdot  {\bf e}_+|.
\]
\end{proposition}

This proposition implies that the positivity of $|\nu_{(2)} \cdot {\bf e}_+|$ is essentially equivalent to 
the positivity of $|\mu \cdot  {\bf e}_+|$, where $\mu$ is a unique direction that the quantity 
$\dt^\varphi u+A(u)\dx^\varphi u$ is continuous across the boundary.

\begin{proof}[Proof of the proposition]
Differentiating the boundary condition in \eqref{IBVPfbbis2} with respect to $t$ and using the 
relation $\dt=\dt^\varphi+(\dt\varphi)\dx^\varphi$, we have 
$\dt^\varphi u + \dot{\ux}\dx^\varphi u = \dt^\varphi u_{\rm i} + \dot{\ux}\dx^\varphi u_{\rm i}$ on $x=0$. 
This and the interior equation in \eqref{IBVPfbbis2} imply 
\begin{equation}\label{rel1}
(\dot{\ux}{\rm Id} - A(u))(\dx^\varphi u - \dx^\varphi u_{\rm i})
 = \dt^\varphi u_{\rm i} + A(u_{\rm i})\dx^\varphi u_{\rm i}
 \quad\mbox{on}\quad x=0.
\end{equation}
Since the matrix $\dot{\ux}{\rm Id} - A(u)$ is invertible, it should hold that 
$(\dt^\varphi u_{\rm i} + A(u_{\rm i})\dx^\varphi u_{\rm i})_{\vert_{x=0}} \ne 0$. 
Therefore, the direction $\mu$ is uniquely determined up to the sign as 
\[
\mu = \frac{ ((\dt^\varphi u_{\rm i} + A(u_{\rm i})\dx^\varphi u_{\rm i})_{\vert_{x=0}})^\perp }{
 |(\dt^\varphi u_{\rm i} + A(u_{\rm i})\dx^\varphi u_{\rm i})_{\vert_{x=0}}| }.
\]
By taking the Euclidean inner product of \eqref{rel1} with $\mu$, we have 
\[
(\dot{\ux}{\rm Id} - A(u_{\vert_{x=0}}))^{\rm T}\mu
 \cdot (\dx^\varphi u - \dx^\varphi u_{\rm i})_{\vert_{x=0}} = 0.
\]
Since both vectors $(\dot{\ux}{\rm Id} - A(u_{\vert_{x=0}}))^{\rm T}\mu$ and 
$(\dx^\varphi u - \dx^\varphi u_{\rm i})_{\vert_{x=0}}$ are nonzero, so that 
\[
(\dx^\varphi u - \dx^\varphi u_{\rm i})_{\vert_{x=0}}^\perp
= \pm\frac{|(\dx^\varphi u - \dx^\varphi u_{\rm i})_{\vert_{x=0}}|}{|(\dot{\ux}{\rm Id}-A(u_{\vert_{x=0}}))^{\rm T}\mu|}
 (\dot{\ux}{\rm Id} - A(u_{\vert_{x=0}}))^{\rm T}\mu.
\]
Particularly, we see on $x=0$ that 
\begin{align*}
\nu_{(2)} \cdot {\bf e}_+ 
&=  (\dx^\varphi u - \dx^\varphi u_{\rm i})^\perp \cdot ({\rm Id} - \dot{\ux}A(u)^{-1})^2 {\bf e}_+ \\
&= (1-\dot{\ux}\lambda_+^{-1})^2 (\dx^\varphi u - \dx^\varphi u_{\rm i})^\perp \cdot {\bf e}_+ \\
&= \pm (1-\dot{\ux}\lambda_+^{-1})^2 
 \frac{|\dx^\varphi u - \dx^\varphi u_{\rm i}|}{|(\dot{\ux}{\rm Id}-A(u))^{\rm T}\mu|}
 (\dot{\ux}-\lambda_+) \mu \cdot {\bf e}_+,
\end{align*}
which gives the desired identity. 
\end{proof}

Once the diffeomorphism $\varphi$ is given, we can regard the initial boundary value problems 
\eqref{qleq1} and \eqref{qleq2} as the same type of problem considered in the previous sections. 
Concerning the compatibility conditions for the problems, 
it is straightforward to show the following lemma.

\begin{lemma}\label{comcon}
Suppose that the data $u^{\rm in} \in H^m(\R_+)$ and $U_{\rm i} \in W^{m,\infty}((0,T)\times(-\delta,\delta))$ 
for the initial boundary value problem \eqref{IBVPfbbis2}--\eqref{IC} satisfy the compatibility conditions 
up to order $m-1$ in the sense of Definition \ref{defCC} and 
that the diffeomorphism $\varphi$ satisfies $\varphi(0,x)=x$ and 
$(\dt^k\varphi)(0,0) = \ux_{(k)}$ for $k=1,\ldots,m-1$. 
\begin{enumerate}
\item[{\bf i.}]
The compatibility conditions for the initial boundary value problem \eqref{qleq1} are 
satisfied up to order $m-1$ in the sense of Definitions \ref{defcompVC}--\ref{defcompQL}. 

\item[{\bf ii.}]
Let $m\geq3$. 
If the initial datum $u_{(2)}^{\rm in}$ is given by \eqref{ukin2} and $u$ satisfies 
$((\dt^\varphi)^ku)_{\vert_{t=0}}=u_{(k)}^{\rm in}$ for $k=0,1,\ldots,m-1$, 
then the compatibility conditions for the initial boundary value problem \eqref{qleq1} are 
satisfied up to order $m-3$ in the sense of Definition \ref{defcompVC}. 
\end{enumerate}
\end{lemma}

\subsubsection{Proof of Theorem \ref{theoIBVP5}}
We will first show the existence of the solution $(u,u_{(2)},\ux)$ to the reduced system 
\eqref{qleq1}--\eqref{qleq3} with the diffeomorphism $\varphi$ given by \eqref{diffeo2} 
under an additional assumption $m\geq4$. 
Then, we will show that $(u,\ux)$ is in fact the solution to the original problem 
\eqref{IBVPfbbis2}--\eqref{IC}. 
In order to reduce the condition on $m$, we will derive an a priori estimate for the solution 
$(u,\ux)$ under the weaker assumption $m\geq2$, which together with Proposition \ref{appdata} 
and standard approximation technique gives the result stated in the theorem.

\medskip
\noindent
{\bf Step 1.}
Let $\mathcal{K}_1$ be a compact and convex set in $\R^2$ satisfying 
$\mathcal{K}_0 \Subset \mathcal{K}_1 \Subset \mathcal{U}$. 
We will construct the solution $(u,\ux)$ satisfying $u(t,x)\in\mathcal{K}_1$ and 
\eqref{disconti}--\eqref{egenv}.

\begin{lemma}\label{prepa}
Under the assumptions of Theorem \ref{theoIBVP5}, there exist positive constants 
$c_0,\varepsilon_0,\delta_0,C_0$, and $T_0 \in (0,T]$ such that if $u(t,x)$ and $\ux(t)$ satisfy 
\begin{equation}\label{precond}
\|u(t)-u^{\rm in}\|_{L^\infty}, |(\dx u (t,\cdot)- \dx u^{\rm in})_{\vert_{x=0}}|, 
 |\ux(t)-\ux_0^{\rm in}|, |\dt\ux(t)-\ux_1^{\rm in}| \leq \delta_0,
\end{equation}
and if $\varphi(t,x)$ is given by \eqref{diffeo2} with the choice $\varepsilon=\varepsilon_0$, 
then for $0\leq t\leq T_0$ we have 
\begin{enumerate}
\setlength{\itemsep}{2pt}
\item[{\bf i.}]
$u(t,x) \in \mathcal{K}_1$, 
\item[{\bf ii.}]
$\lambda_{\pm}(u(t,x)) \geq c_0$, $\lambda_{\pm}(u(t,x)) \mp \dt\varphi(t,x) \geq c_0$,
\item[{\bf iii.}]
$c_0 \leq |(\dx^\varphi u(t,\cdot) - \dx^\varphi u_{\rm i}(t,\cdot))_{\vert_{x=0}}| \leq C_0$,
\item[{\bf iv.}]
$|\nu_{(2)}(t) \cdot \mathbf{e}_{+}(u(t,\cdot)_{\vert_{x=0}})| \geq c_0$,
\item[{\bf v.}]
$\frac12 \leq \dx\varphi(t,x) \leq 2$, $|\dt\varphi(t,x)| \leq C_0$,
\end{enumerate}
where $\nu_{(2)}$ is given by \eqref{nu2g2}. 
\end{lemma}

\begin{proof}
It follows from the assumptions that there exists $c_0>0$ such that 
\[
\begin{cases}
\lambda_{\pm}(u^{\rm in}(x)) \geq 2c_0, \quad
 \lambda_{\pm}({u^{\rm in}}_{\vert_{x=0}}) \mp \ux_1^{\rm in} \geq 4c_0, \\
|(\dx u^{\rm in})_{\vert_{x=0}} - (\dx U_{\rm i})_{\vert_{t=x=0}}| \geq 2c_0, \\
\bigl( 1 - \frac{\ux_1^{\rm in}}{\lambda_{+}({u^{\rm in}}_{\vert_{x=0}})} \bigr)^2
 | ((\dx u^{\rm in})_{\vert_{x=0}} - (\dx U_{\rm i})_{\vert_{t=x=0}})^\perp \cdot 
  \mathbf{e}_{+}({u^{\rm in}}_{\vert_{x=0}}) | \geq 2c_0.
\end{cases}
\]
In view of $\dt\varphi(t,x)=\psi(\frac{x}{\varepsilon})\dt\ux(t)$, 
we proceed to show that if we choose $\varepsilon_0$ sufficiently small, then we have 
\[
\textstyle
\lambda_{\pm}(u^{\rm in}(x)) \mp \psi(\frac{x}{\varepsilon_0})\ux_1^{\rm in} \geq 2c_0.
\]
Since $\psi(\frac{x}{\varepsilon_0})=0$ for $x\geq 2\varepsilon_0$, 
it is sufficient to show this inequality for $0\leq x\leq 2\varepsilon_0$. 
In the case $\ux_1^{\rm in} \leq 0$ we easily get 
\[
\textstyle
\lambda_{+}(u^{\rm in}(x))-\psi(\frac{x}{\varepsilon_0})\ux_1^{\rm in} \geq \lambda_{+}(u^{\rm in}(x)) \geq 2c_0. 
\]
In the case $\ux_1^{\rm in} > 0$, for $0\leq x\leq 2\varepsilon_0$ we see that 
\begin{align*}
\textstyle
\lambda_{+}(u^{\rm in}(x)) - \psi(\frac{x}{\varepsilon_0})\ux_1^{\rm in}
&\geq \lambda_{+}(u^{\rm in}(x)) - \ux_1^{\rm in} \\
&= \lambda_{+}({u^{\rm in}}_{\vert_{x=0}}) - \ux_1^{\rm in} 
 + ( \lambda_{+}(u^{\rm in}(x)) - \lambda_{+}({u^{\rm in}}_{\vert_{x=0}}) ) \\
&\geq 4c_0 - 2\varepsilon_0 \|\nabla u^{\rm in}\|_{L^\infty}\max_{u \in \mathcal{K}_0}|\nabla_u\lambda_+(u)|.
\end{align*}
Therefore, if we choose $\varepsilon_0>0$ so small that 
$\varepsilon_0 \|\nabla u^{\rm in}\|_{L^\infty}\max_{u \in \mathcal{K}_0}|\nabla_u\lambda_+(u)| \leq c_0$, 
then we obtain $\lambda_{+}(u^{\rm in}(x)) - \psi(\frac{x}{\varepsilon_0})\ux_1^{\rm in} \geq 2c_0$. 
Similarly, we can show $\lambda_{-}(u^{\rm in}(x)) + \psi(\frac{x}{\varepsilon_0})\ux_1^{\rm in} \geq 2c_0$ 
so that the claim is proved. 

Now, we note that 
\[
\nu_{(2)}(0) \cdot \mathbf{e}_{+}(u_{\vert_{t=x=0}}) 
= \biggl( 1 - \frac{(\dt\ux)_{\vert_{t=0}}}{\lambda_{+}(u_{\vert_{t=x=0}})} \biggr)^2
 ( (\dx u)_{\vert_{t=x=0}} - (\dx U_{\rm i})_{\vert_{t=0,x=\ux(0)}} )^\perp
  \cdot \mathbf{e}_{+}(u_{\vert_{t=x=0}}), 
\]
where we used $(\dx\varphi)_{\vert_{x=0}}=1$. 
Therefore, by taking $\delta_0$ and $T_0$ sufficiently small,  we obtain the desired results. 
\end{proof}

We will construct the solution $(u,u_{(2)},\ux)$ as a limit of a sequence of approximate solutions 
$\{(u^n,u_{(2)}^n,\ux^n)\}_n$, which is defined as follows. 
We start to construct $\ux^1$ by 
\[
\ux^1(t) = \sum_{k=0}^{m-1}\frac{t^k}{k!} \ux_k^{\rm in}.
\]
Suppose that $\ux^n$ is given so that $(\dt^k\ux^n)_{\vert_{t=0}}=\ux_k^{\rm in}$ for $0\leq k\leq m-1$. 
We define the diffeomorphism $\varphi^n$ by \eqref{diffeo2} with the choice $\varepsilon=\varepsilon_0$, 
where $\varepsilon_0>0$ is the constant stated in Lemma \ref{prepa}. 
Thanks to Theorem \ref{theoIBVP3} together with Lemma \ref{comcon}, using the standard arguments such as 
those in the proof of Theorems \ref{theoIBVP2} and \ref{theoIBVP4}, we can define $u^n$ on a maximal time interval 
$[0,T_*^n)$ as a unique solution to 
\begin{equation}\label{qleq1n}
\begin{cases}
 \dt u^n + \cA(u^n,\partial\varphi^n) \dx u^n = 0 & \mbox{in}\quad (0,T_*^n)\times\R_+, \\
 {u^n}_{\vert_{t=0}} = u^{\rm in}(x) & \mbox{on}\quad \R_+, \\
 \unu \cdot {u^n}_{\vert_{x=0}} = \unu \cdot u_{\rm i}^n & \mbox{on}\quad (0,T_*^n),
\end{cases}
\end{equation}
where $u_{\rm i}^n=U_{\rm i}(t,\ux^n(t))$. 
Then, we see that $((\dt^{\varphi^n})^k u^n)_{\vert_{t=0}}=u_{(k)}^{\rm in}$ for $0\leq k\leq m-1$. 
Therefore, by Theorem \ref{theoIBVP3} together with Lemma \ref{comcon} again, 
we can define $u_{(2)}^n$ as a unique solution to 
\begin{equation}\label{qleq2n}
\begin{cases}
 \dt u_{(2)}^n + \cA(u^n,\partial\varphi^n) \dx u_{(2)}^n
  + B(u^n,\partial^{\varphi^n} u^n)u_{(2)}^n = f_{(2)}^n
   & \mbox{in}\quad (0,T_*^n)\times\R_+, \\
 u_{(2) \vert_{t=0}}^n = u_{(2)}^{\rm in}(x) & \mbox{on}\quad \R_+, \\
 \nu_{(2)}^n \cdot u_{(2) \vert_{x=0}}^n = g_{(2)}^n(t) & \mbox{on}\quad (0,T_*^n), 
\end{cases}
\end{equation}
where $f_{(2)}^n=f_{(2)}^n(u^n,\partial^{\varphi^n} u^n)$ and 
\[
\begin{cases}
\nu_{(2)}^n = \nu_{(2)}(\dt\ux^n,u^n,\dx^{\varphi^n} u^n,\dx^{\varphi^n}u_{\rm i}^n)_{\vert_{x=0}}, \\
g_{(2)}^n = g_{(2)}(\dt\ux^n,u^n,\partial^{\varphi^n} u^n,\partial^{\varphi^n}u_{\rm i}^n,
 \partial^{\varphi^n}\partial^{\varphi^n}u_{\rm i}^n)_{\vert_{x=0}}.
\end{cases}
\]
Then, we define $\ux^{n+1}$ as a unique solution to 
\begin{equation}\label{qleq3n}
\begin{cases}
 \dt^2\ux^{n+1} = \chi^n  \quad\mbox{for}\quad t\in(0,T_*^n), \\
 \ux^{n+1}(0) = 0, \quad (\dt\ux^{n+1})(0) = x_1^{\rm in},
\end{cases}
\end{equation}
where 
\[
\chi^n = \chi(\dt\ux^n,u^n,u_{(2)}^n,\partial^{\varphi^n}u^n,\partial^{\varphi^n}u_{\rm i}^n,
 \partial^{\varphi^n}\partial^{\varphi^n}u_{\rm i}^n)_{\vert_{x=0}}. 
\]
We see that $(\dt^k\ux^{n+1})_{\vert_{t=0}}=\ux_k^{\rm in}$ for $0\leq k\leq m-1$, so that 
we can define $(\ux^n,u^n,u_{(2)}^n)$ on a time interval $[0,T_*^n)$ for all $n\geq1$. 

We prove now that for $M_1,M_2,M_3$ large enough and $T_1$ small enough independent of $n$, 
we have $T_1\leq T_*^n$ and 
\begin{equation}\label{unifest}
\begin{cases}
\opnorm{u^n}_{\WW^{m-1}(T_1)} + |{u^n}_{\vert_{x=0}}|_{m-1,T_1} \leq M_1, \\
\opnorm{u_{(2)}^n}_{\WW^{m-2}(T_1)} + |u_{(2) \vert_{x=0}}|_{m-2,T_1} \leq M_2, \\
|\ux^{n}|_{H^m(0,T_1)} \leq M_3.
\end{cases}
\end{equation}
Here, by taking $T_1=T_1(M_1,M_2,M_3)$ small enough again we see that $u^n(t,x)$ and $\ux^n(t)$ 
satisfy \eqref{precond} so that we can apply Lemma \ref{prepa}. 
In the following, we denote inessential constants independent of $M_1,M_2,M_3$, and $n$ 
by the same symbol $C$, which may change from line to line. 
By \eqref{unifest}, without loss of generality we have also 
\begin{equation}\label{unifest2}
\|u^n\|_{W^{m-2,\infty}(\Omega_{T_1})}, \|u_{(2)}^n\|_{W^{m-3,\infty}(\Omega_{T_1})}, 
 \|\widetilde{\varphi}^n\|_{W^{m-1,\infty}(\Omega_{T_1})} \leq C,
\end{equation}
where $\widetilde{\varphi}^n(t,x)=\varphi^n(t,x)-x=\psi(\frac{x}{\varepsilon_0})\ux^n(t)$, 
so that 
\[
\begin{cases}
\|B(u^n,\partial^{\varphi^n}u^n)\|_{\WW^{m-2}(T_1)}, |\dt^{m-2}\nu_{(2)}^n|_{L^2(0,T_1)} \leq CM_1, \\
|\nu_{(2)}^n|_{W^{m-3,\infty}(0,T_1)} \leq C.
\end{cases}
\]
Therefore, it follows from Lemmas \ref{lemdiffeo2}, \ref{prepa}, and Theorem \ref{theoIBVP3} that 
\begin{align*}
\opnorm{u^n(t)}_{m-1} + |{u^n}_{\vert_{x=0}}|_{m-1,t}
 &\leq Ce^{C(M_1,M_3)t}( 1 + |u_{\rm i}^n|_{H^{m-1}(0,t)} ), \\
\opnorm{u_{(2)}^n(t)}_{m-2} + |u_{(2) \vert_{x=0}}^n|_{m-2,t}
 &\leq Ce^{C(M_1,M_3)t}\biggl( 1 + |\dt^{m-2}\nu_{(2)}^n|_{L^2(0,t)} \\
 &\qquad
  + |g_{(2)}^n|_{H^{m-2}(0,t)} + |f_{(2) \vert_{x=0}}^n|_{m-3,t}
  + \int_0^t \opnorm{f_{(2)}^n(t')}_{m-2}{\rm d}t' \biggr).
\end{align*}
It is easy to see that 
\[
|x^{n+1}|_{H^m(0,T_1)} \leq C(1+|\chi^n|_{H^{m-2}(0,T_1)}).
\]
Here, by \eqref{unifest}--\eqref{unifest2} we have 
\[
\begin{cases}
 |u_{\rm i}^n|_{H^{m-1}(0,T_1)}, |f_{(2) \vert_{x=0}}^n|_{m-3,T_1} \leq C, \\
 |g_{(2)}^n|_{H^{m-2}(0,T_1)},\|f_{(2)}^n\|_{\WW^{m-2}(T_1)} \leq C(1+M_1), \\
 |\chi^n|_{H^{m-2}(0,T_1)} \leq C(1+M_1+M_2).
\end{cases}
\]
Therefore, we obtain 
\[
\begin{cases}
\opnorm{u^n}_{\WW^{m-1}(T_1)} + |{u^n}_{\vert_{x=0}}|_{m-1,T_1} \leq Ce^{C(M_1,M_3)T_1}, \\
\opnorm{u_{(2)}^n}_{\WW^{m-2}(T_1)} + |u_{(2) \vert_{x=0}}^n|_{m-2,T_1} \leq Ce^{C(M_1,M_3)T_1}(1+M_1), \\
|\ux^{n}|_{H^m(0,T_1)} \leq C(1+M_1+M_2).
\end{cases}
\]
Putting $M_1=2C$, $M_2=2C(1+M_1)$, and $M_3=C(1+M_1+M_2)$, and taking $T_1$ sufficiently small, 
we see that \eqref{unifest} holds for all $n$. 

Once we have such uniform bounds for the approximate solutions, by considering the equations for 
$(u^{n+1}-u^n,u_{(2)}^{n+1}-u_{(2)}^n,\ux^{n+1}-\ux^n)$ as in the proof of Theorem \ref{theoIBVP4} 
and by taking $T_1$ sufficiently small, we can show that $\{(u^n,u_{(2)}^n,\ux^n)\}_n$ converges to 
$(u,u_{(2)},\ux)$ in $(\WW^{m-2}(T_1)\cap W^{1,\infty}(\Omega_{T_1})) \times \WW^{m-3}(T_1) \times H^m(0,T_1)$ 
and that the limit is a solution to \eqref{qleq1}--\eqref{qleq3}. 
Moreover, by the standard compactness and regularity arguments we see that the solution satisfies 
$(u,u_{(2)}) \in \WW^{m-1}(T_1) \cap \WW^{m-2}(T_1)$.

\medskip
\noindent
{\bf Step 2.}
We will show that the solution $(u,u_{(2)},\ux)$ to \eqref{qleq1}--\eqref{qleq3} constructed in Step 1 
is in fact a solution to \eqref{IBVPfbbis2}--\eqref{IC} and satisfies $\dt^\varphi\dt^\varphi u=u_{(2)}$. 
Putting $\widetilde{u}_{(2)}=\dt^\varphi\dt^\varphi u$, it is sufficient to show that 
$\widetilde{u}_{(2)}=u_{(2)}$ and the boundary condition $u=u_{\rm i}$ on $x=0$. 

Clearly, $u$ satisfies \eqref{d2u} with $u_{(2)}$ replaced by $\widetilde{u}_{(2)}$ so that 
$\widetilde{u}_{(2)}$ satisfies the same interior equation in \eqref{qleq2} as $u_{(2)}$. 
The boundary condition in \eqref{qleq2} for $u_{(2)}$ and the equation in \eqref{qleq3} for $\ux$ 
are equivalent to 
\begin{equation}\label{equiBC}
({\rm Id} - \dot{\ux}A(u)^{-1})^2 u_{(2)} + \ddot{\ux}( \dx^\varphi u - \dx^\varphi u_{\rm i} )
 = g_1(\dot{\ux},u,\partial^\varphi u,\partial^\varphi\partial^\varphi u_{\rm i})
 \quad\mbox{on}\quad x=0.
\end{equation}
On the other hand, by differentiating the boundary condition in \eqref{qleq1} for $u$ twice 
with respect to $t$ we see that 
\begin{align*}
0 &= \unu \cdot \dt^2(u-u_{\rm i})_{\vert_{x=0}} \\
&= \unu \cdot \bigl( ({\rm Id} - \dot{\ux}A(u)^{-1})^2 \widetilde{u}_{(2)}
 + \ddot{\ux}( \dx^\varphi u - \dx^\varphi u_{\rm i} )
 - g_1(\dot{\ux},u,\partial^\varphi u,\partial^\varphi\partial^\varphi u_{\rm i}) \bigr)_{\vert_{x=0}}.
\end{align*}
Eliminating $\ddot{\ux}$ from these two equations, we obtain 
\[
\unu \cdot ({\rm Id} - \dot{\ux}A(u)^{-1})^2 ( \widetilde{u}_{(2)} - u_{(2)})_{\vert_{x=0}} = 0.
\]
Therefore, $v_{(2)} = \widetilde{u}_{(2)} - u_{(2)}$ is a solution to the initial boundary value problem 
\[
\begin{cases}
 \dt v_{(2)} + \cA(u,\partial\varphi) \dx v_{(2)} + B(u,\partial^\varphi u)v_{(2)} = 0
   & \mbox{in}\quad \Omega_{T_1}), \\
 v_{(2) \vert_{t=0}} = 0 & \mbox{on}\quad \R_+, \\
 \widetilde{\nu}_{(2)} \cdot v_{(2) \vert_{x=0}} = 0 & \mbox{on}\quad (0,T_1), 
\end{cases}
\]
where $\widetilde{\nu}_{(2)}= (({\rm Id} - \dot{\ux}A(u_{\vert_{x=0}})^{-1})^2)^{\rm T}\unu$. 
Here, we have 
\[
\widetilde{\nu}_{(2)} \cdot \mathbf{e}_{+}(u_{\vert_{x=0}})
= \Bigl( 1 - \frac{\dot{\ux}}{\lambda_{+}(u_{\vert_{x=0}})} \Bigr)
 \mathbf{e}_{+}({u^{\rm in}}_{\vert_{x=0}}) \cdot \mathbf{e}_{+}(u_{\vert_{x=0}}),
\]
which is not zero. 
Therefore, we can apply Theorem \ref{theoIBVP3} to the above problem and the uniqueness of the solution 
gives $v_{(2)}=0$, that is, $\widetilde{u}_{(2)} = u_{(2)}$. 
Particularly, \eqref{equiBC} holds with $u_{(2)}$ replaced by $\widetilde{u}_{(2)}$. 

We proceed to show the boundary condition in \eqref{IBVPfbbis2}. 
Putting $w(t)=(u-u_{\rm i})_{\vert_{x=0}}$ we have
\[
\ddot{w} = \bigl( ({\rm Id} - \dot{\ux}A(u)^{-1})^2 \widetilde{u}_{(2)}
 + \ddot{\ux}( \dx^\varphi u - \dx^\varphi u_{\rm i} )
 - g_1(\dot{\ux},u,\partial^\varphi u,\partial^\varphi\partial^\varphi u_{\rm i}) \bigr)_{\vert_{x=0}}
 = 0.
\]
The compatibility conditions implies $w_{\vert_{t=0}}=\dot{w}_{\vert_{t=0}}=0$. 
Therefore, we obtain $w=0$, that is, $u=u_{\rm i}$ on $x=0$, so that 
$(u,\ux)$ is in fact the solution to \eqref{IBVPfbbis2}--\eqref{IC}. 
Uniqueness of the solution follows from that of the reduced problem \eqref{qleq1}--\eqref{qleq3}.

\medskip
\noindent
{\bf Step 3.}
In order to reduce the condition $m\geq4$ to $m\geq2$, we will derive an a priori estimate for 
the solution $(u,\ux)$ under this weaker assumption. 
Although we will again use the reduced system \eqref{qleq1}--\eqref{qleq3}, 
we can now use the relation $\dt^\varphi\dt^\varphi u = u_{(2)}$ to obtain an additional regularity of $u$. 
We will prove again that for $M_1,M_2,M_3$ large enough and $T_1$ small enough, we have 
\begin{equation}\label{unifest3}
\begin{cases}
\opnorm{u}_{\WW^{m-1}(T_1)} + |\vu|_{m-1,T_1} \leq M_1, \\
\opnorm{u_{(2)}}_{\WW^{m-2}(T_1)} + |u_{(2) \vert_{x=0}}|_{m-2,T_1} \leq M_2, \\
|\ux|_{H^m(0,T_1)} \leq M_3.
\end{cases}
\end{equation}
Let $c_0$ and $C_0$ be the constants in Lemma \ref{prepa}. 
By Lemma \ref{lemdiffeo2}, there exists $K_0$ independent of $M_1,M_2,M_3$ such that 
\[
\frac{1}{c_0}, C_0, \opnorm{\partial\widetilde{\varphi}(0)}_{m-2},  |\unu|, 
 \opnorm{u(0)}_{m-1}, \opnorm{u_{(2)}(0)}_{m-2}, \sum_{j=0}^{m-1}|\ux_j^{\rm in}| \leq K_0.
\]
Moreover, by taking $T_1=T_1(M_1,M_2,M_3)$ sufficiently small if necessary, we have 
\begin{equation}\label{estbyK0}
|\nu_{(2)}|_{L^\infty(0,T_1)}, |\ux|_{W^{m-1,\infty}(0,T_1)}, 
 \|\widetilde{\varphi}\|_{W^{m-1,\infty}(\Omega_{T_1})}, \|\dx\widetilde{\varphi}\|_{W^{m-1,\infty}(\Omega_{T_1})}
 \leq C(K_0).
\end{equation}
Let $K$ be a constant such that $K_0,M_1,M_2,M_3 \leq K$.

\begin{lemma}\label{addreg}
For a smooth solution $(u,\ux)$ to \eqref{IBVPfbbis2} with $\varphi$ given by \eqref{diffeo2}
satisfying \eqref{unifest3} and \eqref{estbyK0}, we have 
\[
 \|\dx u\|_{\WW^{m-1}(T_1)}, \|u\|_{W^{m-1,\infty}(\Omega_{T_1})}, |\vu|_{m,T_1} \leq C(K). 
\]
\end{lemma}

\begin{proof}
We begin to evaluate $\opnorm{\dx u(t)}_{m-1}$. 
In view of the identities 
\begin{equation}\label{iden}
\begin{cases}
\dx^2u = (\dx\varphi)^2\dx^\varphi\dx^\varphi u + (\dx^2\varphi)\dx^\varphi u, \\
\dt\dx u = (\dx\varphi)\{ \dt^\varphi\dx^\varphi u + (\dt\varphi)\dx^\varphi\dx^\varphi u
 + (\dx^\varphi\dt\varphi)\dx^\varphi u\},
\end{cases}
\end{equation}
we see that 
\begin{align}\label{estdxu}
\opnorm{\dx u(t)}_{m-1}
&\leq \opnorm{\dx^2u(t)}_{m-2} + \opnorm{\dt\dx u(t)}_{m-2} + \opnorm{\dx u(t)}_{m-2} \\
&\leq C(K_0)( \opnorm{\dx^\varphi\dx^\varphi u(t)}_{m-2} + \opnorm{\dt^\varphi\dx^\varphi u(t)}_{m-2}
 + \opnorm{u(t)}_{m-1}). \nonumber
\end{align}
We note that $u$ satisfies \eqref{d2u}. 
In the case $m\geq3$, by Lemmas \ref{ineq1}--\ref{ineq2} we have 
\[
\opnorm{\dx^\varphi\dx^\varphi u(t)}_{m-2} + \opnorm{\dt^\varphi\dx^\varphi u(t)}_{m-2}
\leq C(\opnorm{u(t)}_{m-2})( \opnorm{u_{(2)}(t)}_{m-2} + \opnorm{\partial^\varphi u(t)}_{m-2}^2 ),
\]
which together with \eqref{estdxu} implies $\opnorm{\dx u(t)}_{m-1} \leq C(K)$. 
In the case $m=2$, by using the Sobolev imbedding theorem 
$\|u\|_{L^\infty} \leq \sqrt{2}\|u\|_{L^2}^{1/2}\|\dx u\|_{L^2}^{1/2}$ we have 
\begin{align*}
\|\dx^\varphi\dx^\varphi u(t)\|_{L^2} + \|\dt^\varphi\dx^\varphi u(t)\|_{L^2}
&\leq C(K_0)( \|u_{(2)}(t)\|_{L^2} + \|\partial u(t)\|_{L^2}\|\dx u(t)\|_{L^\infty} ) \\
&\leq C(K_0)( \|u_{(2)}(t)\|_{L^2} + \opnorm{u(t)}_1^{3/2}\opnorm{\dx u(t)}_1^{1/2} ),
\end{align*}
which together with \eqref{estdxu} implies 
\[
\opnorm{\dx u(t)}_1 \leq C(K_0)( \|u_{(2)}(t)\|_{L^2} + \opnorm{u(t)}_1 + \opnorm{u(t)}_1^3 ) \leq C(K).
\]
Therefore, in any case we have $\opnorm{\dx u(t)}_{m-1} \leq C(K)$, 
which together with the Sobolev imbedding theorem yields 
\[
\|u\|_{W^{m-1,\infty}(\Omega_{T_1})}
\leq C \|u\|_{\WW^{m-1}(T_1)}^{1/2}\|\dx u\|_{\WW^{m-1}(T_1)}^{1/2}
\leq C(K).
\]

We proceed to evaluate $|\vu|_{m,t}$. 
In view of \eqref{iden} and the identity 
\[
\dt^2 u = u_{(2)} + (\dt^2\varphi)\dx^\varphi u + 2(\dt\varphi)\dt^\varphi\dx^\varphi u
 + (\dt\varphi)^2\dx^\varphi\dx^\varphi u,
\]
we see that 
\begin{align*}
|\vu|_{m,t} 
&\leq |(\dt^2 u)_{\vert_{x=0}}|_{m-2,t} + |(\dt\dx u)_{\vert_{x=0}}|_{m-2,t} + |(\dx^2 u)_{\vert_{x=0}}|_{m-2,t} + |\vu|_{m-1,t} \\
&\leq C(K_0)\bigl( |u_{(2) \vert_{x=0}}|_{m-2,t} + |\vu|_{m-1,t} \\
&\quad
  + |(\dt^2\varphi)_{\vert_{x=0}}|_{m-2,t}\|\dx u\|_{L^\infty(\Omega_t)}
  + |(\dx^\varphi\dx^\varphi u)_{\vert_{x=0}}|_{m-2,t} + |(\dt^\varphi\dx^\varphi u)_{\vert_{x=0}}|_{m-2,t} \bigr). 
\end{align*}
Here, we have $|(\dt^2\varphi)_{\vert_{x=0}}|_{m-2,t} \leq C|\ux|_{H^m(0,t)}$. 
Noting again that $u$ satisfies \eqref{d2u} and using Lemma \ref{ineq2} we have 
\[
|(\dx^\varphi\dx^\varphi u)_{\vert_{x=0}}|_{m-2,t} + |(\dt^\varphi\dx^\varphi u)_{\vert_{x=0}}|_{m-2,t}
\leq C(K)( |u_{(2) \vert_{x=0}}|_{m-2,t} + 1 ) \leq C(K).
\]
Therefore, we obtain $|\vu|_{m,T_1} \leq C(K)$. 
\end{proof}

Thanks of this lemma, by taking $T_1$ sufficiently small we have \eqref{precond} and 
\[
\|u\|_{W^{m-2,\infty}(\Omega_{T_1})} \leq C(K_0).
\]
Without loss of generality we can also assume $\|U_{\rm i}\|_{W^{m,\infty}((0,T)\times(-\delta,\delta))} \leq K_0$. 
Since $u$ is a solution to \eqref{qleq1}, we can apply Theorem \ref{theoIBVP3} with $m$ replaced by 
$m-1$ to $u$ and obtain 
\begin{align*}
\opnorm{u(t)}_{m-1} + |\vu|_{m-1,t}
&\leq C(K_0)e^{C(K)t}( \opnorm{u(0)}_{m-1} + |u_{\rm i}|_{H^{m-1}(0,t)} ) \\
&\leq C(K_0)e^{C(K)t}( \opnorm{u(0)}_{m-1} + 1 ).
\end{align*}
We note that $u_{(2)}$ is a solution to \eqref{qleq2} and that in the case of $m\geq3$ we have 
\[
\|B(u,\partial^\varphi u)\|_{\WW^{m-2}(T_1)}, |\nu_{(2)}|_{W^{1,\infty} \cap W^{m-3,\infty}(0,T_1)}, 
 |\dt^{m-2}\nu_{(2)}|_{L^2(0,T_1)} \leq C(K).
\]
Therefore, thanks of Lemma \ref{prepa} we can apply Theorem \ref{theoIBVP3} with $m$ replaced 
by $m-2$ in the case $m\geq3$ and Proposition \ref{propNRJ1} together with Lemma \ref{lemsymmetrizerbis} 
in the case $m=2$ to $u_{(2)}$ and obtain 
\begin{align*}
\opnorm{u(t)}_{m-2} + |\vu|_{m-2,t}
&\leq C(K_0) e^{C(K)t} \biggl( (1+|\dt^{m-2}\nu_{(2)}|_{L^2(0,t)})\opnorm{u_{(2)}(0)}_{m-2} \\
&\quad
 + |g_{(2)}|_{H^{m-2}(0,t)} + |f_{(2) \vert_{x=0}}|_{m-3,t} + \int_0^t \opnorm{f_{(2)}(t')}_{m-2}{\rm d}t' \biggr),
\end{align*}
where the term $|f_{(2) \vert_{x=0}}|_{m-3,t}$ is dropped in the case $m=2$. 
Here, we have 
\[
|\nu_{(2)}|_{W^{m-2,\infty}(0,T_1)}, |g_{(2)}|_{W^{m-2,\infty}(0,T_1)},
 \|f_{(2)}\|_{W^{m-2,\infty}(\Omega_{T_1}) \cap \WW^{m-2}(T_1)} \leq C(K),
\]
so that 
\[
\opnorm{u(t)}_{m-2} + |\vu|_{m-2,t} \leq C(K_0)e^{C(K)t}( 1+C(K)\sqrt{t} ) ( \opnorm{u_{(2)}(0)}_{m-2} + 1 ).
\]
Since $\ux$ is a solution to \eqref{qleq3}, we see that 
\[
|\ux|_{H^m(0,T_1)} \leq C(K_0)( 1 + |u_{(2) \vert_{x=0}}|_{m-2,t} + |\vu|_{m-1,t} ).
\]
Therefore, if we define the constants $M_1,M_2,M_3$ by 
\[
\begin{cases}
M_1 = 2C(K_0)( \opnorm{u(0)}_{m-1} + 1 ), \\
M_2 = 2C(K_0)( \opnorm{u_{(2)}(0)}_{m-2} + 1 ), \\
M_3 = C(K_0)( 1 + M_1 + M_2 ),
\end{cases}
\]
and if we take $T_1=T_1(K)$ sufficiently small, then \eqref{unifest3} holds. 
The proof of Theorem \ref{theoIBVP5} is complete.

\subsubsection{An extension to a system coupled with ODEs}\label{sectext}
In application to physical and engineering problems, the free boundary problem 
\eqref{IBVPfb2}--\eqref{fbBC} appears coupled with a system of ordinary differential equations for the
unknown $W=W(t)$, which takes its value in $\R^N$. 
We will extend Theorem \ref{theoIBVP5} to such a problem. 
More precisely, we consider \eqref{IBVPfb2}--\eqref{fbBC} with the boundary data $U_{\rm i}$ of the form 
$U_{\rm i}(t,x)=G_{\rm i}(W(t),x)$, where $G_{\rm i}(W,x)$ is a given function whereas $W(t)$ satisfies 
\begin{equation}\label{ODE}
\begin{cases}
\dot{W} = F(W,\ux) & \mbox{in}\quad (0,T), \\
W = W^{\rm in} & \mbox{on}\quad \{t=0\}.
\end{cases}
\end{equation}
As before, we will use the diffeomorphism $\varphi(t,\cdot) : \R_{+} \to (\ux(t),\infty)$ given by 
Lemma \ref{lemdiffeo2} and put $u=U\circ\varphi$. 
Then, the problem is recast as 
\begin{equation}\label{nlfbp}
\begin{cases}
 \dt^\varphi u + A(u)\dx^\varphi u = 0 & \mbox{in}\quad \Omega_T, \\
 u_{\vert_{t=0}} = u^{\rm in}(x) & \mbox{on}\quad \R_{+}, \\
 u_{\vert_{x=0}} = u_{\rm i}(t) & \mbox{on}\quad (0,T)
\end{cases}
\end{equation}
with $\ux(0)=0$, 
where $u_{\rm i}(t) = G_{\rm i}(W(t),\ux(t))$.

\begin{assumption}\label{assFG}
Let $\mathcal{W}$ be an open set in $\R^N$, which represents a phase space of $W$. 
We have $G_{\rm i}, F \in W^{m,\infty}(\mathcal{W}\times(-\delta,\delta))$. 
\end{assumption}

\begin{theorem}\label{theoIBVP6}
Let $m\geq 2$ be an integer. 
Suppose that Assumptions \ref{asshypNLFB}--\ref{assFG} are satisfied. 
If $u^{\rm in}\in H^m(\R_+)$ takes its values in a compact and convex set ${\mathcal K}_0\subset \cU$ and 
if the data $u^{\rm in}$ and $W^{\rm in} \in \mathcal{W}$ satisfy 
\begin{enumerate}
\item[{\bf i.}]
$\lambda_{\pm}({u^{\rm in}}_{\vert_{x=0}}) \mp \ux_1^{\rm in} >0$,
\item[{\bf ii.}]
$(\dx u^{\rm in})_{\vert_{x=0}} - (\dx G_{\rm i})_{\vert_{W=W^{\rm in},x=0}} \ne0$,
\item[{\bf iii.}]
$((\dx u^{\rm in})_{\vert_{x=0}} - (\dx G_{\rm i})_{\vert_{W=W^{\rm in},x=0}})^\perp \cdot 
 \mathbf{e}_{+}({u^{\rm in}}_{\vert_{x=0}}) \ne0$,
\end{enumerate}
where $\ux_1^{\rm in}=(\dt\ux)_{\vert_{t=0}}$ will be determined by \eqref{xkin2} below, 
and the compatibility conditions up to order $m-1$ in the sense of Definition \ref{defCC2} below, 
then there exist $T_1 \in (0,T]$ and a unique solution 
$(u,\ux)$ to \eqref{ODE}--\eqref{nlfbp} with 
$u, \dx u \in \WW^{m-1}(T_1)$, $\ux\in H^m(0,T_1)$, $W\in H^{m+1}(0,T_1)$, 
and $\varphi$ given by Lemma \ref{lemdiffeo2}. 
\end{theorem}

\begin{remark}\label{remarkIC2}
As stated in Remark \ref{remarkIC}, the condition {\bf iii} in the theorem can be replaced by 
\begin{enumerate}
\item[{\bf iii$'$.}]
$\mu_0 \cdot \mathbf{e}_{+}({u^{\rm in}}_{\vert_{x=0}}) \ne0$,
\end{enumerate}
where $\mu_0$ is the unit vector satisfying 
$\mu_0 \cdot (\dt U_{\rm i} + A(U_{\rm i})\dx U_{\rm i})_{\vert_{t=x=0}} =0$ with 
$U_{\rm i}(t,x)=G_{\rm i}(W(t),x)$. 
This unit vector $\mu_0$ is uniquely determined up to the sign under the other assumptions 
of the theorem. 
\end{remark}

\begin{proof}[Outline of the proof of Theorem \ref{theoIBVP6}]
The solution $(u,\ux,W)$ can be constructed as a limit of a sequence of approximate solutions 
$\{(u^n,\ux^n,W^n)\}_n$, which are defined by 
\[
\begin{cases}
 \dt u^n + \cA(u^n,\partial\varphi^n)\dx u^n = 0 & \mbox{in}\quad \Omega_T, \\
 {u^n}_{\vert_{t=0}} = u^{\rm in}(x) & \mbox{on}\quad\R_{+}, \\
 {u^n}_{\vert_{x=0}} = u_{\rm i}^n(t) & \mbox{on}\quad (0,T)
\end{cases}
\]
with $\ux^n(0)=0$, 
where $u_{\rm i}^n(t)=G_{\rm i}(W^n(t),\ux^n(t))$ and $\varphi^n$ is given by \eqref{diffeo2} 
with $\varepsilon=\varepsilon_0$ and $\ux$ replaced by $\ux^n$, and 
\[
\begin{cases}
\dot{W}^{n+1} = F(W^n,\ux^n) & \mbox{for}\quad t\in(0,T), \\
W^{n+1}(0) = W^{\rm in}.
\end{cases}
\]
Under the condition $|W^n|_{W^{m-1,\infty}(0,T)}, |\ux^n|_{W^{m-1,\infty}(0,T)} \leq C(K_0)$ we have 
\[
|W^{n+1}|_{H^{m+1}(0,T)} \leq C(K_0)(|W^n|_{H^m(0,T)} + |\ux^n|_{H^m(0,T)} + 1).
\]
Therefore, we can apply Theorem \ref{theoIBVP5} for the existence of the solution $(u^n,\ux^n)$ with 
uniform bounds in appropriate function spaces, so that we can pass to the limit $n\to\infty$ to 
obtain the desired solution. 
\end{proof}

\subsubsection{Compatibility conditions}\label{sectCC}
Suppose that $(u,\ux,W)$ be a smooth solution to \eqref{ODE}--\eqref{nlfbp}. 
As in \S \ref{sscomp}, we define $u_{(k)}^{\rm in}=((\dt^\varphi)^ku)_{\vert_{t=0}}$ by \eqref{ukin2}. 
We denote $W_k^{\rm in}=(\dt^k W)_{\vert_{t=0}}$ and $\ux_k^{\rm in}=(\dt^k\ux)_{\vert_{t=0}}$ as before. 
It follows from $\dot{W}=F(W,\ux)$ that 
\begin{equation}\label{Wkin}
W_{k+1}^{\rm in} = c_{3,k}(W_0^{\rm in},W_1^{\rm in},\ldots,W_k^{\rm in}, 
 \ux_0^{\rm in},\ux_1^{\rm in},\ldots,\ux_k^{\rm in})
\end{equation}
Using the relation $U_{\rm i}(t,x)=G_{\rm i}(W(t),x)$, we have 
\[
(\dt^k\dx^l U_{\rm i})_{\vert_{t=x=0}} = c_{2,k,l}(W_0^{\rm in},W_1^{\rm in},\ldots,W_k^{\rm in}). 
\]
This together with \eqref{xkin} yields 
\begin{align}\label{xkin2}
\ux_{k}^{\rm in} 
&= -\frac{\dx u^{\rm in} - (\dx G_{\rm i})_{\vert_{W=W^{\rm in}}}}{
 |\dx u^{\rm in} - (\dx U_{\rm i})_{\vert_{W=W^{\rm in}}}|^2}
 \cdot \biggl\{ u_{(k)} ^{\rm in} - c_{2,k,0}(W_0^{\rm in},W_1^{\rm in},\ldots,W_k^{\rm in}) \\
&\qquad
 + \sum_{l=2}^k\sum_{\substack{j_0+j_1+\cdots+j_l=k \\ 1\leq j_1,\ldots,j_l}}
  c_{l,j_0,\ldots,j_l}\ux_{j_1}^{\rm in}\cdots\ux_{j_l}^{\rm in}
   \bigl( \dx^l u_{(j_0)}^{\rm in} - c_{2,j_0,l}(W_0^{\rm in},W_1^{\rm in},\ldots,W_{j_0}^{\rm in}) \bigr)
    \biggr\}_{\vert_{x=0}}. \nonumber
\end{align}
Now, we can calculate $\ux_k^{\rm in}$ and $W_k^{\rm in}$ inductively by 
$\ux_0^{\rm in}=0$, $W_0^{\rm in}=W^{\rm in}$, and \eqref{Wkin}--\eqref{xkin2} 
in terms of the data $u^{\rm in}$ and $W^{\rm in}$.

\begin{definition}\label{defCC2}
Let $m\geq1$ be an integer. 
We say that the data $u^{\rm in} \in H^m(\R_+)$ and $W^{\rm in}$ for the problem 
\eqref{ODE}--\eqref{nlfbp} satisfy the compatibility condition at order $k$ if 
$\{u_{(j)}^{\rm in}\}_{j=0}^m$ and $\{\ux_{(j)}^{\rm in}\}_{j=0}^{m-1}$ defined by \eqref{ukin2} 
and \eqref{xkin2} satisfy $u^{\rm in}(0) = G_{\rm i}(W^{\rm in},0)$ in the case $k=0$ and 
\begin{align*}
& (\dx u^{\rm in} - (\dx G_{\rm i})_{\vert_{W=W^{\rm in}}})^\perp \cdot \biggl\{
 u_{(k)}^{\rm in} - c_{2,k,0}(W_0^{\rm in},W_1^{\rm in},\ldots,W_k^{\rm in}) \\
& + \sum_{l=2}^k\sum_{\substack{j_0+j_1+\cdots+j_l=k \\ 1\leq j_1,\ldots,j_l}}
  c_{l,j_0,\ldots,j_l}\ux_{(j_1)}^{\rm in}\cdots\ux_{(j_l)}^{\rm in}
   \bigl( \dx^l u_{(j_0)}^{\rm in} - c_{2,j_0,l}(W_0^{\rm in},W_1^{\rm in},\ldots,W_{j_0}^{\rm in}) \bigr)
  \biggr\}_{\vert_{x=0}}
 = 0
\end{align*}
in the case $k\geq 1$. 
We say also that the data $u^{\rm in}$ and $W_k^{\rm in}$ for the problem \eqref{ODE}--\eqref{nlfbp} 
satisfy the compatibility conditions up to order $m-1$ if they satisfy the compatibility 
conditions at order $k$ for $k=0,1,\ldots,m-1$. 
\end{definition}

Roughly speaking, the definition of $\ux_{k}^{\rm in}$ ensures the equality 
$\dt^k u =\dt^k u_{\rm i}$ at $x=t=0$ in the direction $\dx^\varphi u - \dx^\varphi u_i$, 
whereas the compatibility conditions ensure it in the perpendicular direction 
$(\dx^\varphi u - \dx^\varphi u_i)^\perp$.

\section{Transmission problems}\label{secttransmission}
We proposed in Section \ref{sect2} a general approach to study initial boundary value problems with a possibly 
free boundary for $2\times 2$ hyperbolic systems. 
Our results can easily be extended to systems involving more equations, provided that the diaganalizability 
properties used in Proposition \ref{propBC} to construct the Kreiss symmetrizer are still valid. 
This is for instance the case for transmission problems involving the coupling of two $2\times2$ hyperbolic 
systems across an interface. 
Such problems can be transformed into a $4\times 4$ initial boundary value problems that have the required 
diagonalizability properties. 
Transmission problems being relevant for many applications, we devote this section to their study.

\subsection{Variable coefficients linear $2\times 2$ transmission problems}\label{sectVCtransm}
We consider here a linear transmission problem, where we seek a solution $u$ solving a linear hyperbolic 
system on $\Omega_T^- = (0,T)\times \R_-$, another one (possibly the same) for $\Omega_T^+ = (0,T)\times \R_+$, 
assuming that a transmission condition is provided at the interface $\{x=0\}$ 
\begin{equation}\label{transmVC}
\begin{cases}
 \dt u + \widetilde A(t,x)\dx u + \widetilde B(t,x) u = \widetilde f(t,x) &\mbox{in}\quad \Omega_T^-, \\
 \dt u +  A(t,x)\dx u +  B(t,x) u =  f(t,x) &\mbox{in}\quad \Omega_T^+, \\
 u_{\vert_{t=0}} = u^{\rm in}(x) & \mbox{on}\quad \R_-\cup \R_+, \\
 N_p^{\rm r}(t)u_{\vert_{x=+0}} - N_p^{\rm l}(t) u_{\vert_{x=-0}} = \mybf{g}(t)& \mbox{on}\quad (0,T), 
\end{cases}
\end{equation}
where $u$, $u^{\rm in}$, $f$, and $\widetilde{f}$ are $\R^2$-valued functions, 
$\mybf{g}$ is a $\R^p$-valued function, while $A$, $\widetilde A$, $B$, and $\widetilde B$ take their values 
in the space of $2\times2$ real-valued matrices. 
The matrices $N_p^{\rm l}$ and $N_p^{\rm l}$ that appear in the transmission condition are of size $p\times 2$, 
where $p$ (the number of scalar transmission conditions) depends on the sign of the eigenvalues of 
$\widetilde{A}$ and $A$.

\begin{notation}\label{notanumberp}
We shall consider three possibilities corresponding to the following cases, 
where $\widetilde{\lambda}_{\pm,j}(t,-x)$ and ${\lambda}_{\pm,j}(t,x)$ ($j=1,2,\emptyset$) 
are assumed to be strictly positive for all $(t,x)\in \Omega_T$: 
\begin{itemize}
\item
{\bf Case $p=1$.} There is one outgoing characteristic, that is, one of the following two situations holds: 
\begin{itemize}
\item
The matrices $\widetilde{A}(t,-x)$ and $A(t,x)$ have eigenvalues $\pm\widetilde{\lambda}_\pm(t,-x)$ and 
$-\lambda_{-,j}(t,x)$ ($j=1,2$), respectively.
\item
The matrices $\widetilde{A}(t,-x)$ and $A(t,x)$ have eigenvalues $\widetilde{\lambda}_{+,j}(t,-x)$ ($j=1,2$) 
and $\pm\lambda_{\pm}(t,x)$, respectively.
\end{itemize}
\item
{\bf Case $p=2$.} There are two outgoing characteristics, that is, the matrices $\widetilde{A}(t,-x)$ and 
$A(t,x)$ have eigenvalues $\pm\widetilde{\lambda}_\pm(t,-x)$ and $\pm\lambda_{\pm}(t,x)$, respectively. 
\item
{\bf Case $p=3$.} There are three outgoing characteristics, that is, one of the following two situations holds:
\begin{itemize}
\item
The matrices $\widetilde{A}(t,-x)$ and $A(t,x)$ have eigenvalues $\pm\widetilde{\lambda}_\pm(t,-x)$ and 
$\lambda_{+,j}(t,x)$ ($j=1,2$), respectively. 
\item
The matrices $\widetilde{A}(t,-x)$ and $A(t,x)$ have eigenvalues $-\widetilde{\lambda}_{-,j}(t,-x)$ ($j=1,2$) 
and $\pm\lambda_{\pm}(t,x)$, respectively. 
\end{itemize}
\end{itemize}
Denoting by $\widetilde{\bf e}_{\pm,j}(t,-x)$ and ${\bf e}_{\pm,j}(t,x)$ unit eigenvectors associated to 
the eigenvalues $\widetilde{\lambda}_{\pm,j}(t,-x)$ and ${\lambda}_{\pm,j}(t,x)$ ($j=1,2,\emptyset$), 
we define a $4\times p$ matrix $\mybf{E}_p(t)$ by 
\[
\mybf{E}_p(t) = \left(\begin{array}{cc} \widetilde{\mybf{E}}_-(t) & 0_{2\times p^{\rm r}} \\
0_{2\times p^{\rm l}} & \mybf{E}_+(t)\end{array}\right),
\]
where $0\leq p^{\rm l}\leq 2$ (resp. $0\leq p^{\rm r}\leq 2$) denotes the number of negative eigenvalues of 
$\widetilde{A}(t,0)$ (resp. positive eigenvalues of $A(t,0)$), and $\widetilde{\mybf{E}}_-(t)$ and 
$\mybf{E}_+(t)$ the matrix formed by the corresponding eigenvectors. 
\end{notation}

\begin{remark}
Here we did not list any possible cases, that is, the cases $p=0,4$ are omitted. 
Moreover, even in the case $p=2$ there are two other posibilities. 
Such cases can be treated in the same way so we omit them. 
\end{remark}

It is convenient to recast \eqref{transmVC} as a $4\times 4$ initial boundary value problem by setting 
\begin{equation}\label{leftright}
\begin{array}{llll}
A^{\rm r}(t,x)=A(t,x), & B^{\rm r}(t,x)=B(t,x), & f^{\rm r}(t,x)=f(t,x), & u^{\rm r}(t,x)=u(t,x), \\
A^{\rm l}(t,x)=\widetilde A(t,-x), & B^{\rm l}(t,x)=\widetilde B(t,-x), & f^{\rm l}(t,x)=\widetilde f(t,-x), 
 & u^{\rm l}(t,x)=u(t,-x),
\end{array}
\end{equation}
and 
\begin{equation}\label{notaAB}
\mybf{A} = \left(\begin{array}{cc} -A^{\rm l} & 0_{2\times 2} \\ 0_{2\times 2} & A^{\rm r} \end{array}\right), \qquad
\mybf{B} = \left(\begin{array}{cc} B^{\rm l} & 0_{2\times 2} \\ 0_{2\times 2} & B^{\rm r} \end{array}\right), \qquad
\mybf{u} = \left( \begin{array}{c} u^{\rm l} \\ u^{\rm r} \end{array}\right), \qquad
\mybf{f} = \left( \begin{array}{c} f^{\rm l} \\ f^{\rm r} \end{array}\right).
\end{equation}
The transmission problem \eqref{transmVC} is equivalent to the following initial boundary value problem 
\begin{equation}\label{transmref}
\begin{cases}
\dt \mybf{u} + \mybf{A}(t,x)\dx \mybf{u} + \mybf{B}(t,x)\mybf{u} = \mybf{f}(t,x) &\mbox{in}\quad \Omega_T, \\
\mybf{u}_{\vert_{t=0}} = \mybf{u}^{\rm in}(x) & \mbox{on}\quad  \R_+, \\
\mybf{N}_p(t)\mybf{u}_{\vert_{x=0}} = \mybf{g}(t) & \mbox{on}\quad (0,T), 
\end{cases}
\end{equation}
where $\mybf{u}^{\rm in}(x) = (u^{\rm in}(-x),u^{\rm in}(x))^{\rm T}$ and $\mybf{N}_p$ is the $p\times 4$ matrix 
\begin{equation}\label{notaN}
\mybf{N}_p(t) = \left( \begin{array}{cc} -N_p^{\rm l}(t) & N_p^{\rm r}(t) \end{array} \right).
\end{equation}
This initial boundary value problem has a block structure. 
In order to ensure its well-posedness, we shall make the following assumption, 
which ensures that the sytem of equations is strictly hyperbolic.
Note that the condition on the invertibility of $\mybf{N}_p(t)\mybf{N}_p(t)^{\rm T}$ in the first point is here 
to ensure that $\mybf{N}_p$ is uniformly of rank $p$.

\begin{assumption}\label{asshyptransm}
There exists $c_0>0$ such that the following assertions hold. 
\begin{enumerate}
\item[{\bf i.}]
$A^{\rm l}, A^{\rm r}\in W^{1,\infty}(\Omega_T)$ and $ B^{\rm l}, B^{\rm r}\in L^\infty(\Omega_T)$. 
Moreover, $\mybf{N}_p \in C([0,T])$ and for any $t\in[0,T]$ we have 
\[
\det(\mybf{N}_p(t)\mybf{N}_p(t)^{\rm T}) \geq c_0.
\]
\item[{\bf ii.}]
One of the three cases stated in Notation \ref{notanumberp} holds. Moreover, 
\begin{align*}
& \widetilde{\lambda}_{\pm,j}(t,-x),\lambda_{\pm,j}(t,x)\geq c_0 \quad(j=1,2,\emptyset), \\
& |\widetilde{\lambda}_{\pm,1}(t,-x) - \widetilde{\lambda}_{\pm,2}(t,-x)|, 
 |\lambda_{\pm,1}(t,x) - \lambda_{\pm,2}(t,x)| \geq c_0. 
\end{align*}
\item[{\bf iii.}]
With $\mybf{E}_p(t)$ in Notation \ref{notanumberp}, the $p\times p$ Lopatinski\u{\i} matrix 
$\mybf{L}_p(t) = \mybf{N}_p(t)\mybf{E}_p(t)$ is invertible and for any $t\in [0,T]$ we have 
\[
\Vert \mybf{L}_p(t)^{-1}\Vert_{\R^p\to\R^p} \leq \frac{1}{c_0}.
\]
\end{enumerate}
\end{assumption}

%

We can then derive sharp estimates similar to those derived in Theorem \ref{theoIBVP1} 
for initial boundary value problems. 
The compatibility conditions are not made explicit because they can be obtained as for Definition \ref{defcompVC}.

\begin{theorem}\label{theoIBVP1transm}
Let $m\geq1$ be an integer, $T>0$, and assume that Assumption \ref{asshyptransm} is satisfied for some $c_0>0$. 
Assume moreover that there are constants $0<K_0\leq K$ such that 
\[
\begin{cases}
\frac{1}{c_0}, \| \mybf{A} \|_{L^\infty(\Omega_T)}, |\mybf{N}_p|_{L^\infty(0,T)} \leq K_0, \\
\| \mybf{A} \|_{W^{1,\infty}(\Omega_T)}, \| \mybf{B} \|_{L^\infty(\Omega_T)}, 
 \|(\partial \mybf{A},\partial \mybf{B})\|_{ \WW^{m-1}(T)}, |\mybf{N}_p|_{W^{m,\infty}(0,T)} \leq K. 
\end{cases}
\]
Then, for any data $\mybf{u}^{\rm in} \in H^m(\R_+)$, $\mybf{g}\in H^m(0,T)$, and $\mybf{f}\in H^m(\Omega_T)$ 
satisfying the compatibility conditions up to order $m-1$, 
there exists a unique solution $\mybf{u} \in \WW^m(T)$ to the transmission problem \eqref{transmref}. 
Moreover, the following estimate holds for any $t \in[0,T]$ and any $\gamma \geq C(K)$: 
\begin{align*}
& \opnorm{ \mybf{u}(t) }_{m,\gamma}
 + \biggl( \gamma\int_0^t\opnorm{ \mybf{u}(t') }_{m,\gamma}^2{\rm d}t' \biggr)^\frac12
 + | \mybf{u}_{\vert_{x=0}} |_{m,\gamma,t} \\
&\leq C(K_0)\bigl( \opnorm{ \mybf{u}(0) }_{m} + | \mybf{g} |_{H_\gamma^m(0,t)} 
 + |\mybf{f}_{\vert_{x=0}}|_{m-1,\gamma,t} + S_{\gamma,t}^*(\opnorm{ \dt \mybf{f}(\cdot) }_{m-1}) \bigr). 
\end{align*}
Particularly, we have 
\begin{align*}
& \opnorm{ \mybf{u}(t) }_{m} + | \mybf{u}_{\vert_{x=0}} |_{m,t} \\
&\leq C(K_0)e^{C(K)t} \biggl( \opnorm{ \mybf{u}(0) }_{m} + | \mybf{g} |_{H^m(0,t)} 
 + |\mybf{f}_{\vert_{x=0}}|_{m-1,t} + \int_0^t \opnorm{ \dt \mybf{f}(t') }_{m-1}{\rm d}t' \biggr). 
\end{align*}
\end{theorem}

\subsubsection{A priori estimates}
We prove here an $L^2$ a priori estimate using the following assumption, 
which is the natural generalization of Assumption \ref{assVC} to $4\times 4$ systems.

\begin{assumption}\label{assVCtransm}
There exists a symmetric matrix $\mybf{S}(t,x) \in {\mathcal M}_4(\R)$ such that 
for any $(t,x)\in\Omega_T$ $\mybf{S}(t,x)\mybf{A}(t,x)$ is symmetric and
the following conditions hold. 
\begin{enumerate}
\item[{\bf i.}]
There exist constants $\alpha_0,\beta_0>0$ such that for any 
$(\mybf{v},t,x)\in \R^4\times \Omega_T$ we have 
\[
\alpha_0 |\mybf{v}|^2 \leq \mybf{v}^{\rm T} \mybf{S}(t,x) \mybf{v} \leq \beta_0 |\mybf{v}|^2.
\]
\item[{\bf ii.}]
There exist constants $\alpha_1,\beta_1>0$ such that for any 
$(\mybf{v},t)\in \R^2\times (0,T)$ we have 
\[
\mybf{v}^{\rm T} \mybf{S}(t,0)\mybf{A}(t,0) \mybf{v} \leq -\alpha_1 |\mybf{v}|^2 + \beta_1 |\mybf{N}_p(t) \mybf{v}|^2.
\]
\item[{\bf iii.}]
There exists a constant $\beta_2$ such that 
\[
\| \dt \mybf{S} + \dx (\mybf{SA}) - 2\mybf{SB} \|_{L^2\to L^2} \leq \beta_2. 
\]
\end{enumerate}
\end{assumption}

Under this assumption, 
the $L^2$ a priori estimates of Proposition \ref{propNRJ1} can be straightforwardly generalized.

\begin{proposition}\label{propNRJ1transm}
Under Assumption \ref{assVCtransm}, there are constants  
\[
\mathfrak{c}_0 = C\Bigl( \frac{\beta_0^{\rm in}}{\alpha_0},\frac{\beta_0^{\rm in}}{\alpha_1} \Bigr)
 \quad\mbox{ and }\quad
\mathfrak{c}_1 = C\Big( \frac{\beta_0}{\alpha_0},\frac{\beta_1}{\alpha_0},\frac{\alpha_0}{\alpha_1} \Big)
\]
such that for any $\mybf{u} \in H^1(\Omega_T)$ solving \eqref{transmref}, any $t\in [0,T]$, and any 
$\gamma\geq\frac{\beta_2}{\alpha_0}$, the following inequality holds. 
\begin{align*}
&\opnorm{\mybf{u}(t)}_{0,\gamma} + \biggl(\gamma \int_0^t \opnorm{\mybf{u}(t')}_{0,\gamma}^2 {\rm d}t' \biggr)^\frac12
 + |\mybf{u}_{\vert_{x=0}}|_{L_\gamma^2(0,t)} \\
&\leq \mathfrak{c}_0\|\mybf{u}^{\rm in}\|_{L^2}
 + \mathfrak{c}_1\bigl( |\mybf{g}|_{L_\gamma^2(0,t)} + S_{\gamma,t}^*( \|\mybf{f}(\cdot)\|_{L^2} ) \bigr).
\end{align*}
\end{proposition}

Similarly, the following generalization of Proposition \ref{propVC2} does not raise any difficulty, 
and we therefore omit the proof.

\begin{proposition}\label{propVC2transm}
Let $m\geq1$ be an integer, $T>0$, and assume that Assumption \ref{assVCtransm} is satisfied. 
Assume moreover that there are two constants $0<K_0\leq K$ such that 
\[
\begin{cases}
\mathfrak{c}_0, \mathfrak{c}_1, \|\mybf{A}\|_{L^\infty(\Omega_T)}, 
 \|\mybf{A}^{-1}\|_{L^\infty(\Omega_T)}, |\mybf{N}_p|_{L^\infty(0,T)}\leq K_0, \\
\frac{\beta_2}{\alpha_0}, \|\mybf{A}\|_{W^{1,\infty}(\Omega_T)}, \|\mybf{B}\|_{L^\infty(\Omega_T)}, 
 \|(\partial \mybf{A},\partial \mybf{B})\|_{\WW^{m-1}(T)}, |\mybf{N}_p|_{W^{m,\infty}(0,T)}\leq K,
\end{cases}
\]
where $\mathfrak{c}_0$ and $\mathfrak{c}_1$ are as in Proposition \ref{propNRJ1transm}. 
Then, every solution $\mybf{u}\in H^{m+1}(\Omega_{T})$ to the initial boundary value problem 
\eqref{transmref} satisfies, for any $t \in [0,T]$ and any $\gamma \geq C(K)$, 
\begin{align*}
&\opnorm{ \mybf{u}(t) }_{m,\gamma}
 + \biggl( \gamma\int_0^t \opnorm{ \mybf{u}(t') }_{m,\gamma}^2{\rm d}t' \biggr)^\frac12
 + |\mybf{u}_{\vert_{x=0}}|_{m,\gamma,t} \\
&\leq C(K_0)\bigl( \opnorm{ \mybf{u}(0) }_{m} + |\mybf{g}|_{H_\gamma^m(0,t)} 
 + |\mybf{f}_{\vert_{x=0}}|_{m-1,\gamma,t} + S_{\gamma,t}^*(\opnorm{ \dt \mybf{f}(t') }_{m-1}) \bigr).
\end{align*}
\end{proposition}

\subsubsection{Proof of Theorem \ref{theoIBVP1transm}}
As for the proof of Theorem \ref{theoIBVP1transm}, we just have to prove that the assumptions made 
in the statement of Theorem \ref{theoIBVP1transm} imply that Assumption \ref{assVCtransm} is satisfied. 
This is what the following lemma claims; 
its proof requires the construction of a Kreiss symmetrizer yielding maximal dissipativity on the boundary.

\begin{lemma}\label{lemsymmetrizertransm}
Let $c_0>0$ be such that Assumption \ref{asshyptransm} is satisfied. 
There exist a symmetrizer $\mybf{S}\in W^{1,\infty}(\Omega_T)$ and constants 
$\alpha_0,\alpha_1$ and $\beta_0,\beta_1,\beta_2$ such that Assumption \ref{assVCtransm} is satisfied.
Moreover, we have 
\[
\mathfrak{c}_0 \leq C\Bigl( \frac{1}{c_0}, \| \mybf{A}_{\vert_{t=0}} \|_{L^\infty(\R_+)} \Bigr)
\quad\mbox{and}\quad 
\mathfrak{c}_1 \leq C\Bigl( \frac{1}{c_0},\|\mybf{A}\|_{L^{\infty}(\Omega_T)}, |\mybf{N}_p|_{L^\infty(0,T)}\Bigr),
\]
where $\mathfrak{c}_0$ and $\mathfrak{c}_1$ are as defined in Proposition \ref{propNRJ1transm}, and we also have 
\[
\frac{\beta_2}{\beta_0} \leq C\Bigl( \frac{1}{c_0}, \|\mybf{A}\|_{W^{1,\infty}(\Omega_T)},
 \|\mybf{B}\|_{L^\infty(\Omega_T)}\Bigr).
\]
\end{lemma}

\begin{proof}
Most of the proof is similar to the proof of Lemma \ref{lemsymmetrizer} and Proposition \ref{propBC} 
and we therefore omit the details. 
The only new point is to show that it is possible to construct a symmetrizer $\mybf{S}$ satisfying 
${\bf ii}$ in Assumption \ref{asshyptransm}. 
We show here how to prove this point, namely, that there exist constants $\alpha_1,\beta_1>0$ such that for any 
$(\mybf{v},t)\in \R^4\times (0,T)$ we have 
\[
\mybf{v}^{\rm T} \mybf{S}(t,0)\mybf{A}(t,0) \mybf{v}
 \leq -\alpha_1 |\mybf{v}|^2 + \beta_1 |\mybf{N}_p(t) \mybf{v}|^2.
\]
Let us denote by $\widetilde{\boldsymbol{\pi}}_{\pm,j}(t,x)$ and $\boldsymbol{\pi}_{\pm,j}(t,x)$ 
the eigenprojectors associated to the eigenvalues $\widetilde{\lambda}_{\pm,j}$ and $\lambda_{\pm,j}$ 
(with $j=1,2,\emptyset$); they are of the form 
\[
\widetilde{\boldsymbol{\pi}}_{\pm,j} = \left(
 \begin{array}{cc}
  \widetilde{\pi}_{\pm,j} & 0_{2\times 2} \\
  0_{2\times 2} & {0}_{2\times 2}
 \end{array}\right)
\quad\mbox{ and }\quad
\boldsymbol{\pi}_{\pm,j} = \left(
 \begin{array}{cc}
  0_{2\times 2} & 0_{2\times 2} \\
  0_{2\times 2} & \pi_{\pm,j}
 \end{array}\right),
\]
where $\widetilde{\pi}_{\pm,j}(t,x)$ and $\pi_{\pm,j}(t,x)$ are the corresponding eigenprojectors of 
$\widetilde{A}(t,x)$ and $A(t,x)$. 
Distinguishing the three cases stated in Notation \ref{notanumberp} and writing as in \eqref{leftright} 

\begin{align*}
& \lambda_{\pm,j}^{\rm l}(t,x) = \widetilde{\lambda}_{\pm,j}(t,-x), \quad
 \lambda_{\pm,j}^{\rm r}(t,x) = \lambda_{\pm,j}(t,x), \\
& \boldsymbol{\pi}_{\pm,j}^{\rm l}(t,x) = \widetilde{\boldsymbol{\pi}}_{\pm,j}(t,-x), \quad
 \boldsymbol{\pi}_{\pm,j}^{\rm r}(t,x) = \boldsymbol{\pi}_{\pm,j}(t,x),
\end{align*}
the spectral decomposition of the matrix $\mybf{A}$ is given by 
\[
\mybf{A} = 
\begin{cases}
 \lambda_-^{\rm l}\bpi_-^{\rm l} - \lambda_+^{\rm l}\bpi_+^{\rm l}
  - \lambda_{-,1}^{\rm r}\bpi_{-,1}^{\rm r} - \lambda_{-,2}^{\rm r}\bpi_{-,2}^{\rm r}
   & (\mbox{frist case of $p=1$}), \\
 \lambda_+^{\rm r}\bpi_+^{\rm r} - \lambda_{+,1}^{\rm l}\bpi_{+,1}^{\rm l}
  - \lambda_{+,2}^{\rm l}\bpi_{+,2}^{\rm l} - \lambda_{-}^{\rm r}\bpi_{-}^{\rm r}
   & (\mbox{second case of $p=1$}), \\
 \lambda_-^{\rm l}\bpi_-^{\rm l} + \lambda_{+}^{\rm r}\bpi_{+}^{\rm r}
  - \lambda_+^{\rm l}\bpi_+^{\rm l} - \lambda_{-}^{\rm r}\bpi_{-}^{\rm r} & (p=2), \\
 \lambda_{-}^{\rm l}\bpi_{-}^{\rm l} + \lambda_{+,1}^{\rm r}\bpi_{+,1}^{\rm r}
  + \lambda_{+,2}^{\rm r}\bpi_{+,2}^{\rm r} - \lambda_+^{\rm l}\bpi_+^{\rm l}
   & (\mbox{first case of $p=3$}), \\
 \lambda_{-,1}^{\rm l}\bpi_{-,1}^{\rm l} + \lambda_{-,2}^{\rm l}\bpi_{-,2}^{\rm l}
  + \lambda_{+}^{\rm r}\bpi_{+}^{\rm r} - \lambda_-^{\rm r}\bpi_-^{\rm r}
   & (\mbox{second case of $p=3$}).
\end{cases}
\]
We construct the symmetrizer $\mybf{S}$ in the form 
\[
\mybf{S} = 
\begin{cases}
 (\bpi_-^{\rm l})^{\rm T}\bpi_-^{\rm l} + M\bigl\{ (\bpi_+^{\rm l})^{\rm T}\bpi_+^{\rm l}
  + (\bpi_{-,1}^{\rm r})^{\rm T}\bpi_{-,1}^{\rm r} + (\bpi_{-,2}^{\rm r})^{\rm T}\bpi_{-,2}^{\rm r} \bigr\}
   & (\mbox{frist case of $p=1$}), \\
 (\bpi_+^{\rm r})^{\rm T}\bpi_+^{\rm r} + M\bigl\{ (\bpi_{+,1}^{\rm l})^{\rm T}\bpi_{+,1}^{\rm l}
  + (\bpi_{+,2}^{\rm l})^{\rm T}\bpi_{+,2}^{\rm l} + (\bpi_{-}^{\rm r})^{\rm T}\bpi_{-}^{\rm r} \bigr\}
   & (\mbox{second case of $p=1$}), \\
 (\bpi_-^{\rm l})^{\rm T}\bpi_-^{\rm l} + (\bpi_+^{\rm r})^{\rm T}\bpi_+^{\rm r}
  + M\bigl\{ (\bpi_+^{\rm l})^{\rm T}\bpi_+^{\rm l} + (\bpi_-^{\rm r})^{\rm T}\bpi_-^{\rm r} \bigr\} & (p=2), \\
 (\bpi_{-}^{\rm l})^{\rm T}\bpi_{-}^{\rm l} + (\bpi_{+,1}^{\rm r})^{\rm T}\bpi_{+,1}^{\rm r}
  + (\bpi_{+,2}^{\rm r})^{\rm T}\bpi_{+,2}^{\rm r} + M (\bpi_+^{\rm l})^{\rm T}\bpi_+^{\rm l}
   & (\mbox{first case of $p=3$}), \\
 (\bpi_{-,1}^{\rm l})^{\rm T}\bpi_{-,1}^{\rm l} + (\bpi_{-,2}^{\rm l})^{\rm T}\bpi_{-,2}^{\rm l}
  + (\bpi_{+}^{\rm r})^{\rm T}\bpi_{+}^{\rm r} + M (\bpi_-^{\rm r})^{\rm T}\bpi_-^{\rm r}
   & (\mbox{second case of $p=3$}),
\end{cases}
\]
where $M>0$ will be determined later.

{\it From now on, we focus on the case $p=2$, the adaptations to the cases $p=1$ and $p=3$ being straightforward.} 
Then, we have 
\[
\mybf{SA} = \lambda_-^{\rm l}(\bpi_-^{\rm l})^{\rm T}\bpi_-^{\rm l}
  + \lambda_+^{\rm r}(\bpi_+^{\rm r})^{\rm T}(\bpi_+^{\rm r})
 - M\bigl\{ \lambda_+^{\rm l}(\bpi_+^{\rm l})^{\rm T}\bpi_+^{\rm l}
  + \lambda_-^{\rm r}(\bpi_-^{\rm r})^{\rm T}\bpi_-^{\rm r} \bigr\}.
\]
We begin to show that for $\mybf{v} \in \ker \mybf{N}_p(t)$ we have 
\[
|\mybf{v}|^2 \leq -C\mybf{v}^{\rm T}(\mybf{SA})(t,0)\mybf{v}.
\]
For any $\mybf{v} = \begin{pmatrix} v^{\rm l} \\ v^{\rm r} \end{pmatrix} \in \R^4$, we have 
\begin{align*}
-\mybf{v}^{\rm T} \mybf{SA} \mybf{v}
&= - \lambda_-^{\rm l}(\bpi_-^{\rm l}\mybf{v})^{\rm T}\bpi_-^{\rm l}\mybf{v}
 - \lambda_+^{\rm r}(\bpi_+^{\rm r}\mybf{v})^{\rm T}\bpi_+^{\rm r}\mybf{v}
 + M\bigl\{ \lambda_+^{\rm l}(\bpi_+^{\rm l}\mybf{v})^{\rm T}\bpi_+^{\rm l}\mybf{v}
  + \lambda_-^{\rm r}(\bpi_-^{\rm r}\mybf{v})^{\rm T}\bpi_-^{\rm r}\mybf{v} \bigr\} \\
&= - \lambda_-^{\rm l}|\pi_-^{\rm l}v^{\rm l}|^2 - \lambda_+^{\rm r}|\pi_+^{\rm r}v^{\rm r}|^2
 + M\bigl\{ \lambda_+^{\rm l}|\pi_+^{\rm l}v^{\rm l}|^2
  + \lambda_-^{\rm r}|\pi_-v^{\rm r}|^2 \bigr\}.
\end{align*}
We decompose $v^{\rm l}$ and $v^{\rm r}$ as 
\begin{equation}\label{decom}
\begin{cases}
 v^{\rm l} = c_+^{\rm l}{\bf e}_+^{\rm l} + c_-^{\rm l}{\bf e}_-^{\rm l}, \\
 v^{\rm r} = c_+^{\rm r}{\bf e}_+^{\rm r} + c_-^{\rm r}{\bf e}_-^{\rm r}, 
\end{cases}
\end{equation}
where $\pi_{\pm}^{\rm l}v^{\rm l} = c_{\pm}^{\rm l}{\bf e}_{\pm}^{\rm l}$ and 
$\pi_{\pm}^{\rm r}v^{\rm r} = c_{\pm}^{\rm r}{\bf e}_{\pm}^{\rm r}$. 
Particularly, we have $|\pi_{\pm}^{\rm l}v^{\rm l}| = |c_{\pm}^{\rm l}|$ and 
$|\pi_{\pm}^{\rm r}v^{\rm r}| = |c_{\pm}^{\rm r}|$, so that 
\[
-\mybf{v}^{\rm T} \mybf{SA}\mybf{v}
= - \lambda_-^{\rm l}|c_-^{\rm l}|^2 - \lambda_+^{\rm r}|c_+^{\rm r}|^2
 + M\bigl\{ \lambda_+^{\rm l}|c_+^{\rm l}|^2  + \lambda_-^{\rm r}|c_-^{\rm r}|^2 \bigr\}.
\]
Now, suppose that $\mybf{v} \in \ker \mybf{N}_p(t)$. 
Then, we have 
\[
\mybf{N}_p\mybf{v} = - N_p^{\rm l}v^{\rm l} + N_p^{\rm r}v^{\rm r} = 0.
\]
Plugging \eqref{decom} into the above relation, we have 
\[
- c_+^{\rm l}N_p^{\rm l}{\bf e}_+^{\rm l} - c_-^{\rm l}N_p^{\rm l}{\bf e}_-^{\rm l}
 + c_+^{\rm r}N_p^{\rm r}{\bf e}_+^{\rm r} + c_-^{\rm r}N_p^{\rm r}{\bf e}_-^{\rm r} = 0,
\]
which we can rewrite, using the Lopatinski\u{\i} matrix, 
\[
\mybf{L}_p(t)
 \begin{pmatrix} c_-^{\rm l} \\ c_+^{\rm r} \end{pmatrix}
= \begin{pmatrix} N_p^{\rm l}{\bf e}_+^{\rm l} & -N_p^{\rm r}{\bf e}_-^{\rm r} \end{pmatrix}
 \begin{pmatrix} c_+^{\rm l} \\ c_-^{\rm r} \end{pmatrix}. 
\]
Under the uniform Kreiss--Lopatinski\u{\i} condition made in Assumption \ref{asshyptransm}, we deduce
\[
|c_-^{\rm l}|^2 + |c_+^{\rm r}|^2
 \leq C( |c_+^{\rm l}|^2 + |c_-^{\rm r}|^2 ),
\]
where $C$ depends only on $| \mybf{N}_p |_{L^\infty(0,T)}$ and $1/c_0$, or equivalently, 
\[
|\pi_-^{\rm l}v^{\rm l}|^2 + |\pi_+^{\rm r}v^{\rm r}|^2
 \leq C( |\pi_+^{\rm l}v^{\rm l}|^2 + |\pi_-^{\rm r}v^{\rm r}|^2 ).
\]
Therefore, if we take $M$ sufficiently large, then for any $\mybf{v} \in \ker\mybf{N}_p(t)$ we have 
\[
|\mybf{v}|^2 \leq -C\mybf{v}^{\rm T}(\mybf{SA})(t,0)\mybf{v}.
\]

Next, we will show that for any $\mybf{v} \in \R^4$ we have 
\[
\mybf{v}^{\rm T}(\mybf{SA})(t,0)\mybf{v} \leq -\alpha_1|\mybf{v}|^2 + \beta_1|\mybf{N}_p(t)\mybf{v}|^2.
\]
To this end, we use the assumption that 
\begin{equation}
| \det ( \mybf{N}_p(t) \mybf{N}_p(t)^{\rm T} ) | \geq c_0.
\end{equation}
This condition means that the $2\times4$ matrix $\mybf{N}_p(t)$ has rank 2 uniformly in time. 
For any $\mybf{v} \in \R^4$, we decompose it as 
\[
\mybf{v} = \mybf{v}_1 + \mybf{v}_2 \quad\mbox{with}\quad
 \mybf{v}_2 = \mybf{N}_p^{\rm T}( \mybf{N}_p\mybf{N}_p^{\rm T} )^{-1} \mybf{N}_p\mybf{v}.
\]
Then, we have 
\[
\mybf{v}_1 \in \ker \mybf{N}_p, \qquad \mybf{N}_p\mybf{v} = \mybf{N}_p\mybf{v}_2,
\]
so that 
\begin{align*}
|\mybf{v}|^2 &\leq C(|\mybf{v}_1|^2 + |\mybf{v}_2|^2) \\
&\leq -C\mybf{v}_1^{\rm T}\mybf{SA}\mybf{v}_1 + C|\mybf{v}_2|^2 \\
&= -C(\mybf{v}-\mybf{v}_2)^{\rm T}\mybf{SA}(\mybf{v}-\mybf{v}_2) + C|\mybf{v}_2|^2 \\
&\leq -C\mybf{v}^{\rm T}\mybf{SA}\mybf{v} + \frac12|\mybf{v}|^2 + C|\mybf{v}_2|^2.
\end{align*}
Since $|\mybf{v}_2| \leq C|\mybf{N}_p\mybf{v}|$, we obtain the desired estimate. 
\end{proof}

\subsection{Application to quasilinear $2\times 2$ transmission problems}\label{sectappltransmQL}
As done in \S \ref{sectapplQL} in the case of initial boundary value problems, 
we can use the linear estimates of Theorem \ref{theoIBVP1transm} to solve quasilinear problems. 
More precisely, after reduction to a $4\times 4$ initial boundary value problem as indicated 
in \S \ref{sectVCtransm}, let us consider 
\begin{equation}\label{systQLtransm}
\begin{cases}
\dt \mybf{u} + \mybf{A}(\mybf{u})\dx \mybf{u} + \mybf{B}(t,x)\mybf{u} = \mybf{f}(t,x) & \mbox{in}\quad \Omega_T, \\
\mybf{u}_{\vert_{t=0}} = \mybf{u}^{\rm in}(x) & \mbox{on}\quad \R_+,\\
\mybf{N}_{p}(t) \mybf{u}_{\vert_{x=0}} = \mybf{g}(t) & \mbox{on}\quad (0,T),
\end{cases}
\end{equation}
where $\mybf{u} = (u^{\rm l},u^{\rm r})^{\rm T}$, $\mybf{u}^{\rm in}$, and $\mybf{f}$ are $\R^4$-valued functions, 
and $\mybf{g}$ is a $\R^p$-valued function, 
while $\mybf{A}(\mybf{u}) = \mbox{diag}(-\widetilde{A}(\rm u^l), A(u^{\rm r}))$ and 
$\mybf{B} = \mbox{diag}(B^{\rm l},B^{\rm r})$ take their values in the space of $4\times4$ real-valued matrices 
and $\mybf{N}_p$ is a $p\times 4$ matrix, where $p$ is the number of outgoing characteristics 
(i.e., the number of positive eigenvalues of $\mybf{A}(\mybf{u})$).

\begin{notation}\label{notanumberp2}
Adaptating Notation \ref{notanumberp} in a straightforward way, 
we consider three different possibilities ($p=1,2,3$) depending on the sign of the eigenvalues of 
$\widetilde{A}(u^{\rm l})$ and $A(u^{\rm r})$. 
Correspondingly, a $4\times p$ matrix $\mybf{E}_p(\mybf{u}_{\vert_{x=0}})$ is formed as in 
Notation \ref{notanumberp} with the eigenvectors associated to the eigenvalues defining outgoing characteristics, 
and we define the Lopatinski\u{\i} matrix by 
$\mybf{L}_p(t,\mybf{u}_{\vert_{x=0}}) = \mybf{N}_p(t)\mybf{E}_p(\mybf{u}_{\vert_{x=0}})$. 
\end{notation}


We also make the following assumption on the hyperbolicity of the system and on the boundary condition.

\begin{assumption}\label{asshypQLtransm}
Let $\widetilde{\cU}$ and $\mathcal{U}$ be open sets in $\R^2$ and $p\in \{1,2,3\}$ such that the following 
conditions hold with $\boldsymbol{\cU} = \widetilde{\cU}\times \cU$ representing a phase space of $\mybf{u}$. 
\begin{enumerate}
\setlength{\itemsep}{3pt}
\item[{\bf i.}]
$\mybf{A} \in C^\infty(\boldsymbol{\cU})$. 
\item[{\bf ii.}]
The integer $p$ is such that for any $\mybf{u} = (u^{\rm l},u^{\rm r})^{\rm T} \in \boldsymbol{\cU}$ the matrices 
$\widetilde{A}(u^{\rm l})$ and $A(u^{\rm r})$ satisfy one of the three conditions of Notation \ref{notanumberp}. 
\item[{\bf iii.}]
For any $t\in[0,T]$ and any $\mybf{u} \in \boldsymbol{\cU}$, 
the Lopatinski\u{\i} matrix $\mybf{L}_p(t,\mybf{u})$ is invertible.
\end{enumerate}
\end{assumption}

The main result is the following. 
The compatibility conditions mentioned in the statement of the theorem can be obtained as for 
Definition \ref{defcompQL}. 
It can be deduced from Theorem \ref{theoIBVP1transm} in the same way that 
Theorem \ref{theoIBVP2} was deduced from Thoerem \ref{theoIBVP1} and we therefore omit the proof.

\begin{theorem}\label{theoIBVP2transm}
Let $m\geq 2$ be an integer and assume that Assumption \ref{asshypQLtransm} is satisfied with some $p\in\{1,2,3\}$. 
Assume moreover that $\mybf{B}\in L^\infty(\Omega_T)$, $\partial \mybf{B}\in \WW^{m-1}(T)$, 
and $\mybf{N}_p\in W^{m,\infty}(0,T)$. 
If $\mybf{u}^{\rm in }\in H^m(\R_+)$ takes its values in $\widetilde{\mathcal{K}}_0\times \mathcal{K}_0$ 
with $\widetilde{\mathcal{K}}_0 \subset \widetilde{\cU}$ and ${\mathcal{K}}_0 \subset \cU$ compact and convex sets, 
and if the data $\mybf{u}^{\rm in}$, $\mybf{f}\in H^m(\Omega_T)$, and $\mybf{g}\in H^m(0,T)$ 
satisfy the compatibility conditions up to order $m-1$, 
then there exist $T_1 \in (0,T]$ and a unique solution $\mybf{u}\in \WW^m(T_1)$ 
to the transmission problem \eqref{systQLtransm}. 
Moroever, the trace of $\mybf{u}$ at $x=0$ belongs to $H^m(0,T_1)$ and $|\mybf{u}_{\vert_{x=0}}|_{m,T_1}$ is finite. 
\end{theorem}

\subsection{Variable coefficients $2\times2$ transmission problems on moving domains}\label{secttransmmov}
As for the initial boundary value problems considered previously, we consider here the case of variable coefficients 
transmission problems on a moving domain as a preliminary step to treat free boundary transmission problems. 
We consider therefore a transmission problem with transmission conditions given at a moving boundary located at 
$x=\ux(t)$ with $\ux(\cdot)$ a given function. 
As in \S \ref{sectVCm}, we consider variable coefficients matrices of the form $A(t,x)=A(\uU(t,x))$, etc. 
Let us consider therefore 
\begin{equation}\label{transmmov}
\begin{cases}
 \dt U + \widetilde{A}({\uU})\dx U + \widetilde{{\mathtt B}}U = \widetilde{F}
  & \mbox{in }\quad  (-\infty,\ux(t)) \quad \mbox{ for } \quad t\in(0,T), \\
 \dt U + {A}(\uU)\dx U + {{\mathtt B}}U = {F}
  & \mbox{in }\quad  (\ux(t),+\infty) \quad \mbox{ for } \quad t\in(0,T), \\
 U_{\vert_{t=0}} = u^{\rm in}(x) & \mbox{on }\quad \mathbb{R}_-\cup\mathbb{R}_+, \\
 N_p^ {\rm r}(t)U_{\vert_{x=\ux(t)+0}}-N_p^{\rm l}(t)U_{\vert_{x=\ux(t)-0}} = \mybf{g}(t)
  & \mbox{on } \quad (0,T),
\end{cases}
\end{equation}
where, without loss of generality, we assumed that $\ux(0)=0$, 
and with notations inherited from the previous sections. 
As in \S \ref{sectVCm}, we use a diffeomorphism $\varphi(t,\cdot): {\mathbb R}\to {\mathbb R}$ such that 
$\varphi(0,\cdot) = \mbox{Id}$ and that for any $t\in[0,T]$ we have 
\[
\varphi(t,0)=\ux(t), \qquad \varphi(t,\cdot): {\mathbb R}_- \to (-\infty,\ux(t)), \quad\mbox{ and }\quad
\varphi(t,\cdot): {\mathbb R}_+ \to (\ux(t),+\infty).
\]
Writing as before $u=U\circ \varphi$, $\dt^\varphi u=(\dt U)\circ \varphi$, etc., and with $\dx^\varphi$ and 
$\dt^\varphi$ as defined in \eqref{dtphi}, we transform \eqref{transmmov} into a transmission problem with 
a fix interface located at $x=0$. 
Using the same procedure as in \S \ref{sectVCtransm} and with the same notations as in \eqref{leftright} 
(writing also $\varphi^{\rm l}(t,x) = \varphi(t,-x)$ and  $\varphi^{\rm r}(t,x) = \varphi(t,x)$ for $x>0$), 
this transmission problem can be recast as a $4\times 4$ initial boundary value problem on $(0,T)\times \R_+$, 
namely 
%
\begin{equation}\label{bigIBVP}
\begin{cases}
 \dt \mybf{u} +{\bm{\mathcal{A}}}(\ubu,\partial \bvarphi)\dx \mybf{u} + \mybf{B}(t,x)\mybf{u} = \mybf{f}(t,x)
  & \mbox{ in }\quad \Omega_T, \\
 \mybf{u}_{\vert_{t=0}} = \mybf{u}^{\rm in}(x)  & \mbox{ on }\quad {\mathbb R}_+, \\
 \mybf{N}_p(t)\mybf{u}_{\vert_{x=0}} = \mybf{g}(t) & \mbox{ on }\quad (0,T),
\end{cases}
\end{equation}
with $\mybf{u} = (u^{\rm l},u^{\rm r})^{\rm T}$, $\bvarphi = (\varphi^{\rm l},\varphi^{\rm r})^{\rm T}$, and 
\[
{\bm{\mathcal{A}}}(\ubu,\partial \bvarphi) = \left(
 \begin{array}{cc}
  -{\mathcal A}^{\rm l}(\uu^{\rm l},\partial\varphi^{\rm l}) & 0_{2\times2} \\
  0_{2\times2} &{\mathcal A}^{\rm r}(\uu^{\rm r},\partial\varphi^{\rm r})
 \end{array}\right)
\]
as well as
\[
{\mathcal A}^{\rm l}(\uu^{\rm l},\partial\varphi^{\rm l})
 = \frac{1}{\abs{\dx \varphi^{\rm l}}}\big( \widetilde{A}({\uu}^{\rm l})-(\dt \varphi^{\rm l})\mbox{Id}\big), \qquad
{\mathcal A}^{\rm r}(\uu^{\rm r},\partial\varphi^{\rm r})
 = \frac{1}{\dx \varphi^{\rm r}}\big( A(\uu^{\rm r})-(\dt \varphi^{\rm r})\mbox{Id}\big),
\]
while $\mybf{B}$ and $\mybf{f}$ as in \S \ref{sectVCtransm}. 
The matrix $\mybf{N}_p$ is as in \eqref{notaN} and still denotes a $p\times 4$ matrix, 
but the difference is that the value of $p$ depends not only on the eigenvalues of $\widetilde{A}(u)$ and $A(u)$, 
but also on the speed $\dot \ux$ of the interface. 
For the sake of simplicity, we shall consider here the case where $\widetilde{A}(u)$ and $A(u)$ have both 
a positive and a negative eigenvalue, and shall consider two cases depending on the speed of the interface.

\begin{definition}\label{defLax}
Denoting by $\pm\widetilde{\lambda}_\pm(\uu^{\rm l})$ and $\pm \lambda_\pm(\uu^{\rm r})$ the eigenvalues of 
$\widetilde{A}(\uu^{\rm l})$ and $A(\uu^{\rm r})$, respectively 
(with $\widetilde{\lambda}_\pm(\uu^{\rm l}),\lambda_\pm(\uu^{\rm r})>0)$, we define two regimes: 
\begin{itemize}
\item
{\bf Subsonic regime.} We say that $\ubu = (\uu^{\rm l},\uu^{\rm r})^{\rm T}$ and ${\chi}\in \R$ are in the 
\emph{subsonic regime} if the following condition holds. 
\[
\widetilde{\lambda}_\pm(\uu^{\rm l})\mp\chi>0
 \quad\mbox{ and }\quad
  {\lambda}_\pm(\uu^{\rm r})\mp\chi>0.
\]
\item
{\bf Lax regime.} We say that $\ubu = (\uu^{\rm l},\uu^{\rm r})^{\rm T}$ and ${\chi}\in \R$ are in the 
\emph{Lax regime} if the following condition holds. 
\[
\widetilde{\lambda}_\pm(\uu^{\rm l})\mp\chi>0
 \quad\mbox{ and }\quad -\lambda_+(\uu^{\rm r})+\chi>0,
\]
or
\[
-\widetilde{\lambda}_-(\uu^{\rm l})-\chi>0
 \quad\mbox{ and }\quad{\lambda}_\pm(\uu^{\rm r})\mp\chi>0.
\]
\end{itemize}
\end{definition}

\begin{remark}
This terminology is of course inherited from the study of shocks \cite{Lax}. 
The linearized equations around a shock can indeed be put under the form \eqref{transmmov}. 
We refer to \S \ref{sectshocks} where we prove the stability of one-dimensional shocks 
for nonlinear $2\times2$ hyperbolic systems.
\end{remark}

Since the eigenvalues of the matrix ${\bm{\mathcal{A}}}(\ubu,\partial \bvarphi)$ are given by 
\[
\frac{1}{|\dx \varphi^{\rm l}|}\bigl( \pm\widetilde{\lambda}_\mp(\uu^{\rm l}) + \dt \varphi^{\rm l} \bigr)
\quad\mbox{ and }\quad
\frac{1}{\dx \varphi^{\rm r}}\bigl( \pm{\lambda}_\pm(\uu^{\rm r})-\dt \varphi^{\rm r} \bigr),
\]
the number $p$ of outgoing characteristics for \eqref{bigIBVP} is equal to $2$ in the subsonic regime, 
and to $1$ in the Lax regime. 
As in Notation \ref{notanumberp}, we form a $4\times p$ matrix $\mybf{E}_p (\ubu_{\vert_{x=0}})$ given by 
\begin{align*}
\mybf{E}_2 (\ubu_{\vert_{x=0}}) = \left(
\begin{array}{cc}
 \widetilde{\bf e}_-({\uu^{\rm l}}_{\vert_{x=0}}) & 0_{2\times1} \\
 0_{2\times1} & {\bf e}_+({\uu^{\rm r}}_{\vert_{x=0}})
\end{array}\right) 
\end{align*}
in the subsonic regime, and 
\begin{align*}
\mybf{E}_1 (\ubu_{\vert_{x=0}}) = \left(\begin{array}{c}
 \widetilde{\bf e}_-({\uu^{\rm l}}_{\vert_{x=0}}) \\ 0_{2\times1} \end{array}\right) 
\quad\mbox{ or }\quad
\mybf{E}_1 (\ubu_{\vert_{x=0}}) = \left(\begin{array}{c}
 0_{2\times1} \\ {\bf e}_+({\uu^{\rm r}}_{\vert_{x=0}}) \end{array}\right) 
\end{align*}
(depending on which of the two conditions in Definition \ref{defLax} is satisfied) in the Lax regime. 
As in Assumption \ref{asshyptransm}, we define a Lopatinski\u{\i} matrix 
$\mybf{L}_p(t,\ubu_{\vert_{x=0}})$ by 
\begin{equation}\label{interfmatr}
\mybf{L}_p(t,\ubu_{\vert_{x=0}}) = \mybf{N}_p(t)\mybf{E}_p (\ubu_{\vert_{x=0}}).
\end{equation}
In order to be able to apply Theorem \ref{theoIBVP1transm} to this initial boundary value problem, 
we make the following assumption. 
It is the natural generalization of Assumption \ref{asshypm} to transmission problems.

\begin{assumption}\label{asstransmred}
We have $\ubu = (\uu^{\rm l},\uu^{\rm r})^{\rm T} \in W^{1,\infty}(\Omega_T)$, $\ux\in C^1([0,T])$, $\ux(0)=0$, 
and the diffeomorphisms $\varphi^{\rm l}$ and $\varphi^{\rm r}$ are in $C^1(\Omega_T)$. 
Moreover, there exists $c_0>0$ such that the following three conditions hold. 
\begin{enumerate}
\setlength{\itemsep}{3pt}
\item[{\bf i.}]
There exist open sets $\widetilde{\cU},\cU \subset \R^2$  such that, with 
$\boldsymbol{\cU} = \widetilde{\cU}\times \cU$, we have $\mybf{A} \in C^\infty(\bcU)$ and for any 
$\mybf{u} = (u^{\rm l},u^{\rm r})^{\rm T} \in \bcU$, the matrices $\widetilde{A}({u^{\rm l}})$ and $A(u^{\rm r})$ 
have eigenvalues $\widetilde{\lambda}_+(u^{\rm l}), -\widetilde{\lambda}_-(u^{\rm l})$ and 
$\lambda_+(u^{\rm r}), -\lambda_-(u^{\rm r})$, respectively. 
Moreover, $\ubu$ takes its values in a compact set $\boldsymbol{\mathcal{K}}_0 \subset \bcU$ and for any 
$(t,x)\in\Omega_T$ we have 
\[
\widetilde{\lambda}_\pm(\uu^{\rm l}(t,x))\geq c_0\quad\mbox{ and }\quad \lambda_\pm(\uu^{\rm r}(t,x))\geq c_0,
\]
and one of the following conditions holds
\begin{align*}
a) \qquad \widetilde{\lambda}_\pm(\uu^{\rm l}(t,x))\mp \dt \varphi^{\rm l} (t,x)\geq c_0
 & \quad\mbox{ and }\quad  \lambda_\pm(\uu^{\rm r}(t,x))\mp \dt \varphi^{\rm r} (t,x)\geq c_0, \\
b) \qquad\widetilde{\lambda}_\pm(\uu^{\rm l}(t,x))\mp \dt \varphi^{\rm l} (t,x)\geq c_0
 & \quad\mbox{ and }\quad -\lambda_+(\uu^{\rm r}(t,x))+\dt \varphi^{\rm r} (t,x)  \geq c_0, \\
c) \,\,\,\,\, -\widetilde{\lambda}_-(\uu^{\rm l}(t,x))- \dt \varphi^{\rm l} (t,x)\geq c_0
 & \quad\mbox{ and }\quad \lambda_\pm(\uu^{\rm r}(t,x))\mp \dt \varphi^{\rm r} (t,x)\geq c_0.
\end{align*}
\item[{\bf ii.}]
The Lopatinski\u{\i} matrix $\mybf{L}_p(t,\ubu_{\vert_{x=0}})$ associated to the condition $a)$, $b)$, or $c)$ 
constructed in \eqref{interfmatr} is invertible and for any $t \in [0,T]$ we have 
\[
\| \mybf{L}_p(t,\ubu_{\vert_{x=0}}(t))^{-1} \|_{\R^p \to \R^p} \leq \frac{1}{c_0}.
\]

\item[{\bf iii.}]
The Jacobian of the diffeomorphism is uniformly bounded from below and from above, that is, 
for any $(t,x)\in\Omega_T$ we have 
\[
c_0 \leq - \dx \varphi^{\rm l} (t,x) \leq \frac{1}{c_0} \quad \mbox{ and }\quad
 c_0 \leq \dx \varphi^{\rm r} (t,x) \leq \frac{1}{c_0}.
\]
\end{enumerate}
\end{assumption}

The equivalent of Theorem \ref{theoIBVP3} for transmission problems is the following. 
We do not make explicit the compatibility condition in the statement of the theorem 
because they are obtained along a procedure similar to the one used for Definition \ref{defcompVC}.

\begin{theorem}\label{theoIBVP3transm}
Let $m\geq1$ be an integer, $T>0$, and assume that Assumption \ref{asstransmred} is satisfied for some $c_0>0$. 
Assume moreover that there are constants $0<K_0\leq K$ such that 
\[
\begin{cases}
\frac{1}{c_0}, \opnorm{\partial \varphi^{\rm l,r}(0) }_{m-1}, \|\partial \varphi^{\rm l,r}\|_{L^\infty(\Omega_T)}, 
 \|\mybf{A}\|_{L^\infty(\bm{{\mathcal K}}_0)}, |\mybf{N}_p|_{L^\infty(0,T)} \leq K_0, \\
\| \partial\widetilde{\varphi}^{\rm l,r} \|_{\WW^{m-1}(T)}, \| \dt\varphi^{\rm l,r} \|_{H^m(\Omega_T)}, 
 | (\partial^m \varphi^{\rm l,r})_{\vert_{x=0}} |_{L^\infty(0,T)}\leq K, \\
\| \ubu \|_{W^{1,\infty}(\Omega_T)\cap \WW^m(T)}, \| \mybf{B} \|_{W^{1,\infty}(\Omega_T)}, 
 \| \partial \mybf{B} \|_{ \WW^{m-1}(T)}, |\mybf{N}_p|_{W^{1,\infty}\cap W^{m-1,\infty}(0,T)}, 
 |\dt^m\mybf{N}_p|_{L^2(0,T)} \leq K, 
\end{cases}
\]
where $\widetilde{\varphi}^{\rm r}(t,x)=\varphi^{\rm r}(t,x)-x$ and 
$\widetilde{\varphi}^{\rm l}(t,x)=\varphi^{\rm l}(t,x)+x$.
Then, for any data $\mybf{u}^{\rm in} \in H^m(\R_+)$, $\mybf{g}\in H^m(0,T)$, and $\mybf{f}\in H^m(\Omega_T)$ 
satisfying the compatibility conditions up to order $m-1$, 
there exists a unique solution $\mybf{u} \in \WW^m(T)$ to the transmission problem \eqref{transmref}. 
Moreover, the following estimate holds for any $t \in [0,T]$ and any $\gamma \geq C(K)$: 
\begin{align*}
& \opnorm{ \mybf{u}(t) }_{m,\gamma}
 + \biggl( \gamma\int_0^t\opnorm{ \mybf{u}(t') }_{m,\gamma}^2{\rm d}t' \biggr)^\frac12
 + | \mybf{u}_{\vert_{x=0}} |_{m,\gamma,t} \\
&\leq C(K_0)\bigl( (1 + |\dt^m \mybf{N}_p|_{L^2(0,t)})\opnorm{ \mybf{u}(0) }_{m}
 + | \mybf{g} |_{H_\gamma^m(0,t)} 
 + | \mybf{f}_{\vert_{x=0}} |_{m-1,\gamma,t} + S_{\gamma,t}^*(\opnorm{  \mybf{f}(\cdot) }_{m}) \bigr). 
\end{align*}
Particularly, we also have 
\begin{align*}
& \opnorm{ \mybf{u}(t) }_{m} + | \mybf{u}_{\vert_{x=0}} |_{m,t} \\
&\leq C(K_0)e^{C(K)t} \biggl( (1 + |\dt^m \mybf{N}_p|_{L^2(0,t)})\opnorm{ \mybf{u}(0) }_{m}
 + | \mybf{g} |_{H^m(0,t)} 
 + | \mybf{f}_{\vert_{x=0}} |_{m-1,t} + \int_0^t \opnorm{  \mybf{f}(t') }_{m}{\rm d}t' \biggr). 
\end{align*}
\end{theorem}

\subsubsection{Proof of Theorem \ref{theoIBVP3transm}}
As for Theorem \ref{theoIBVP3transm}, we do not seek a direct estimate on $\mybf{u} = (u^{\rm l},u^{\rm r})$ 
in $\WW^m(T)$, but $\WW^{m-1}(T)$ estimates of $\mybf{u}$ and 
$\dot{\mybf{u}}^\bvarphi = (\dt^{\varphi^{\rm l}}u^{\rm l}, \dt^{\varphi^{\rm r}}u^{\rm r})$. 
The $\WW^{m-1}(T)$ estimate of $\mybf{u}$ is obtained exactly as in Step 1 of the proof of Proposition \ref{propAl} 
and requires a variant of Lemma \ref{lemsymmetrizerbis} which is easily obtained by choosing a symmetrizer 
$\boldsymbol{\mathcal S}$ given in the subsonic case $p=2$ 
(with straightforward adadptation in the Lax regime $p=1$) by 
\begin{equation}\label{defbigsym}
\boldsymbol{\mathcal S}
 = (-\dx\varphi^{\rm l})\bigl[ (\bpi_-^{\rm l})^{\rm T}\bpi_-^{\rm l} + M(\bpi_+^{\rm l})^{\rm T}\bpi_+^{\rm l}\bigr]
 + (\dx\varphi^{\rm r})\bigl[(\bpi_+^{\rm r})^{\rm T}\bpi_+^{\rm r}+M(\bpi_-^{\rm r})^{\rm T}\bpi_-^{\rm r}\bigr]
\end{equation}
and by using Theorem \ref{theoIBVP1transm}. 
In order to obtain the $\WW^{m-1}(T)$ estimates of $\dot{\mybf{u}}^\bvarphi$, 
we first remark that $\dot{\mybf{u}}^\bvarphi$ solves 
\begin{equation}\label{bigIBVP1}
\begin{cases}
\dt \dot{\mybf{u}}^\bvarphi + {\bm{\mathcal{A}}}(\ubu,\partial \bvarphi)\dx \dot{\mybf{u}}^\bvarphi
 + \mybf{B}_{(1)}\dot{\mybf{u}}^\bvarphi = \mybf{f}_{(1)} & \mbox{ in }\quad \Omega_T, \\
{ \dot{\mybf{u}}^\bvarphi }_{\ \; \vert_{t=0}} = \mybf{u}^{\rm in}_{(1)} & \mbox{ on }\quad {\mathbb R}_+, \\
\mybf{N}_{(1)}(t){\dot{\mybf{u}}^\bvarphi}_{\ \; \vert_{x=0}} = \mybf{g}_{(1)}(t) & \mbox{ on }\quad (0,T),
\end{cases}
\end{equation}
where $\mybf{B}_{(1)} = \mbox{diag}(B_{(1)}^{\rm l},B_{(1)}^{\rm r})$ and 
$\mybf{f}_{(1)} = (f^{\rm l}_{(1)},f^{\rm r}_{(1)} )$ are straightforwardly deduced from \eqref{equx} 
while $\mybf{g}_{(1)} = (g_{(1)}^{\rm l}, g_{(1)}^{\rm r})$ and 
$\mybf{N}_{(1)} = \bigl(-N^{\rm l}_{(1)}(t) \quad N^{\rm r}_{(1)}(t)\bigr)$ 
are obtained using a procedure similar to the one used to derive \eqref{nu1}. 
In particular 
\[
N^{\rm l}_{(1)}(t) = N_p^{\rm l}\bigl(1-\dot\ux \widetilde{A}({\uu^{\rm l}}_{\vert_{x=0}})^{-1}\bigr), \qquad
N^{\rm r}_{(1)}(t) = N_p^{\rm r}\bigl(1-\dot\ux {A}({\uu^{\rm r}}_{\vert_{x=0}})^{-1}\bigr).
\]
In order to apply Theorem \ref{theoIBVP1transm} to \eqref{bigIBVP1}, it is necessary to show that 
the third point in Assumption \ref{asshyptransm} is satisfied. 
We therefore consider the Lopatinski\u{\i} matrix $\mybf{L}_{(1)}(t,\ubu_{\vert_{x=0}})$ 
associated to \eqref{bigIBVP1}, namely, 
\[
\mybf{L}_{(1)}(t,\ubu_{\vert_{x=0}}) = 
 \begin{pmatrix} -N^{\rm l}_{(1)}(t) & N^{\rm r}_{(1)}(t)\end{pmatrix} \mybf{E}_p(\ubu_{\vert_{x=0}}). 
\]
When $p=2$ (the case $p=1$ is a straightforward adaptation), one has therefore 
\[
\mybf{L}_{(1)}(t,\ubu_{\vert_{x=0}}) = \mybf{L}_{p}(t,\ubu_{\vert_{x=0}})
\begin{pmatrix}
 1 - \frac{\dot{\ux}}{\widetilde{\lambda}_-({\uu^{\rm l}}_{\vert_{x=0}})} & 0 \\
 0 & 1 - \frac{\dot{\ux}}{{\lambda}_+({\uu^{\rm r}}_{\vert_{x=0}})}
\end{pmatrix}
\]
and the required bound on $\mybf{L}_{(1)}(t,\ubu_{\vert_{x=0}})^{-1}$ is therefore a direct consequence of 
Assumption \ref{asstransmred}. 
It is therefore possible to apply Theorem \ref{theoIBVP1transm} and to obtain an $\WW^{m-1}(T)$ bound on 
$\dot{\mybf{u}}^\bvarphi$ by a close adaptation of the proof of Proposition \ref{propAl}. 
Thanks to the block structure of the equations, the end of the proof follows the same lines as the proof of 
Theorem \ref{theoIBVP3}, and we therefore omit the details.

\subsection{Application to free boundary transmission problems with a transmission condition of ``kinematic'' type}
\label{secttransmkin}
We consider here a general class of free boundary quasilinear transmission problem in which two quasilinear 
hyperbolic systems at the left and at the right of a moving interface located at $x=\ux(t)$ on which 
transmission conditions are provided 
\begin{equation}\label{transmmovQL}
\begin{cases}
\dt U + \widetilde{A}({U})\dx U = 0 & \mbox{in }\quad  (-\infty,\ux(t)) \quad \mbox{ for } \quad t\in(0,T), \\
\dt U + {A}(U)\dx U = 0 & \mbox{in }\quad (\ux(t),+\infty) \quad \mbox{ for } \quad t\in(0,T), \\
U_{\vert_{t=0}} = u^{\rm in}(x) & \mbox{on }\quad \mathbb{R}_-\cup\mathbb{R}_+, \\
\underline{N}_p^{\rm r}U_{\vert_{x=\ux(t)+0}} - \underline{N}_p^{\rm l}U_{\vert_{x=\ux(t)-0}}
 = \mybf{g}(t) & \mbox{on } \quad (0,T),
\end{cases}
\end{equation}
where we assumed that $\ux(0)=0$ without loss of generality. 
Moreover, we assume that the position of the interface is given through a nonlinear equation of the form 
\begin{equation}\label{eqinterf}
\dot\ux=\chi(U_{\vert_{x=\ux(t)-0}},U_{\vert_{x=\ux(t)+0}})
\end{equation}
for some smooth function $\chi$ defined on a domain of $\R^2\times \R^2$. 
The same reduction as in \S \ref{secttransmmov}, and using the same notations, leads us to consider the 
$4\times 4$ initial boundary value problem 
\begin{equation}\label{bigIBVPQL}
\begin{cases}
\dt \mybf{u} + {\bm{\mathcal{A}}}(\mybf{u},\partial \bvarphi)\dx \mybf{u} = 0 & \mbox{ in }\quad \Omega_T, \\
\mybf{u}_{\vert_{t=0}} = \mybf{u}^{\rm in}(x)  & \mbox{ on }\quad {\mathbb R}_+, \\
\underline{\mybf{N}}_p\mybf{u}_{\vert_{x=0}} = \mybf{g}(t) & \mbox{ on }\quad (0,T),
\end{cases}
\end{equation}
where $\underline{\mybf{N}}_p = \big(-\underline{N}_p^{\rm l} \ \ \underline{N}_p^{\rm r}\big)$ is here, 
for the sake of simplicity, a {\it constant} $p\times 4$ matrix (the value of $p$ is discussed below). 
These equations are complemented by the evolution equation 
\begin{equation}\label{eqinterf2}
\dot\ux = \chi(\mybf{u}_{\vert_{x=0}}).
\end{equation}
This boundary condition, of ``kinematic'' type, leads us to work with the following generalization of the 
``Lagrangian'' diffeomorphism \eqref{diffeo}, 
\begin{equation}\label{choicediffeo}
\varphi(t,x) = x + \psi\Bigl(\frac{x}{\eps}\Bigr)\int_0^t \chi(\mybf{u}(t',|x|)){\rm d}t',
\end{equation}
where $\psi\in C_0^\infty(\R)$ is an even cut-off function such that $\psi(x)=1$ for $\abs{x}\leq1$ and $=0$ for 
$\abs{x}\geq 2$, while $\eps$ is chosen small enough to have $\mybf{u}$ close enough to its initial boundary 
value when $x$ is in the support of $\psi$ and $t$ small enough. 
Contrary to \eqref{diffeo}, this cut-off is necessary here because $\chi$ might not be defined at the origin 
(this is for instance the case in \S \ref{sectshocks} for the evolution of shocks). 
In particular, we have 
\[
\varphi^{\rm l}(t,x) = -x+\psi\Bigl(\frac{x}{\eps}\Bigr)\int_0^t \chi(\mybf{u}(t',x)){\rm d}t' \quad\mbox{ and }\quad
\varphi^{\rm r}(t,x) = x+\psi\Bigl(\frac{x}{\eps}\Bigr)\int_0^t \chi(\mybf{u}(t',x)){\rm d}t',
\]
and $\varphi^{\rm l,r}$ satisfy the same kind of bounds as those given in Lemma \ref{lemdiffeo} 
(with $\widetilde{\varphi}^{\rm r}(t,x) = \varphi^{\rm r}(t,x)-x$ and 
$\widetilde{\varphi}^{\rm l}(t,x) = \varphi^{\rm l}(t,x)+x$). 
The well-posedness of \eqref{bigIBVPQL}--\eqref{choicediffeo} also requires the following assumption.

\begin{assumption}\label{asshypQLFBtransm}
Let $\widetilde{\cU}$ and $\cU$ be two open sets in $\R^2$ and let $\boldsymbol{\cU}=\widetilde{\cU}\times\cU$ 
representing a phase space of $\mybf{u}$. 
Let $\widetilde{\cU}_I \subset \widetilde{\cU}$ and $\cU_I \subset \cU$ be also open sets and let 
$\boldsymbol{\cU}_I = \widetilde{\cU}_I\times\cU_I$ representing a phase space of $\mybf{u}_{\vert_{x=0}}$.
The following conditions hold: 
\begin{enumerate}
\setlength{\itemsep}{3pt}
\item[{\bf i.}]
$\mybf{A} \in C^\infty(\boldsymbol{\cU})$ and $\chi\in C^\infty(\boldsymbol{\cU}_I)$. 
\item[{\bf ii.}]
For all $\mybf{u} = (u^{\rm l},u^{\rm r})^{\rm T}\in\bcU$, the matrices $\widetilde{A}({u^{\rm l}})$ 
and $A(u^{\rm r})$ have eigenvalues $\widetilde{\lambda}_+(u^{\rm l}), -\widetilde{\lambda}_-(u^{\rm l})$ 
and $\lambda_+(u^{\rm r}), -\lambda_-(u^{\rm r})$, respectively, satisfying 
\[
\widetilde{\lambda}_\pm(u^{\rm l})>0 \quad\mbox{ and }\quad {\lambda}_\pm(u^{\rm r})>0;
\]
moreover, one of the following situations 
for any $\mybf{u} = (u^{\rm l},u^{\rm r})^{\rm T}\in\bcU_I$ holds: 
\begin{align*}
a) \qquad \widetilde{\lambda}_\pm(u^{\rm l})\mp \chi(\mybf{u}) > 0
 & \quad\mbox{ and }\quad 
 \lambda_\pm(u^{\rm r})\mp \chi(\bu) > 0, \\
b)  \qquad\widetilde{\lambda}_\pm(u^{\rm l}) \mp \chi(\mybf{u}) > 0 
 & \quad\mbox{ and }\quad \lambda_+(u^{\rm r}) - \chi(\mybf{u}) < 0, \\
c)\qquad  \widetilde{\lambda}_-(u^{\rm l})+\chi(\mybf{u}) < 0 
 & \quad\mbox{ and }\quad  \lambda_\pm(u^{\rm r}) \mp \chi(\mybf{u}) > 0. 
\end{align*}
\item[{\bf iii.}]
For any $\mybf{u}\in \bcU_I$, the Lopatinski\u{\i} matrix $\mybf{L}_p(\mybf{u})$ associated to 
the condition $a)$, $b)$, or $c)$ constructed in \eqref{interfmatr} is invertible 
(note that $p=2$ under condition $a)$ and $p=1$ under conditions $b)$ and $c)$). 
\end{enumerate}
\end{assumption}

\begin{remark}
With the terminology introduced in the previous section, condition $a)$ corresponds to an interface moving 
at subsonic speed, while conditions $b)$ and $c)$ correspond to interfaces moving at supersonic speed 
(to the right for condition $a)$ and to the left for condition $b)$) and satisfying Lax's conditions. 
\end{remark}

We can now state the following theorem, which can be deduced from Theorem \ref{theoIBVP3transm} in exactly 
the same way as Theorem \ref{theoIBVP4} is deduced from Theorem \ref{theoIBVP3} for free boundary initial 
value problem with an evolution equation of kinematic type for the location of the boundary.

\begin{theorem}\label{theoIBVP4transm}
Let $m\geq 2$ be an integer. 
Suppose that Assumption \ref{asshypQLFBtransm} is satisfied. 
If $\mybf{u}^{\rm in}\in H^m(\R_+)$ takes its values in $\widetilde{\mathcal K}_0\times {\mathcal K}_0$ with 
$\widetilde{\mathcal K}_0\subset\widetilde{\cU}$ and $ {\mathcal K}_0\subset \cU$ compact and convex sets, 
if $\mybf{u}^{\rm in}(0) \in \bcU_I$,
and if the data $\mybf{u}^{\rm in}$ and $\mybf{g} \in H^m(0,T)$ satisfy the compatibility conditions up to 
order $m-1$, then there exist $T_1 \in (0,T]$ and a unique solution 
$(\mybf{u},\ux)$ to \eqref{transmmovQL}--\eqref{eqinterf} with $\mybf{u}\in \WW^m(T_1)$, $\ux\in H^{m+1}(0,T_1)$, 
and $\varphi$ given by \eqref{choicediffeo}. 
\end{theorem}

\section{Waves interacting with a lateral piston}\label{sectlatpis}
We analyze here a particular example of wave-structure interaction in which the fluid occupies 
a semi-infinite canal over a flat bottom which is delimited by a lateral wall that can move horizontally. 
When the wall is in forced motion, this situation corresponds to a wave-maker device 
often used to generate waves in wave-flumes \cite{katell2002accuracy,orszaghova2012paddle}. 
We are more interested here in the case where the lateral wall moves under the action of the hydrodynamic force 
created by the waves and of a spring force that tends to bring it back to its equilibrium position. 
This configuration corresponds to a wave absorption mechanism and can also be seen as a simplified model 
of wave energy convertor, such as the Oyster. 
Such a configuration has been studied numerically in various references \cite{he2009nonlinear, korobkin2009motion, 
khakimzyanov2017numerical}, but there is no mathematical result available yet. 
Note also that this problem is related to the piston problem for isentropic gas dynamics whose linear analysis 
can be found in \cite{gerlach1984two} and weak solutions constructed in \cite{takeno1995free}. 
Our goal in this section is to provide a well-posedness result for this wave-structure interaction under 
the shallow water approximation, i.e., assuming that the evolution of the free surface is governed by 
the nonlinear shallow water equations. 
The configuration under study here is described in Figure \ref{fpiston}.

\medskip
\begin{figure}[h]
\begin{center}
\includegraphics[width=0.7\linewidth]{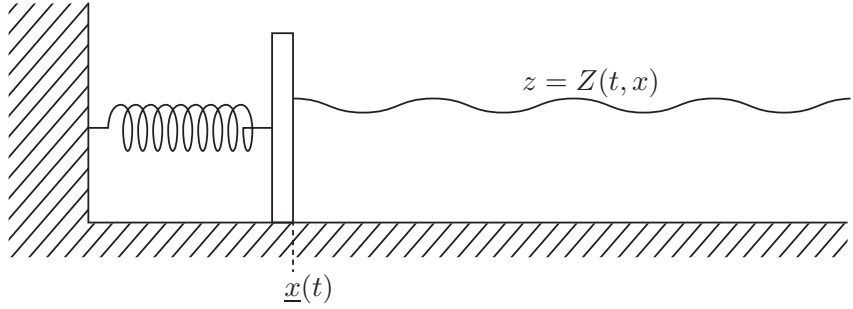}
\end{center}
\setlength{\unitlength}{1pt}
\begin{picture}(0,0)
\put(-55,6){$\ux(t)$}
\put(35,84){$z=Z(t,x)$}
\end{picture}
\caption{Waves interacting with a lateral piston}\label{fpiston}
\end{figure}

\subsection{Presentation of the problem}\label{sectprespb1}
In the canal, of mean depth $h_0$ and delimited on the left by the moving wall located at $x=\ux(t)$, 
the waves are described by the nonlinear shallow water equations. 
It is convenient to write them in $(H,\ovV)$ variables, where $H(t,x)=h_0+Z(t,x)$ is the water depth, 
$Z(t,x)$ is the surface elevation of the water, and $\ovV(t,x)$ is the vertically averaged horizontal velocity 
\begin{equation}\label{SW}
\begin{cases}
\dt H + \dx (H\ovV) = 0 & \mbox{in}\quad (\ux(t),\infty), \\
\dt \ovV + \ovV\dx\ovV + \mathtt{g}\dx H = 0 & \mbox{in}\quad (\ux(t),\infty),
\end{cases}
\end{equation}
where $\mathtt{g}$ is the gravitational constant; 
with this formulation, the boundary condition at the left boundary at the canal will be imposed as the kinematic type: 
the velocity $\ovV$ matches the velocity $\dot{\ux}$, that is, 
\begin{equation}\label{BCSW}
\ovV(t,\ux(t)) = \dot{\ux}(t).
\end{equation}
Since the wall moves under the action of the hydrodynamic force exerted by the fluid and of the spring force, 
its position $\ux(t)$ satisfies Newton's equation 
$$
\mathtt{m}\ddot{\ux} = -\mathtt{k}(\ux-\ux_0) + F_{\rm hyd},
$$
where $\mathtt{m}$ is the mass of the moving wall, $\mathtt{k}$ the stiffness of the spring force, 
$\ux_{0}$ its reference position, and $F_{\rm hyd}$ the hydrodynamic force. 
This force corresponds to the horizontal pressure forces integrated on the vertical wall. 
Assuming, in accordance with the modeling of the flow by the nonlinear shallow water equations, 
that the pressure is hydrostatic, we get 
\begin{align*}
F_{\rm hyd} &= \int_{-h_0}^{Z(t,\ux(t))} \rho\mathtt{g} (Z(t,\ux(t)) - z'){\rm d}z' \\
&= \frac{1}{2}\rho\mathtt{g} (h_0 + Z(t,\ux(t)))^2.
\end{align*}
At rest, we have $H=h_0$ and the equilibrium position $\ux_{\rm eq}$ is therefore given by 
$$
\ux_{\rm eq} - \ux_0 = \frac{1}{2}\frac{\rho\mathtt{g} h_0^2}{\mathtt{k}}
$$
so that Newton's equation can be put under the form 
\begin{equation}\label{NewtonSW}
\mathtt{m}\ddot{\ux} = - \mathtt{k} (\ux-\ux_{\rm eq})
 + \frac{1}{2}\rho\mathtt{g} \bigl( (h_0+Z_{\vert_{x=\ux}})^2 - h_0^2 \bigr).
\end{equation}
The free boundary problem we have to solve consists therefore in the equations \eqref{SW}--\eqref{NewtonSW} 
complemented by the initial conditions 
\begin{equation}\label{ICSW}
(Z,\ovV)_{\vert_{t=0}} = (Z^{\rm in},\ovV^{\rm in}) \quad\mbox{on}\quad \R_+, \qquad
 (\ux,\dot{\ux})_{\vert_{t=0}} = (0,\ux_1^{\rm in}),
\end{equation}
where we assumed without loss of generality that the wall is initially located at $x=0$.

\subsection{Reformulation of the equations}
As in \S \ref{sectVCm}, the first step is to use a diffeomorphism $\varphi(t,\cdot): \R_+\to (\ux(t),\infty)$ 
and to work with the transform variables 
$$
\zeta(t,x) = Z(t,\varphi(t,x)), \qquad \ovv(t,x) = \ovV(t,\varphi(t,x))
$$
with $h=h_0+\zeta$. 
The boundary condition \eqref{BCSW} which can be rewritten as 
$$
\dot{\ux}(t) = \ovv(t,0)
$$
leads us to work with the Lagrangian diffeomorphism 
\begin{equation}\label{diffL}
\varphi(t,x) = x + \int_0^t \ovv(t',x){\rm d}t',
\end{equation}
which satisfies the properties stated in Lemma \ref{lemdiffeo}. 
After composition with $\varphi$, the problem under consideration is reduced to the initial boundary value problem 
\begin{equation}\label{SW2}
\begin{cases}
\dsp \dt \zeta + h \dx^\varphi \ovv = 0 & \mbox{in}\quad \Omega_T, \\
\dsp \dt \ovv + \mathtt{g} \dx^\varphi \zeta = 0 & \mbox{in}\quad \Omega_T, \\
(\zeta,\ovv)_{\vert_{t=0}} = (\zeta^{\rm in},\ovv^{\rm in}) & \mbox{on}\quad \R_{+}, \\
\ovv_{\vert_{x=0}} = \dot{x} & \mbox{on}\quad (0,T),
\end{cases}
\end{equation}
coupled to the ODE
\begin{equation}\label{NewtonSW2}
\begin{cases}
\mathtt{m}\ddot{\ux} = -\mathtt{k}(\ux-\ux_{\rm eq})
 + \frac{1}{2}\rho\mathtt{g} \bigl( (h_0+\zeta_{\vert_{x=0}})^2 - h_0^2 \bigr)
 \quad\mbox{for}\quad t\in(0,T), \\
(\ux,\dot{\ux})_{\vert_{t=0}} = (0,\ux_1^{\rm in}),
\end{cases}
\end{equation}
where we used the same notation as in \eqref{dtphi}, that is, $\dx^\varphi = \frac{1}{\dx\varphi}\dx$. 
The initial boundary value problem \eqref{SW2} is of course of the form \eqref{IBVPmT} with $u=(\zeta,\ovv)^{\rm T}$, $\nu=(0,1)^{\rm T}$, and 
\begin{equation}\label{Au}
A(u) = \begin{pmatrix} \ovv & h \\ \mathtt{g} & \ovv \end{pmatrix},
\end{equation}
whose eigenvalues $\pm\lambda_{\pm}(u)$ and the corresponding unit eigen vectors $\mathbf{e}_{\pm}(u)$ 
are given by 
\[
\lambda_{\pm}(u) = \sqrt{\mathtt{g}h} \pm \ovv, \qquad 
\mathbf{e}_{\pm}(u) =\frac{1}{\sqrt{\mathtt{g}+h}
 \binom{\sqrt{h}}{\pm\sqrt{\mathtt{g}}}. }
\]
Therefore, the positivity of $|\nu\cdot\mathbf{e}_{+}(u_{\vert_{x=0}})|$ stated in Assumption \ref{asshypm} 
is automatically satisfied under the positivity of $h$.

Here, we will show another equivalent formulation to \eqref{SW2}--\eqref{NewtonSW2}. 
The following lemma shows that \eqref{NewtonSW2} provides an expression for $\dot\ux$ 
in terms of $\zeta_{\vert_{x=0}}$.

\begin{lemma}\label{lemmaG}
Let $m\geq1$ be an integer, $\ux_1^{\rm in}\in \R$, and assume that $\zeta_{\rm b}\in H^m(0,T)$. 
Then there exists a unique solution $\ux \in H^{m+2}(0,T)$ to 
\[
\begin{cases}
{\mathtt m}\ddot{\ux}  = -{\mathtt k} (\ux - \ux_{\rm eq})
 + \frac12 \rho {\mathtt g} \bigl( \zeta_{\rm b}^2 + 2h_0\zeta_{\rm b} \bigr), \\
(\ux,\dot \ux)_{\vert_{t=0}} = (0,\ux_1^{\rm in}),
\end{cases}
\]
so that we can define a mapping $\mathcal{G} : H^m(0,T) \ni \zeta_{\rm b} \mapsto \dot{\ux} \in H^{m+1}(0,T)$, 
which satisfies 
\[
|\mathcal{G}(\zeta_{\rm b})|_{H^{m+1}(0,t)}
 \leq C\bigl( \sqrt{t}(|\ux_{\rm eq}|+|\ux_1^{\rm in}|)
  + (1+t)( 1 + |\zeta_{\rm b}|_{W^{[m/2],\infty}(0,t)} ) |\zeta_{\rm b}|_{H^m(0,t)} \bigr)
\]
for any $t\in[0,T]$, where $C>0$ is a constant depending only on $\mathtt{m},\mathtt{k},\rho\mathtt{g},h_0$, 
and $m$. 
\end{lemma}

\begin{proof}
The existence and uniqueness of the solution $\ux$ is obvious, so that we focus on the derivation of the estimate. 
Replacing $\ux$ with $\ux+\ux_{\rm eq}$, it is sufficient to consider the problem 
\[
\begin{cases}
{\mathtt m}\ddot{\ux}  = -\mathtt{k}\ux + f, \\
(\ux,\dot \ux)_{\vert_{t=0}} = (\ux_{\rm eq},\ux_1^{\rm in}),
\end{cases}
\]
where $f=\frac12 \rho {\mathtt g} \bigl( \zeta_{\rm b}^2 + 2h_0\zeta_{\rm b} \bigr)$. 
Then, we see that 
\[
\frac12\frac{\rm d}{{\rm d}t}( \mathtt{m} \dot{\ux}(t)^2 + \mathtt{k} \ux(t)^2 ) = f(t)\dot{x}(t),
\]
from which we deduce that
\begin{align*}
|\dot{\ux}(t)| + |\ux(t)| 
&\leq C\Bigl( |\ux_1^{\rm in}| + |\ux_{\rm eq}| + \int_0^t|f(t')|{\rm d}t' \Bigr) \\
&\leq C( |\ux_1^{\rm in}| + |\ux_{\rm eq}| + \sqrt{t}|f|_{L^2(0,t)} ),
\end{align*}
so that 
\[
|\ux|_{H^1(0,t)} \leq C\bigl( \sqrt{t}(|\ux_1^{\rm in}| + |\ux_{\rm eq}|) + t|f|_{L^2(0,t)} \bigr).
\]
On the other hand, it follows from the equation directly that 
\[
|\dt^{k+2}\ux|_{L^2(0,t)} \leq C(|\dt^k\ux|_{L^2(0,t)}+|\dt^kf|_{L^2(0,t)})
\]
for $k=0,1,2,\ldots$. 
Using these inductively, we obtain 
\[
|\ux|_{H^{m+2}(0,t)}
 \leq C\bigl( \sqrt{t}(|\ux_1^{\rm in}| + |\ux_{\rm eq}|) + t|f|_{L^2(0,t)} + |f|_{H^m(0,t)} \bigr),
\]
which together with 
$|f|_{H^m(0,t)} \leq C( 1 + |\zeta_{\rm b}|_{W^{[m/2],\infty}(0,t)} ) |\zeta_{\rm b}|_{H^m(0,t)}$ 
gives the desired estimate. 
\end{proof}

It follows from the lines above that the problem presented in \S \ref{sectprespb1} can be recast 
under the following form 
\begin{equation}\label{pb1recast}
\begin{cases}
\dt u + \cA(u,\partial\varphi)\dx u = 0 & \mbox{in}\quad \Omega_T, \\
u_{\vert_{t=0}} = u^{\rm in} & \mbox{on}\quad \R_{+}, \\
\unu\cdot u_{\vert_{x=0}} = \mathcal{G}(\unu^\perp\cdot u_{\vert_{x=0}}) & \mbox{on}\quad (0,T),
\end{cases}
\end{equation}
where $\unu=(0,1)^{\rm T}$ and $\varphi$ is given by \eqref{diffL}, with a boundary equation given by 
\begin{equation}\label{BCSWt}
\dot\ux=\unu\cdot u_{\vert_{x=0}}, \qquad \ux_{\vert_{t=0}}=0.
\end{equation}
Here, we emphasize that the notation for the matrix $\cA(u,\varphi)$ is the same as in \eqref{IBVPmT} with 
the matrix $A(u)$ defined by \eqref{Au}. 
However, thanks to our choice of the Lagrangian diffeomorphism $\varphi$, 
the term $\dt\varphi$ is cancelled and does not appear in the equation. 
The problem is therefore a small variant of the free boundary problem considered in \S \ref{sectFB1}, 
the difference being that the boundary condition $\unu\cdot u_{\vert_{x=0}}=g(t)$ is replaced by a semi-linear 
and nonlocal boundary condition $\unu\cdot u_{\vert_{x=0}} = {\mathcal G}(\nu^\perp\cdot u_{\vert_{x=0}})$. 
Of course, \eqref{pb1recast}--\eqref{BCSWt} is equivalent to \eqref{SW2}--\eqref{NewtonSW2}.

\subsection{Compatibility condition}
As usual, compatibility conditions are required to have regular solutions. 
However, we can derive the conditions easier than the problem considered in \S \ref{sectFB1} because 
the equation does not contain the term $\dt\varphi$. 
Denoting $u_k=\dt^k u$, we get classically by induction that $u_k$ is a polynomial expression of space derivatives 
of $u$ of order at most $k$, and of space and time derivatives of $(\dx\varphi)^{-1}$ of order at most $k-1$. 
Remarking further that $\dx^{j}\dt^{l+1} \varphi= \dx^j\dt^l\ovv$ and $\dx^{j+1}\varphi_{\vert_{t=0}}=\delta_{j,0}$, 
where $\delta_{j,0}$ is the Kronecker symbol, it follows that at $t=0$, we have an expression for 
$u_k^{\rm in}={u_k}_{\vert_{t=0}}$ as 
\begin{equation}\label{pb1comp}
u_k^{\rm in}=c_{1,k}(u^{\rm in},\dx u^{\rm in},\dots, \dx^k u^{\rm in})
\end{equation}
with $c_{1,k}$ a polynomial expression of its arguments such that the total number of derivatives of $u^{\rm in}$ 
involved in each monomial is at most $k$. 
Using the equation in \eqref{NewtonSW2} we can express $\ux_k^{\rm in}$ for $k\geq2$ in terms of the initial data as 
\begin{equation}\label{xk+2in}
\ux_{k+2}^{\rm in} = c_{2,k}(\ux_1^{\rm in},\zeta^{\rm in},\zeta_1^{\rm in},\ldots,\zeta_k^{\rm in})_{\vert_{x=0}}
\end{equation}
with $c_{2,k}$ a polynomial expression of its arguments. 
The compatibility condition is obtained by differentiating the boundary condition 
$\ovv_{\vert_{x=0}} = \dot{\ux}$ with respect to $t$ and taking its trace at $t=0$.

\begin{definition}\label{defcomppb1}
Let $m\geq1$ be an integer. 
We say that the initial data $u^{\rm in} = (\zeta^{\rm in},\ovv^{\rm in})^{\rm T} \in H^m(\R_+)$ and 
$\ux_1^{\rm in}\in \R$ for the initial boundary value problem \eqref{SW2}--\eqref{NewtonSW2} satisfy 
the compatibility condition at order $k$ if $\{u_j^{\rm in}\}_{j=0}^m$ and $\{\ux_j^{\rm in}\}_{j=1}^{m+1}$ 
defined by \eqref{pb1comp}--\eqref{xk+2in} satisfy 
\[
{\ovv_k^{\rm in}}_{\vert_{x=0}} = \ux_{k+1}^{\rm in}. 
\]
We also say that the initial data $u^{\rm in}$ and $\ux_1^{\rm in}$ satisfy the compatibility conditions 
up to $m-1$ if they satisfy the compatibility conditions at order $k$ for $k=0,1,\ldots,m-1$. 
\end{definition}

\begin{remark}
The local existence theorem given below requires that the compatibility conditions are satisfied at order $m-1$ with $m\geq 2$. In the case $m=2$, the compatibility conditions are
$$
{\ovv^{\rm in}}_{\vert_{x=0}} = \ux_1^{\rm in} \quad\mbox{and}\quad 
-\mathtt{g} (\dx \zeta^{\rm in})_{\vert_{x=0}}
 = \mathtt{k}\ux_{\rm eq} + \frac{\rho\mathtt{g}}{2\mathtt{m}}
  \bigl( (\zeta^{\rm in})^2 + 2h_0\zeta^{\rm in} \bigr)_{\vert_{x=0}}.
$$
\end{remark}

\subsection{Local well-posedness}
We can now state the main result of this section, which shows the local well-posedness of the wave-structure 
interaction problem presented in \S \ref{sectprespb1}.

\begin{theorem}
Let $m\geq 2$ be an integer.
If the initial data $(\zeta^{\rm in},\ovv^{\rm in})^{\rm T}\in H^m(\R_+)$ and  $\ux_1^{\rm in} \in \R$ satisfy 
\[
\inf_{x\in\R_{+}}\bigl( \sqrt{\mathtt{g}(h_0+\zeta^{\rm in}(x))} - |\ovv^{\rm in}(x)| \bigr) > 0
\]
and the compatibility conditions up to order $m-1$ in the sense of Definition \ref{defcomppb1}, 
then there exist $T>0$ and a unique solution $(\zeta,\ovv,\ux)$ to \eqref{SW2}--\eqref{NewtonSW2} with 
$(\zeta,\ovv)\in \WW^m(T)$ and $\ux\in H^{m+2}(0,T)$, and $\varphi$ given by \eqref{diffL}. 
\end{theorem}

\begin{proof}
The proof is a small variant of the proof of Theorem \ref{theoIBVP4}. 
We define the phase space $\cU$ of $u=(\zeta,\ovv)^{\rm T}$ by 
\[
\cU = \{ u=(\zeta,\ovv)^{\rm T} \in \R^2 \,|\, \sqrt{\mathtt{g}(h_0+\zeta)} - |\ovv|>0 \}.
\]
Then, we can readily check that all the conditions in Assumption \ref{asshypQLFB} are satisfied with 
$\chi(u)=\ovv$ and $\unu=(0,1)^{\rm T}$. 
Moreover, once $u^n=(\zeta^n,\ovv^n)^{\rm T} \in \WW^m(T)$ is given so that 
\begin{equation}\label{unifest4}
\begin{cases}
(\dt^k u^n)_{\vert_{t=0}}=u_k^{\rm in} \quad\mbox{for}\quad k=0,1,\ldots,m-1, \\
\|u^n\|_{\WW^m(T)} + |{u^n}_{\vert_{x=0}}|_{m,T} \leq M_1,
\end{cases}
\end{equation}
we can check that the data $u^{\rm in}$ and $g^n(t)=\mathcal{G}(\unu^\perp\cdot {u^n}_{\vert_{x=0}})$ for the problem 
\[
\begin{cases}
\dt u + \cA(u,\partial\varphi)\dx u = 0 & \mbox{in}\quad \Omega_T, \\
u_{\vert_{t=0}} = u^{\rm in}(x) & \mbox{on}\quad \R_{+}, \\
\unu\cdot u_{\vert_{x=0}} = g^n(t) & \mbox{on}\quad (0,T),
\end{cases}
\]
\[
\dot\ux = \unu\cdot u_{\vert_{x=0}}, \qquad \ux_{\vert_{t=0}}=0,
\]
satisfy the compatibility conditions up to order $m-1$ in the sense of Definition \ref{defcompfbp}, 
and we can apply Theorem \ref{theoIBVP4} to show a unique existence of the solution 
$u=(\zeta,\ovv)^{\rm T} \in \WW^m(T_1)$ and $\ux \in H^{m+1}(0,T_1)$ to this problem for some $T_1 \in (0,T]$ 
depending on $M_1$. 
We denote by $u^{n+1}$ this solution $u$. 
Furthermore, we see that $u^{n+1}$ satisfies $(\dt^k u^{n+1})_{\vert_{t=0}}=u_k^{\rm in}$ for $k=0,1,\ldots,m-1$ and 
\[
\|u^{n+1}\|_{\WW^m(T_1)} + |{u^{n+1}}_{\vert_{x=0}}|_{m,T_1} \leq C_1(|\mathcal{G}(\unu^\perp\cdot {u^n}_{\vert_{x=0}})|_{H^m(0,T_1)}).
\]
Here, by Lemma \ref{lemmaG} we have 
\[
|\mathcal{G}(\unu^\perp\cdot {u^n}_{\vert_{x=0}})|_{H^{m+1}(0,T_1)} \leq C(M_1,T_1).
\]
On the other hand, we have 
\begin{align*}
|\mathcal{G}(\unu^\perp\cdot {u^n}_{\vert_{x=0}})|_{H^m(0,T_1)}
&\leq \sqrt{T_1}\sum_{j=1}^{m+1}|\ux_j^{\rm in}|
 + T_1|\mathcal{G}(\unu^\perp\cdot {u^n}_{\vert_{x=0}})|_{H^{m+1}(0,T_1)},
\end{align*}
where we used $(\dt^k \mathcal{G}(\unu^\perp\cdot {u^n}_{\vert_{x=0}}))_{\vert_{t=0}} = \ux_{k+1}^{\rm in}$ 
for $k=0,1,\ldots,m$. 
Therefore, for any fixed $M_0>0$ if we define $M_1>0$ by $M_1=C_1(M_0)$ and choose $T_1=T_1(M_0)$ 
sufficiently small, then we have 
\[
|\mathcal{G}(\unu^\perp\cdot {u^n}_{\vert_{x=0}})|_{H^m(0,T_1)} \leq M_0,
\]
so that $u^{n+1}$ satisfies \eqref{unifest4} with $T$ replaced by $T_1$. 
Now, we can iterate the above procedure to construct a sequence of approximate solutions 
$\{(\zeta^n,\ovv^n,\ux^n)\}_n$, which satisfy the uniform bounds. 
As in the proof of Theorem \ref{theoIBVP4}, we can prove the convergence of these approximate solutions 
to the solution $(\zeta,\ovv,\ux)$ to \eqref{pb1recast}--\eqref{BCSWt}. 
This solution satisfies $\dot{\ux} = \mathcal{G}(\unu^\perp\cdot {u}_{\vert_{x=0}}) \in H^{m+1}(0,T_1)$, 
so that we have the regularity $\ux \in H^{m+2}(0,T_1)$. 
\end{proof}

\section{Shallow water model with a floating body on the water surface} \label{sectfloat}
We turn to analyze other examples of wave-structure interaction in which the fluid occupies an infinite canal 
and a floating rigid body is placed on the water surface. 
We follow the approach proposed in \cite{Lannes2017} where the free surface Euler equations are reformulated 
in terms of the free surface elevation and of the horizontal water flux.
Under this approach, the pressure exerted by the fluid on the floating body can be viewed as the Lagrange 
multiplier associated to the constraint that, under the body, the surface of the fluid coincides 
with the bottom of the body.

As shown in \cite{Lannes2017}, this approach can be used also in the shallow water approximation, 
replacing the free surface Euler equations by the much simpler nonlinear shallow water equations. 
This is the framework that we shall consider here, addressing three cases; the floating body is fixed, 
the motion of the body is prescribed, and the body moves freely according to Newton's laws 
under the action of the gravitational force and the pressure from the air and from the water. 
The case of a floating body moving only vertically and with vertical lateral walls has been considered in 
\cite{Lannes2017} in $1D$, in \cite{Bocchi} for a $2D$ configuration with radial symmetry, and numerical computations have been proposed in \cite{Bosi}.
For such configurations, the horizontal projection of the portion of the solid in contact with the water is 
independent of time. 
We consider here the more complex situation of nonvertical lateral walls: 
even in the case of a fixed object, determining the portion of the solid in contact with the water 
is then a free boundary problem that is difficult to handle; in the numerical study \cite{Godlewski} for instance, the authors use a compressible approximation of the equations in order to remove this issue.
The configuration under study here is described in Figure \ref{ffloating}.

\bigskip
\begin{figure}[h]
\begin{center}
\includegraphics[width=0.7\linewidth]{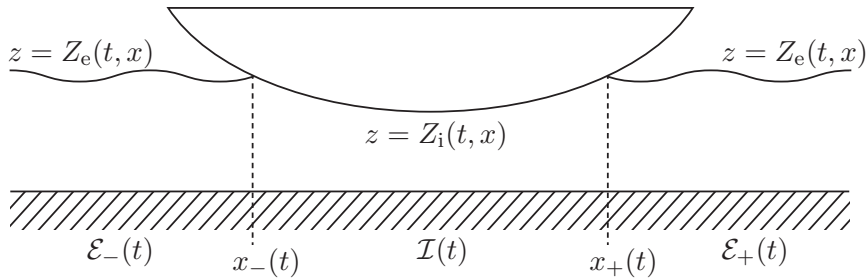}
\end{center}
\setlength{\unitlength}{1pt}
\begin{picture}(0,0)
\put(-75,6){$x_{-}(t)$}
\put(60,6){$x_{+}(t)$}
\put(-5,10){$\mathcal{I}(t)$}
\put(-130,10){$\mathcal{E}_{-}(t)$}
\put(110,10){$\mathcal{E}_{+}(t)$}
\put(-25,54){$z=Z_{\rm i}(t,x)$}
\put(110,85){$z=Z_{\rm e}(t,x)$}
\put(-160,85){$z=Z_{\rm e}(t,x)$}
\end{picture}
\caption{Waves interacting with a floating body}\label{ffloating}
\end{figure}

\subsection{Presentation of the equations for the water}\label{sectpresfloat}
We consider the two-dimensional water waves over a flat bottom with a floating body on the water surface 
under the assumption that there are only two contact points where the water, the air, and the body meet. 
These contact points at time $t$ are denoted by $x_{-}(t)$ and $x_{+}(t)$, which satisfy $x_{-}(t)<x_{+}(t)$. 
Let $\mathcal{I}(t)$ and $\mathcal{E}(t)$ be the projections on the horizontal line of the parts where 
the water surface contacts with the floating structure and the air, respectively, that is, 
\[
\begin{cases}
\mathcal{I}(t) = (x_{-}(t),x_{+}(t)), \\
\mathcal{E}(t) = \mathcal{E}_{-}(t) \cup \mathcal{E}_{+}(t), \quad 
 \mathcal{E}_{-}(t)=(-\infty,x_{-}(t)), \quad \mathcal{E}_{-}(t)=(x_{+}(t),\infty). 
\end{cases}
\]
The corresponding water regions to $\mathcal{I}(t)$ and $\mathcal{E}(t)$ will be called the interior and 
the exterior regions, respectively. 
We consider the case where overhanging waves do not occur and suppose that the surface elevation of the water 
in the exterior region is denoted by $Z_{\rm e}(t,x)$ and that the underside of the floating body is 
parameterized by $Z_{\rm i}(t,x)$, where $x$ is the horizontal coordinate. 
Let $h_0$ be the mean depth of the water, so that the water depth in the interior and exterior regions are 
given by $H_{\rm i}(t,x) = h_0 + Z_{\rm i}(t,x)$ and $H_{\rm e}(t,x) = h_0 + Z_{\rm e}(t,x)$, respectively. 
We denote by $\ovV(t,x)$ the vertically averaged horizontal velocity and put $Q = H\ovV$, which is the 
horizontal flux of the water. 
The restrictions of $Q$ to the interior and the exterior regions will be denoted by 
$Q_{\rm i}$ and $Q_{\rm e}$, respectively. 
Let $\underline{P}_{\rm i}(t,x)$ be the pressure of the water at the underside of the floating body. 
This pressure is an important unknown quantity and should be determined together with the motion of the water. 
In the case where the floating body moves freely, 
the body interacts with the water through the force exerted by this pressure. 
The shallow water model was derived from the full water wave equations by using the assumption that 
$\dx \big( \int_{-h_0}^\zeta V(t,x,z)^2 dz\big) \approx \dx \big( H \ovV^2\big)$, 
where $V(t,x,z)$ denotes the horizontal component of the velocity field in the fluid, 
and that the pressure $P(t,x,z)$ can be approximated by the hydrostatic pressure, that is, 
\[
P(t,x,z) = 
\begin{cases}
 P_{\rm atm} - \rho\mathtt{g}(z-Z_{\rm e}(t,x)) & \mbox{in}\quad \mathcal{E}(t), \\
 \underline{P}_{\rm i}(t,x) - \rho\mathtt{g}(z-Z_{\rm i}(t,x)) & \mbox{in}\quad \mathcal{I}(t),
\end{cases}
\]
where $\rho$ is the density of the water, $\mathtt{g}$ the gravitational constant, and 
$P_{\rm atm}$ the atmospheric pressure (see \cite{Lannes2017}). 
Then, the shallow water model for the water has the form 
\begin{equation}\label{eqext}
\begin{cases}
\dt Z_{\rm e} + \dx Q_{\rm e} = 0 & \mbox{in}\quad \mathcal{E}(t), \\
\dt Q_{\rm e} + \dx \bigl( \frac{Q_{\rm e}^2}{H_{\rm e}} + \frac12\mathtt{g}H_{\rm e}^2 \bigr) = 0 
 &\mbox{in}\quad \mathcal{E}(t),
\end{cases}
\end{equation}
in the exterior region, while under the object we have
\begin{equation}\label{eqint}
\begin{cases}
\dt Z_{\rm i} + \dx Q_{\rm i} = 0 & \mbox{in}\quad \mathcal{I}(t), \\
\dt Q_{\rm i} + \dx \bigl( \frac{Q_{\rm i}^2}{H_{\rm i}} + \frac12\mathtt{g}H_{\rm i}^2 \bigr)
 = -\frac{1}{\rho}H_{\rm i}\dx \underline{P}_{\rm i}
 &\mbox{in}\quad \mathcal{I}(t),
\end{cases}
\end{equation}
with transmission conditions
\begin{equation}\label{BC1}
H_{\rm e} = H_{\rm i}, \quad Q_{\rm e} = Q_{\rm i}, \quad \underline{P}_{\rm i} = P_{\rm atm} 
\quad \mbox{on} \quad \Gamma(t),
\end{equation}
where $\Gamma(t)=\partial\mathcal{I}(t)=\partial\mathcal{E}(t)$ denotes the contact points. 
We also need to prescribe equations of the motion of the floating body. 
Such equations will be given in the following sections according to the cases where the floating body 
is fixed, the motion of the body is prescribed, or the body moves freely.

\subsubsection{Basic structure of the equations}\label{sectbseq}
Once the equations of the motion of the floating body are given, as we will see in the following sections, 
we can solve the equations in the interior region \eqref{eqint} and the problem will be reduced to 
the type considered in \S \ref{sectVCm2} with $U=(Z_{\rm e},Q_{\rm e})^{\rm T}$. 
We note that \eqref{eqext} can be written in the matrix form 
\[
\dt U + A(U) \dx U = 0.
\]
As was explained in Example \ref{ex1}, the eigenvalues $\lambda_{\pm}(U)$ of the coefficient matrix 
$A(U)$ and the corresponding unit eigenvectors $\mathbf{e}_{\pm}(U)$ are given by 
\[
\lambda_{\pm}(U) = \sqrt{\mathtt{g}H_{\rm e}} \pm \frac{Q_{\rm e}}{H_{\rm e}}, \qquad
 \mathbf{e}_{\pm}(U) = \frac{1}{\sqrt{1+\lambda_{\pm}(U)^2}}
  \binom{1}{\pm\lambda_{\pm}(U)}.
\]
Moreover, the unit vector $\mu_0$ defined in Remark \ref{remarkIC2} is in this case given by 
$\mu_0=(1,0)^{\rm T}$, so that the condition $\mu_0 \cdot \mathbf{e}_{+}(U) \ne0$ is automatically satisfied. 
As was explained in \S \ref{sectVCm2}, the discontinuity of $\dx U$ at the contact points plays an 
important role to determine the contact points $x_{\pm}$. 
Concerning this discontinuity condition, we have the following proposition.

\begin{proposition}\label{propdisconti}
Suppose that $U_{\rm e} = (Z_{\rm e},Q_{\rm e})^{\rm T}$, $U_{\rm i} = (Z_{\rm i},Q_{\rm i})^{\rm T}$, 
$\underline{P}_{\rm i}$, and $x_{\pm}$ satisfy \eqref{eqext}--\eqref{BC1}. 
Then, the condition $\dx U_{\rm e} - \dx U_{\rm i} \ne 0$ on $\Gamma(t)$ 
is equivalent to $\dx Z_{\rm e} - \dx Z_{\rm i} \ne 0$ on $\Gamma(t)$. 
\end{proposition}

\begin{proof}
Differentiating the boundary condition $Z_{\rm e}(t,x_{\pm}(t)) = Z_{\rm i}(t,x_{\pm}(t))$ with respect to $t$, 
we obtain 
\[
\dt Z_{\rm e} + \dot{x}_{\pm}\dx Z_{\rm e}
 = \dt Z_{\rm i} + \dot{x}_{\pm}\dx Z_{\rm i} \quad\mbox{on}\quad \Gamma(t).
\]
By the continuity equations in the interior and the exterior regions, we have 
$\dt Z_{\rm e} = -\dx Q_{\rm e}$ and $\dt Z_{\rm i} = -\dx Q_{\rm i}$, so that 
\[
\dot{x}_{\pm} ( \dx Z_{\rm e} - \dx Z_{\rm i} ) = \dx Q_{\rm e} - \dx Q_{\rm i} \quad\mbox{on}\quad \Gamma(t).
\]
This gives the desired result. 
\end{proof}

\subsection{The case of a fixed floating body}\label{sectfixfloat}
In the case where the body is fixed, we impose the condition 
\begin{equation}\label{body1}
Z_{\rm i}=Z_{\rm lid} \quad\mbox{on}\quad \mathcal{I}(t),
\end{equation}
where $Z_{\rm lid}=Z_{\rm lid}(x)$ is a given function defined on an open interval $I_{\rm f}$.

\subsubsection{Reformulation of the equations}\label{sectreform}
We begin to solve the equations in the interior region \eqref{eqint}. 
It follows from \eqref{body1} that $H_{\rm i}(t,x)=h_0+Z_{\rm lid}(x)$ does not depend on $t$, so that 
the continuity equation in \eqref{eqint} yields $\dx Q_{\rm i}=0$. 
This means that $Q_{\rm i}$ does not depend on $x$, so that we can write $Q_{\rm i}(t,x)=q_{\rm i}(t)$. 
Plugging this into the momentum equation in \eqref{eqint} we have 
\[
\dot{q}_{\rm i} + \dx\Bigl( \frac{q_{\rm i}^2}{H_{\rm i}} + \frac12\mathtt{g}H_{\rm i}^2 \Bigr)
 = -\frac{1}{\rho}H_{\rm i} \dx\underline{P}_{\rm i},
\]
which is equivalent to 
\[
\frac{\dot{q}_{\rm i}}{H_{\rm i}}
 + \dx\Bigl( \frac12\frac{q_{\rm i}^2}{H_{\rm i}^2} + \mathtt{g}H_{\rm i} \Bigr)
  = -\frac{1}{\rho} \dx\underline{P}_{\rm i}.
\]
Therefore, $\underline{P}_{\rm i}$ satisfies a simple boundary value problem 
\begin{equation}\label{bvpp1}
\begin{cases}
 \dx\underline{P}_{\rm i} = -\rho \bigl( \frac{\dot{q}_{\rm i}}{H_{\rm i}}
 + \dx\bigl( \frac12\frac{q_{\rm i}^2}{H_{\rm i}^2} + \mathtt{g}H_{\rm i} \bigr) \bigr)
  & \mbox{in}\quad \mathcal{I}(t), \\
 \underline{P}_{\rm i} = P_{\rm atm} & \mbox{on}\quad \Gamma(t).
\end{cases}
\end{equation}

\begin{notation}
For a function $F=F(t,x)$, we put $\jump{F}=F(t,x_{-}(t))-F(t,x_{+}(t))$. 
\end{notation}

Integrating the first equation in \eqref{bvpp1} and using the boundary condition, we obtain 
\begin{equation}\label{eqqi}
\dot{q}_{\rm i}\int_{\mathcal{I}(t)}\frac{1}{H_{\rm i}}
 + \jump{ \frac12\frac{q_{\rm i}^2}{H_{\rm i}^2} + \mathtt{g}H_{\rm i} } = 0,
\end{equation}
which is a solvability condition of the boundary value problem \eqref{bvpp1} for $\underline{P}_{\rm i}$. 
Conversely, once $q_{\rm i}$ and $x_{\pm}$ are given so that \eqref{eqqi} holds, 
we can resolve \eqref{bvpp1} for the pressure $\underline{P}_{\rm i}$ explicitly as 
\begin{align*}
\underline{P}_{\rm i}(t,x)
&= P_{\rm atm} -\rho\biggl\{ \dot{q}_{\rm i}(t) \int_{x_{-}(t)}^x\frac{{\rm d}x'}{H_{\rm i}(x')} \\
&\quad
 + \frac12q_{\rm i}(t)^2\biggl( \frac{1}{H_{\rm i}(x)^2} - \frac{1}{H_{\rm i}(x_{-}(t))^2} \biggr)
 + \mathtt{g}( H_{\rm i}(x) - H_{\rm i}(x_{-}(t)) ) \biggr\}.
\end{align*}
Therefore, the equations in the interior region \eqref{eqint} are reduced to a scalar ordinary 
differential equation \eqref{eqqi}.

We turn to reformulate the equations in the exterior region \eqref{eqext}. 
As in \S \ref{sectVCm2}, we will use a coordinate transformation to reduce the equations on 
the unknown region $\mathcal{E}(t)$ to those on a fixed region $\underline{\mathcal{E}}$. 
Let $\ux_{-}^{\rm in}$ and $\ux_{+}^{\rm in}$ be the initial contact points at time $t=0$ such that 
$\ux_{-}^{\rm in} < \ux_{+}^{\rm in}$ and put $\underline{\mathcal{E}}_{-} = (-\infty,\ux_{-}^{\rm in})$, 
$\underline{\mathcal{E}}_{+} = (\ux_{+}^{\rm in},\infty)$, 
and $\underline{\mathcal{E}} = \underline{\mathcal{E}}_{-} \cup \underline{\mathcal{E}}_{+}$. 
We use a diffeomorphism $\varphi(t,\cdot) : \underline{\mathcal{E}} \to \mathcal{E}(t)$ and put 
$\zeta_{\rm e} = Z_{\rm e}\circ\varphi$, $h_{\rm e} = H_{\rm e}\circ\varphi$, 
$q_{\rm e} = Q_{\rm e}\circ\varphi$, and $\zeta_{\rm i} = Z_{\rm i}\circ\varphi$. 
Such a diffeomorphism $\varphi$ can be constructed as in \eqref{diffeo2}, that is, 
\begin{equation}\label{diffeo3}
\varphi(t,x)=
\begin{cases}
x + \psi(\frac{x-\underline{x}_{-}^{\rm in}}{\varepsilon})(x_{-}(t)-\underline{x}_{-}^{\rm in})
 & \mbox{for}\quad x\in \underline{E}_{-}, \\
x + \psi(\frac{x-\underline{x}_{+}^{\rm in}}{\varepsilon})(x_{+}(t)-\underline{x}_{+}^{\rm in})
 & \mbox{for}\quad x\in \underline{E}_{+}, 
\end{cases}
\end{equation}
with an appropriate choice of $\varepsilon=\varepsilon_0$ and a cut-off function $\psi \in C_0^\infty(\R)$ 
satisfying $\psi(x)=1$ for $|x| \leq 1$. 
As before, we will use the notation $\dx^\varphi$ and $\dt^\varphi$ which were defined by \eqref{dtphi}. 
Now, the problem under consideration is reduced to 
\begin{equation}\label{teqext}
\begin{cases}
\dt^\varphi \zeta_{\rm e} + \dx^\varphi q_{\rm e} = 0 & \mbox{in}\quad \underline{\mathcal{E}}, \\
\dt^\varphi q_{\rm e} + 2 \frac{q_{\rm e}}{h_{\rm e}}\dx^\varphi q_{\rm e}
 + \Bigl( \mathtt{g}h_{\rm e} - \frac{q_{\rm e}^2}{h_{\rm e}^2} \Bigr)\dx^\varphi \zeta_{\rm e} = 0
 & \mbox{in}\quad \underline{\mathcal{E}}, \\
\zeta_{\rm e} = \zeta_{\rm i}, \quad q_{\rm e} = q_{\rm i} & \mbox{on}\quad \partial\underline{\mathcal{E}},
\end{cases}
\end{equation}
with the interior value $q_i$ of the horizontal water flux given by
\begin{equation}\label{eqqi2}
\dot{q}_{\rm i} = -\frac{1}{\int_{\mathcal{I}(t)}\frac{1}{H_{\rm i}}}
 \jump{ \frac12\frac{q_{\rm i}^2}{H_{\rm i}^2} + \mathtt{g}H_{\rm i} }.
\end{equation}
We impose the initial conditions of the form 
\begin{equation}\label{ICs1}
(\zeta_{\rm e},q_{\rm e})_{\vert_{t=0}} = (\zeta_{\rm e}^{\rm in},q_{\rm e}^{\rm in})
 \quad\mbox{in}\quad \underline{\mathcal{E}}, \qquad {x_{\pm}}_{\vert_{t=0}} = \ux_{\pm}^{\rm in},
 \qquad {q_{\rm i}}_{\vert_{t=0}} = q_{\rm i}^{\rm in}.
\end{equation}

\subsubsection{Local well-posedness}
The equations in \eqref{teqext} can be written in the matrix form 
\[
\dt^\varphi u + A(u) \dx^\varphi u = 0,
\]
where $u=(\zeta_{\rm e},q_{\rm e})^{\rm T}$, so that \eqref{teqext}--\eqref{ICs1} is almost 
the same type as the problem \eqref{ODE}--\eqref{nlfbp} considered in \S \ref{sectext}. 
Therefore, the compatibility conditions for \eqref{teqext}--\eqref{ICs1} can be defined in the 
same way as Definition \ref{defCC2} in \S \ref{sectCC}. 
Here, we calculate $\ux_{\pm,1}^{\rm in} = (\dt x_{\pm})_{\vert_{t=0}}$ in terms of the initial data. 
Differentiating the boundary condition $\zeta_{\rm e}=\zeta_{\rm i}$ with respect to $t$, we have 
$\dt \zeta_{\rm e} = \dt \zeta_{\rm i}$ on $\partial\underline{\mathcal{E}}$, which is equivalent to 
$\dt^\varphi \zeta_{\rm e} + \dot{x}_{\pm}\dx^\varphi \zeta_{\rm e}
 = \dt^\varphi \zeta_{\rm i} + \dot{x}_{\pm}\dx^\varphi \zeta_{\rm i}$ on $\partial\underline{\mathcal{E}}$. 
By using $\dt^\varphi\zeta_{\rm e}=-\dx^\varphi q_{\rm e}$ and $\dt^\varphi\zeta_{\rm i}=0$, we see that 
$(\dx^\varphi \zeta_{\rm e}-\dx^\varphi \zeta_{\rm i})\dot{x}_{\pm} = \dx^\varphi q_{\rm e}$ on 
$\partial\underline{\mathcal{E}}$. 
Therefore, we obtain 
\begin{equation}\label{dtuxin}
\ux_{\pm,1}^{\rm in} = \biggl( \frac{\dx q_{\rm e}^{\rm in}}{\dx \zeta_{\rm e}^{\rm in}
 - \dx Z_{\rm lid}} \biggr)_{\vert_{\partial\underline{\mathcal{E}}_{\pm}}}.
\end{equation}
In view of this and the consideration in \S \ref{sectbseq}, we impose the following assumption on the data.

\begin{assumption}\label{assondata}
The data $(\zeta_{\rm e}^{\rm in},q_{\rm e}^{\rm in})$, $\ux_{\pm}^{\rm in}$, and $Z_{\rm lid}$ 
satisfy the following conditions. 
\begin{enumerate}
\item[{\bf i.}]
$\ux_{-} < \ux_{+}$, 
\item[{\bf ii.}]
$\inf_{x \in I_{\rm f}}( h_0 + Z_{\rm lid}(x)) > 0$, \;
 $\inf_{x \in \underline{\mathcal{E}}}( h_0 + \zeta_{\rm e}^{\rm in}(x)) > 0$,
\item[{\bf iii.}]
$\inf_{x \in \underline{\mathcal{E}}}\bigl( \sqrt{\mathtt{g}(h_0 + \zeta_{\rm e}^{\rm in}(x))}
 - \frac{|q_{\rm e}^{\rm in}(x)|}{h_0 + \zeta_{\rm e}^{\rm in}(x)} \bigr) > 0$,
\item[{\bf iv.}]
$\bigl( \sqrt{ \mathtt{g}(h_0 + \zeta_{\rm e}^{\rm in}) }
 - \bigl| \frac{q_{\rm e}^{\rm in}}{h_0 + \zeta_{\rm e}^{\rm in}} - \ux_{\pm,1}^{\rm in} \bigr|
  \bigr)_{\vert_{\partial\underline{\mathcal{E}}}} > 0$,
\item[{\bf v.}]
$(\dx Z_{\rm lid} - \dx \zeta_{\rm e}^{\rm in})_{\vert_{\partial\underline{\mathcal{E}}}} \ne 0$
\end{enumerate}
\end{assumption}

We can now state one of our main result in this section, which shows the well-posedness of the shallow water 
model with a fixed floating structure on the water surface.

\begin{theorem}\label{theoIBVP8}
Let $m\geq2$ be an integer and $I_{\rm f}$ an open interval. 
If the initial data $(\zeta_{\rm e}^{\rm in},q_{\rm e}^{\rm in}) \in H^m(\underline{\mathcal{E}})$, 
$\ux_{\pm}^{\rm in} \in I_{\rm f}$, 
$q_{\rm i}^{\rm in} \in \R$, and $Z_{\rm lid} \in W^{m,\infty}(I_{\rm f})$ satisfy the conditions in 
Assumption \ref{assondata}, where $\ux_{\pm,1}^{\rm in}$ is defined by \eqref{dtuxin}, 
and the compatibility conditions up to order $m-1$, 
then there exist $T>0$ and a unique solution $(\zeta_{\rm e},q_{\rm e},x_{\pm},q_{\rm i})$ to 
\eqref{teqext}--\eqref{ICs1} with $\varphi$ given by \eqref{diffeo3} in the class 
$\zeta_{\rm e},q_{\rm e} \in 
\cap_{j=0}^{m-1}C^j([0,T];H^{m-j}(\underline{\mathcal{E}}))$, $x_{\pm} \in H^m(0,T)$, and 
$q_{\rm i} \in H^{m+1}(0,T)$. 
\end{theorem}

\begin{proof}
Given $q_{\rm i} \in W^{m,\infty}(0,T)$, \eqref{teqext} forms the same type problem in each exterior regions 
$\underline{\mathcal{E}}_{-}$ and $\underline{\mathcal{E}}_{+}$ as the problem considered in \S \ref{sectVCm2}, 
so that we can apply Theorem \ref{theoIBVP5} to show the existence of the solution 
$(\zeta_{\rm e},q_{\rm e},x_{\pm})$ to \eqref{teqext} under the initial conditions in \eqref{ICs1} 
satisfying $x_{\pm} \in H^m(0,T_1)$ for some $T_1 \in (0,T]$. 
Conversely, given $x_{\pm} \in H^m(0,T)$, we can easily show the existence of the solution 
$q_{\rm i} \in H^{m+1}(0,T_1)$ to \eqref{eqqi2} under the initial condition in \eqref{ICs1} 
for some $T_1 \in (0,T]$. 
Iterating this procedure as in the proof of Theorem \ref{theoIBVP6} we can construct a sequence of approximate 
solutions, which converges to the desired solution. 
\end{proof}

\subsection{The case of a floating body with a prescribed motion}\label{sectprescfloat}
Since the floating body is allowed only to a solid motion, its motion is completely determined by 
$(x_G(t),z_G(t))$ the coordinates of the center of mass and $\theta(t)$ the rotational angle of the body. 
Without loss of generality, we have $\theta_{\vert_{t=0}}=0$. 
Suppose that the underside of the floating body is initially parameterized by $Z_{\rm lid}(x)$ on an open 
interval $I_{\rm f}$, that is, ${Z_{\rm i}}_{\vert_{t=0}}=Z_{\rm lid}$. 
Consider a point of the underside of the body and denote the coordinates of the point at $t=0$ by $(X,Z)$. 
Let the coordinates of the point at time $t$ be $(x,z)$. 
Then, it holds that 
\[
Z = Z_{\rm lid}(X), \qquad z = Z_{\rm i}(t,x),
\]
and that 
\[
\begin{pmatrix} x - x_G(t) \\ z - z_G(t) \end{pmatrix}
=
\begin{pmatrix}
 \cos\theta(t) & -\sin\theta(t) \\
 \sin\theta(t) &  \cos\theta(t)
\end{pmatrix}
\begin{pmatrix} X - x_G(0) \\ Z - z_G(0) \end{pmatrix}.
\]
Therefore, we obtain 
\begin{align}\label{expression1}
& (Z_{\rm i}(t,x) - z_G(t))\cos\theta(t) - (x - x_G(t))\sin\theta(t) + z_G(0) \\
&= Z_{\rm lid}\bigl( (x - x_G(t))\cos\theta(t) + (Z_{\rm i}(t,x) - z_G(t))\sin\theta(t) + x_G(0) \bigr).
\nonumber
\end{align}
This is the equation for the motion of the body and gives an expression of $Z_{\rm i}$ implicitly 
in terms of $x_G,z_G,\theta$, and $Z_{\rm lid}$.

\subsubsection{Reformulation of the equations}\label{sectreform2}
Proceeding as in \S \ref{sectreform}, it is possible to reformulate the equations in compact form. 
Due to the various degrees of freedom of the solid, 
the computations are a bit technical and are postponed to Appendix \ref{appreform}. 
It is shown there that the surface elevation and the horizontal water flux
in the interior region are given by 
$$
\begin{cases}
Z_{\rm i}(t,x) = \psi_{\rm lid}\bigl( x-x_G(t),\theta(t) \bigr) + z_G(t),\\
Q_{\rm i}(t,x) = \begin{pmatrix} \mathbf{U}_G(t) \\ \omega(t) \end{pmatrix} \cdot 
 \mathbf{T}( \mathbf{r}_G(t,x) ) + \overline{q}_{\rm i}(t),
 \end{cases}
$$
for some smooth enough function $\psi_{\rm lid}$ and some function $\overline{q}_{\rm i}(t)$ of $t$ 
solving an ODE of the form 
\[
\dt \overline{q}_{\rm i}
 = F(\overline{q}_{\rm i},x_G,z_G,\theta,\mathbf{U}_G,\omega,\dt\mathbf{U}_G,\dt\omega,x_{-},x_{+})
\]
with $F$ in the class $W^{m,\infty}$ under the assumption $Z_{\rm lid} \in W^{m,\infty}(I_{\rm f})$. 
As in the previous section, we use the same diffeomorphism 
$\varphi(t,\cdot) : \underline{\mathcal{E}} \to \mathcal{E}(t)$ defined by \eqref{diffeo3} to transform 
the equations in exterior region \eqref{eqext} and put 
$\zeta_{\rm e} = Z_{\rm e}\circ\varphi$, $h_{\rm e} = H_{\rm e}\circ\varphi$, 
$q_{\rm e} = Q_{\rm e}\circ\varphi$, $\zeta_{\rm i} = Z_{\rm i}\circ\varphi$, and 
$q_{\rm i} = Q_{\rm i}\circ\varphi$. 
Now, the problem under consideration is reduced to 
\begin{equation}\label{teqext2}
\begin{cases}
\dt^\varphi \zeta_{\rm e} + \dx^\varphi q_{\rm e} = 0 & \mbox{in}\quad \underline{\mathcal{E}}, \\
\dt^\varphi q_{\rm e} + 2 \frac{q_{\rm e}}{h_{\rm e}}\dx^\varphi q_{\rm e}
 + \Bigl( \mathtt{g}h_{\rm e} - \frac{q_{\rm e}^2}{h_{\rm e}^2} \Bigr)\dx^\varphi \zeta_{\rm e} = 0
 & \mbox{in}\quad \underline{\mathcal{E}}, \\
\zeta_{\rm e} = \zeta_{\rm i}, \quad q_{\rm e} = q_{\rm i} & \mbox{on}\quad \partial\underline{\mathcal{E}},
\end{cases}
\end{equation}
and 
\begin{equation}\label{eqqi3}
\dt \overline{q}_{\rm i}
 = F(\overline{q}_{\rm i},x_G,z_G,\theta,\mathbf{U}_G,\omega,\dt\mathbf{U}_G,\dt\omega,x_{-},x_{+}). 
\end{equation}
We also impose the initial conditions of the form 
\begin{equation}\label{ICs2}
(\zeta_{\rm e},q_{\rm e})_{\vert_{t=0}} = (\zeta_{\rm e}^{\rm in},q_{\rm e}^{\rm in})
 \quad\mbox{in}\quad \underline{\mathcal{E}}, \qquad {x_{\pm}}_{\vert_{t=0}} = \ux_{\pm}^{\rm in},
 \qquad {\overline{q}_{\rm i}}_{\vert_{t=0}} = \overline{q}_{\rm i}^{\rm in}. 
\end{equation}

\subsubsection{Local well-posedness}
\eqref{teqext2}--\eqref{ICs2} is again almost the same type as the problem \eqref{ODE}--\eqref{nlfbp} 
considered in \S \ref{sectext}. 
Therefore, the compatibility conditions for \eqref{teqext2}--\eqref{ICs2} can be defined in the 
same way as Definition \ref{defCC2} in \S \ref{sectCC}. 
Here, we calculate $\ux_{\pm,1}^{\rm in} = (\dt x_{\pm})_{\vert_{t=0}}$ in terms of the initial data. 
Differentiating the boundary condition $Z_{\rm e}(t,x_{\pm}(t))=Z_{\rm i}(t,x_{\pm}(t))$ with respect to 
$t$ and using the equation $\dt Z_{\rm e} + \dx Q_{\rm e}=0$, we obtain 
$(\dx Z_{\rm e} - \dx Z_{\rm i})_{\vert_{\partial\underline{\mathcal{E}}_{\pm}}} \dt x_{\pm}
 = (\dx Q_{\rm e} + \dt Z_{\rm i})_{\vert_{\partial\underline{\mathcal{E}}_{\pm}}}$, 
so that 
\begin{equation}\label{dtuxin2}
\ux_{\pm,1}^{\rm in} = \biggl( 
 \frac{Z_{\rm i,1}^{\rm in} + \dx q_{\rm e}^{\rm in}}{\dx \zeta_{\rm e}^{\rm in} - \dx Z_{\rm lid}}
  \biggr)_{x=x_{\pm}},
\end{equation}
where $Z_{\rm i,1}^{\rm in}=(\dt Z_{\rm i})_{\vert_{t=0}}$ is given by 
\[
Z_{\rm i,1}^{\rm in}(x) = \biggl( \mathbf{U}_{G}^{\rm in} + \omega^{\rm in}
 \begin{pmatrix} Z_{\rm lid}(x) - z_G^{\rm in} \\ -(x-x_G^{\rm in}) \end{pmatrix} \biggr) \cdot
 \begin{pmatrix} -\dx Z_{\rm lid}(x) \\ 1 \end{pmatrix}.
\]
with $(x_G^{\rm in},z_G^{\rm in},\mathbf{U}_{G}^{\rm in},\omega^{\rm in})
 = (x_G,z_G,\mathbf{U}_{G},\omega)_{\vert_{t=0}}$. 
Here, we used \eqref{zi2}. 
We can now state one of our main result in this section, which shows the well-posedness of the shallow water 
model with a floating body on the water surface whose motion is prescribed.

\begin{theorem}\label{theoIBVP9}
Let $m\geq2$ be an integer and $I_{\rm f}$ an open interval. 
If the data $(\zeta_{\rm e}^{\rm in},q_{\rm e}^{\rm in}) \in H^m(\underline{\mathcal{E}})$, 
$\ux_{\pm}^{\rm in} \in I_{\rm f}$, $\overline{q}_{\rm i}^{\rm in} \in \R$, $Z_{\rm lid} \in W^{m,\infty}(I_{\rm f})$, 
and $x_G,z_G,\theta \in H^{m+2}(0,T)$ satisfy the conditions in 
Assumption \ref{assondata}, where $\ux_{\pm,1}^{\rm in}$ is defined by \eqref{dtuxin2}, 
and the compatibility conditions up to order $m-1$, 
then there exist $T_1 \in (0,T]$ and a unique solution $(\zeta_{\rm e},q_{\rm e},x_{\pm},\overline{q}_{\rm i})$ to 
\eqref{teqext2}--\eqref{ICs2} with $\varphi$ given by \eqref{diffeo3} in the class 
$\zeta_{\rm e},q_{\rm e} \in 
\cap_{j=0}^{m-1}C^j([0,T_1];H^{m-j}(\underline{\mathcal{E}}))$, $x_{\pm} \in H^m(0,T_1)$, and 
$\overline{q}_{\rm i} \in H^{m+1}(0,T_1)$. 
\end{theorem}

\subsection{The case of a freely floating body}\label{sectfreefloat}
Finally, we consider the case where the floating body moves freely according to the Newton's laws 
under the action of the gravitational force and the pressure from the air and from the water. 
Let $\mathfrak{m}$ and $\mathfrak{i}_0$ be the mass and the inertia coefficient of the body. 
Then, Newton's laws for the conservation of linear and angular momentum have the form 
\begin{equation}\label{Newton's law}
\begin{cases}
\mathfrak{m} \dt \mathbf{U}_G
 = -\mathfrak{mg}\mathbf{e}_z + \int_{\mathcal{I}(t)}( \underline{P}_{\rm i} - P_{\rm atm} )N_{\rm lid}, \\
\mathfrak{i}_0 \dt \omega = - \int_{\mathcal{I}(t)}( \underline{P}_{\rm i} - P_{\rm atm} )
 \mathbf{r}_G^{\perp} \cdot N_{\rm lid},
\end{cases}
\end{equation}
which together with \eqref{expression1} constitute the equations of motion for the floating body.

\subsubsection{Reformulation of the equations}
Proceeding as in \S \ref{sectreform} and \S \ref{sectreform2}, and with the same notations, 
the problem under consideration can be reduced to 
\begin{equation}\label{teqext3}
\begin{cases}
\dt^\varphi \zeta_{\rm e} + \dx^\varphi q_{\rm e} = 0 & \mbox{in}\quad \underline{\mathcal{E}}, \\
\dt^\varphi q_{\rm e} + 2 \frac{q_{\rm e}}{h_{\rm e}}\dx^\varphi q_{\rm e}
 + \Bigl( \mathtt{g}h_{\rm e} - \frac{q_{\rm e}^2}{h_{\rm e}^2} \Bigr)\dx^\varphi \zeta_{\rm e} = 0
 & \mbox{in}\quad \underline{\mathcal{E}}, \\
\zeta_{\rm e} = \zeta_{\rm i}, \quad q_{\rm e} = q_{\rm i} & \mbox{on}\quad \partial\underline{\mathcal{E}},
\end{cases}
\end{equation}
and with $W=(\overline{q}_{\rm i},x_G,z_G,\theta,\mathbf{U}_G,\omega)$ solving an ordinary differential equation 
of the form 
\begin{equation}\label{eqdtW}
\dt W = F(W,x_{-},x_{+})
\end{equation}
with $F$ in the class $W^{m,\infty}$ under the assumption $Z_{\rm lid} \in W^{m,\infty}(I_{\rm f})$ 
(see \eqref{eqqi4bis}--\eqref{eqbodybis} for more precisions). 
The details of this technical reduction, which takes advantage of the so-called added mass effect, 
are postponed to Appendix \ref{appreform2}. 
We also impose the initial conditions of the form 
\begin{equation}\label{ICs3}
\begin{cases}
 (\zeta_{\rm e},q_{\rm e})_{\vert_{t=0}} = (\zeta_{\rm e}^{\rm in},q_{\rm e}^{\rm in})
  \quad\mbox{in}\quad \underline{\mathcal{E}}, \qquad {x_{\pm}}_{\vert_{t=0}} = \ux_{\pm}^{\rm in}, \\
 {\overline{q}_{\rm i}}_{\vert_{t=0}} = \overline{q}_{\rm i}^{\rm in},
  \qquad (x_G,z_G,\theta,\mathbf{U}_G,\omega)_{\vert_{t=0}}
   = (x_G^{\rm in},z_G^{\rm in},0,\mathbf{U}_G^{\rm in},\omega^{\rm in}).
\end{cases}
\end{equation}

\subsubsection{Local well-posedness}
%
Therefore, \eqref{teqext3}--\eqref{ICs3} is again almost the same type as the problem 
\eqref{ODE}--\eqref{nlfbp} considered in \S \ref{sectext}. 
Therefore, the compatibility conditions for \eqref{teqext3}--\eqref{ICs3} can be defined in the 
same way as Definition \ref{defCC2} in \S \ref{sectCC}. 
Moreover, $\ux_{\pm,1}^{\rm in} = (\dt x_{\pm})_{\vert_{t=0}}$ can be given by \eqref{dtuxin2}. 
We can now state one of our main result in this section, which shows the well-posedness of the shallow water 
model with a freely floating body on the water surface.

\begin{theorem}\label{theoIBVP10}
Let $m\geq2$ be an integer and $I_{\rm f}$ an open interval. 
If the data $(\zeta_{\rm e}^{\rm in},q_{\rm e}^{\rm in}) \in H^m(\underline{\mathcal{E}})$, 
$\ux_{\pm}^{\rm in} \in I_{\rm f}$, 
$(q_{\rm i}^{\rm in},x_G^{\rm in},z_G^{\rm in},\mathbf{U}_G^{\rm in},\omega^{\rm in}) \in \R^6$, 
and $Z_{\rm lid} \in W^{m,\infty}(I_{\rm f})$ satisfy the conditions in 
Assumption \ref{assondata}, where $\ux_{\pm,1}^{\rm in}$ is defined by \eqref{dtuxin2}, 
and the compatibility conditions up to order $m-1$, 
then there exist $T>0$ and a unique solution 
$(\zeta_{\rm e},q_{\rm e},x_{\pm},\overline{q}_{\rm i},x_G,z_G,\theta)$ to 
\eqref{teqext3}--\eqref{ICs3} with $\varphi$ given by \eqref{diffeo3} in the class 
$\zeta_{\rm e},q_{\rm e} \in 
\cap_{j=0}^{m-1}C^j([0,T];H^{m-j}(\underline{\mathcal{E}}))$, $x_{\pm} \in H^m(0,T)$, 
$\overline{q}_{\rm i} \in H^{m+1}(0,T)$, and 
$x_G,z_G,\theta \in H^{m+2}(0,T)$. 
\end{theorem}

\section{Several examples of transmission problems}\label{sectsev}
We present here several applications of the results proved in Section \ref{secttransmission} 
on transmission problems. 
The first one, in \S \ref{sectdiscf}, is a transmission problem with a fixed interface: 
it corresponds to a conservation law with a flux which is discontinuous across the interface. 
A typical example of application is given by the propagation of shallow water waves over a step-like 
discontinuous topography. 
The second application, in \S \ref{sectshocks}, is a very classical free boundary transmission problem: 
we show how the issue of the stability of one-dimensional shocks for $2\times2$ conservations laws 
falls in the general framework of \S \ref{secttransmkin}. 
This provides an elementary proof of these results, with an improved regularity threshold. 
The case of classical (Lax) shock is considered in \S \ref{sectLax}, while nonclassical, undercompressive, 
shocks are dealt with in \S \ref{sectunder}.

\subsection{Systems of conservation laws with discontinuous flux}\label{sectdiscf}
Let us consider here a system of two conservation laws, with a flux depending on the position. 
For instance, let us consider a flux $\widetilde{f}$ on $\R^-$, and $f$ on $\R_+$, that is, 
\begin{equation}\label{discconserv}
\begin{cases}
 \dt u + \dx \widetilde{f}(u) = 0 & \mbox{ in }\quad (0,T)\times\R_-, \\
 \dt u + \dx {f}(u) = 0 & \mbox{ in }\quad (0,T)\times\R_+,
\end{cases}
\end{equation}
where $\widetilde{f}: \widetilde{\cU} \to \R^2$ and ${f}: {\cU} \to \R^2$ are smooth mappings 
defined on open subsets $\widetilde{\cU}$ and $\cU$ of $\R^2$. 
In addition, $p$ transmission conditions are given at $x=0$ ($p=1,2,3$), 
\begin{equation}\label{transmappl}
N_p^{\rm r}(t) u_{\vert_{x=+0}} - N_p^{\rm l}(t)u_{\vert_{x=-0}} = \mybf{g}(t),
\end{equation}
where $N_p^{\rm l}$ and $N_p^{\rm r}$ are $p\times 2$ matrices.

\begin{remark}
A natural condition is to impose the continuity of the fluxes at the interface, 
$\widetilde{f}({u^{\rm l}}_{\vert_{x=0}}) = f({u^{\rm r}}_{\vert_{x=0}})$, 
which is a nonlinear transmission condition. 
One can in general use a nonlinear change of variables as in \S \ref{sectapplQL} or \S \ref{sectshocks} 
to reduce to the case of a linear transmission condition. 
\end{remark}

Denoting $\widetilde{A}(u) = \widetilde{f}'(u)$ and $A(u) = f'(u)$, and using the same notations 
as in \S \ref{sectappltransmQL}, the system takes the form \eqref{systQLtransm}, namely, 
\begin{equation}\label{systQLtransmappl}
\begin{cases}
\dt \mybf{u} + \mybf{A}(\mybf{u})\dx \mybf{u} = 0 & \mbox{in}\quad \Omega_T, \\
\mybf{u}_{\vert_{t=0}} = \mybf{u}^{\rm in}(x) & \mbox{on}\quad \R_+, \\
\mybf{N}_{p}(t) \mybf{u}_{\vert_{x=0}} = \mybf{g}(t) & \mbox{on}\quad (0,T),
\end{cases}
\end{equation}
and Theorem \ref{theoIBVP2transm} can therefore be applied.

\begin{example}[Shallow water equations with a discontinuous topography]
Let us consider the shallow water equations with a depth at rest $\widetilde{h}_0$ for $x<0$ and $h_0$ for $x>0$. 
\medskip
\begin{figure}[h]
\begin{center}
\includegraphics[width=0.7\linewidth]{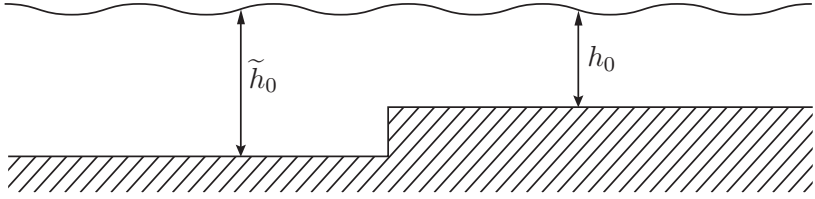}
\end{center}
\setlength{\unitlength}{1pt}
\begin{picture}(0,0)
\put(-65,55){$\widetilde{h}_0$}
\put(63,63){$h_0$}
\end{picture}
\vspace{-5mm}
\caption{Shallow water with a discontinuous topography}\label{FigSW}
\end{figure}
The configuration under study here is described in Figure \ref{FigSW}. 
This is a particular example of \eqref{discconserv} with 
\[
\widetilde{f}(\zeta,q)
 = (q,\frac{1}{\widetilde{h}_0+\zeta}q^2+\frac{1}{2}{\mathtt g}(\widetilde{h}_0+\zeta)^2)^{\rm T}
 \quad\mbox{ and }\quad 
f(\zeta,q) = (q,\frac{1}{{h}_0+\zeta}q^2+\frac{1}{2}{\mathtt g}({h}_0+\zeta)^2)^{\rm T}.
\]
If $\widetilde{\lambda}_\pm(u^{\rm l})
 = \sqrt{\mathtt{g}(\widetilde{h}_0+\zeta^{\rm l})}\pm \frac{q^{\rm l}}{\widetilde{h}_0+\zeta^{\rm l}} >0$ 
and $\lambda_\pm(u^{\rm r}) = \sqrt{\mathtt{g}({h}_0+\zeta^{\rm r})}\pm \frac{q^{\rm r}}{h_0+\zeta^{\rm r}} >0$, 
then one has $p=2$ in Assumption \ref{asshypQLtransm} and two transmission conditions are needed; 
they are naturally given by the continuity of the surface elevation $\zeta$ and of 
the horizontal water flux $q$, that is, 
\[
{u^{\rm l}}_{\vert_{x=0}} = {u^{\rm r}}_{\vert_{x=0}}.
\]
In order to apply Theorem \ref{theoIBVP2transm}, we need to check the invertibility of the Lopatinski\u{\i} matrix 
(third point in Assumption \ref{asshypQLtransm}), which is given here by 
\[
\mybf{L}(\mybf{u}_{\vert_{x=0}}) = \begin{pmatrix} -\widetilde{{\bf e}}_-({u^{\rm l}}_{\vert_{x=0}})
 & {\bf e}_+({u^{\rm r}}_{\vert_{x=0}}) \end{pmatrix},
\]
where $\widetilde{{\bf e}}_-(u)$ denotes a unit eigenvector associated to the eigenvalue $-\widetilde{\lambda}_-(u)$ 
of $\widetilde{A}(u)$ and ${\bf e}_+(u)$ a unit eigenvector associated to the eigenvalue ${\lambda}_+(u)$ 
of ${A}(u)$. 
Using the expression for the eigenvectors provided in Example \ref{ex1}, the invertibility of the Lopatinski\u{\i} 
matrix reduces to the condition 
$| \widetilde{\lambda}_-({u^{\rm l}}_{\vert_{x=0}}) + {\lambda}_+({u^{\rm r}}_{\vert_{x=0}}) | > 0$, 
which is always satisfied. 
One can therefore apply Theorem \ref{theoIBVP2transm}. 
\end{example}

\subsection{Stability of one-dimensional shocks}\label{sectshocks}
Let us consider again a system of two conservation laws 
\begin{equation}\label{conserv}
\dt f_0 (U) + \dx f(U) = 0,
\end{equation}
where $f_0,f: \cU \to \R^2$ are smooth mappings defined on an openset $\cU$ in $\R^2$ and a $2\times2$ matrix 
$f_0'(U)$ is assumed to be invertible. 
The problem of showing the stability  of shocks for \eqref{conserv} consists in finding a 
curve $\ux: [0,T]\to \R$ and $U$ such that $U$ is $C^1$ and solve \eqref{conserv} on 
$\{(t,x)\in (0,T)\times \R\,;\, x< \ux(t)\}$ and $\{(t,x)\in (0,T)\times \R \,;\, x> \ux(t)\}$, 
and satisfy the Rankine--Hugoniot condition 
\begin{equation}\label{RH}
\dot\ux \bigl( f_0(U_{\vert_{x=\ux(t)+0}}) - f_0(U_{\vert_{x=\ux(t)-0}}) \bigr)
 = f(U_{\vert_{x=\ux(t)+0}}) - f(U_{\vert_{x=\ux(t)-0}}).
\end{equation}
This condition can be split into a nonlinear transmission condition 
\[
\Phi( U_{\vert_{x=\ux(t)-0}},U_{\vert_{x=\ux(t)+0}} ) = 0 \quad \mbox{ with } \quad
 \Phi(u^{\rm l},u^{\rm r}) = \bigl[ f(u^{\rm r}) - f(u^{\rm l}) \bigr] \cdot
  \bigl[ f_0(u^{\rm r}) - f_0(u^{\rm l}) \bigr]^\perp
\]
and the evolution equation $\dot\ux = \chi\bigl( U_{\vert_{x=\ux(t)-0}},U_{\vert_{x=\ux(t)+0}} \bigr)$ with 
\begin{equation}\label{defchishock}
\chi(u^{\rm l},u^{\rm r}) = \bigl[ f(u^{\rm r}) - f(u^{\rm l}) \bigr] \cdot
 \frac{f_0(u^{\rm r}) - f_0(u^{\rm l})}{ |f_0(u^{\rm r}) - f_0(u^{\rm l})|^2 }.
\end{equation}
Denoting $A(U) = \bigl( f_0'(U) \bigr)^{-1} f'(U)$, we are therefore led to consider the transmission problem 
\[
\begin{cases}
\dt U + A(U)\dx U = 0 & \mbox{in }\quad  (-\infty,\ux(t)) \quad \mbox{ for } \quad t\in(0,T), \\
\dt U + A(U)\dx U = 0 & \mbox{in }\quad  (\ux(t),+\infty) \quad \mbox{ for } \quad t\in(0,T), \\
U_{\vert_{t=0}} = u^{\rm in}(x) & \mbox{on }\quad  {\mathbb R}, \\
\Phi\bigl(U_{\vert_{x=\ux(t)-0}},U_{\vert_{x=\ux(t)+0}}\bigr) = 0 &\mbox{on } \quad (0,T).
\end{cases}
\]
As for \eqref{bigIBVPQL}, we use the diffeomorphism \eqref{choicediffeo} to recast this transmission problem 
as an initial boundary value problem 
\begin{equation}\label{bigIBVPshock}
\begin{cases}
\dt \mybf{u} + {\bm{\mathcal{A}}}(\mybf{u},\partial \bvarphi)\dx \mybf{u} = 0 & \mbox{ in }\quad \Omega_T, \\
\mybf{u}_{\vert_{t=0}} = \mybf{u}^{\rm in} & \mbox{ on }\quad {\mathbb R}_+, \\
\Phi(\mybf{u}_{\vert_{x=0}}) = 0 & \mbox{ on }\quad (0,T)
\end{cases}
\end{equation}
with $\ux$ given by the resolution of 
\begin{equation}\label{uxeq}
\dot\ux = \chi( \mybf{u}_{\vert_{x=0}} ), \qquad \ux(0) = 0,
\end{equation}
where $\chi$ given by \eqref{defchishock}.

There are several kinds of shock. 
The most famous are the so-called Lax shocks which move at a supersonic speed; 
more precisely, the number of positive eigenvalues for ${\bm{\mathcal{A}}}(\mybf{u},\partial \bvarphi)$ 
in \eqref{bigIBVPshock} is equal to one and one boundary condition is needed; 
it is provided by the condition $\Phi(\mybf{u}_{\vert_{x=0}}) = 0$ in \eqref{bigIBVPshock}. 
There are also undercompressive shocks that travel at a subsonic speed. 
The number of positive eigenvalues for ${\bm{\mathcal{A}}}(\mybf{u},\partial \bvarphi)$ 
in \eqref{bigIBVPshock} is then equal to two and {\it two} boundary conditions are therefore necessary. 
One needs therefore an additional boundary condition to the condition $\Phi(\mybf{u}_{\vert_{x=0}}) = 0$ 
that comes from the Rankine--Hugoniot condition.

\subsubsection{The stability of Lax shocks}\label{sectLax}
As said above, for Lax shocks, the number of positive eigenvalues for 
${\bm{\mathcal{A}}}(\mybf{u},\partial \bvarphi)$ in \eqref{bigIBVPshock} is equal to one; 
this correponds to $p=1$ and condition $b)$ or $c)$ in Assumption \ref{asshypQLFBtransm}. 
The Kreiss--Lopatinski\u{\i} condition in the third point of Assumption \ref{asshypQLFBtransm} is therefore scalar. 
It is explicited in the assumption below for right-going and left-going Lax shocks where for all function $g$ 
defined on $\bcU$, we use the notation 
\[
\jump{g}=g(u^{\rm r})-g(u^{\rm l}). 
\]

\begin{assumption}\label{asshypQLFBshock}
Let $\widetilde{\cU}$ and $\cU$ be open sets in $\R^2$ and put $\boldsymbol{\cU}=\widetilde{\cU}\times\cU$ 
representing a phase space of $\mybf{u}$. 
Let $\widetilde{\cU}_I \subset \widetilde{\cU}$ and $\cU_I \subset \cU$ be also open sets and put 
$\boldsymbol{\cU}_I = \widetilde{\cU}_I\times\cU_I$ representing a phase space of $\mybf{u}_{\vert_{x=0}}$.
The following conditions hold: 
\begin{enumerate}
\setlength{\itemsep}{3pt}
\item[{\bf i.}]
$\mybf{A}(\mybf{u}) = \mbox{\rm diag}(-A(u^{\rm l}),A(u^{\rm r})) \in C^\infty(\boldsymbol{\cU})$ and 
$\Phi,\chi\in C^\infty(\boldsymbol{\cU}_I)$. 
\item[{\bf ii.}]
For any $\mybf{u} = (u^{\rm l},u^{\rm r})^{\rm T}\in\boldsymbol{\cU}$, the matrix $A(u^{\rm l,r})$ 
has eigenvalues $\lambda_+(u^{\rm l,r}) $ and $ -\lambda_-(u^{\rm l,r})$ with $\lambda_\pm(u^{\rm l,r})>0$. 
Moreover, one of the following conditions 
for all $\mybf{u}= (u^{\rm l},u^{\rm r})^{\rm T}\in\boldsymbol{\cU}_I$ holds: 
\begin{enumerate}
\item[-]
Right-going Lax shock
\[
\begin{cases}
\lambda_\pm(u^{\rm l}) \mp \chi(\mybf{u}) >0 \quad\mbox{ and }\quad
 \lambda_+(u^{\rm r}) - \chi(\mybf{u}) < 0, \\
\bigl|\bigl(f_0'(u^{\rm l}){\bf e}_-(u^{\rm l})\bigr) \cdot \jump{f_0}^\perp\bigr| > 0.
 \end{cases}
\]
\item[-] Left-going Lax shock
\[
\begin{cases}
\lambda_-(u^{\rm l}) + \chi(\mybf{u}) < 0 \quad\mbox{ and }\quad 
\lambda_\pm(u^{\rm r}) \mp \chi(\mybf{u}) >0,  \\
\bigl|\bigl(f_0'(u^{\rm r}){\bf e}_+(u^{\rm r})\bigr) \cdot \jump{f_0}^\perp\bigr| > 0.
 \end{cases}
\]
\end{enumerate}
\item[{\bf iii.}]
There exists a $C^\infty$-mapping $\Theta : \boldsymbol{\cU} \to \R^4$ such that it defines a diffeomorphism from 
$\boldsymbol{\cU}$ onto its image and for any $\mybf{u} = (u^{\rm l},u^{\rm r})^{\rm T}\in\boldsymbol{\cU}_I$ 
we have 
\[
\Theta(\mybf{u}) = \bigl( \Phi(\mybf{u}),\chi(\mybf{u}),u^{\rm r}\bigr)^{\rm T}.
\]
\end{enumerate}
\end{assumption}

\begin{remark}\label{remdiffeo}
Up to shrinking $\widetilde{\cU}$ and $\cU$, the third point is always satisfied. 
Indeed, as remarked in \cite{metivier2001}, this follows from the local inversion theorem 
since $\Theta'(\mybf{u})$ is invertible at any point $\mybf{u}$ satisfying $\Phi(\mybf{u})=0$. 
In order to check this point, it is enough to prove that the partial derivative of the mapping 
$\mybf{u} \mapsto (\Phi(\mybf{u}),\chi(\mybf{u}))$ with respect to $u^{\rm l}$ is invertible. 
Denoting by $W(\mybf{u})$ a $2\times2$ matrix defined by 
\[
W(\mybf{u})F = \Bigl(F \cdot \jump{f_0}^\perp, \frac{1}{|\jump{f_0}|^2} F \cdot \jump{f_0} \Bigr)^{\rm T},
\]
this partial derivative is given by the linear mapping 
\begin{align*}
\dot u^{\rm l}\mapsto
& (d_{u^{\rm l}} W(\mybf{u})[\dot u^{\rm l}])\jump{f} - W(\mybf{u})f'(u^{\rm l})\dot u^{\rm l} \\
&= \chi(\bu)(d_{u^{\rm l}} W(\mybf{u})[\dot u^{\rm l}])\jump{f_0}
 - W(\mybf{u})f_0'(u^{\rm l})A(u^{\rm l})\dot u^{\rm l};
\end{align*}
observing by differentiating the identity $W(\mybf{u})\jump{f_0} = (0,1)^{\rm T}$ that 
\[
d_{u^{\rm l}} W(\mybf{u})[\dot u^{\rm l}]\jump{f_0} = W(\mybf{u})f_0'(u^{\rm l})\dot u^{\rm l}, 
\]
the partial derivative can be written as 
\begin{align*}
\dot u^{\rm l}\mapsto & W(\mybf{u})f_0'(u^{\rm l})
 \bigl(\chi(\mybf{u})\mbox{\rm Id} - A(u^{\rm l}) \bigr)\dot u^{\rm l},
\end{align*}
which is invertible by the second point of Assumption \ref{asshypQLFBtransm}. 
\end{remark}

We can now state the following stability result for Lax shocks.

\begin{theorem}\label{theoshock}
Let $m\geq 2$ be an integer. 
Suppose that Assumption \ref{asshypQLFBshock} is satisfied. 
If $\mybf{u}^{\rm in}\in H^m(\R_+)$ takes its values in $\widetilde{{\mathcal K}}_0\times {\mathcal K}_0$ with 
$\widetilde{{\mathcal K}}_0\subset\widetilde{\cU}_0$ and ${\mathcal K}_0\subset\cU_0$ compact and convex sets, 
if $\mybf{u}^{\rm in}(0) \in \boldsymbol{\cU}_I$,
and if it satisfies the compatibility conditions at order $m-1$, 
then there exists $T>0$ and a unique solition $(\mybf{u},\ux)$ to \eqref{bigIBVPshock}--\eqref{uxeq} with 
$\mybf{u} \in \WW^m(T)$ and $\ux\in H^{m+1}(0,T)$, and $\varphi$ given by \eqref{choicediffeo}. 
Moreover, $\mybf{u}_{\vert_{x=0}}\in H^m(0,T)$.
\end{theorem}

\begin{remark}\label{remimp}
The stability of multidimensional shocks was proved in \cite{majda1,majda2,majda3}, 
with improvements in \cite{metivier2001}. 
In space dimension one, this result shows the stability in $\WW^m(T)$ for $m\geq 3$ 
provided that the data is in $H^{m+1/2}(\R_+)$. 
Our proof, which takes advantage of the specificities of the one-dimensional case, is much more elementary 
and provides an improvement of these classical results since we only need $m\geq 2$ 
(and therefore one compatibility condition less) with data in $H^m(\R_+)$ (and therefore no loss of regularity). 
\end{remark}

\begin{proof}
There are two steps in the proof. 
We first transform the problem \eqref{bigIBVPshock} into an initial boundary value problem with a 
{\it linear} boundary condition, and we then prove that Assumption \ref{asshypQLFBtransm} is satisfied 
so that we can conclude with Theorem \ref{theoIBVP4transm}. 
Using the third point of Assumption \ref{asshypQLFBshock}, it is equivalent to solve the initial boundary 
value problem satisfied by $\mybf{v} = \Theta(\mybf{u})$, namely, 
\begin{equation}\label{bigIBVPshocklin}
\begin{cases}
\dt \mybf{v} + {\bm{\mathcal{A}}}^\sharp(\mybf{v},\partial \bvarphi)\dx \mybf{v} = 0
 & \mbox{ in }\quad \Omega_T, \\
\mybf{v}_{\vert_{t=0}} = \mybf{v}^{\rm in} & \mbox{ on }\quad {\mathbb R}_+, \\
{\bf e}_1^\sharp \cdot \mybf{v}_{\vert_{x=0}} = 0 & \mbox{ on }\quad (0,T),
\end{cases}
\end{equation}
with $\ux$ given by the resolution of 
\begin{equation}\label{uxeqb}
\dot\ux = {\bf e}_2^\sharp \cdot \mybf{v}_{\vert_{x=0}}, \qquad \ux(0)=0,
\end{equation}
where $({\bf e}^\sharp_1,{\bf e}^\sharp_2,{\bf e}^\sharp_3,{\bf e}^\sharp_4)$ 
denotes the canonical basis of $\R^4$ and 
\[
\boldsymbol{{\mathcal A}}^\sharp(\mybf{v}, \partial \bvarphi)
 = \bigl( d_{\boldsymbol{v}}\Theta^{-1}(\mybf{v}) \bigr)^{-1}
 {\bm{\mathcal{A}}}(\Theta^{-1}(\mybf{v}),\partial \bvarphi)\bigl( d_{\boldsymbol{v}}\Theta^{-1}(\mybf{v}) \bigr).
\]
In particular, the eigenvalues of $\boldsymbol{{\mathcal A}}^\sharp(\mybf{v}, \partial \bvarphi)$ 
are the same as those of $\boldsymbol{{\mathcal A}}(\mybf{u}, \partial \bvarphi)$ and 
if $\mybf{E}$ is an eigenvector of $\boldsymbol{{\mathcal A}}(\mybf{u}, \partial \bvarphi)$, 
the corresponding eigenvector of $\boldsymbol{{\mathcal A}}^\sharp(\mybf{v}, \partial \bvarphi)$ 
is $\mybf{E}^\sharp = \Theta'(\mybf{u})\mybf{E}$. 
By the second point of Assumption \ref{asshypQLFBshock}, the system \eqref{bigIBVPshocklin} satisfies therefore 
condition $b)$ or $c)$ in Assumption \ref{asshypQLFBtransm} and the Lopatinski\u{\i} matrix reduces to 
a scalar denoted $L^\sharp(\mybf{v}_{\vert_{x=0}})$, 
\[
L^\sharp(\mybf{v}_{\vert_{x=0}}) = {\bf e}^\sharp_1 \cdot \mybf{E}^\sharp_{\rm out}(\mybf{v}_{\vert_{x=0}}),
\]
where $\mybf{E}^\sharp_{\rm out}(\mybf{v})$ is the eigenvector of 
$\boldsymbol{{\mathcal A}}^\sharp(\mybf{v}, \partial \bvarphi)$ associated to its unique positive eigenvalue. 
From the discussion above, one has 
$\mybf{E}^\sharp_{\rm out}(\mybf{v}) = \Theta'(\mybf{u})\mybf{E}_{\rm out}(\mybf{u})$, 
where $\mybf{E}_{\rm out}(\mybf{u})$ is the eigenvector associated to the unique positive eigenvalue of 
$\boldsymbol{{\mathcal A}}(\mybf{u}, \partial \bvarphi)$. 
We have therefore 
\begin{align*}
L^\sharp(\mybf{v}) 
&= \Theta'(\mybf{u})^{\rm T}{\bf e}^\sharp_1 \cdot \mybf{E}_{\rm out}(\mybf{u}), \\
&= \nabla_{\boldsymbol{u}}\Phi(\mybf{u}) \cdot \mybf{E}_{\rm out}(\mybf{u}).
\end{align*}
Let us assume for instance that the first condition holds in the second point of 
Assumption \ref{asshypQLFBshock} (the adaptation if the second condition holds is straightforward). 
One then has $\mybf{E}_{\rm out}(\mybf{u}) = \left(\begin{array}{c}{\bf e}_-(u^{\rm l}) \\ 0\end{array}\right)$ 
(where as usual ${\bf e}_-(u^{\rm l})$ is the eigenvector associate to the eigenvalue $-\lambda_-(u^{\rm l})$ 
of $A(u^{\rm l})$) and, with computations similar to those performed in Remark \ref{remdiffeo}, we obtain 
\begin{align*}
L^\sharp(\mybf{v}) 
&= \jump{f_0}^\perp\cdot f_0'(u^{\rm l})(\chi(\mybf{u})\mbox{Id}-A(u^{\rm l})){\bf e}_-(u^{\rm l}) \\
&=  (\chi(\mybf{u}) + \lambda_-(u^{\rm l})) \jump{f_0}^\perp \cdot f_0'(u^{\rm l}){\bf e}_-(u^{\rm l});
\end{align*}
the second point of the assumption implies that this quantity is nonzero, and we can therefore conclude with Theorem \ref{theoIBVP4transm}.
\end{proof}

\subsubsection{The stability of undercompressive shocks}\label{sectunder}
In some applications, one can encounter shock waves that violate Lax's conditions. 
This is for instance the case for magnetohydrodynamics, or phase transitions in elastodynamics, 
or van der Waals fluids. 
In the particular case of {\it undercompressive shocks}, Lax's conditions are violated but 
condition $a)$ is satisfied in Assumption \ref{asshypQLFBtransm}. 
This means that $p=2$ (the number of positive eigenvalues for ${\bm{\mathcal{A}}}(\boldsymbol{u},\partial \bvarphi)$ 
in \eqref{bigIBVPshock} is equal to two) 
and therefore that the system of equations \eqref{bigIBVPshock}--\eqref{uxeq} is now {\it underdeterminated}. 
An additional boundary condition is therefore necessary.

This additional condition requires some additional modeling and depends on the context: 
it often comes from considerations based on the theory of viscosity-capillarity, 
see for instance \cite{slemrod1983,truskinovsky1994} for isothermal phase transitions 
or \cite{abeyaratne1991} for elastic rods. 
If such an additional boundary condition is provided and if it satisfies an appropriate stability condition as in 
\S \ref{secttransmkin} then the undercompressive shocks are stable. 
This extension of Majda's work on Lax's shock was proposed in \cite{freistuhler1998,freistuhler1998}, 
and studied in \cite{colombo1999} in the one-dimensional case. 
The extension to several dimensions was performed in \cite{benzoni1998} 
(derivation of the Kreiss--Lopatinski\u{\i} condition), \cite{benzoni1999} (linear estimates) and 
\cite{coulombel2003} (nonlinear estimates). 
We show here that the framework developed in \S \ref{secttransmkin} can be used to improve these results 
for the stability of one-dimensional undercompressive shocks.

We shall consider here an general framework where the additional boundary conditions we use to complement 
\eqref{bigIBVPshock}--\eqref{uxeq} is of the form 
\begin{equation}\label{eqPsi}
\Psi(\mybf{u}_{\vert_{x=0}}) = 0,
\end{equation}
where $\Psi$ is a smooth function satisfiying the assumption below. 
Note in particular that for undercompressive shocks, the Lopatinski\u{\i} matrix in the third point of 
Assumption \ref{asshypQLFBtransm} is a $2\times2$ matrix; 
its invertibility corresponds to the condition stated in the second point of the assumption below.

\begin{assumption}\label{asshypQLFBshockunder}
Let $\widetilde{\cU}$ and $\cU$ be open sets in $\R^2$ and put $\boldsymbol{\cU}=\widetilde{\cU}\times\cU$ 
representing a phase space of $\mybf{u}$. 
Let $\widetilde{\cU}_I \subset \widetilde{\cU}$ and $\cU_I \subset \cU$ be also open sets and put 
$\boldsymbol{\cU}_I = \widetilde{\cU}_I\times\cU_I$ representing a phase space of $\mybf{u}_{\vert_{x=0}}$.
The following conditions hold: 
\begin{enumerate}
\setlength{\itemsep}{3pt}
\item[{\bf i.}]
$\mybf{A}(\mybf{u}) = \mbox{\rm diag}(-A(u^{\rm l}),A(u^{\rm r})) \in C^\infty(\boldsymbol{\cU})$ and 
$\Phi,\Psi,\chi\in C^\infty(\boldsymbol{\cU}_I)$. 
\item[{\bf ii.}]
For any $\mybf{u} = (u^{\rm l},u^{\rm r})^{\rm T} \in \boldsymbol{\cU}$, the matrix $A(u^{\rm l,r})$ 
has eigenvalues $\lambda_+(u^{\rm l,r}) $ and $ -\lambda_-(u^{\rm l,r})$ with $\lambda_\pm(u^{\rm l,r})>0$. 
Moreover, for any $\mybf{u} = (u^{\rm l},u^{\rm r})^{\rm T} \in \boldsymbol{\cU}_I$
the following conditions hold:
\[
\lambda_\pm(u^{\rm l})\mp\chi(\boldsymbol{u}) > 0 \quad\mbox{ and }\quad 
 \lambda_\pm(u^{\rm r})\mp\chi(\boldsymbol{u}) > 0
\]
and the Lopatinski\u{\i} matrix 
\[
\left(
\begin{array}{cc}
 \bigl( \chi(\mybf{u})+\lambda_-(u^{\rm l}) \bigr)
  \bigl( f_0'(u^{\rm l}){\bf e}_-(u^{\rm l}) \bigr) \cdot \jump{f_0}^\perp
  & -\bigl( \chi(\mybf{u})-\lambda_+(u^{\rm r}) \bigr)
   \bigl( f_0'(u^{\rm r}){\bf e}_+(u^{\rm r}) \bigr) \cdot \jump{f_0}^\perp \\
 \nabla_{u^{\rm l}}\Psi \cdot {\bf e}_-(u^{\rm l}) & \nabla_{u^{\rm r}}\Psi \cdot {\bf e}_+(u^{\rm r})
\end{array}
\right)
\]
is invertible. 
\item[{\bf iii.}]
There exists a $C^\infty$-mapping $\Theta: \bcU\to \R^4$ such that it defines a diffeomorphism from $\bcU$ onto 
its image and for all $\mybf{u} = (u^{\rm l},u^{\rm r})^{\rm T} \in \boldsymbol{\cU}_I$ we have 
\[
\Theta(\mybf{u}) = \big(\Phi(\mybf{u}),\Psi(\mybf{u}),\theta(\mybf{u}) \big)^{\rm T}
\]
with a mapping $\theta: \bcU\to \R^2$. 
\end{enumerate}
\end{assumption}

\begin{remark}\label{remdiffeo2}
Up to shrinking $\widetilde{\cU}$ and $\cU$, the third point is always satisfied. 
Indeed, the second point of the assumption shows that $d_{\boldsymbol{u}}(\Phi,\Psi)$ has rank $2$ so that 
$\mybf{u}\mapsto (\Phi(\mybf{u}),\Psi(\mybf{u}))$ can  be completed to form a local diffeomorphism. 
\end{remark}

An easy adaptation of the proof of Theorem \ref{theoshock} yields the following stability result 
for undercompressive shocks. 
The same improvements as those described in Remark \ref{remimp} hold with respect the result obtained 
by considering the one-dimensional case in \cite{coulombel2003}.

\begin{theorem}\label{theoshockunder}
Let $m\geq 2$ be an integer. 
Suppose that Assumption \ref{asshypQLFBshockunder} is satisfied. 
If $\mybf{u}^{\rm in} \in H^m(\R_+)$ takes its values in $\widetilde{{\mathcal K}}_0\times {\mathcal K}_0$ 
with  $\widetilde{{\mathcal K}}_0\subset\widetilde{\cU}_0$ and ${\mathcal K}_0\subset\cU_0$ compact and convex sets, 
if $\mybf{u}^{\rm in}(0) \in \bcU_I$, 
and if it satisfies the compatibility conditions at order $m-1$, 
then there exists $T>0$ and a unique solition $(\mybf{u},\ux)$ to \eqref{bigIBVPshock}--\eqref{uxeq} 
complemented by \eqref{eqPsi}, with $\mybf{u} \in \WW^m(T)$ and $\ux\in H^{m+1}(0,T)$, 
and $\varphi$ given by \eqref{choicediffeo}. Moreover, $\boldsymbol{u}_{\vert_{x=0}}\in H^m(0,T)$. 
\end{theorem}

\appendix
\section{Reformulation of the equations of motion in the case an object with prescribed motion}\label{appreform}
We will begin to show that \eqref{expression1} determines $Z_{\rm i}(t,x)$ under the assumptions 
that the center of mass is close to its initial position, that the rotational angle is small, and 
that $Z_{\rm lid} \in W^{m,\infty}(I_{\rm f})$. 
By extending $Z_{\rm lid}$ outside of the interval $I_{\rm f}$ appropriately, we can assume that 
$Z_{\rm lid} \in W^{m,\infty}(\R)$. 
Then, we have the following lemma.

\begin{lemma}\label{lemgrap}
Let $m\geq1$ be an integer and suppose that $Z_{\rm lid} \in C^1 \cap W^{m,\infty}(\R)$. 
There exist $\theta_0 \in (0,\frac{\pi}{2})$ and 
$\psi_{\rm lid} \in C^1 \cap W_{\rm loc}^{m,\infty}(\R\times[-\delta_0,\delta_0])$ such that 
as long as $|\theta(t)| \leq \theta_0$ we can solve \eqref{expression1} for $Z_{\rm i}(t,x)$ 
uniquely in the form 
\begin{equation}\label{zi}
Z_{\rm i}(t,x) = \psi_{\rm lid}\bigl( x-x_G(t),\theta(t) \bigr) + z_G(t).
\end{equation}
\end{lemma}

\begin{proof}
We consider an auxiliary function 
\[
\Psi(z,x,\theta) = z\cos\theta - x\sin\theta  + z_G(0) - Z_{\rm lid}(x\cos\theta + z\sin\theta + x_G(0)),
\]
which belongs to the class $C^1 \cap W_{\rm loc}^{m,\infty}(\R^3)$. 
For $\theta \in (-\frac{\pi}{2},\frac{\pi}{2})$, we see that 
\begin{align*}
\dz \Psi(z,x,\theta) &= \cos\theta - (\dx Z_{\rm lid})(x\cos\theta + z\sin\theta + x_G(0)) \sin\theta \\
&\geq (1-\|\dx Z_{\rm lid}\|_{L^{\infty}(\mathbb{R})}\tan|\theta|)\cos\theta.
\end{align*}
In view of this we take $\theta_0 \in (0,\frac{\pi}{2})$ such that 
$\|\dx Z_{\rm lid}\|_{L^{\infty}(\mathbb{R})}\tan\theta_0 < 1$. 
Then, it holds that $\dz \Psi(z,x,\theta)>0$ as long as $|\theta| \leq \theta_0$. 
Therefore, the implicit function theorem gives the desired result. 
\end{proof}

We proceed to solve the equations in the interior region \eqref{eqint}. 
Let $N_{\rm i}$ be a normal vector on the underside of the floating body and $\mathbf{r}_G(t,x)$ 
a position vector of the point on the underside of the body relative to the center of mass, that is, 
\[
N_{\rm i}(t,x) = \begin{pmatrix} -\dx Z_{\rm i}(t,x) \\ 1 \end{pmatrix}, \quad
\mathbf{r}_G(t,x) = \begin{pmatrix} x - x_G(t) \\ Z_{\rm i}(t,x) - z_G(t) \end{pmatrix}.
\]
Here, we have $\dx \mathbf{r}_G^{\perp} = N_{\rm i}$. 
Denoting 
\[
\mathbf{T}( \mathbf{r}_G )
 = \begin{pmatrix} -\mathbf{r}_G^{\perp} \\ \frac12|\mathbf{r}_G|^2 \end{pmatrix},
\]
we have 
\begin{equation}\label{formula1}
\dx \mathbf{T}( \mathbf{r}_G )
 = \begin{pmatrix} -N_{\rm i} \\ \mathbf{r}_G^{\perp} \cdot N_{\rm i} \end{pmatrix}.
\end{equation}
Let $\mathbf{U}_G(t) = (u_G(t),w_G(t))^{\rm T}$ and $\omega(t)$ be the velocity of the center of mass 
and the angular velocity of the body, respectively, that is, 
$u_G=\dt x_G$, $w_G=\dt z_G$, and $\omega=\dt\theta$. 
Differentiating \eqref{expression1} with respect to $t$ and $x$, we see that 
\begin{equation}\label{zi2}
\dt Z_{\rm i} = ( \mathbf{U}_G - \omega \mathbf{r}_G^{\perp} )\cdot N_{\rm i}
 = -\dx \biggl( \begin{pmatrix} \mathbf{U}_G \\ \omega \end{pmatrix} \cdot 
 \mathbf{T}( \mathbf{r}_G ) \biggr).
\end{equation}
which together with the continuity equation in \eqref{eqint} yields that there exists a 
function $\overline{q}_{\rm i}(t)$ of $t$ such that 
\begin{equation}\label{qi}
Q_{\rm i}(t,x) = \begin{pmatrix} \mathbf{U}_G(t) \\ \omega(t) \end{pmatrix} \cdot 
 \mathbf{T}( \mathbf{r}_G(t,x) ) + \overline{q}_{\rm i}(t).
\end{equation}
Plugging this into the momentum equation in \eqref{eqint}, we see that $\underline{P}_{\rm i}$ 
satisfies a simple boundary value problem 
\begin{equation}\label{bvpp2}
\begin{cases}
 \dx\underline{P}_{\rm i} = -\frac{\rho}{H_{\rm i}} ( \dt \overline{q}_{\rm i} + F^{\rm I} + F^{\rm II} + F^{\rm III} )
  & \mbox{in}\quad \mathcal{I}(t), \\
 \underline{P}_{\rm i} = P_{\rm atm} & \mbox{on}\quad \Gamma(t),
\end{cases}
\end{equation}
where 
\[
\begin{cases}
F^{\rm I}(t,x) = \dx \Bigl( \frac{Q_{\rm i}(t,x)^2}{H_{\rm i}(t,x)} + \frac12 \mathtt{g} H_{\rm i}^2 \Bigr), \\
F^{\rm II}(t,x) = \begin{pmatrix} \dt \mathbf{U}_G(t) \\ \dt \omega(t) \end{pmatrix} \cdot 
 \mathbf{T}( \mathbf{r}_G(t,x) ), \\
F^{\rm III}(t,x) = \begin{pmatrix} \mathbf{U}_G(t) \\ \omega(t) \end{pmatrix} \cdot 
 \dt \mathbf{T}( \mathbf{r}_G(t,x) ).
\end{cases}
\]
In view of 
\[
\dt \mathbf{T}( \mathbf{r}_G(t,x) ) = M( \mathbf{r}_G(t,x), N_{\rm lid}(t,x) )
 \begin{pmatrix} \mathbf{U}_G(t) \\ \omega(t) \end{pmatrix},
\]
where 
\[
M( \mathbf{r}_G(t,x), N_{\rm lid}(t,x) ) = 
\begin{pmatrix}
 \mathbf{e}_x \cdot N_{\rm lid} & 0 & -\mathbf{r}_G^{\perp} \cdot N_{\rm lid} \\
 1 & 0 & 0 \\
 -\mathbf{r}_G^{\perp}\cdot N_{\rm lid} & 0 & -(\mathbf{e}_z \cdot \mathbf{r}_G)
  (\mathbf{r}_G^{\perp} \cdot N_{\rm lid})
\end{pmatrix}
\]
with $\mathbf{e}_x=(1,0)^{\rm T}$ and $\mathbf{e}_z=(0,1)^{\rm T}$, 
we can rewrite $F^{\rm I}$ and $F^{\rm III}$ as 
\begin{equation}\label{F^I}
\begin{cases}
\displaystyle
F^{\rm I} = \overline{q}_{\rm i}^2 \dx\biggl( \frac{1}{H_{\rm i}} \biggr)
 + 2\overline{q}_{\rm i}\begin{pmatrix} \mathbf{U}_G \\ \omega \end{pmatrix} \cdot
  \dx \biggl( \frac{\mathbf{T}(\mathbf{r}_G)}{H_{\rm i}} \biggr) \\
\qquad\;
\displaystyle
 + \begin{pmatrix} \mathbf{U}_G \\ \omega \end{pmatrix} \cdot
  \biggl( \dx \biggl( \frac{\mathbf{T}(\mathbf{r}_G) \otimes \mathbf{T}(\mathbf{r}_G)}{H_{\rm i}} \biggr) \biggr)
   \begin{pmatrix} \mathbf{U}_G \\ \omega \end{pmatrix}
   + \frac12 \mathtt{g} \dx( H_{\rm i}^2 ), \\
F^{\rm II} = \begin{pmatrix} \dt \mathbf{U}_G \\ \dt \omega \end{pmatrix} \cdot 
 \mathbf{T}( \mathbf{r}_G ), \quad
F^{\rm III} = \begin{pmatrix} \mathbf{U}_G \\ \omega \end{pmatrix} \cdot 
 M( \mathbf{r}_G, N_{\rm lid} ) \begin{pmatrix} \mathbf{U}_G \\ \omega \end{pmatrix}.
\end{cases}
\end{equation}

\begin{notation}
For a function $F=F(t,x)$, we put 
$\langle F \rangle = \frac{1}{\int_{\mathcal{I}(t)}\frac{1}{H_{\rm i}}}\int_{\mathcal{I}(t)}\frac{F}{H_{\rm i}}$ 
and $F^* = F - \langle F \rangle$. 
\end{notation}

We see easily that the boundary value problem \eqref{bvpp2} for $\underline{P}_{\rm i}$ is solvable 
if and only if $\overline{q}_{\rm i}$ saisfies 
\begin{align*}
\dt \overline{q}_{\rm i}
&= - ( \langle F^{\rm I} \rangle + \langle F^{\rm II} \rangle + \langle F^{\rm III} \rangle ) \\
&= -\overline{q}_{\rm i}^2 \left\langle \dx\biggl( \frac{1}{H_{\rm i}} \biggr) \right\rangle
 - 2\overline{q}_{\rm i} \begin{pmatrix} \mathbf{U}_G \\ \omega \end{pmatrix} \cdot
  \left\langle \dx \biggl( \frac{\mathbf{T}(\mathbf{r}_G)}{H_{\rm i}} \biggr)\right\rangle \\
&\quad\;
 - \begin{pmatrix} \mathbf{U}_G \\ \omega \end{pmatrix} \cdot
  \biggl\langle \dx \biggl( \frac{\mathbf{T}(\mathbf{r}_G) \otimes \mathbf{T}(\mathbf{r}_G)}{H_{\rm i}}
   \biggr) \biggr\rangle
   \begin{pmatrix} \mathbf{U}_G \\ \omega \end{pmatrix}
   - \frac12 \mathtt{g} \langle \dx ( H_{\rm i}^2 ) \rangle \\
&\quad\;
 - \begin{pmatrix} \dt \mathbf{U}_G \\ \dt \omega \end{pmatrix} \cdot
  \langle \mathbf{T}( \mathbf{r}_G ) \rangle
 - \begin{pmatrix} \mathbf{U}_G \\ \omega \end{pmatrix} \cdot 
  \langle M( \mathbf{r}_G, N_{\rm lid} ) \rangle \begin{pmatrix} \mathbf{U}_G \\ \omega \end{pmatrix}.
\end{align*}
Thanks of Lemma \ref{lemgrap}, this can be written in the form 
\[
\dt \overline{q}_{\rm i}
 = F(\overline{q}_{\rm i},x_G,z_G,\theta,\mathbf{U}_G,\omega,\dt\mathbf{U}_G,\dt\omega,x_{-},x_{+})
\]
with $F$ in the class $W^{m,\infty}$ under the assumption $Z_{\rm lid} \in W^{m,\infty}(I_{\rm f})$. 
As in the previous section, we use the same diffeomorphism 
$\varphi(t,\cdot) : \underline{\mathcal{E}} \to \mathcal{E}(t)$ defined by \eqref{diffeo3} to transform 
the equations in exterior region \eqref{eqext} and put 
$\zeta_{\rm e} = Z_{\rm e}\circ\varphi$, $h_{\rm e} = H_{\rm e}\circ\varphi$, 
$q_{\rm e} = Q_{\rm e}\circ\varphi$, $\zeta_{\rm i} = Z_{\rm i}\circ\varphi$, and 
$q_{\rm i} = Q_{\rm i}\circ\varphi$. 
We remind here that $Z_{\rm i}$ and $Q_{\rm i}$ are given by \eqref{zi} and \eqref{qi}, respectively. 
Now, as claimed in \S \ref{sectreform2}, the problem under consideration is reduced to 
$$
\begin{cases}
\dt^\varphi \zeta_{\rm e} + \dx^\varphi q_{\rm e} = 0 & \mbox{in}\quad \underline{\mathcal{E}}, \\
\dt^\varphi q_{\rm e} + 2 \frac{q_{\rm e}}{h_{\rm e}}\dx^\varphi q_{\rm e}
 + \Bigl( \mathtt{g}h_{\rm e} - \frac{q_{\rm e}^2}{h_{\rm e}^2} \Bigr)\dx^\varphi \zeta_{\rm e} = 0
 & \mbox{in}\quad \underline{\mathcal{E}}, \\
\zeta_{\rm e} = \zeta_{\rm i}, \quad q_{\rm e} = q_{\rm i} & \mbox{on}\quad \partial\underline{\mathcal{E}},
\end{cases}
$$
and
$$
\dt \overline{q}_{\rm i} = - ( \langle F^{\rm I} \rangle + \langle F^{\rm II} \rangle + \langle F^{\rm III} \rangle ).
$$
%

\section{Reformulation of the equations of motion in the case of a freely floating object}\label{appreform2}
As before, we can solve the equations in the interior region \eqref{eqint}. 
Thanks of Lemma \ref{lemgrap}, we can express $Z_{\rm i}$ in terms of $x_G,z_G,\theta$, and $Z_{\rm lid}$ 
as \eqref{zi}. 
By the continuity equation in \eqref{eqint}, there exists a function $\overline{q}_{\rm i}(t)$ of $t$ 
such that $Q_{\rm i}$ is expressed as \eqref{qi}. 
Then, by the momentum equation in \eqref{eqint}, the pressure $\underline{P}_{\rm i}$ satisfies the 
boundary value problem \eqref{bvpp2}, whose solvability is guaranteed by \eqref{eqqi3}. 
Then, $\underline{P}_{\rm i}$ satisfies 
\[
\dx\underline{P}_{\rm i} = -\frac{\rho}{H_{\rm i}}( (F^{\rm I})^* + (F^{\rm II})^* + (F^{\rm III})^* ).
\]
On the other hand, by using \eqref{formula1} and integration by parts we can rewrite \eqref{Newton's law} as 
\[
\begin{pmatrix}
 \mathfrak{m} {\rm Id}_{2\times2} & 0 \\
 0 & \mathfrak{i}_0
\end{pmatrix}
\dt \begin{pmatrix}  \mathbf{U}_G \\ \omega \end{pmatrix}
= \begin{pmatrix} -\mathfrak{mg}\mathbf{e}_z \\ 0 \end{pmatrix}
 + \int_{\mathcal{I}(t)} ( \dx \underline{P}_{\rm i} )( \mathbf{T}( \mathbf{r}_G ) )^*,
\]
where we used the boundary condition $\underline{P}_{\rm i} = P_{\rm atm}$ on $\Gamma(t)$. 
Eliminating the pressure $\underline{P}_{\rm i}$ from these two equations, we have 
\[
\begin{pmatrix}
 \mathfrak{m} {\rm Id}_{2\times2} & 0 \\
 0 & \mathfrak{i}_0
\end{pmatrix}
\dt \begin{pmatrix}  \mathbf{U}_G \\ \omega \end{pmatrix}
= \begin{pmatrix} -\mathfrak{mg}\mathbf{e}_z \\ 0 \end{pmatrix}
 - \rho\int_{\mathcal{I}(t)}( (F^{\rm I})^* + (F^{\rm II})^* + (F^{\rm III})^* )
 \frac{(\mathbf{T}( \mathbf{r}_G ))^*}{H_{\rm i}}. 
\]
Here, we see that 
\begin{align*}
\int_{\mathcal{I}(t)} (F^{\rm II})^* \frac{(\mathbf{T}( \mathbf{r}_G ))^*}{H_{\rm i}}
= \int_{\mathcal{I}(t)} \frac{(\mathbf{T}( \mathbf{r}_G ))^* \otimes (\mathbf{T}( \mathbf{r}_G ))^*}{H_{\rm i}}
 \dt \begin{pmatrix} \mathbf{U}_G \\ \omega \end{pmatrix},
\end{align*}
so that 
\[
(\mathcal{M}_0 + \mathcal{M}_{\rm a}(H_{\rm i},\mathbf{r}_G))
\dt \begin{pmatrix}  \mathbf{U}_G \\ \omega \end{pmatrix}
= \begin{pmatrix} -\mathfrak{mg}\mathbf{e}_z \\ 0 \end{pmatrix}
 - \rho\int_{\mathcal{I}(t)}( (F^{\rm I})^* + (F^{\rm III})^* )
 \frac{(\mathbf{T}( \mathbf{r}_G ))^*}{H_{\rm i}},
\]
where 
\begin{equation}\label{mass}
\mathcal{M}_0 = 
\begin{pmatrix}
 \mathfrak{m} I_{2\times2} & 0 \\
 0 & \mathfrak{i}_0
\end{pmatrix}, \qquad
\mathcal{M}_{\rm a}(H_{\rm i},\mathbf{r}_G) = 
\rho\int_{\mathcal{I}(t)} \frac{(\mathbf{T}( \mathbf{r}_G ))^* \otimes (\mathbf{T}( \mathbf{r}_G ))^*}{H_{\rm i}},
\end{equation}
and 
\[
\begin{cases}
\displaystyle
(F^{\rm I})^* = \overline{q}_{\rm i}^2 \biggl( \dx \biggl( \frac{1}{H_{\rm i}} \biggr) \biggr)^*
 + 2\overline{q}_{\rm i} \begin{pmatrix} \mathbf{U}_G \\ \omega \end{pmatrix} \cdot
  \biggl( \dx \biggl( \frac{\mathbf{T}(\mathbf{r}_G)}{H_{\rm i}} \biggr) \biggr)^* \\
\qquad\;
\displaystyle
 + \begin{pmatrix} \mathbf{U}_G \\ \omega \end{pmatrix} \cdot
  \biggl( \dx \biggl( \frac{\mathbf{T}(\mathbf{r}_G) \otimes \mathbf{T}(\mathbf{r}_G)}{H_{\rm i}} \biggr) \biggr)^*
   \begin{pmatrix} \mathbf{U}_G \\ \omega \end{pmatrix}
   + \frac12 \mathtt{g} ( \dx (H_{\rm i}^2) )^*, \\
(F^{\rm III})^* = \begin{pmatrix} \mathbf{U}_G \\ \omega \end{pmatrix} \cdot 
 (M( \mathbf{r}_G, N_{\rm lid} ))^* \begin{pmatrix} \mathbf{U}_G \\ \omega \end{pmatrix}.
\end{cases}
\]

\begin{remark}
We note that the matrix $\mathcal{M}_{\rm a}(H_{\rm i},\mathbf{r}_G)$ is symmetric and nonnegative, 
so that $\mathcal{M}_0 + \mathcal{M}_{\rm a}(H_{\rm i},\mathbf{r}_G)$ is positive definite and invertible. 
Expressing the contribution of the force $F^{\rm II}$ under the form 
$\mathcal{M}_{\rm a}(H_{\rm i},\mathbf{r}_G) \dt \begin{pmatrix} \mathbf{U}_G \\ \omega \end{pmatrix}$ 
plays therefore a stabilizing effect which corresponds to the added-mass effect of paramount importance 
for the study of fluid-structure interactions (see for inctance \cite{Causin20054506,glass2014point}). 
\end{remark}

As before, we use the diffeomorphism 
$\varphi(t,\cdot) : \underline{\mathcal{E}} \to \mathcal{E}(t)$ defined by \eqref{diffeo3} to transform 
the equations in exterior region \eqref{eqext} and put 
$\zeta_{\rm e} = Z_{\rm e}\circ\varphi$, $h_{\rm e} = H_{\rm e}\circ\varphi$, 
$q_{\rm e} = Q_{\rm e}\circ\varphi$, $\zeta_{\rm i} = Z_{\rm i}\circ\varphi$, and 
$q_{\rm i} = Q_{\rm i}\circ\varphi$. 
We remind here that $Z_{\rm i}$ and $Q_{\rm i}$ are given by \eqref{zi2} and \eqref{qi}, respectively. 
Now, the problem under consideration is reduced to 
\begin{equation}\label{teqext3bis}
\begin{cases}
\dt^\varphi \zeta_{\rm e} + \dx^\varphi q_{\rm e} = 0 & \mbox{in}\quad \underline{\mathcal{E}}, \\
\dt^\varphi q_{\rm e} + 2 \frac{q_{\rm e}}{h_{\rm e}}\dx^\varphi q_{\rm e}
 + \Bigl( \mathtt{g}h_{\rm e} - \frac{q_{\rm e}^2}{h_{\rm e}^2} \Bigr)\dx^\varphi \zeta_{\rm e} = 0
 & \mbox{in}\quad \underline{\mathcal{E}}, \\
\zeta_{\rm e} = \zeta_{\rm i}, \quad q_{\rm e} = q_{\rm i} & \mbox{on}\quad \partial\underline{\mathcal{E}},
\end{cases}
\end{equation}
\begin{equation}\label{eqqi4bis}
\dt \overline{q}_{\rm i} = - ( \langle F^{\rm I} \rangle + \langle F^{\rm II} \rangle + \langle F^{\rm III} \rangle ),
\end{equation}
\begin{equation}\label{eqbodybis}
\dt \begin{pmatrix}  \mathbf{U}_G \\ \omega \end{pmatrix}
= (\mathcal{M}_0 + \mathcal{M}_{\rm a}(H_{\rm i},\mathbf{r}_G))^{-1}
 \biggl\{ \begin{pmatrix} -\mathfrak{mg}\mathbf{e}_z \\ 0 \end{pmatrix}
 - \rho\int_{\mathcal{I}(t)}( (F^{\rm I})^* + (F^{\rm III})^* )
 \frac{(\mathbf{T}( \mathbf{r}_G ))^*}{H_{\rm i}} \biggr\}.
\end{equation}
%

\bibliographystyle{alpha}
\bibliography{BIB_IL}

\newcommand{\etalchar}[1]{$^{#1}$}
\begin{thebibliography}{HKH{\etalchar{+}}09}

\bibitem[AK91]{abeyaratne1991}
Rohan Abeyaratne and James~K Knowles.
\newblock Kinetic relations and the propagation of phase boundaries in solids.
\newblock {\em Archive for rational mechanics and analysis}, 114(2):119--154,
  1991.

\bibitem[Ali89]{alinhac1989}
S.~Alinhac.
\newblock Existence d'ondes de rar{\'e}faction pour des syst{\`e}mes
  quasi-lin{\'e}aires hyperboliques multidimensionnels.
\newblock {\em Commun. in Partial Differential Equations}, 14(2):173--230,
  1989.

\bibitem[BEKER]{Bosi}
U.~Bosi, A.~Engsig-Karup, C.~Eskilsson, and Mario Ricchiuto.
\newblock A spectral/hp element depth-integrated model for nonlinear wave-body
  interaction.
\newblock {\em hal-01760366}.

\bibitem[BG98]{benzoni1998}
Sylvie Benzoni-Gavage.
\newblock Stability of multi-dimensional phase transitions in a van der waals
  fluid.
\newblock {\em Nonlinear Analysis: Theory, Methods \& Applications},
  31(1-2):243--263, 1998.

\bibitem[BG99]{benzoni1999}
Sylvie Benzoni-Gavage.
\newblock Stability of subsonic planar phase boundaries in a van der waals
  fluid.
\newblock {\em Archive for rational mechanics and analysis}, 150(1):23--55,
  1999.

\bibitem[BGS07]{benzoniserre2007}
Sylvie Benzoni-Gavage and Denis Serre.
\newblock {\em Multi-dimensional hyperbolic partial differential equations:
  First-order Systems and Applications}.
\newblock Oxford University Press on Demand, 2007.

\bibitem[Boc18]{Bocchi}
Edoardo Bocchi.
\newblock Floating structures in shallow water: local well-posedness in the
  axisymmetric case.
\newblock {\em arXiv:1802.07643}, 2018.

\bibitem[CC99]{colombo1999}
Rinaldo~M Colombo and Andrea Corli.
\newblock Continuous dependence in conservation laws with phase transitions.
\newblock {\em SIAM Journal on Mathematical Analysis}, 31(1):34--62, 1999.

\bibitem[CGN05]{Causin20054506}
P.~Causin, J.F. Gerbeau, and F.~Nobile.
\newblock Added-mass effect in the design of partitioned algorithms for
  fluid--structure problems.
\newblock {\em Computer Methods in Applied Mechanics and Engineering},
  194(42--44):4506 -- 4527, 2005.

\bibitem[Cou03]{coulombel2003}
Jean-Fran{\c{c}}ois Coulombel.
\newblock Stability of multidimensional undercompressive shock waves.
\newblock {\em Interfaces and Free Boundaries}, 5(4):367--390, 2003.

\bibitem[Fre98]{freistuhler1998}
H~Freist{\"u}hler.
\newblock Some results on the stability of non-classical shock waves.
\newblock {\em Journal of partial differntial equations}, 11:25--38, 1998.

\bibitem[Ger84]{gerlach1984two}
J{\"u}rgen Gerlach.
\newblock Two linearized models for a hyperbolic free boundary value problem.
\newblock {\em Zeitschrift f{\"u}r angewandte Mathematik und Physik ZAMP},
  35(2):181--192, 1984.

\bibitem[GMS14]{glass2014point}
Olivier Glass, Alexandre Munnier, and Franck Sueur.
\newblock Point vortex dynamics as zero-radius limit of the motion of a rigid
  body in an irrotational fluid.
\newblock {\em arXiv preprint arXiv:1402.5387}, 2014.

\bibitem[GPSMW]{Godlewski}
Edwige Godlewski, Martin Parisot, Jacques Sainte-Marie, and Fabien Wahl.
\newblock Congested shallow water type model: roof modelling in free surface
  flow.
\newblock {\em hal-01368075v2}.

\bibitem[HKH{\etalchar{+}}09]{he2009nonlinear}
Guanghua He, Masashi Kashiwagi, Changhong Hu, et~al.
\newblock Nonlinear solution for vibration of vertical elastic plate by initial
  elevation of free surface.
\newblock In {\em The Nineteenth International Offshore and Polar Engineering
  Conference}. International Society of Offshore and Polar Engineers, 2009.

\bibitem[KD17]{khakimzyanov2017numerical}
Gayaz Khakimzyanov and Denys Dutykh.
\newblock Numerical modelling of surface water wave interaction with a moving
  wall.
\newblock {\em arXiv preprint arXiv:1706.08790}, 2017.

\bibitem[KE02]{katell2002accuracy}
GuizieN Katell and Barth{\'e}lemy Eric.
\newblock Accuracy of solitary wave generation by a piston wave maker.
\newblock {\em Journal of hydraulic research}, 40(3):321--331, 2002.

\bibitem[KSS09]{korobkin2009motion}
AA~Korobkin, SV~Stukolov, and IV~Sturova.
\newblock Motion of a vertical wall fixed on springs under the action of
  surface waves.
\newblock {\em Journal of applied mechanics and technical physics},
  50(5):841--849, 2009.

\bibitem[Lan17]{Lannes2017}
David Lannes.
\newblock On the dynamics of floating structures.
\newblock {\em Annals of PDE}, 3(1):11, 2017.

\bibitem[Lax57]{Lax}
Peter~D Lax.
\newblock Hyperbolic systems of conservation laws ii.
\newblock {\em Communications on pure and applied mathematics}, 10(4):537--566,
  1957.

\bibitem[LY85]{li1985boundary}
Ta-Tsien Li and Wen-Ci Yu.
\newblock Boundary value problems for quasilinear hyperbolic systems.
\newblock {\em Duke University Mathematics ser. 5}, 1985.

\bibitem[Maj83a]{majda1}
Andrew Majda.
\newblock {\em The existence of multi-dimensional shock fronts}, volume 281.
\newblock American Mathematical Soc., 1983.

\bibitem[Maj83b]{majda2}
Andrew Majda.
\newblock {\em The stability of multi-dimensional shock fronts}, volume 275.
\newblock American Mathematical Soc., 1983.

\bibitem[Maj12]{majda3}
Andrew Majda.
\newblock {\em Compressible fluid flow and systems of conservation laws in
  several space variables}, volume~53.
\newblock Springer Science \& Business Media, 2012.

\bibitem[M{\'e}t01]{metivier2001}
Guy M{\'e}tivier.
\newblock Stability of multidimensional shocks.
\newblock {\em Advances in the theory of shock waves}, pages 25--103, 2001.

\bibitem[M{\'e}t12]{metivier2012}
Guy M{\'e}tivier.
\newblock {\em Small Viscosity and Boundary Layer Methods: Theory, Stability
  Analysis, and Applications}.
\newblock Springer Science \& Business Media, 2012.

\bibitem[Mok87]{mokrane1987}
Ahmed Mokrane.
\newblock {\em Probl{\`e}mes mixtes hyperboliques non lin{\'e}aires}.
\newblock PhD thesis, Rennes 1, 1987.

\bibitem[OBT12]{orszaghova2012paddle}
Jana Orszaghova, Alistair~GL Borthwick, and Paul~H Taylor.
\newblock From the paddle to the beach--a boussinesq shallow water numerical
  wave tank based on madsen and s{\o}rensen's equations.
\newblock {\em Journal of Computational Physics}, 231(2):328--344, 2012.

\bibitem[PT13]{petcu2013one}
Madalina Petcu and Roger Temam.
\newblock The one-dimensional shallow water equations with transparent boundary
  conditions.
\newblock {\em Mathematical Methods in the Applied Sciences},
  36(15):1979--1994, 2013.

\bibitem[RMey]{RauchMassey}
Jeffrey~B. Rauch and Frank~J. Massey.
\newblock Differentiability of solutions to hyperbolic initial boundary value
  problems.
\newblock {\em Transactions of the American Mathematical Society}, Jeffrey B.
  Rauch and Frank J. Massey:303--318, Jeffrey B. Rauch and Frank J. Massey.

\bibitem[Sch86]{schochet1986}
Steve Schochet.
\newblock The compressible euler equations in a bounded domain: existence of
  solutions and the incompressible limit.
\newblock {\em Comm. Math. Phys.}, 104(1):49--75, 1986.

\bibitem[Sle83]{slemrod1983}
Marshall Slemrod.
\newblock Admissibility criteria for propagating phase boundaries in a van der
  waals fluid.
\newblock {\em Archive for Rational Mechanics and Analysis}, 81(4):301--315,
  1983.

\bibitem[Tak95]{takeno1995free}
Shigeharu Takeno.
\newblock Free piston problem for isentropic gas dynamics.
\newblock {\em Japan journal of industrial and applied mathematics}, 12(2):163,
  1995.

\bibitem[Tru94]{truskinovsky1994}
L~Truskinovsky.
\newblock About the ``normal growth'' approximation in the dynamical theory of
  phase transitions.
\newblock {\em Continuum mechanics and thermodynamics}, 6(3):185--208, 1994.

\end{thebibliography}

\end{document}